\documentclass[11pt]{article}
\usepackage{fullpage, amsmath}
\usepackage{amsthm,amsfonts,url,amssymb}
\usepackage{mathrsfs}
\usepackage{amssymb,amscd}
\usepackage{color}
\usepackage{chemarrow}
\usepackage{pictexwd,dcpic}
\usepackage{extarrows}
\usepackage[colorlinks,linkcolor=red,anchorcolor=green,citecolor=blue]{hyperref}
\usepackage{enumitem}
\usepackage{lipsum}
\usepackage[all,cmtip]{xy}
\usepackage{leftidx}
\usepackage[version=3]{mhchem}
\usepackage{imakeidx}
\usepackage{pdfpages}
\usepackage{etoc}
\usepackage{titlesec}
\usepackage{mathdots}

\usepackage{xpatch}

\xpatchbibmacro{date+extrayear}{%
	\printtext[parens]%
}{%
	\setunit{\addperiod\space}%
	\printtext%
}{}{}

\newtheorem{theorem}{Theorem}[section]
\newtheorem{corollary}[theorem]{Corollary}
\newtheorem{lemma}[theorem]{Lemma}
\newtheorem{proposition}[theorem]{Proposition}

\newtheorem{hypothesis}[theorem]{Hypothesis}
\newtheorem{remark}[theorem]{Remark}

\newtheorem{example}[theorem]{Example}
\newtheorem{conjecture}[theorem]{Conjecture}

\newcommand{\hooklongrightarrow}{\lhook\joinrel\longrightarrow}
\newcommand{\twoheadlongrightarrow}{\relbar\joinrel\twoheadrightarrow}

\newcommand{\ra}{\rightarrow}
\newcommand{\lra}{\longrightarrow}

\newcommand{\ul}{\underline}

\newcommand{\F}{\mathbb F}
\newcommand{\bG}{\mathbb G}

\newcommand{\bM}{\mathbb M}
\newcommand{\N}{\mathbb N}
\newcommand{\Q}{\mathbb Q}

\newcommand{\Z}{\mathbb Z}

\newcommand{\bP}{\mathbb P}

\newcommand{\cL}{\mathcal L}
\newcommand{\co}{\mathcal O}
\newcommand{\cO}{\mathcal O}

\newcommand{\cR}{\mathcal R}
\newcommand{\cH}{\mathcal H}
\newcommand{\cC}{\mathcal C}
\newcommand{\cS}{\mathcal S}
\newcommand{\cD}{\mathcal D}

\newcommand{\cM}{\mathcal M}
\newcommand{\cF}{\mathcal F}

\newcommand{\cE}{\mathcal E}

\newcommand{\cZ}{\mathcal Z}

\newcommand{\fh}{\mathfrak h}
\newcommand{\fm}{\mathfrak{m}}
\newcommand{\ub}{\mathfrak b}

\newcommand{\fp}{\mathfrak p}

\newcommand{\ug}{\mathfrak g}

\newcommand{\ft}{\mathfrak t}
\newcommand{\fa}{\mathfrak a}
\newcommand{\fz}{\mathfrak z}

\newcommand{\fc}{\mathfrak c}

\newcommand{\fd}{\mathfrak d}

\newcommand{\sW}{\mathscr W}

\newcommand{\sF}{\mathscr F}

\newcommand{\EE}{\mathrm{E}}
\def\To#1{\buildrel\hbox{\tiny{$#1$}}\over\longrightarrow}
\newcommand{\plim}{\varprojlim}
\def\d{\delta}
\def\a{\alpha}
\def\onto{\twoheadrightarrow}
\def\L{\Lambda}
\newcommand{\brho}{\overline{\rho}}
\def\Fov{{\overline{F}}}
\newcommand{\A}{\mathbb A}
\DeclareMathOperator{\loc}{{\mathrm{loc}}}
\DeclareMathOperator{\tw}{{\mathrm{tw}}}
\newcommand{\bQp}{\overline{\Q}_p}
\newcommand{\defn}{\overset{\rm{def}}{=}}
\DeclareMathOperator{\Ker}{{\mathrm{Ker}}}
\def\s{\sigma}

\DeclareMathOperator{\la}{\mathrm la}

\DeclareMathOperator{\gl}{\mathfrak gl}

\DeclareMathOperator{\GL}{\mathrm GL}
\DeclareMathOperator{\gr}{\mathrm gr}

\DeclareMathOperator{\Fil}{\mathrm Fil}
\DeclareMathOperator{\Res}{\mathrm Res}

\DeclareMathOperator{\Gal}{\mathrm Gal}
\DeclareMathOperator{\Hom}{\mathrm Hom}

\DeclareMathOperator{\End}{\mathrm End}

\DeclareMathOperator{\rig}{\mathrm rig}
\DeclareMathOperator{\an}{\mathrm an}
\DeclareMathOperator{\Spec}{\mathrm Spec}

\DeclareMathOperator{\dR}{\mathrm dR}

\DeclareMathOperator{\Ind}{\mathrm Ind}
\DeclareMathOperator{\unr}{\mathrm unr}

\DeclareMathOperator{\Fer}{\mathrm Ker}
\DeclareMathOperator{\pr}{\mathrm pr}

\DeclareMathOperator{\ord}{\mathrm ord}

\DeclareMathOperator{\Ext}{\mathrm Ext}

\DeclareMathOperator{\Spf}{\mathrm Spf}
\DeclareMathOperator{\Ima}{\mathrm Im}
\DeclareMathOperator{\SL}{\mathrm SL}
\DeclareMathOperator{\lalg}{\mathrm lalg}

\DeclareMathOperator{\id}{\mathrm id}

\DeclareMathOperator{\dett}{\mathrm det}
\DeclareMathOperator{\alg}{\mathrm alg}

\DeclareMathOperator{\soc}{\mathrm soc}

\DeclareMathOperator{\sss}{\mathrm ss}

\DeclareMathOperator{\st}{\mathrm st}
\DeclareMathOperator{\St}{\mathrm St}

\DeclareMathOperator{\WD}{\mathrm WD}

\DeclareMathOperator{\rk}{\mathrm rk}

\DeclareMathOperator{\diag}{\mathrm diag}

\DeclareMathOperator{\cont}{\mathrm cont}

\DeclareMathOperator{\val}{\mathrm val}
\DeclareMathOperator{\PS}{\mathrm PS} \DeclareMathOperator{\SSS}{\mathrm SS}
\DeclareMathOperator{\fss}{\mathrm fs}
\DeclareMathOperator{\Mod}{\mathrm Mod}
\DeclareMathOperator{\RHom}{\mathrm RHom}
\DeclareMathOperator{\Coker}{\mathrm Coker}
\DeclareMathOperator{\sm}{\mathrm sm}
\DeclareMathOperator{\tri}{\mathrm tri}
\DeclareMathOperator{\cind}{c-\mathrm{ind}}
\DeclareMathOperator{\pcr}{\mathrm pcr}
\begin{document}
	\title{Towards the wall-crossing of locally $\Q_p$-analytic representations of $\GL_n(K)$ for a $p$-adic field $K$} 
\author{Yiwen Ding\\ \\ (with an appendix by Yiwen Ding, Yongquan Hu and Haoran Wang)}
\date{}
\maketitle
\begin{abstract}Let $K$ be a finite extension of $\Q_p$. We study the locally $\Q_p$-analytic representations $\pi$ of $\GL_n(K)$ of integral weights that appear in spaces of $p$-adic automorphic representations.  We conjecture that the translation of $\pi$ to the singular block has an internal structure which is compatible with certain algebraic representations of $\GL_n$, analogously to the mod $p$ local-global compatibility conjecture of Breuil-Herzig-Hu-Morra-Schraen. We next make some conjectures and speculations on the wall-crossings of $\pi$. In particular, when $\pi$ is associated to a two dimensional de Rham Galois representation,  we make conjectures and speculations on the relation between the Hodge filtrations of $\rho$  and the wall-crossings of $\pi$, which have a flavour of  the Breuil-Strauch conjecture. We  collect some results towards the conjectures and speculations. 	
\end{abstract}
\tableofcontents
\section{Introduction}	
Let $K$ be a finite extension of $\Q_p$. The $p$-adic Langlands program aims at understanding the correspondence between certain $p$-adic representations of $\GL_n(K)$ and the $n$-dimensional $p$-adic representations of $\Gal_K$, the absolute Galois group of $K$. For the $p$-adic representations of $\GL_n(K)$, one approach  is to consider unitary Banach representations. These representations possess an integral structure, presumbed to be compatible with the integral structure on the Galois side, thereby  closely aligning  with the mod $p$ Langlands program.  Another approach focuses on the locally $\Q_p$-analytic representations of $\GL_n(K)$. This aspect has the advantage of making the $p$-adic Hodge theoretic property of the Galois representations more  transparent.  Note that one can pass from unitary Banach representations to locally $\Q_p$-analytic representations by taking locally $\Q_p$-analytic vectors. 

A correspondence for both aspects was established for $\GL_2(\Q_p)$ through the work of Berger, Breuil, Colmez, Pa{\v{s}}k{\=u}nas et al. 
Subsequently, significant progress has been made for both aspects concerning $\GL_n(K)$. In our discussion, we highlight some of them that relate to the results in this paper. Despite the correspondence being somewhat mysterious, in \cite{CEGGPS1}, using Taylor-Wiles-Kisin patching method, one can associate to a $\Gal_K$-representation $\rho$ a unitary Banach representation $\widehat{\pi}(\rho)$.  We denote by $\pi(\rho):=\widehat{\pi}(\rho)^{\Q_p-\an}$ its locally $\Q_p$-analytic vectors
of $\widehat{\pi}(\rho)$. The construction depends on a lot of auxiliary global data. 
The study of the (global-local) compatibility  between $\widehat{\pi}(\rho)$ (or $\pi(\rho)$) and the local $\rho$ is a central problem in $p$-adic Langlands program. 

For $\widehat{\pi}(\rho)$, in the recent breakthrough work \cite{BHHMS2}, it is conjectured (and proved in some cases of $\GL_2(K)$) that (the mod $p$ reduction of) $\widehat{\pi}(\rho)$ has a symmetric internal structure related to certain algebraic representation of $\GL_n$. On the locally analytic setting, a key starting point is that  when $\rho$ is de Rham, the locally algebraic vectors of $\pi(\rho)$ can be characterized using the Weil-Deigne representation associated to $\rho$ via the classical local Langlands correspondence. Passing from $\rho$ to its Weil-Deligne representation results in a loss of information on Hodge filtration. Therefore, a major focus in the locally analytic aspect of $p$-adic Langlands program is to recover the information on Hodge filtration in the larger $\pi(\rho)$. Towards this, in potentially crystalline case, there were results on Breuil's locally analytis socle conjecture (cf. \cite{BHS3}, \cite{BD3}, \cite{Wu21}), and furthermore on the finite slope part of $\pi(\rho)$ (\cite{BH2}), all of which are related to the relative position of Hodge filtrations and Weil filtrations.  In semi-stable case, there were results on  $\cL$-invariants, those parametrizing the Hodge-filtration (\cite{Ding7} \cite{BD1} \cite{He1}). And finally in \cite{Br16} (see also \cite{BD2}), Breuil made a conjecture on how the extension group in some subrepresentations of $\pi(\rho)$ may be related with the Hodge filtrations.   We remark that all these results are about a small piece of $\pi(\rho)$, which is already very complicated and seems to exhibit a nature quite  different from the Banach setting. 

In the paper, we propose an approach to study the entire $\pi(\rho)$. Note first by \cite[Thm.~1.5]{DPS1}, $\pi(\rho)$ admits an infinitesimal character $\chi$. It appears that, although the explicit structure  $\pi(\rho)$ is dauntingly complicated, its translation into the singular block, denoted by $\pi(\Delta)$,  should have a much simpler and cleaner structure (see Conjecture \ref{introConj1} (1) for the precise definition, and see \cite{JLS} for a general formalism of translations of locally analytic representations). A rough reason is that a large amount of the irreducible constituents of $\pi(\rho)$ are  annihilated under the translation. We conjecture that the information lost in this step is precisely the Hodge filtration information on $\rho$. Subsequently, we retranslate $\pi(\Delta)$ back to the weight $\chi$-block and explore how the interplay between the resulting representation and $\pi(\rho)$ might reveal the information of Hodge filtration. 

We introduce some notation before giving more details. Let $E$ be a sufficiently large finite extension of $\Q_p$, which will be the coefficient field of the representations. Let $\Sigma_K:=\{\sigma: K \hookrightarrow E\}$. Let $D:=D_{\rig}(\rho)$ be the associated $(\varphi, \Gamma)$-module over the Robba ring $\cR_{K,E}$ (associated to $K$ with $E$-coefficients). We will frequently write $\pi(D)$ for $\pi(\rho)$ (to emphasize the relation between locally analytic representations and $(\varphi, \Gamma)$-modules over Robba rings). Assume $\rho$ has integral Sen weights $\textbf{h}=(h_{1,\sigma}\geq \cdots \geq h_{n,\sigma})_{\sigma\in \Sigma_K}$. Let  $\theta_K:=(n-1, \cdots, 0)_{\sigma\in \Sigma_K}$, and $\lambda:=\textbf{h}-\theta_K$. Let $\ug_K:=\gl_n(K) \otimes_{\Q_p} E$, and $\cZ_K$ be the centre of $\text{U}(\ug_K)$. For an integral weight $\mu$, we denote by $\chi_{\mu}$ the infinitesimal character of $\cZ_K$ acting on $\text{U}(\ug_K) \otimes_{\text{U}(\ub_K)} \mu$ where $\ub$ is the Lie algebra of upper triangular Borel subgroup $B$, and $\ub_K:=\ub(K) \otimes_{\Q_p} E$.  If $\mu$ is moreover dominant (with respect to $B$), denote by $L(\mu)$ the algebraic representation of $\Res^K_{\Q_p} \GL_n$ of highest weight $\mu$. By \cite[Thm.~1.5]{DPS1}, $\cZ_K$ acts on $\pi(\rho)$ via  $\chi_{\lambda}$.  

By Fontaine's classification of $B_{\dR}$-representations \cite{Fo04}, there exists a unique $(\varphi, \Gamma)$-module $\Delta$ of constant weights $0$ such that $D[\frac{1}{t}]=\Delta[\frac{1}{t}]$ (cf. Lemma \ref{pDE0}). Throughout the paper, we assume $\Delta$ has distinct irreducible constituents. When $D$ is de Rham, $\Delta$ is the so-called $p$-adic differential equation associated to $D$ (\cite{Ber08a}). In this case, passing from $D$ to $\Delta$, one loses exactly the information on Hodge filtrations. Another extreme example is that when $\dim_E D_{\dR}(D)_{\sigma}=1$ for all $\sigma$. In this case, one can recover $D$ from $\Delta$:  $D$ is the \textit{unique} $(\varphi, \Gamma)$-submodule of $\Delta[\frac{1}{t}]$ of Sen weights $\textbf{h}$. 

\begin{conjecture}[Singular skeleton]\label{introConj1}
	(1) The locally $\Q_p$-analytic representation  $T_{\lambda}^{-\theta_K} \pi(D)\cong (\pi(D) \otimes_E L(\textbf{h})^{\vee})[\cZ_K=\chi_{-\theta_K}]$ of $\GL_n(K)$ depends only on $\Delta$, which we denote by $\pi(\Delta)$.
	
	(2) The representation $\pi(\Delta)$ has finite length and is weakly compatible with $\Delta$.
\end{conjecture}

The notion of the compatibility in (2) is borrowed from \cite{BHHMS2}. We give a quick explanation. If $\Delta$ is irreducible, this just means $\pi(\Delta)$ is also irreducible. When $\Delta$ is reducible, a filtration of submodules of $\Delta$ (with irreducible graded pieces) corresponds to a standard parabolic subgroup $P$ of $\GL_n$ together with a Zariski closed subgroup $\widetilde{P}\subset P$ (containing the standard Levi subgroup $M_P$ of $P$) encoding the relation between the filtration with the canonical socle filtration. Then one can look the restriction of the fundamental algebraic representation $L^{\otimes}$ of $\GL_n(K)$ to $\widetilde{P}(K)$, and define the notion of  good subrepresentations of $L^{\otimes}|_{\widetilde{P}}$: those whose restriction to $Z_{M_P}(\Q_p)\hookrightarrow Z_{M_P}(K)$ is a direct sum of certain isotypic components of $L^{\otimes}|_{Z_{M_P}(\Q_p)}$, where $Z_{M_P}$ is the centre of $M_P$. 
The compatible property means there is a \textit{nice} one-to-one correspondence between the subrepresentations  of $\pi(\Delta)$ and the good subrepresentations of  $L^{\otimes}|_{\widetilde{P}}$. In particular, the irreducible constituents of $\pi(\Delta)$ are in one-to-one correspondence with the isotypic components of $L^{\otimes}|_{Z_{M_P}}$. For example, if $\Delta$ is generic potentially crystabelline hence semi-simple, then $\widetilde{P}=M_P$ and $\pi(\Delta)$ would be also semi-simple. In contrary, $\pi(D)$ itself usually has a lot of internal extensions. If $n=2$ and $\Delta$ is generic crystabelline, $\pi(\Delta)$ would be the direct sum of two locally $\Q_p$-analytic principal series and $(d_K-1)$-number of supersingular representations.  
\begin{theorem}
	Conjecture \ref{introConj1} holds for $\GL_2(\Q_p)$.
\end{theorem}For $\GL_2(\Q_p)$, by \cite{CEGGPS2}, $\widehat{\pi}(\rho)$ coincides with the Banach representation associated to $\rho$ via the $p$-adic Langlands correspondence (\cite{Colm10a}). By \cite{Ding14}, $\pi(\Delta)$ is just the locally analytic representation associated to $\Delta$ via the locally analytic $p$-adic Langlands correspondence (cf. \cite{Colm16}), and all the properties in (2) are directly derived from the known facts in the correspondence (cf. \cite[\S.~0.3]{Colm16}, \cite[Thm.~0.3 (i)]{Colm18}). The argument in \cite{Ding14} crucially uses Colmez's construction of $\pi(\rho)$. However, when $\rho$ appears in the completed cohomology group of modular curves, it seems plausible to use Pan's geometric approach \cite{Pan4} to prove (1). For (2), the finite length property (of $\pi(\Delta)$) can be proved using global method. In fact, by \cite[Thm.~1.5]{DPS} (or Corollary \ref{C-CMdual}), the canonical dimension of $\pi(\rho)$ is $1$,\footnote{The canonical dimension is usually defined for duals of the representations, considering them as modules over the distribution algebra. However, we adopt the same terminology for a representation itself.}
so is $\pi(\Delta)$ (cf. Corollary \ref{transpure}). As any irreducible constituent of $\pi(\Delta)$ can not be locally algebraic hence has dimension at least $1$, $\pi(\Delta)$ has finite length. For $\pi(D)$ itself, it was known $\pi(D)$ has finite length (cf. \cite{Colm18}), but I don't know a proof without using the $p$-adic local Langlands correspondence (as $\pi(D)$ can have zero dimensional subquotients, in contrast to $\pi(\Delta)$). 

There is a natural equivalence of categories between $p$-adic differential equations and Weil-Deligne representations (\cite[Thm.~A]{Ber08a}, \cite[Prop.~4.1]{BS07}). Consequently, when $D$ is de Rham, the correspondence $\Delta\leftrightarrow \pi(\Delta)$ may be viewed as a singular weight version of the classical local Langlands correspondence. Denoting by $\pi_{\infty}(\Delta)$  the smooth representation of $\GL_n(K)$ corresponding to $\Delta$, we may recover $\pi_{\infty}(\Delta)$ from $\pi(\Delta)$ in the following way (where $(-)_{\cZ_K}$ denotes the $\cZ_K$-coinvariant quotient).
\begin{conjecture}
There exists $r\geq 1$ such that $\big((\pi(\Delta) \otimes_E L(\theta_K))_{\cZ_K}\big)^{\sm}  \cong \pi_{\infty}(\Delta)^{\oplus r}$.
\end{conjecture}
The conjecture holds for $\GL_2(\Q_p)$ with $r=2$. In general, one should have $r\geq n!$. When $n=2$, one may even expect $r=2d_K$. Under certain hypotheses, we show  it is the case when $K=\Q_{p^2}$ and $\rho$ is generic  crystabelline (cf. (\ref{diag2})). 

When $\Delta$ is trianguline, mimicking \cite{BH}, we can associate to $\Delta$ an explicit locally $\Q_p$-analytic representation $\pi(\Delta)^{\fss}$ of $\GL_n(K)$. Remark that the compatibility between  $\pi(\Delta)^{\fss}$ and $\Delta$ is exactly the same as the compatibility between $\widehat{\pi}(\rho')^{\ord}$ and $\rho'$ for ordinary $\Gal_{K'}$-representations $\rho'$, studied in \cite{BH}. See Example \ref{Efs} for some examples.   	Conjecture \ref{introConj1} (2) implies that $\pi(\Delta)^{\fss}$ is a subrepresentation of $\pi(\Delta)$.  Under some assumptions, we prove it is indeed the case (cf. Theorem \ref{thmlgfs}, Theorem \ref{lgSt}):
\begin{theorem}[Local-global compatibility]
Suppose $D$ appears on the patched eigenvariety of \cite{BHS1}.

(1) Let $D$ be trianguline and generic.\footnote{The trianguline $D$ is called generic, if the irreducible constituents $\cR_{K,E}(\phi_i)$ of $\Delta$ are distinct and $\phi_i\phi_j^{-1} \neq |\cdot |_K^{\pm 1}$.} Suppose  all the refinements of $D$ appear on the patched eigenvariety, then $\pi(\Delta)^{\fss}\hookrightarrow (\pi(D) \otimes_E L(\textbf{h})^{\vee})[\cZ_K=\chi_{-\theta_K}]$.

(2) Let $D$ be semi-stable non-crystalline (up to twist) with $N^{n-1}\neq 0$, and suppose $D$ is  non-critical. Then $\pi(\Delta)^{\fss}\hookrightarrow (\pi(D) \otimes_E L(\textbf{h})^{\vee})[\cZ_K=\chi_{-\theta_K}]$.
\end{theorem}
For trianguline and generic $D$, we do know that all the refinements of $D$ appear on the so-called trianguline variety, the Galois avatar of the patched eigenvariety. When $D$ is crystabelline and generic, by \cite{BHS2} \cite{BHS3}, $D$ appearing on the patched eigenvariety implies all of its refinements appear. In this case, $\pi(\Delta)^{\fss}$ is simply a direct sum of $n!$($=\#S_n$)-locally analytic principal series, corresponding to the $n!$-distinct refinements of $D$. In contrast, if we consider the maximal subrepresentation of $\pi(D)$ whose irreducible constituents are subquotients of locally analytic principal series, then its structure would be far more complicated (see \cite{BH}, \cite{Ding7}, Theorem \ref{Tsurplus2}).
In fact, as the information on the Hodge filtration of $D$ is lost in $\Delta$, the extra socle phenomenon (cf. \cite{Br13I}) and the $\cL$-invariant problem  should  all disappear in $\pi(\Delta)$.

Now we translate $\pi(\Delta)$ back to the $\chi_{\lambda}$-block. By the result of Bernstein-Gelfand on projective generator (\cite{BG}) (cf. Proposition \ref{propsemi}), the natural injection ($\{-\}$ denoting the generalized eigenspace)
\begin{equation*}
\pi(\Delta)\cong (\pi(D)\otimes_E L(\textbf{h})^{\vee})[\cZ_K=\chi_{-\theta_K}] \hooklongrightarrow 	(\pi(D)\otimes_E L(\textbf{h})^{\vee})\{\cZ_K=\chi_{-\theta_K}\} 
\end{equation*}
is an isomorphism. In particular, $\pi(\Delta)$ is a direct summand of $\pi(D) \otimes_E L(\textbf{h})^{\vee}$. 	Let 
\begin{equation*}
\Theta (\pi(D)):=\big((\pi(D) \otimes_E L(\textbf{h})^{\vee})\{\cZ_K=\chi_{-\theta_K}\}\otimes_E L(\textbf{h})\big)\{\cZ_K=\chi_{\lambda}\}\cong (\pi(\Delta) \otimes_E L(\textbf{h})\big)\{\cZ_K=\chi_{\lambda}\},
\end{equation*}
which is hence a direct summand of $\pi(D) \otimes_E L(\textbf{h})^{\vee} \otimes_E L(\textbf{h})$. The representation $\Theta(\pi(D))$ is usually referred to as a \textit{wall-crossing} of $\pi(D)$. Remark that Conjecture \ref{introConj1} (1) implies $\Theta(\pi(D))$ depends only on $\Delta$ and $\lambda$. The diagonal map $E \ra L(\textbf{h})^{\vee} \otimes_E  L(\textbf{h})\cong \End_E(L(\textbf{h}))$ (resp. the trace map $L(\textbf{h})^{\vee} \otimes_E L(\textbf{h}) \ra E$) induces a $\GL_n(K)$-equivariant map
\begin{equation*}
\iota: \pi(D) \lra \Theta(\pi(D)) \text{ (resp. } \kappa: \Theta(\pi(D))\lra \pi(D)\text{)}.	\end{equation*}
As $\cZ_K$ acts on $\pi(D)$ via $\chi_{\lambda}$, the map $\kappa$ factors through $\pi(\Delta, \lambda):=\Theta(\pi(D))_{\chi_{\lambda}} \ra \pi(D)$ where $(-)_{\chi_{\lambda}}$ denotes the $(\cZ_K=\chi_{\lambda})$-coinvariant quotient. 
We denote by $\pi_0(\Delta, \lambda)\subset \Theta \pi(D)$ be the image of $\iota$. It is clear that $\cZ_K$ acts on both $\pi(\Delta, \lambda)$ and $\pi_0(\Delta, \lambda)$ via the character $\chi_{\lambda}$. 

\begin{conjecture}\label{introconj3}

(1) $\pi_0(\Delta, \lambda)$ and $\pi(\Delta, \lambda)$ depend only on $\Delta$ and $\lambda$. 

(2) The map $\kappa: \pi(\Delta, \lambda) \ra \pi(D)$ is surjective.

(3) Assume $\textbf{h}$ is strictly dominant. The followings are equivalent:

(i) $\dim_E D_{\dR}(\rho)_{\sigma}=1$ for all $\sigma\in \Sigma_K$, 

(ii) $\pi(\Delta, \lambda) \xrightarrow[\sim]{\kappa} \pi(D)$,

(iii) $\pi(D) \xrightarrow[\sim]{\iota} \pi_0(\Delta, \lambda)$.	\end{conjecture}
The $\pi(\Delta, \lambda)$-part in (1) is clearly a direct consequence of Conjecture \ref{introConj1} (1). (2) is equivalent to the surjectivity of $\Theta(\pi(D))\ra \pi(D)$. Note while it is obvious  that $\pi(D) \otimes_E L(\textbf{h})^{\vee} \otimes_E L(\textbf{h}) \twoheadrightarrow \pi(D)$, the surjectivity is not straightforward when restricted to the direct facgtor $\Theta(\pi(D))$. By the conjecture, we have 
\begin{equation}\label{eqinto1}
\pi(\Delta, \lambda) {\buildrel \kappa \over \twoheadlongrightarrow} \pi(D) {\buildrel \iota \over \twoheadlongrightarrow} \pi_0(\Delta, \lambda).
\end{equation}
As $\pi(\Delta, \lambda)$ and $\pi_0(\Delta, \lambda)$ are both conjectured to  depend only on $\Delta$ and $\lambda$ while $\pi(D)$ should carry the full information of $D$, the above two quotient maps should carry  the information on the Hodge filtration of $D$. Conjecture \ref{introconj3} (1) (2) imply that $\pi(\Delta, \lambda)$ is a \textit{universal} object: for any (\'etale) $(\varphi, \Gamma)$-module $D'$ of Sen weights $\textbf{h}$ such that $D'[1/t]\cong \Delta[1/t]$,  we should have $\pi(\Delta, \lambda) \twoheadrightarrow \pi(D')$.  The statement in (3) is compatible with the Galois phenomenon: (1) is equivalent to that  $D$ is determined by $\Delta$ and $\lambda$. Remark that (3) is closely related to the ``local avatar" of the Fontaine-Mazur conjecture: $\rho$ is de Rham of distinct Hodge-Tate weights if and only if $\pi(D)$ has non-zero locally algebraic vectors (e.g. see the discussion below Theorem \ref{thminto3}).  Finally,  for general $D$, one may  expect that $\pi(\Delta, \lambda)$ and $\pi(D)$ always have the same irreducible constituents (with possibly different multiplicities).

\begin{theorem}\label{thminto3}
(1) Conjecture \ref{introconj3} holds for $\GL_2(\Q_p)$. 

(2) Under mild hypotheses, Conjecture \ref{introconj3} (2) and the part ``(i) $\Rightarrow$ (ii) $\Rightarrow$ (iii)" of (3) hold for $\GL_2(\Q_{p^2})$.
\end{theorem} 
The $\GL_2(\Q_p)$-case follows from the results in \cite{Ding14} and \cite{Colm18}. By a Lie calculation for $\gl_2$ in the appendix,  $\Fer(\iota)$, $\Fer(\kappa|_{\pi(\Delta, \lambda)})$ and $\Coker(\kappa)$  are all generated by $\text{U}(\gl_{2,\sigma})$-finite vectors, for $\sigma\in \Sigma_K$ (where $\gl_{2,K}\cong \prod_{\sigma\in \Sigma_K} \gl_{2,\sigma}$). When $K$ is unramified over $\Q_p$, by Appendix \ref{AppCM}, the dual $\pi(D)^*$ is Cohen-Macaulay of dimension $d_K$ under mild hypotheses. When $d_K=2$, using results in \cite{DPS} and the essential self-duality of $\pi(D)^*$ (cf. Appendix \ref{AppCM}), one can show that any $\text{U}(\gl_{2,\sigma})$-finite sub of $\pi(D)^*$ has  dimension at most one, which has to be zero as $\pi(D)^*$ is Cohen-Macaulay hence pure.  Together with the aforementioned Lie result, we see $\Coker[\pi(\Delta, \lambda) \ra \pi(D)]=0$ (see Theorem \ref{Tinj}).  If $\dim_E D_{\dR}(D)_{\sigma}=1$, then $\pi(D)$ would not have non-zero $\text{U}(\ug_{\sigma})$-finite vectors (cf. Proposition  \ref{pdR}). Using certain duality, one can deduce $\pi(\Delta, \lambda)$ does not have non-zero $\text{U}(\ug_{\sigma})$-finite vectors neither. We then deduce, again using the Lie results,  that (ii) and (iii) in  (3) hold (see Theorem \ref{TndR}). Note that the same arguments also (re)prove Conjecture \ref{introconj3} (2) and (i) $\Rightarrow$ (ii) $\Rightarrow$ (iii) in  (3) for $\GL_2(\Q_p)$ without using $(\varphi, \Gamma)$-modules. 

We discuss how the maps (\ref{eqinto1}) may reveal the Hodge filtration.  We first look at the $\GL_2(\Q_p)$-case. Suppose $D$ is de Rham, by \cite{Ding14} \cite{Colm18}, $\iota$ and $\kappa$ give non-split exact sequences:
\begin{equation}\label{IE0}
0 \lra \pi_{\infty}(\Delta) \otimes_E L(\lambda)  \lra \pi(\Delta, \lambda)  \xlongrightarrow{\kappa} \pi(D) \lra 0,
\end{equation}
\begin{equation}\label{IE1}
0 \lra \pi_{\infty}(\Delta) \otimes_E L(\lambda)   \lra \pi(D) \xlongrightarrow{\iota} \pi_0(\Delta, \lambda) \lra 0.
\end{equation}
By \cite{BD2} \cite{Ding12}, we actually have natural isomorphisms
\begin{equation}\label{IE2}
\Hom_{\GL_2(\Q_p)}( \pi_{\infty}(\Delta) \otimes_E L(\lambda) ,\pi(\Delta, \lambda))\cong \Hom_{(\varphi, \Gamma)}(\Delta, \cR_E/t^{h_1-h_2}),
\end{equation}
\begin{equation}\label{IE3}
\Ext^1_{\GL_2(\Q_p)}(\pi_0(\Delta, \lambda), \pi_{\infty}(\Delta) \otimes_E L(\lambda) )\cong \Ext^1_{(\varphi, \Gamma)}(\cR_E/t^{h_1-h_2},\Delta).
\end{equation}
The isomorphism class $[\pi(D)]$ corresponds respectively to the class $[t^{-h_2} D]$ (here we use the isomorphism class of the kernel to denote an isomorphism $\Hom$-class) via (\ref{IE2}) and $[t^{-h_1} D]$ via (\ref{IE3}). By \cite[Lem.~5.1.1, Prop.~5.1.2]{BD2}, there is a natural isomorphism of $E$-vector spaces 
\begin{equation}
\label{IE4}\Ext^1_{(\varphi, \Gamma)}(\cR_E/t^{h_1-h_2},\Delta) \xlongrightarrow{\sim} D_{\dR}(\Delta),
\end{equation} sending the $E$-line $E[D']$ to the one dimensional $\Fil^{h_1-h_2} D_{\dR}(D')\subset D_{\dR}(D')\cong D_{\dR}(\Delta)$.
By an easy variation of the arguments of \textit{loc. cit.}, there is also a natural isomorphism between $\Hom_{(\varphi, \Gamma)}(\Delta, \cR_E/t^{h_1-h_2})$ and $D_{\dR}(\Delta)$ satisfying similar properties. Throughout the paper, rather than using the terms ``Hodge filtration" directly, we will primarily use the terms of $(\varphi, \Gamma)$-modules. It is recommended for the reader to keep in mind the isomorphism in (\ref{IE4}), particularly noting their close relation. 


We move to the case of $\GL_2(\Q_{p^2})$. We first look at the Galois side. Let $\Sigma_K=\{\sigma, \tau\}$. Using Fontaine's classification of $B_{\dR}$-representation of $\Gal_K$, one can show there is a unique rank two $(\varphi, \Gamma)$-module $D_{\sigma}$ over $\cR_{K,E}$ of constant Sen $\tau$-weight $0$, and Sen $\sigma$-weights $\textbf{h}_{\sigma}$ such that $D_{\sigma}[\frac{1}{t}]\cong D[\frac{1}{t}]$. We write $D_{\emptyset}:=\Delta$ and $D_{\Sigma_K}:=D$. 
We assume $\textbf{h}$ is  strictly dominant (i.e. $D$ has distinct Sen weights).
\begin{table}[!h]\centering
\begin{tabular}{|c|c|c|c|}
	\hline
	$\Delta=D_{\emptyset}$ & $D_{\sigma}$ & $D_{\tau}$ & $D=D_{\Sigma_K}$ \\
	\hline
	$(0,0)_{\sigma}$, $(0,0)_{\tau}$ & $(h_{1,\sigma}, h_{2,\sigma})_{\sigma}$, $(0,0)_{\tau}$ & $(0,0)_{\sigma}$, $(h_{1,\tau}, h_{2,\tau})_{\tau}$ &$(h_{1,\sigma}, h_{2,\sigma})_{\sigma}$, $(h_{1,\tau}, h_{2,\tau})_{\tau}$ \\
	\hline
\end{tabular} 
\end{table}
Remark that when $D$ is de Rham non-crystabelline, it is not difficult to see $D$ is uniquely determined by $D_{\sigma}$ and $D_{\tau}$. However, it  does not hold  when $D$ is crystabelline; in such cases,  there is an extra parameter (that we call a Hodge parameter) which parametrizes the relative position of the Hodge filtrations for different embeddings. 

For the $\GL_2(\Q_{p^2})$-side, roughly speaking, the exact sequences (\ref{IE0}) (\ref{IE1}) will expand to two squares. In fact,  using the wall-crossing functors for the embeddings, we can construct two (commutative) squares consisting of exact sequences, denoted respectively by $\boxdot^+(\pi(D))$, $\boxdot^-(\pi(D))$:
\begin{equation}\label{diag000}
\begindc{\commdiag}[250]
\obj(0,0)[a1]{$\pi(D_{\tau}, \tau,\lambda)$}
\obj(3,0)[b1]{$\pi(D_{\tau}, \lambda)$}
\obj(6,0)[c1]{$\pi(D)$}
\obj(9,0)[a2]{$\pi_0(\Delta, \tau, \lambda)$}
\obj(12,0)[b2]{$\pi_0(D_{\sigma}, \lambda)$}
\obj(15,0)[c2]{$\pi_0(\Delta, \lambda)$}
\obj(0,2)[d1]{$\pi(\Delta, \tau, \lambda)$}
\obj(3,2)[e1]{$\pi(\Delta, \lambda)$}
\obj(6,2)[f1]{$\pi(D_{\sigma},\lambda)$}
\obj(9,2)[d2]{$\pi_0(D_{\tau}, \lambda)$}
\obj(12,2)[e2]{$\pi(D)$}
\obj(15,2)[f2]{$\pi_0(D_{\tau},\lambda)$}
\obj(0,4)[g1]{$\pi(\Delta, \emptyset, \lambda)$}
\obj(3,4)[h1]{$\pi(\Delta,\sigma, \lambda)$}
\obj(6,4)[i1]{$\pi(D_{\sigma},\sigma, \lambda)$}
\obj(9,4)[g2]{$\pi_0(\Delta, \emptyset, \lambda)$}
\obj(12,4)[h2]{$\pi_0(D_{\sigma}, \sigma, \lambda)$}
\obj(15,4)[i2]{$\pi_0(\Delta, \sigma, \lambda)$}
\mor{a1}{b1}{}
\mor{b1}{c1}{}
\mor{d1}{a1}{}
\mor{e1}{b1}{}
\mor{f1}{c1}{}
\mor{d1}{e1}{}
\mor{e1}{f1}{}
\mor{g1}{d1}{}
\mor{h1}{e1}{}
\mor{i1}{f1}{}
\mor{g1}{h1}{}
\mor{h1}{i1}{}
\mor{a2}{b2}{}
\mor{b2}{c2}{}
\mor{d2}{a2}{}
\mor{e2}{b2}{}
\mor{f2}{c2}{}
\mor{d2}{e2}{}
\mor{e2}{f2}{}
\mor{g2}{d2}{}
\mor{h2}{e2}{}
\mor{i2}{f2}{}
\mor{g2}{h2}{}
\mor{h2}{i2}{}
\enddc.
\end{equation}	
Indeed, for each $\sigma_1\in \Sigma_K$, $\pi(D_{\sigma_1},\lambda)$ and $\pi_0(D_{\sigma_1},\lambda)$ are constructed exactly in the same way as $\pi(\Delta, \lambda)$ and $\pi_0(\Delta,\lambda)$, by replacing the algebraic representation $L(\textbf{h})$ by its $\sigma_1$-factor $L_{\sigma_1}(\textbf{h}_{\sigma_1})$. And all the other terms are the respective kernels. The  ``$D_{\sigma}$" (same for $D$, $\Delta$, $D_{\tau}$) in the notation suggests the corresponding representation conjecturally depends only on $D_{\sigma}$. The second label ``$\sigma$" (and similarly for ``$\tau$", ``$\emptyset$"), within  the parentheses of some representations, signifies  that the  representation is locally $\sigma$-analytic, up to twist by an algebraic representation (where ``locally $\emptyset$-analytic'' $=$ ``smooth"). A representation without such a label, like $\pi(\Delta, \lambda)$, is considered just locally $\Q_p$-analytic. For general $K$, a similar construction yields two $d_K$-dimensional hypercubes.  For further details, we refer to \S~\ref{Shfh} (note that some notation may differ). 

The following conjecture describes how the squares are related with the Hodge-filtrations of $D$. We only discuss the horizontal sequences, the vertical sequences being similar. We refer to Conjecture \ref{conjGL2} in the context for a version for general $\GL_2(K)$.

\begin{conjecture}[Hodge filtration hypercubes] \label{ICBS}Suppose $\Delta$ is de Rham.

(1) (Going from $\Delta$ to $D_{\sigma}$) Let $r=\begin{cases}
	1 & \text{$\Delta$ is indecomposable} \\
	2 & \text{$\Delta$ is decomposable}
\end{cases}$. 
There are natural isomorphisms (where $t_{\sigma}\in \cR_{K,E}$ is the $\sigma$-factor of $t$ defined  in \cite[Notation~6.27]{KPX}):
\begin{equation*}
	\Hom_{\GL_2(K)}\big(\pi(\Delta, \emptyset, \lambda),\pi(\Delta, \sigma, \lambda)\big) \xlongrightarrow{\sim} 
	\Hom_{(\varphi, \Gamma)}\big(\Delta, \cR_{K,E}/t_{\sigma}^{h_{1,\sigma}-h_{2,\sigma}} \big),
\end{equation*}
\begin{equation*}
	\Hom_{\GL_2(K)}\big(\pi(\Delta, \tau, \lambda),\pi(\Delta, \lambda)\big) \xlongrightarrow{\sim} 
	\Hom_{(\varphi, \Gamma)}\big(\Delta, \cR_{K,E}/t_{\sigma}^{h_{1,\sigma}-h_{2,\sigma}} \big)^{\oplus r},
\end{equation*}
\begin{equation*}
	\Big(\text{resp. }\Ext^1_{\GL_2(K)}\big(\pi_0(\Delta, \sigma, \lambda), \pi_0(\Delta, \emptyset, \lambda)\big) \xlongrightarrow{\sim} \Ext^1_{(\varphi, \Gamma)}\big(t_{\sigma}^{h_{2,\sigma}}\cR_{K,E}/t_{\sigma}^{h_{1,\sigma}}, \Delta\big),
\end{equation*}
\begin{equation*}
	\Ext^1_{\GL_2(K)}\big(\pi_0(\Delta, \Sigma_K, \lambda), \pi_0(\Delta, \tau,  \lambda)\big) \xlongrightarrow{\sim} \Ext^1_{(\varphi, \Gamma)}\big(t_{\sigma}^{h_{2,\sigma}}\cR_{K,E}/t_{\sigma}^{h_{1,\sigma}}, \Delta\big)^{\oplus r}\Big)
\end{equation*}
satisfying that for any de Rham rank two $(\varphi, \Gamma)$-module $D'$ of weight $\textbf{h}$ with $D'_{\emptyset}\cong \Delta$, the isomorphism classes\footnote{Again, we use the kernel to denote an isomorphism class in $\Hom$.}  $[ \pi(D'_{\sigma}, \sigma, \lambda)]$ and $[\pi(D',\lambda)]$ \big(resp. $[\pi_0(D'_{\sigma},\sigma, \lambda)]$ and $[\pi_0(D'_{\sigma},\lambda)]$\big) are sent to $[t_{\sigma}^{-h_{2,\sigma}}D'_{\sigma}]$ and $[t_{\sigma}^{-h_{2,\sigma}}D'_{\sigma}]^{\oplus r}$ \big(resp. $[t_{\sigma}^{-h_{1,\sigma}}D'_{\sigma}]$ and $[t_{\sigma}^{-h_{1,\sigma}}D'_{\sigma}]^{\oplus r}$\big) respectively, where $D'_{\sigma}$ is defined in a similar way as $D_{\sigma}$.

(2) (Going from $D_{\tau}$ to $D$) Let $r=\begin{cases}
	1 & \text{$D_{\tau}$ is indecomposable}\\
	2 & \text{$D_{\tau}$ is decomposable}
\end{cases}$. There are natural isomorphisms:
\begin{equation*}
	\Hom_{\GL_2(K)}\big( \pi(D_{\tau}, \tau, \lambda), \pi(D_{\tau}, \lambda)\big) \xlongrightarrow{\sim} 
	\Hom_{(\varphi, \Gamma)}\big(D_{\tau}, \cR_{K,E}/t_{\tau}^{h_{1,\tau}-h_{2,\tau}} \big)^{\oplus r},
\end{equation*}
\begin{equation*}
	\Big(\text{resp. }\Ext^1_{\GL_2(K)}\big(\pi_0(D_{\tau}, \lambda), \pi_0(D_{\tau}, \tau, \lambda)\big) \xlongrightarrow{\sim} \Ext^1_{(\varphi, \Gamma)}\big(t_{\tau}^{h_{2,\tau}}\cR_{K,E}/t_{\tau}^{h_{1,\tau}}, D_{\tau}\big)^{\oplus r}\Big),
\end{equation*}
satisfying that for any de Rham rank two $(\varphi, \Gamma)$-module $D'$ of weight $\textbf{h}$ with $D'_{\tau}\cong D_{\tau}$, the isomorphism class $[\pi(D')]$ is sent to $[t_{\tau}^{-h_{2,\tau}}D']^{\oplus r}$  (resp. $[t_{\tau}^{-h_{1,\tau}}D']^{\oplus r}$).
\end{conjecture}
In short, to go from $\Delta$ to $D$, one can first use horizontal (resp. vertical) sequences to go from $\Delta$ to $D_{\sigma}$ (resp. to $D_{\tau}$), and then apply vertical (resp. horizontal) sequences to go from $D_{\sigma}$ (resp. from $D_{\tau}$) to $D$. 
The conjecture generalizes \cite[Conj.~1.1]{Br16} (see also \cite[Conj.~5.3.1]{BD2}), which can be viewed as a generalization of the Breuil-Strauch conjecture (and its $\Hom$-version). Indeed, the extensions conjectured in \cite[Conj.~1.1]{Br16} should be the left vertical and upper horizontal sequences in $\boxdot^-(\pi(D))$ (i.e. the third isomorphism in Conjecture \ref{ICBS} (2)). We refer to Remark \ref{ReconjGL2} (4) for a discussion on how  these are related with Hodge filtrations. The curious reader may find the multiplicity $r$ in the conjecture a little strange, but it would fit well with a hypothetical multi-variable $(\varphi, \Gamma)$-module avatar of $\boxdot^{\pm} (\pi(D))$ (see also the discussion below Theorem \ref{ITBS}). When $\Delta$ is crystabelline and $D_{\tau}$ non-split, the isomorphisms in (3) reveal the Hodge parameter of $D$. Towards Conjecture \ref{ICBS}, we have the following theorem:
\begin{theorem}\label{ITBS}Suppose $K=\Q_{p^2}$, and $\Delta$ is generic and crystabelline. 
We have under Hypothesis \ref{sigma}  (we omit the subscript $\GL_2(K)$ for $\Hom$ and $\Ext^1$):

(1) $\dim_E\Hom\big(\pi(\Delta, \emptyset, \lambda), \pi(\Delta, \sigma, \lambda)\big)=\dim_E \Ext^1\big(\pi_0(\Delta, \sigma, \lambda), \pi_0(\Delta, \emptyset, \lambda)\big)=2$.

(2) $\dim_E\Hom\big(\pi(\Delta, \tau, \lambda), \pi(\Delta,  \lambda)\big)=\dim_E \Ext^1\big(\pi_0(\Delta,\lambda), \pi_0(\Delta, \tau, \lambda)\big)=4$.

(3) $\dim_E\Hom\big(\pi(D_{\tau}, \tau, \lambda), \pi(D_{\tau},  \lambda)\big)=\dim_E \Ext^1\big(\pi_0(D_{\tau},  \lambda), \pi_0(D_{\tau}, \tau, \lambda)\big)=\begin{cases}
	2  & D_{\tau} \text{non-split} \\
	4 & D_{\tau} \text{ split.}
\end{cases}$
\end{theorem}
When $[K:\Q_p]=2$ and $\pi(D)$ is cut out from the completed cohomology of unitary Shimura curves (and $\pi(D)^{\lalg}\neq 0$),\footnote{This is in fact a bit different from the setting of \cite{CEGGPS1}. However, all the discussed results generalize to it as well.} the results in \cite{QS} (generalizing \cite{Pan3}) give a full description of  $\pi(D_{\sigma},\sigma,\lambda)$ and $\pi(D_{\tau},\tau,\lambda)$, which turn out to consist of subquotients of locally $\Q_p$-analytic principal series. This strongly suggests  a similar conclusion for  the general crystabelline $\GL_2(\Q_{p^2})$-case, and we adopt as   Hypothesis \ref{sigma}. Via a close study of $\boxdot^{\pm}(\pi(D))$  in \S~\ref{S423} and using a local-global compatibility result on ``surplus" locally algebraic constituents in \S~\ref{AppC}, we deduce a full and concrete description of all the representations in $\boxdot^{\pm}(\pi(D))$ except for those that are ``genuinely" locally $\Q_p$-analytic: $\pi(M,\lambda)$ and $\pi_0(M,\lambda)$ for $M\in \{\Delta, D_{\sigma}, D_{\tau},D\}$. Theorem \ref{ITBS} (1) follows from the classical facts on extensions of locally analytic  principal series (e.g. see \cite{Sch10}).
The proof of (2) and (3) are however not so straightforward (see Proposition \ref{PpiDelta}, Theorem \ref{TLin}), as the representations $\pi(M,\lambda)$  in (2) and (3) contain the hypothetical \textit{supersingular} constituents. We refer to \S~\ref{secpirho} for more discussions on the internal structure of $\pi(D)$ using $\boxdot^{\pm} ( \pi(D))$ (see in particular, the diagrams in (\ref{diag2}) and  (\ref{Dpirho})).

We now turn  to several topics closely related to the results discussed above, which have not been explored in this paper. We anticipate working on some of them  in the future. 

In \cite{BD2}, we defined a locally analytic generalization of Colmez's functor. Applying the functor respectively on the sequences (\ref{IE0}) and (\ref{IE1}), we get exact sequences $0 \ra t^{-h_2}D \ra \Delta \ra \cR_E/t^{h_1-h_2} \ra 0$, and $0 \ra  \Delta \ra t^{-h_1}D \ra t^{h_2}\cR_E/t^{h_1} \ra 0$ (which in fact  induce isomorphisms (\ref{IE2})  and (\ref{IE3})). One may ask for a $(\varphi, \Gamma)$-module avatar of $\boxdot^{\pm}(\pi(D))$. Note however,  say when $[K:\Q_p]=2$, the naive square built using the $(\varphi, \Gamma)$-modules $\begindc{\commdiag}[180]
\obj(0,0)[a]{$t_{\sigma}^{-h_{\sigma, 2}}t_{\tau}^{-h_{\tau,2}}D$}
\obj(4,-1)[b]{$t_{\sigma}^{-h_{\sigma,2}}D_{\sigma}$}
\obj(4,1)[c]{$t_{\tau}^{-h_{\tau,2}}D_{\tau}$}
\obj(7,0)[d]{$\Delta$}
\mor{a}{b}{}
\mor{a}{c}{}
\mor{b}{d}{}
\mor{c}{d}{}
\enddc$ does not fit well with $\boxdot^+(\pi(D))$ (and similarly for $\boxdot^-(\pi(D))$. Instead, the hypercubes appear to be more compatible with  (hypothetical) multi-variable $(\varphi,\Gamma)$-modules. This would particularly explain the multiplicity $r$ in Conjecture \ref{ICBS}. See \cite{BHHMS3} for the mod $p$ setting.

For $\GL_2(\Q_p)$, the sequences (\ref{IE0}) (\ref{IE1}) admit  geometric realizations, see \cite{DLB} (for the de Rham non-trianguline case in the cohomology of Drinfeld spaces) and \cite[\S~7.3]{Pan3} (in the completed cohomology of modular curves). We expect the hypercubes $\boxdot^{\pm}(\pi(D))$ also admit  geometric realizations.   

Finally, for general $\GL_n(K)$, one may also consider factorizing the maps $\iota$ and $\kappa$ using various wall-crossing functors, which will provide a bunch of representation. A natural question is how these representations reveal the information on Hodge filtrations when $D$ is de Rham. We leave further exploration of this topic for future work. 

We refer to the body of the context for more detailed and precise statements. One main difference from what's mentioned in the introduction is that we use the wall-crossing functors on the dual of $\pi(D)$, not directly on $\pi(D)$ itself.

\subsubsection*{Acknowledgement} 

I thank Christophe Breuil, Yongquan Hu, Lue Pan, Zicheng Qian, Benchao Su, Zhixiang Wu and Liang Xiao  for helpful discussions and answering my questions. Special thanks to Christophe Breuil for sending me a personal note and  for his	 comments on a preliminary version of the paper, and to Yongquan Hu and Haoran Wang for providing the Appendix. This work is supported by   the NSFC Grant No. 8200800065 and No. 8200907289.
\subsection*{Notation}
Let $K$ be a finite extension of $\Q_p$, and $E$ be a sufficiently large finite extension of $\Q_p$ containing all the Galois conjugate of $K$. Let $\Sigma_K:=\{\sigma: K \hookrightarrow \overline{\Q_p}\}=\{\sigma: K \hookrightarrow E\}$.  Let $d_K:=[K:\Q_p]$, and $f_K$ be the unramified degree of $K$ over $\Q_p$. Let $\val_K: K^{\times} \ra \Z$ be the normalized additive valuation (sending uniformizers to $1$), and $|\cdot |_K: K^{\times} \ra E^{\times}$ be the normalized multiplicative valuation (that is the unramified character sending uniformizers to $p^{-f_K}$). We normalize the local class field theory by sending a uniformizer of $K$ to a (lift of) geometric Frobenius.  Let $\varepsilon: \Gal_K  \ra \Q_p^{\times}$ to denote the cyclotomic character, that we also view as a character of $K^{\times}$. We let $\cR_{K,E}$ be the $E$-coefficients Robba ring associated to $K$, and when $K=\Q_p$, we write $\cR_E:=\cR_{\Q_p,E}$. 

Let $T\subset \GL_n$ be the subgroup of diagonal matrices, $X(T):=\Hom(T,\bG_m)$ which is isomorphic to $\oplus_{i=1}^n \Z e_i$, where $e_i$ denotes the character $\diag(x_1, \cdots, x_n) \mapsto x_i$. Let $B\subset \GL_n$ be the Borel subgroup of upper triangular matrices. Let $R=\{e_i-e_j, i\neq j\}\subset X(T)$ be the set of roots of $(\GL_n,T)$, and $R^+$ the set of positive roots (with respect to $B$), i.e. $R^+=\{e_i-e_j, i<j\}$. Let $\theta=(n-1, n-2, \cdots, 0)$ be the half-sum of the positive roots.  For a simple root $\alpha=e_i-e_{i+1}$ of $\GL_n$, let $\lambda_{\alpha}:=e_1+\cdots+e_i$, which is a fundamental weight of $\GL_n$. 	Let $\sW\cong S_n$ be the Weil group of $\GL_n$, which acts naturally on $R$ via $w(e_i-e_j)=e_{w(i)}-e_{w(j)}$. 	Let $\gl_n$, $\ft$, $\ub$ be the Lie algebra of $\GL_n$, $T$, $B$ respectively. Let $\cZ$ be the centre of $\text{U}(\ug)$. 

For a standard parabolic subgroup $P\supset B$ of $\GL_n$. Let $M_P\supset T$ be the standard Levi subgroup of $P$, which has the form $M_P=M_{P,1} \times \cdots \times M_{P,k}$, with $M_{P,i}\cong \GL_{n_i}$. Denote by $S(P)\subset S$ the set of simple roots of $M_P$, $R(P)^+\subset R^+$ the set of positive roots of $M_P$. Let $\sW(P)$ be the Weyl group of $M_P$.

For an algebraic group $H$ over $\Q_p$ (which we also view as an algebraic group over extensions of $\Q_p$ by base-change), denote by $H^{\Gal(K/\Q_p)}:=\Res^K_{\Q_p} H$.  We let $\sW_K$ be the Weyl group of $\GL_n^{\Gal(K/\Q_p)}$. For the Lie algebra $\fh$ of $H$ over $K$, let $\fh_K:=\fh \otimes_{\Q_p} E$ which is the Lie algebra of $H^{\Gal(K/\Q_p)}$ over $E$. Let $\cZ_K$ be the centre of $\text{U}(\gl_{n,K})$, we have  $\cZ_K\cong \otimes_{\sigma\in \Sigma_K} \cZ_{\sigma}$ where $\cZ_{\sigma}\cong \cZ \otimes_{K,\sigma} E$. We let $\theta_K:=(n-1, \cdots, 0)_{\sigma\in \Sigma_K}$, which is the sum of positive roots of $\GL_n^{\Gal(K/\Q_p)}$.

For an integral weight $\lambda$ of $T(K)$, let $M(\lambda):=\text{U}(\gl_{n,K}) \otimes_{\text{U}(\ub_K)} \lambda$ and $M^-(\lambda):=\text{U}(\gl_{n,K}) \otimes_{\text{U}(\ub^-_K)}\lambda$ where $\ub^-$ is the Lie algebra of the opposite Borel subgroup $B^-$. Let $L(\lambda)$ (resp. $L^-(\lambda)$) be the simple quotient of $M(\lambda)$ (resp. of $M^-(\lambda)$). If $\lambda$ is dominant (with respect to $B$), then $L(\lambda)$ is finite dimensional, which is actually the algebraic representation of $\GL_n^{\Gal(K/\Q_p)}$ of highest weight $\lambda$. We use $z^{\lambda}$ to denote the algebraic character $T(K)$ of weight $\lambda$, i.e. if $\lambda=(\lambda_{i,\sigma})_{\substack{i=1, \cdots, n\\ \sigma\in \Sigma_K}}$, then $z^{\lambda}=\otimes_{i=1}^n \sigma^{\lambda_{i,\sigma}}$.  

For an integral weight $\lambda$ of $T(K)$, we let $\chi_{\lambda}$ be the infinitesimal character of $\cZ_K$ associated to $\lambda$, that is the character of $\cZ_K$ on $M(\lambda)$.  Denote by $\Mod(\text{U}(\ug_K)_{\chi_{\lambda}})\subset \Mod(\text{U}(\ug_K))$ the full subcategory of $\text{U}(\ug_K)$-modules,  consisting of those on which $\cZ$ acts by $\chi_{\lambda}$. For integral weights $\lambda$, $\mu$, let $T_{\lambda}^{\mu}: \Mod(\text{U}(\ug_K)_{\chi_{\lambda}}) \ra \Mod(\text{U}(\ug_K))$  be the  translation functor sending  $M$ to $(M \otimes_E L(\nu))\{\cZ_K=\chi_{\mu}\}$, where $\nu$ is the (unique) dominant weight in $\{w(\lambda-\mu)\ |\ w\in \sW_K\}$, and $\{\cZ_K=\chi_{\mu}\}$ denotes the generalized eigenspace.

Besides the standard action of $\sW_K$ on the weights, we will also frequently use the dot action of $\sW_K$ on the weights: $w\cdot \lambda=w(\lambda+\theta_K)-\theta_K$.  In particular, $w\cdot (-\theta_K)=-\theta_K$. 

For a locally $K$-analytic group $H$, denote by $\cD(H,E)$ the locally $\Q_p$-analytic distribution algebra of $H$ over $E$ (cf. \cite[\S~2]{ST02}), which is the strong dual of the space $\cC^{\Q_p-\la}(H,E)$ of locally $\Q_p$-analytic $E$-valued functions on $H$. Denote by $\cM_H$ the category of abstract $\cD(H,E)$-module. 
For a locally $\Q_p$-analytic representation $V$ of $H$ on space of compact type, denote by $V^*$ the strong dual of $V$, which is naturally equipped with a separately continuous $\cD(H,E)$-action (cf. \textit{loc. cit.}). There is a natural  $\Q_p$-linear action of the Lie algebra $\fh$ of $H$ on $V$ (resp.  $V^*$), which  induces a $\text{U}(\fh_K)$-action on $V$ (resp. $V^*$).

\section{$(\varphi,\Gamma)$-modules of constant weights}\label{Sphigamma}
We discuss the change of weights of $(\varphi,\Gamma)$-modules.  Let $D$ be a  $(\varphi, \Gamma)$-module of rank $n$ over $\cR_{K,E}$.  Suppose $D$ has integer Sen weights $\textbf{h}=(h_{1,\sigma}\geq \cdots \geq h_{n,\sigma})_{\sigma\in \Sigma_K}$.
\begin{lemma}\label{pDE0}
There exists a unique $(\varphi, \Gamma)$-module $\Delta$ of constant Hodge-Tate weight $0$ over $\cR_{K,E}$ such that $\Delta[\frac{1}{t}]\cong D[\frac{1}{t}]$. 
\end{lemma} 
\begin{proof}
Consider the $B$-pair $(W_e(D), W_{\dR}^+(D))$ associated to $D$ (cf. \cite{Ber08II}). By Fontaine's classification of $B_{\dR}$-representations of $\Gal_K$ (cf. \cite[Thm.~3.19]{Fo04}), there exists a  $B_{\dR}^+$-subrepresentation $\Lambda \subset W_{\dR}^+[1/t]$ such that $\Lambda$ has constant Sen weight $0$. Let $\Delta$ be the $(\varphi, \Gamma)$-module associated to the $B$-pair $(W_e(D), \Lambda)$, which clearly satisfies the properties in the lemma. The uniqueness follows from the uniqueness of $\Lambda$, which is a direct consequence of \cite[Thm.~3.19]{Fo04}. One can also prove it as follows. let $\Delta'$ be another $(\varphi, \Gamma)$-module satisfying the same properties. For any $n>0$, by \cite[Lem.~5.1.1]{BD2}, $H^0_{(\varphi, \Gamma)}((\Delta)'^{\vee} \otimes_{\cR_{K,E}} \Delta) \ra  H^0_{(\varphi, \Gamma)}((\Delta)'^{\vee} \otimes_{\cR_{K,E}} t^{-n}\Delta)$ is an isomorphism. We deduce 
\begin{equation*}
	\Hom_{(\varphi, \Gamma)}(\Delta', \Delta)\xlongrightarrow{\sim} \Hom_{(\varphi, \Gamma)}(\Delta', \Delta[1/t]).
\end{equation*}
In particular, the tautological injection $\Delta' \hookrightarrow \Delta[1/t]$ factors through an injection $\Delta' \hookrightarrow \Delta$, which has to be an isomorphism by comparing the Sen weights.
\end{proof}
\begin{remark}
(1)	Suppose $D$ is de Rham, then $\Delta$ is the so-called $p$-adic differential equation associated to $D$ (cf. \cite{Ber08a}). Passing from $D$ to $\Delta$, we lose exactly the information of Hodge filtrations of $D$. 

(2) If $\dim_E D_{\dR}(D)_{\sigma}=1$ for all $\sigma\in \Sigma_K$ (which is actually the most non-de Rham case, as $D$ has integral Sen weights), then by  \cite[Thm.~3.19]{Fo04}, one can show that $\prod_{\sigma\in \Sigma_K} t_{\sigma}^{-h_{n,\sigma}} D$ is the unique $(\varphi, \Gamma)$-submodule of $\Delta$  of Sen weights $(h_{1,\sigma}-h_{n,\sigma}, \cdots, 0)_{\sigma\in \Sigma_K}$. In particular, in this case, passing from $D$ to $\Delta$ does not lose extra information of $D$ than the Sen weights. 
\end{remark}
We define the irreducible constituents for a $(\varphi, \Gamma)$-module $\Delta$ of constant weights $0$ over $\cR_{K,E}$. We call a filtration  $\sF=\{0= \sF_0 \Delta \subsetneq \cdots \subsetneq \sF_k \Delta=\Delta\}$ of saturated $(\varphi, \Gamma)$-submodules of $\Delta$ \textit{minimal} if $\gr_i \sF$ are all irreducible. Note that all these $\gr_{i} \sF$ have constant  Sen weight $0$.
\begin{lemma}
Let $\sF$ be a minimal filtration of $\Delta$, then the set $\{\gr_i \sF\}_i$ is independent of the choice of minimal filtrations on $\Delta$.
\end{lemma} 
\begin{proof}
Suppose $\sF'$ is another minimal filtration on $\Delta$. Using d\'evissage, there exists $i$ such that $\Hom_{(\varphi, \Gamma)}(\sF_0 \Delta, \gr_i \sF')\neq 0$. As both $\sF_0 \Delta$ and $\gr_i \sF'$ are irreducible of constant weight $0$, we see $\sF_0 \Delta \cong \gr_i \sF'$. Then we can repeat the argument for $\Delta/\sF_0 \Delta$ (with the induced filtrations $\sF$ and $\sF'$). The lemma follows by induction.
\end{proof}
We call elements in $\{\gr_i \sF\}$ \textit{irreducible constituents} of $\Delta$.  Similarly, a filtration of $(\varphi, \Gamma)$-submodules over $\cR_{K,E}[\frac{1}{t}]$ of $\Delta[\frac{1}{t}]$ is called minimal, if the graded pieces are irreducible, which are called irreducible constituents of $\Delta[\frac{1}{t}]$. 
For a  filtration $\sF$ on $\Delta$, we define a filtration $\sF[\frac{1}{t}]$ of $(\varphi, \Gamma)$-module over $\cR_{K,E}[\frac{1}{t}]$ on $\Delta[\frac{1}{t}]$. The following lemma is clear. 
\begin{lemma}
The map $\sF \mapsto \sF[\frac{1}{t}]$ is a bijection of the minimal filtrations on  $\Delta$ and the minimal filtrations  on   $\Delta[\frac{1}{t}]$. 
\end{lemma}
For a minimal filtration $\sF$ of $\Delta$, we let $P_{\sF}\supset B$ be the associated standard parabolic subgroup of $\GL_n$. For $i=1,\cdots, n$, we let $\beta(i)\in \{1,\cdots, k\}$ such that $e_{ii} \in M_{P_{\sF},\beta(i)}$ (the $\beta(i)$-th factor of the Levi subgroup $M_{P_{\sF}}$ of $P_{\sF}$). In the following, we assume the irreducible constituents of $\Delta$ are all \textit{distinct}. Put  $C_{\sF}$ be the subset of $R^+$ consisting of  $e_i-e_j$ for $i,j\in \{1\cdots, n\}$, $i<j$ which satisfies that $\Delta$ admits a subquotient $\Delta'$ such that $\soc \Delta'\cong \gr_{\beta(i)} \sF$ and ${\rm cosoc} \Delta' \cong \gr_{\beta(j)} \sF$. 
\begin{lemma}\label{rootCF}
The set $C_{\sF}$ is a closed subset of $R^+$ relative to $P_{\sF}$ in the sense of \cite[Def.\ 2.3.1]{BHHMS2}.
\end{lemma}
\begin{proof}
It is trivial that $R_{P_{\sF}}^+\subset C_{\sF}$. Suppose $e_i-e_j$ and $e_j-e_{j'}$ both lie in $C_{\sF}$. Let $M_1$ be the submodule of $\Delta$ of cosocle $\gr_{\beta(j')} \sF$, and $M_2$ be the quotient of $\Delta$ of socle $\gr_{\beta(i)}$. As $e_i-e_j\in C_{\sF}$ (resp. $e_j-e_{j'}\in C_{\sF}$), $\Delta_j$ is an irreducible constituent of $M_2$ (resp. $M_1$). As the irreducible constituents of $\Delta$ are assumed to be distinct, the composition $M_1 \hookrightarrow \Delta \twoheadrightarrow M_2$ is non-zero. Its image has socle $\gr_{\beta(i)} \sF$ and cosocle $\gr_{\beta(j')} \sF$. So $e_i-e_{j'}\in C_{\sF}$. For $w\in \sW(P_{\sF})$ and $e_i-e_j\in C_{\sF} \setminus R_{P_{\sF}}^+$, we see $\beta(i)< \beta(j)$. As $\beta(w(i))=\beta(i)$ and $\beta(w(j))=\beta(j)$, we deduce $e_{w(i)}-e_{w(j)}\in C_{\sF} \setminus R_{P_{\sF}}^+$. 
\end{proof}
\begin{remark}
Assume $\Delta$ is de Rham, and let $\WD(\Delta)$ be the associated Weil-Deligne representation. Then $C_{\sF}$ is closely related to the monodromy operator $N$ on  $\WD(\Delta)$. For example, if $N=0$, then the (any) minimal filtration $\sF$ on $\Delta$ splits and  $C_{\sF}=R(P_{\sF})^+$.
\end{remark}
Let $\widetilde{P}_{\sF}\subset P_{\sF}$ be the Zariski closed subgroup of $P_{\sF}$ (containing $M_{P_{\sF}}$) associated to $C_{\sF}$ by \cite[Lem. 2.3.1.4]{BHHMS2}. Let 
\begin{equation}\label{wtildPF}\sW_{\widetilde{P}_{\sF}}:=\{w\in \sW\ |\ w(S(P_{\sF}))\subset S, w(C_{\sF}\setminus R(P_{\sF})^+)\subset R^+\}.\end{equation} 
For $w\in \sW$ such that $w(S(P_{\sF}))\subset S$, let $^w P_{\sF}$ be the standard parabolic subgroup of simple roots $w(S(P_{\sF}))$. For such $w$, and  $i,j\in \{1,\cdots, n\}$, if $\beta(i)=\beta(j)$, then $\beta(w(i))=\beta(w(j))$. Hence $w$ corresponds to an element $w^{\natural}\in S_k$. It is easy to see the map $w\mapsto w^{\natural}$ is a bijection from $\{w\in \sW\ |\ w(S(P_{\sF}))\subset S\}$ to $S_k$. 
\begin{lemma} \label{refine0}(1) Let $w\in\sW$ such that $w(S(P_{\sF}))\subset S$. Then $w\in \sW_{\widetilde{P}_{\sF}}$ if and only if there exists a filtration $\sF'$ on $\Delta$ such that $\gr_i \sF'\cong \gr_{(w^{\natural})^{-1}(i)} \sF$. 

(2) Let $w\in \sW_{\widetilde{P}_{\sF}}$ and $\sF'$ be the associated filtration on $\Delta$ as in (1). Then $P_{\sF'}=\! ^wP_{\sF}$ and $\widetilde{P}_{\sF'}=w\widetilde{P}_{\sF}w^{-1}$. 
\end{lemma} 
\begin{proof}(1) Suppose $w\in \sW_{\widetilde{P}_{\sF}}$.
As $w(e_i-e_{w^{-1}(1)})=e_{w(i)}-e_1$, it is easy to see any $e_{i}-e_{w^{-1}(1)}$ for $i<w^{-1}(1)$ can not lie in $C_{\sF}$. So $\gr_{\beta(w^{-1}(1))}\sF=\gr_{(w^{\natural})^{-1}(1)} \sF$ is a submodule of $\Delta$, denoted by $\sF'_1 \Delta$. Let $n_i:=\#\beta^{-1}(\{i\})$.  	As $w(S(P_{\sF}))\subset S$, $w^{-1}(1+j)=w^{-1}(1)+j$ for all $0\leq j \leq n_{\beta(w^{-1}(1))}-1$. Consider henceforth $w^{-1}(n_{\beta(w^{-1}(1))}+1)$. By similar arguments for  $e_i-e_{w^{-1}(n_{\beta(w^{-1}(1))}+1)}$, we see  $\gr_{\beta(w^{-1}(n_{\beta(w^{-1}(1))}+1)} \sF=\gr_{(w^{\natural})^{-1}(2)}\sF$ is a submodule of $\Delta/\sF'_1\Delta$. We put $\sF'_2 \Delta$ the submodule of $\Delta$ with irreducible constituents $\sF'_1 \Delta$ and $\gr_{(w^{\natural})^{-1}(2)} \sF$. Continuing with the argument, we get the wanted filtration $\sF'$. 

Conversely, if $e_i-e_j\in C_{\sF}$, it suffices to show $w(i)<w(j)$. As $w(S(P_{\sF}))\subset S$, it suffices to show $w^{\natural}(\beta(i))<w^{\natural}(\beta(j))$ if $\beta(i)<\beta(j)$ and $e_i-e_j\in C_{\sF}$. If $w^{\natural}(\beta(i))>w^{\natural}(\beta(j))$, then $\gr_{\beta(j)} \sF=\gr_{(w^{\natural})^{-1}(w^{\natural}(\beta(j)))}\sF =\gr_{w^{\natural}(\beta(j))} \sF'\hookrightarrow \Delta/\sF'_{w^{\natural}(\beta(j))-1}$, where the latter contains $\gr_{w^{\natural}(\beta(i))} \sF'=\gr_{\beta(i)} \sF$ as irreducible constituent.   But as we assume all the irreducible constituents of $\Delta$ are distinct, this implies there can not exist subquotients of $\Delta$ of socle $\gr_{\beta(i)} \sF$ and of cosocle $\gr_{\beta(j)} \sF$, contradicting $e_i-e_j\in C_{\sF}$.

(2) follows by definition.
\end{proof}
Now we consider the case where $P_{\sF}=B$, which is usually  refereed to as the trianguline case. There exist smooth characters $\phi_i: K^{\times} \ra E^{\times}$ such that $\Delta_i \cong \cR_{K,E}(\phi_i)$. For $w\in \sW$, recall a smooth character of $T(K)$,  $w(\phi):=\otimes_{i=1}^n \phi_{w^{-1}(i)}$ is called a refinement of $\Delta[\frac{1}{t}]$ (resp. of $\Delta$) if $\Delta[\frac{1}{t}]$ (resp. $\Delta$) admits a successive extension of $\cR_{K,E}(\phi_{w^{-1}(i)})[\frac{1}{t}]$ (resp. of $\cR_{K,E}(\phi_{w^{-1}(i)})$). In particular, $\phi=\otimes_{i=1}^n \phi_i$ is a refinement of $\Delta[\frac{1}{t}]$.  
\begin{lemma}\label{reftri}
For $w\in \sW$, the followings are equivalent. 

(1) $w(\phi)$ is a refinement of $\Delta[\frac{1}{t}]$.

(2) $w(\phi)$ is a refinement of $\Delta$. 

(3) $w(C_{\sF}) \subset R^+$.
\end{lemma}
\begin{proof} By Lemma \ref{refine0} (1), (2) $\Leftrightarrow$ (3). We show (1) $\Leftrightarrow$ (2). If $(\phi_{w^{-1}(i)})$ is a refinement of $\Delta$, by inverting $t$, it is also clear $(\phi_{w^{-1}(i)})$ is a refinement of $\Delta[\frac{1}{t}]$. Suppose $\sF'$ is a filtration on $\Delta[\frac{1}{t}]$ such that $\gr_i \sF'\cong \cR_{K,E}(\phi_{w^{-1}(i)})[\frac{1}{t}]$. Then $\sF'_i \Delta:=\Delta\cap \sF'_i \Delta[\frac{1}{t}]$ defines a filtration of saturated $(\varphi, \Gamma)$-submodules of $\Delta$. Hence $\sF'_i \Delta$ and $\gr_i \sF'$ also have constant Sen weights zero for all $i=0, \cdots, n-1$. As $\gr_i \sF'\Delta \hookrightarrow \gr_i \sF' \Delta[\frac{1}{t}]$, $\gr_i \sF'\Delta\cong \cR_{K,E}(\phi_{w^{-1}(i)})$. 
\end{proof}
\begin{remark}\label{Rrefine}
By similar arguments and comparing the Sen weights, $w(\phi)$ is a refinement of $\Delta[\frac{1}{t}]\cong D[\frac{1}{t}]$ is equivalent to that there exists $w'\in \sW_K$ such that $D$ admits a filtration of $i$-th graded piece isomorphic to $\cR_{K,E}(\phi_{w^{-1}(i)} z^{w'(\textbf{h})_i})$. In this case,  $w(\phi) =(\phi_{w^{-1}(i)})$ is usually referred to as  a refinement of $D$. Note that when passing from $D$ to $\Delta$, the information of $w'$ is lost. 
\end{remark}

\section{Conjectures and results on the singular skeletons}
Let $\pi(D)$ be the locally $\Q_p$-analytic representation of $\GL_n(K)$ associated (via the theory $p$-adic automorphic representations, cf. \cite{CEGGPS1}) to an $n$-dimension $\Gal_K$-representation $\rho$ of integral Sen weights. We consider the translation of $\pi(D)$ to the singular block. We discuss some basis properties of the translation to singular block in \S~\ref{S3.1}.  In \S~\ref{S3.2},  we propose a local-global compatibility conjecture on the translation, which may be viewed as a locally analytic version of \cite[Conj.~2.5.1]{BHHMS2}. In \S~\ref{S3.3}, we prove some results towards the conjecture on the finite slope part. 	
\subsection{Translation to the singular block} \label{S3.1}
\subsubsection{Preliminaries}
Let $\lambda$ be an integral weight of $\ft_K$ such that $w_0 \cdot \lambda$ is anti-dominant. Consider the translation functor \begin{eqnarray*}T_{\lambda}^{-\theta_K}=T_{w_0 \cdot \lambda}^{-\theta_K}: \Mod(\text{U}(\ug_K)_{\chi_{\lambda}}) &\lra& \Mod(\text{U}(\ug_K)) \\
M &\mapsto& \big(M \otimes_E L(-\theta_K-w_0\cdot \lambda)\big)\{\cZ_K=\chi_{-\theta_K}\}.
\end{eqnarray*}
\begin{proposition}\label{propsemi}For $M\in \Mod(\text{U}(\ug_K)_{\chi_{\lambda}})$, the map \begin{equation}\label{Zsemi}(M\otimes_E L(-\theta_K-w_0 \cdot \lambda))[\cZ_K=\chi_{-\theta_K}]\hooklongrightarrow (M\otimes_E L(-\theta_K-w_0 \cdot \lambda))\{\cZ_K=\chi_{-\theta_K}\}=T_{\lambda}^{-\theta_K}(M)\end{equation}
is an isomorphism.
\end{proposition}
\begin{proof}
First we prove the statement holds for Verma modules in $\Mod(\text{U}(\ug_K)_{\chi_{\lambda}})$. For $w\in \sW_K$, $M(w\cdot \lambda) \otimes_E L(-\theta_K-w_0 \cdot \lambda)$ admits a filtration with quotients isomorphic to $M(w\cdot \lambda+\mu)$ where $\mu$ run through weights of $L(-\theta_K-w_0 \cdot \lambda)$ (cf. \cite[Thm.~3.6]{Hum08}). If $M(w\cdot \lambda+\mu)$ is a generalized $\chi_{-\theta_K}$-eigenspace of $\cZ_K$, we have $\mu=-\theta_K-w\cdot \lambda$. As the weight $-\theta_K-w\cdot \lambda=(ww_0)(-\theta_K-w_0 \cdot \lambda)$ has multiplicity one in $L(-\theta_K-w_0\cdot \lambda)$. We deduce $T_{\lambda}^{-\theta_K} M(w\cdot \lambda)\cong M(-\theta_K)$, in particular (\ref{Zsemi}) is an isomorphism for $M=M(w\cdot \lambda)$. 

Now for general $M$, it suffices to show that $\cZ_K$-action on $T^{-\theta_K}_{\lambda}(M)$ is semi-simple. Consider the set  $\Hom(T^{-\theta_K}_{\lambda}, T^{-\theta_K}_{\lambda})$  of the endomorphisms of the functor $T^{-\theta_K}_{\lambda}$. There is a natural morphism $f: \cZ_K \ra  \Hom(T^{-\theta_K}_{\lambda}, T^{-\theta_K}_{\lambda})$, induced by the $\cZ_K$ -action on each $T^{-\theta_K}_{\lambda}(M')$.  We have hence 
\begin{equation*}
	\cZ_K\lra 	\Hom(T^{-\theta_K}_{\lambda}, T^{-\theta_K}_{\lambda})\cong \Hom_{\text{U}(\ug)}(T^{-\theta_K}_{\lambda} M(\lambda), T^{-\theta_K}_{\lambda}M(\lambda))\cong \Hom_{\text{U}(\ug)}(M(-\theta_K),M(-\theta_K)),
\end{equation*}
where the first isomorphism follows from  \cite[Thm.~3.5]{BG}. We deduce the map $\cZ_K \ra  \Hom(T^{-\theta_K}_{\lambda}, T^{-\theta_K}_{\lambda})$ factors through $\cZ_K/J_{-\theta_K}$, where $J_{-\theta_K}\subset \cZ_K$ is the maximal ideal corresponding to $\chi_{-\theta_K}$. As  the map $\cZ_K \ra \Hom_{\text{U}(\ug)}(T^{-\theta_K}_{\lambda} M, T^{-\theta_K}_{\lambda} M)$ obviously factors through $f$, $T^{-\theta_K}_{\lambda} M$ is annihilated by $J_{-\theta_K}$. The proposition follows.
\end{proof}
\begin{remark}\label{transzero}
Note we have (by the proof) $T_{\lambda}^{-\theta_K} L(w\cdot \lambda)\cong \begin{cases} 0 & w\neq w_0 \\ L(-\theta_K) & w=w_0
\end{cases} $.
\end{remark}

We next discuss translations
of parabolic inductions to the singular block. Let $P$ be a standard parabolic subgroup of $\GL_n$, with $M_P\supset T$ the standard Levi subgroup. 
\begin{lemma}\label{lemind}
Let $\pi_{M_P}$ be a space of compact type over $E$, equipped with a locally $\Q_p$-analytic representation of $M_P(K)$, and $V$ be an algebraic representation of $\Res_{\Q_p}^K \GL_n$. Then we have an isomorphism
\begin{equation}\label{indten}
	(\Ind_{P^-(K)}^{\GL_n(K)} \pi_{M_P})^{\Q_p-\an} \otimes_E V \xlongrightarrow{\sim} (\Ind_{P^-(K)}^{\GL_n(K)} \pi_{M_P} \otimes_E V)^{\Q_p-\an}, f \otimes v \mapsto [g\mapsto f(g) \otimes v]
\end{equation}
where $P^-(K)$ acts on the second $V$ via restriction of the $G(K)$-action. 
\end{lemma}
\begin{proof}It is easy to check the map is well-defined. We construct an inverse of the map. Let $e_1, \cdots, e_m$ be a basis of $V$. For $F\in (\Ind_{P^-(K)}^{\GL_n(K)} \pi_{M_P} \otimes_E V)^{\Q_p-\an}$, let $f_i: \GL_n(K) \ra \pi_{M_P}$ such that $F(g)=\sum_{i=1}^m f_i(g) \otimes g(e_i)$. Note that $f_i$ is locally $\Q_p$-analytic. Indeed, $f_i$ is equal to the composition $\GL_n(K) \xrightarrow{g\mapsto F(g) \otimes g(e_i^*)}\pi_{M_P} \otimes_E V \otimes_E V^{\vee} \twoheadrightarrow \pi_{M_P}$. It is then straightforward to check the map $F\mapsto \sum_{i=1}^m f_i \otimes e_i$ gives an inverse of (\ref{indten}). The lemma follows.
\end{proof}
Let $\fm_{\fp}$ be the Lie algebra of $M_P$. 
Denote by $\cZ_{\fm_{\fp},K}$ the centre of $\text{U}(\fm_{\fp,K})$. The Harish-Chandra isomorphism for $\gl_n$ and $\fm_{\fp}$ induce an isomorphism $\cZ_K\xrightarrow{\sim} \cZ_{\fm_{\fp},K}^{\sW(P)_K\backslash \sW_K}$. For a weight $\mu$ of $\ft_K$, we denote by $\chi_{M_P,\mu}$ the associated character of $\cZ_{\fm_{\fp},K}$. 
Let $\pi_{M_P}$ be a space of compact type over $E$ equipped with a locally $\Q_p$-analytic representation of $M_P(K)$. Let $\lambda$ be as in the beginning of the section, and assume that $\cZ_{\fm_\fp,K}$ acts on $\pi_{M_P}$ via the character $\chi_{M_P, w\cdot \lambda}$ for some $w\in \sW_K$ (noting there is no ambiguity for the dot action: we always use the $\gl_n$-dot action, and when $w\in \sW(P)_K$, its $\gl_n$-dot action coincides with its $\fm_{\fp,K}$-dot action). 
\begin{lemma}\label{lemind2}
$\cZ_K$ acts on  $(\Ind_{P^-(K)}^{\GL_n(K)} \pi_{M_P})^{\Q_p-\an}$ via the character $\chi_{\lambda}=\chi_{w\cdot \lambda}$.
\end{lemma}
\begin{proof}
It suffices to show $\cZ_K$ acts on the dual $\big((\Ind_{P^-(K)}^{\GL_n(K)} \pi_{M_P})^{\Q_p-\an}\big)^*$ via $\chi_{\lambda^*}$. Let $H:=\GL_n(\co_K)$, and $P_0^-:=H\cap P^-(K)$. By \cite[Prop.~5.3]{Koh2011} (see also \cite[Prop.~2.1]{Sch11}), there is a natural $\cD(H, E)$-equivariant map
\begin{equation*}
	\cD(H, E) \otimes_{\cD(P_0^-,E)} \pi_{M_P}^* \lra \big((\Ind_{P^-(K)}^{\GL_n(K)} \pi_{M_P})^{\Q_p-\an}\big)^*
\end{equation*}
which moreover has dense image. It suffices to show $\cZ_K$ acts on the source via $\chi_{\lambda^*}$. But this follows easily from the Harish-Chandra isomorphisms and the fact that $\cZ_K$ lies in the centre of $\cD(\GL_n(K),E)$ (cf. \cite{Koh2}).
\end{proof}
We choose $w$ such that $w\cdot \lambda$ is anti-dominant for $\fm_{\fp,K}$ (with $\chi_{M_P,w\cdot \lambda}$ unchanged). Indeed, we just need to multiply the original $w$ on the left by a certain element in $\sW(P)_K$. Then $-\theta_K-w\cdot \lambda$ is dominant for $\fm_{\fp,K}$, and we have $T_{w\cdot \lambda}^{-\theta_K}(-)=(- \otimes_E L(-\theta_K-w\cdot \lambda)_P)\{\cZ_{\fm_{\fp},K}=-\theta_K\}\cong (- \otimes_E L(-\theta_K-w\cdot \lambda)_P)[\cZ_{\fm_{\fp},K}=-\theta_K]$, where the second isomorphism follows by the same argument as in Proposition \ref{propsemi}. 
\begin{proposition}\label{tranInd}
We have a natural  isomorphism
\begin{equation*}
	T_{\lambda}^{-\theta_K} (\Ind_{P^-(K)}^{\GL_n(K)} \pi_{M_P})^{\Q_p-\an} \cong (\Ind_{P^-(K)}^{\GL_n(K)} T_{w \cdot \lambda}^{-\theta_K} (\pi_{M_P}))^{\Q_p-\an}.
\end{equation*}
\end{proposition}
\begin{proof}
By Lemma \ref{lemind}, we have
\begin{equation*}
	T_{\lambda}^{-\theta_K} (\Ind_{P^-(K)}^{\GL_n(K)} \pi_{M_P})^{\Q_p-\an} \cong  (\Ind_{P^-(K)}^{\GL_n(K)} \pi_{M_P} \otimes_E L(-\theta_K-w_0 \cdot \lambda))^{\Q_p-\an}\{\cZ_{K}=\chi_{-\theta_K}\}. 
\end{equation*}
We can write $\pi_{M_P} \otimes_E L(-\theta_K-w_0 \cdot \lambda)$ as a successive extension of $(\pi_{M_P} \otimes_E L(\mu)_P)\{\cZ_{\fm_{\fp},K}=\chi_{\gamma}\}$, where $L(\mu)_P$ run through irreducible factors of $L(-\theta_K-w_0 \cdot \lambda)|_{M_P(K)}$. By Lemma \ref{lemind2}, only the term with  $\chi_{\gamma}=\chi_{M_P,-\theta_K}$  can contribute to $T_{\lambda}^{-\theta_K} (\Ind_{P^-(K)}^{\GL_n(K)} \pi_{M_P})^{\Q_p-\an}$.  Let $L(\mu)_P$ be such that $(\pi_{M_P} \otimes_E L(\mu)_P)\{\cZ_{\fm_{\fp},K}=\chi_{M_P,-\theta_K}\}\neq 0$. By \cite[Thm.~2.5]{BG}, there exist a weight $\mu'$ in $L(\mu)_P$ and $w'\in \sW(P)_K$ such that $w'\cdot (w \cdot \lambda)+\mu'=-\theta_K$ hence $\mu'=-\theta_K-(w'w)\cdot \lambda=(w'ww_0)(-\theta_K-w_0\cdot \lambda)$ is an extremal weight in $L(-\theta_K-w_0\cdot \lambda)$. Using the $\sW(P)_K$-action, we see $-\theta_K-w \cdot \lambda$ appears in $L(\mu)_P$.  As $-\theta_K-w\cdot \lambda=(ww_0)(-\theta_K-w_0 \cdot \lambda)$ is a highest weight with respect to $(ww_0) \ub_K (ww_0)^{-1}$, and  $(ww_0) \ub_K (ww_0)^{-1} \cap \fm_{\fp,K}=\ub_K\cap \fm_{\fp,K}$ (by our assumption on $w$), we see $-\theta_K-w\cdot \lambda$ is a highest weight with respect to $\ub_K\cap \fm_{\fp,K}$. Hence $L(\mu)_P\cong L(-\theta_K-w \cdot \lambda)_P$. As the weight $-\theta_K-w\cdot \lambda$ has multiplicity one in $L(-\theta_K-w_0 \cdot \lambda)$, we get 
\begin{multline*}
	(\Ind_{P^-(K)}^{\GL_n(K)} \pi_{M_P} \otimes_E L(-\theta_K-w_0 \cdot \lambda))^{\Q_p-\an}\{\cZ_{K}=\chi_{-\theta_K}\}\\
	\cong \big(\Ind_{P^-(K)}^{\GL_n(K)} (\pi_{M_P} \otimes_E L(-\theta_K-w \cdot \lambda)_P)\{\cZ_{\fm_{\fp},K}=\chi_{-\theta_K}\}\big)^{\Q_p-\an}.
\end{multline*}The proposition follows.
\end{proof}
Next we discuss the effect of the translations on duals. We first consider the Schneider-Teitelbaum duals. We write $G:=\GL_n(K)$ and let $H$ be an open (compact) uniform subgroup of $G$. Let $\cC_c^{\Q_p-\la}(G,E)$ be the space of locally $\Q_p$-analytic functions on $G$ with compact support, which is equipped with a natural locally convex topology and is a space of compact type (cf. \cite[Rem.~2.1]{ST-dual}). Let $\cD_c(G,E):=\cC_c^{\Q_p-\la}(G,E)^*$. By the dicussion in \cite[\S~2]{ST-dual}, the right (resp.  left) translation of $G$ on $\cC_c^{\Q_p-\la}(G,E)$ induces a separately continuous right (resp. left) $\cD(G,E)$-module structure on $\cD_{c}(G,E)$.  We denote by $\cM_G$ the category of abstract left $\cD(G,E)$-modules. 
\begin{proposition}\label{STdualtran}
Let $M\in \cM_G$ and $V$ be a finite dimensional $G$-representation. There is a natural  isomorphism in $D^b(\cM_G)$:
\begin{equation*}
	\RHom_{\cD(G,E)}(M \otimes_E V, \cD_c(G,E)) \xlongrightarrow{\sim} \RHom_{\cD(G,E)}(M, \cD_c(G,E)) \otimes_E V^{\vee},
\end{equation*}
where the right hand side is equipped with the diagonal action,  the (left) $\cD(G,E)$-action on the left hand side and on the first term of the right hand side is induced via the natural involution from the right $\cD(G,E)$-action on $\cD_c(G,E)$.
\end{proposition}
\begin{proof}
As the functor $-\otimes_EV: \cM_G \ra \cM_G$ is exact and preserves  projective objects, it suffices to show there are natural  isomorphisms of  left  $\cD(G,E)$-modules
\begin{multline}\label{adjun1}
	\Hom_{\cD(G,E)}(M \otimes_E V, \cD_c(G,E))\cong \Hom_{\cD(G,E)}(M, V^{\vee} \otimes_E \cD_c(G,E))\\
	\cong \Hom_{\cD(G,E)}(M, \cD_c(G,E)) \otimes_E V^{\vee}
\end{multline}
where the $\cD(G,E)$-action on the second term is induced via involution from its right action on $\cD_c(G,E)$ and the trivial action on $V^{\vee}$ (the other actions being the same as in the proposition).
Indeed the first isomorphism of $E$-vector spaces follows from \cite[Thm.~6.3.1]{AS}. It sends  $f$ to the composition $M \ra M \otimes_E V \otimes_E V^{\vee} \xrightarrow{f \otimes \id} \cD_c(G,E) \otimes V^{\vee}$. It is clear that the isomorphism is $\cD(G,E)$-equivariant. The map $\cC^{\Q_p-\la}_c(G,V)\ra \cC^{\Q_p-\la}_c(G,V)$, $f\mapsto [g \mapsto g(f(g))]$ is a topological isomorphism, which is $G$-equivariant if  $G$ acts on the first term via $(gf)(h)=gf(hg)$ and on the second term via $(gf) (h)=f(hg)$, and is also $G$-equivariant if $G$ acts on the first term via $(gf)(h)=f(g^{-1} h)$ and on the second term via $(gf)(h)=gf(g^{-1}h)$. Taking dual we get an isomorphism 
\begin{equation}
	j: \cD_c(G,E)\otimes_E V^{\vee} \xlongrightarrow{\sim} \cD_c(G,E)\otimes_E V^{\vee}
\end{equation}which, by the above discussion, is an isomorphism of $\cD(G,E)$-bi-modules: $j((X\mu Y^{\iota})\otimes (Y^{\iota}v))=Xj(\mu \otimes v)Y^{\iota}$, where $X, Y\in \cD(G,E)$, $Y^{\iota}$ denotes the involution of $Y$, and on the right hand side $X$ acts diagonally and $Y^{\iota}$ only acts on $\cD_c(G,E)$ by right multiplication.  It then  induces the second $\cD(G,E)$-equivariant isomorphism in (\ref{adjun1}).  The proposition follows.
%
\end{proof}
By the definition of $\cD(G,E)$-action, it is clear that if $\cZ_K$ acts on $M$ via $\chi_{\mu_1}$ for an integral weight $\mu_1$, then $\cZ_K$ acts on $\Ext^i_{\cD(G,E)}(M,\cD_c(G,E))$ via $\chi_{\mu_1^*}$. Proposition \ref{STdualtran} implies: 
\begin{corollary}\label{corSTdualtran}Let $\mu_1, \mu_2$ be two integral weights, and assume $\cZ_K$ acts on $M$ via $\chi_{\mu_1}$. Then we have
\begin{equation*}
	\Ext^i_{\cD(G,E)}(T_{\mu_1}^{\mu_2} M, \cD_c(G,E)) \cong  T_{\mu_1^*}^{\mu_2^*}\big(\Ext^i_{\cD(G,E)}(M, \cD_c(G,E))\big).
\end{equation*}
\end{corollary}
\begin{proof}
One just needs to show that if $\nu$ is the dominant weight in the $\sW_K$-orbit of $\mu_2-\mu_1$, then $\nu^*$ is the dominant weight in the $\sW_K$-orbit of $\mu_2^*-\mu_1^*$. But this is clear. 
\end{proof}
For a left $\cD(H,E)$-module $M$, the \textit{grade} of $M$ is defined by  $$j_{\cD(H,E)}(M):= \min\{l \geq 0\ |\ \Ext_{\cD(H,E)}^l(M, \cD(H,E))\neq 0\}.$$ Recall that $M$ is called \textit{pure}, if $\Ext^l_{\cD(H,E)}(\Ext^l_{\cD(H,E)}(M, \cD(H,E)))=0$ for any $l\neq j_{\cD(H,E)}(M)$. And $M$ is called \textit{Cohen-Macaulay}, if $\Ext^l_{\cD(H,E)}(M,\cD(H,E))\neq 0$ if and only if $l=j_{\cD(H,E)}(M)$. If $M$ is coadmissible, by \cite[Thm.~8.9]{ST03}, $j_{\cD(H,E)}(M)\leq n^2 d_K$ if $M\neq 0$. We define $\dim M:=n^2 d_K -j_{\cD(H,E)}(M)$. By \cite[Prop.~8.7]{ST03}, the coadmissible $M$ admits a natural dimension filtration (similarly as finitely generated modules over a noetherian Auslander regular ring).
By Proposition \ref{STdualtran} and \cite[Prop.~2.3]{ST-dual}, we have:
\begin{corollary}\label{transpure}
Let $M\in \cM_G$ be a coadmissible $\cD(H,E)$-module, and $V$ be a finite dimensional  representation of  $G$, then $\dim M \otimes_E V=\dim M$. Moreover, $M$ is pure (resp. Cohen-Macaulay) if and only if $M\otimes_E V$ is pure (resp. Cohen-Macaulay). 
\end{corollary}

Now we consider the topological dual. Let $V$ a a locally $\Q_p$-analytic representation of $G$ on space of compact type. Let $\mu_1$, $\mu_2$ be two integral weights. Suppose $\cZ_K$ acts on $V$ via $\chi_{\mu_1}$, which implies that $\cZ_K$ acts on $V^*$ via $\chi_{\mu_1^*}$. 

\begin{lemma}\label{1dual} There is  a natural isomorphism $T_{\mu_1^*}^{\mu_2^*} V^*\cong  (T_{\mu_1}^{\mu_2} V)^*$.
\end{lemma} 
\begin{proof}Let $\nu$ be the dominant weight in the class $\{w(\mu_2-\mu_1)\}_{w\in \sW_K}$. Let $e_1, \cdots, e_m$ be a basis of $L(\mu)$, and $e_1^*, \cdots, e_m^*$ be the dual basis of $L(\mu)^{\vee}\cong L(\mu^*)$. 
Consider the map $\Hom_E^{\cont}(V \otimes_E L(\mu), E)\cong \Hom_E^{\cont}(V, E) \otimes_E L(\mu^*)$, $f\mapsto \sum_{i=1}^m f_i \otimes e_i^*$ with $f(\sum_{i=1}^m v_i \otimes e_i)=\sum_{i=1}^m f_i(v_i)$. The map is clearly a topological isomorphism. It is straightforward to check it is also $G$-equivariant. By comparing the generalized eigenspaces for $\cZ_K$, the lemma follows.
\end{proof}
\subsubsection{Singular skeleton}
Let $\rho: \Gal_K \ra \GL_n(E)$ be a continuous representation. We assume $\rho$ has  integral Sen weights $\textbf{h}=(h_{i,\sigma})_{\substack{i=1,\cdots, n\\ \sigma\in \Sigma_K}}$ ($\textbf{h}$ being dominant).  
Let $\widehat{\pi}(\rho)$ be the unitary Banach $\GL_n(K)$-representation associated to $\rho$  in \cite{CEGGPS1} (see \cite[\S~2.10, \S~6]{CEGGPS1}). We refer to \textit{loc. cit.} for details. But note the construction of $\widehat{\pi}(\rho)$ depends on many auxiliary data, and it is even not clear if it is always non-zero. We assume  however $\widehat{\pi}(\rho)\neq 0$. For example, this is the case if $n=2$ and $K/\Q_p$ is unramified by \cite[Thm.~1.3]{BHHMS1} (under mild assumptions). Remark that when $K=\Q_p$, $\widehat{\pi}(\rho)$ coincides with the $\GL_2(\Q_p)$-representation associated to $\rho$ via the $p$-adic Langlands correspondence (cf. \cite{CEGGPS2}). We use this fact without further mention. Let $D:=D_{\rig}(\rho)$ be the associated $(\varphi, \Gamma)$-module (of rank $n$) over $\cR_{K,E}$, and  $\pi(D)$ be the locally $\Q_p$-analytic vectors of $\widehat{\pi}(\rho)$. By \cite[Thm.~7.1 (i)]{ST}, $\pi(D)$ is dense in $\widehat{\pi}(\rho)$.

Let $\lambda:=\textbf{h}-\theta_K$ (hence $\lambda_{i,\sigma}=h_{i,\sigma}-n+i$).  Note that $\lambda$ is dominant if and only if  $\text{h}$ is strictly dominant.  Note also $w_0\cdot \lambda$ is antidominant (in the sense of \cite{Hum08}).  Recall the following theorem of \cite{DPS}:
\begin{theorem}$\cZ_K$ acts on $\pi(D)$ via the character $\chi_{\lambda}$.
\end{theorem}
\begin{remark}\label{Rcenc}
For a $(\varphi, \Gamma)$-module $D'$ of rank $n$ over $\cR_{K,E}$, let $\delta_{D'}: K^{\times} \ra E$ be the character such that $\cR_{K,E}(\delta_{D'} \varepsilon^{\frac{n(n-1)}{2}}) = \wedge^n D'$.	By similar (and easier) density arguments as in  \cite{DPS},  the centre $Z(K)$ of $\GL_n(K)$ acts on $\pi(D)$ via the character $\delta_D$.
\end{remark}

Let $\Delta$ be the unique$(\varphi, \Gamma)$-module of constant weight $0$ over $\cR_{K,E}$ such that $\Delta[\frac{1}{t}] \cong D[\frac{1}{t}]$ (Lemma \ref{pDE0}). Throughout the paper, we assume $\Delta$ has distinct irreducible constituents. Note that $-\theta_K-w_0 \cdot \lambda=-w_0(\textbf{h})=(-h_{n+1-i,\sigma})_{\substack{i=1,\cdots,n\\ \sigma\in \Sigma_K}}$ is dominant. 
\begin{conjecture}
The $\GL_n(K)$-representation $T_{w_0\cdot \lambda}^{-\theta_K} \pi(D)\cong (\pi(D) \otimes_E L(-w_0(\textbf{h}))[\cZ_K=\chi_{-\theta_K}]$ depends only on $\Delta$. 
\end{conjecture}	
We denote $\pi(\Delta):=(\pi(D) \otimes_E L(-w_0(\textbf{h})))[\cZ_K=\chi_{-\theta_K}]$ (but noting the construction of $\pi(\Delta)$ is global, and depends \textit{a priori} on a choice of Galois representations). We have by \cite[Thm.~1.1 (1)]{Ding14} (for $D$ indecomposable case) and \cite[Thm.~3.8]{Ding14} (for splitting  $D$): \footnote{As we assume $\Delta$ has distinct irreducible constituents, $D$ does not contain pathological submodules in the sense of \cite{Colm18}.}
\begin{theorem}\label{Tgl2qp1}
The conjecture holds for $\GL_2(\Q_p)$. In fact, $\pi(\Delta)$ is the locally analytic representation associated to $\Delta$ via the $p$-adic Langlands correspondence.\footnote{If $\Delta\cong \cR_E(\phi_1) \oplus \cR_E(\phi_2)$, we put $\pi(\Delta):= (\Ind_{B^-(\Q_p)}^{\GL_2(\Q_p)} (\phi_1\varepsilon^{-1}) \otimes \phi_2 )^{\an} \oplus (\Ind_{B^-(\Q_p)}^{\GL_2(\Q_p)} (\phi_2 \varepsilon^{-1}) \otimes \phi_1 )^{\an}. $}
\end{theorem}
\begin{remark}\label{remsin}
(1) If $\rho$ itself has constant Sen weight $0$, then $\pi(\Delta)=\pi(D)$. 


(2) For any irreducible $\Delta$ (of constant weights $0$), there exists a smooth character of $K^{\times}$ such that $\Delta \otimes_{\cR_{K,E}} \cR_{K,E}(\phi)$ is \'etale hence isomorphic to $D_{\rig}(\rho)$ for some $\Gal_K$-representation $\rho$. We put in this case
\begin{equation*}
	\pi(\Delta):=\pi(D) \otimes_E \phi \circ \dett.
\end{equation*}
In particular, we have  a candidate $\pi(\Delta)$  for any such $\Delta$. If $\Delta$ is of rank $1$ hence isomorphic to $\cR_{K,E}(\phi)$ for some $\phi$, we have $\pi(\Delta)\cong \phi$.


(3) By Remark \ref{Rcenc}, $\pi(\Delta)$ has central character $\delta_{\Delta}$.
\end{remark}
\subsection{Local-global compatibility conjectures}\label{S3.2}We discuss the internal structure of $\pi(\Delta)$. In particular, we conjecture it is compatible with certain algebraic representation of $\GL_n$ (and the internal structure of $\Delta$),  which is completely analogous to the mod $p$ conjecture of \cite{BHHMS2}. 

\subsubsection{Good subquotients of $L^{\otimes}$}


For $\alpha=e_i-e_{i+1}\in S$, let $\lambda_{\alpha}:=e_1+ \cdots +e_i$ which is  a fundamental weight  of $\GL_n$. Put
\begin{equation}\label{funalg}
L^{\otimes}:=\otimes_{\sigma\in \Sigma_K}(\otimes_{\alpha\in S} L(\lambda_{\alpha})_{\sigma})
\end{equation}
that is the tensor product of all fundamental representations of  $\GL_n^{\Gal(K/\Q_p)}$.  In this section, we briefly recall good subquotients of $L^{\otimes}$ introduced and studied in \cite[\S~2.2]{BHHMS2}.

Let $P$ be a standard parabolic subgroup of $\GL_n$. For $\alpha=e_i-e_{i+1}\in S(P)$, put $\lambda_{\alpha,P}:=\sum_{e_j-e_{i+1}\in R(P)^+}e_j$, which are fundamental weights for $M_P$ such that $\langle \alpha, \beta\rangle\leq 0$ for $\beta\in S\setminus S(P)$. Let $\theta_P:=\sum_{\alpha \in S(P)} \lambda_{\alpha}$, and $\theta^P:=\theta_{G}-\theta_P$ (noting $\theta_G=\sum_{\alpha\in S} \lambda_{\alpha}$).  By the discussion below \cite[(45)]{BHHMS2}, $\theta^P$ naturally extends to an element $\Hom_{\rm{gr}}(M_P, \bG_m)$. If $M_P\cong M_1 \times \cdots M_d$ with $M_i\cong \GL_{n_i}$, we denote by $(\theta^P)_i$ the corresponding character of $M_i$. 

Consider $L^{\otimes}|_{Z_{M_P}}$ (on which $Z_{M_P}$ acts via the diagonal map $Z_{M_P}\hookrightarrow Z_{M_P}^{\Gal(K/\Q_p)}$ ). For an isotypic component $C_P$ of $L^{\otimes}|_{Z_{M_P}}$, one can associate as in \cite[\S~2.2.2]{BHHMS2} a  standard parabolic subgroup $P(C_P)$ and a (non-empty) set $\sW(C_P)\subset \{w\in \sW\ |\ w(S(P))\subset S\}$ (cf. \cite[(39)]{BHHMS2}).  We refer the reader to \textit{loc. cit.} for the precise definition. Let  $w\in \sW$ such that $w(S(P))\subset S$, denote by $^wP$ the standard parabolic subgroup associated to the subset $w(S(P))\subset S$. Remark that for $w\in \sW(C_P)$, we have $^w P\subset P(C_P)$ (cf. \cite[(40)]{BHHMS2}, in particular, $\# S(P(C_P))\geq \#S(P)$), and if $^wP=P(C_P)$, then $\sW(C_P)=\{w\}$ (cf. \cite[Lem.~2.2.3.1]{BHHMS2}).

Note that $C_P$ inherits a natural action of $M_P^{\Gal(K/\Q_p)}$ (\cite[Lem.\ 2.2.1.2]{BHHMS2}), and we refer to \cite[Thm. 2.2.3.9]{BHHMS2} for a description of the corresponding (algebraic) representation. For $w\in \sW$ such that $w(S(P))\subset S$, and an algebraic representation $R$ of $M_P^{\Gal(K/\Q_p)}$, let $w(R)$ be the algebraic representation of $M_{^wP}^{\Gal(K/\Q_p)}$ with $w(R)(g)=R(w^{-1} g w)$ (noting $M_{^w P}=w M_P w^{-1}$). 
The image of $C_P$ under the diagonal action of $w$ on $L^{\otimes}$ is isomorphic to $w(C_P)$ as representation of $M_{^w P}^{\Gal(K/\Q_p)}$.  Write $M_{P(C_P)}=M_1 \times \cdots M_d$ with $M_i \cong \GL_{n_i}$ and let $L_i^{\otimes}$ be the algebraic representation of $M_i$ defined as in  (\ref{funalg}) with $\GL_n$ replaced by $M_i$. For $w\in \sW(C_P)$, $^w P \cap M_{P(C_P)}$ is a standard parabolic subgroup of $M_{P(C_P)}$ hence has the form $\prod_{i=1}^d (^wP)_i\subset \prod_{i=1}^d M_i$. By \cite[Thm.~2.2.3.9]{BHHMS2},  there exist an isotypic component $C_{w,i}$ of $L_i^{\otimes}|_{Z_{M_{(^{w} P)_i}}}$ as algebraic representation of $M_{(^{w} P)_i}^{\Gal(K/\Q_p)}$ for $i=1,\cdots, d$, such that 
\begin{equation}\label{wCP}
w(C_P)=\otimes_{i=1}^d (C_{w,i} \otimes (\underbrace{(\theta^{P(C_P)})_i \otimes \cdots \otimes (\theta^{P(C_P)})_i}_{\Gal(K/\Q_p)}))
\end{equation}
as algebraic representation of $M_{^{w} P}^{\Gal(K/\Q_p)}=\prod_i M_{(^{w} P)_i}^{\Gal(K/\Q_p)}$.  


Let $\widetilde{P}$ be a Zariski closed algebraic subgroup of $P$ containing $M_P$, let $X\supset R(P)^+$ be the associated closed subset of $R^+$. Let 
\begin{equation}\label{wtildP}\sW_{\widetilde{P}}:=\{w\in \sW\ |\ w(S(P))\subset S, w(X\setminus R(P)^+)\subset R^+\}.\end{equation} As in \cite[Def.\ 2.2.1.3]{BHHMS2}, a subquotient (resp. subrepresentation, resp. quotient) of $L^{\otimes}|_{\widetilde{P}^{\Gal(K/\Q_p)}}$ is called \textit{good} if its restriction to $Z_{M_P}$ is a direct sum of the isotypic components of   $L^{\otimes}|_{Z_{M_P}}$. By \cite[Lem.\ 2.2.1.5]{BHHMS2}, there exists a filtration on $L^{\otimes} |_{\widetilde{P}^{\Gal(K/\Q_p)}}$ by good subrepresentations such that the graded pieces exhaust the isotypic components of $L^{\otimes}|_{Z_{M_P}}$. On the graded pieces, the $\widetilde{P}^{\Gal(K/\Q_p)}$-action factors through the natural action of $M_P^{\Gal(K/\Q_p)}$ via  $\widetilde{P}^{\Gal(K/\Q_p)} \twoheadrightarrow M_P^{\Gal(K/\Q_p)}$.

\subsubsection{Compatibility of $\pi(\Delta)$ and $\Delta$}
Imitating \cite[\S~2.4]{BHHMS2}, we define the notion of compatibility with $\Delta$ for  locally $\Q_p$-analytic representations of $\GL_n(K)$.

Let $\pi$ be a finite length admissible locally $\Q_p$-analytic representation of $\GL_n(K)$ over $E$. 
Let $P$ be a standard parabolic subgroup of $\GL_n$, $\widetilde{P}$ be a Zariski-closed subgroup of $P$ containing $M_P$. 
Suppose that there exists   a bijection $\Phi$ of partially ordered sets from the set of subrepresentations  of $\pi$ to the set of good subrepresentations of $L^{\otimes}|_{\widetilde{P}^{\Gal(L/\Q_p)}}$ (both ordered by inclusions). Note that $\Phi$ induces a bijection, stilled denoted by $\Phi$, from the set of subquotients  of $\pi$  to the set of direct sums of isotypic components of $L^{\otimes}|_{Z_{M_P}}$ (by \cite[Lem. 2.2.1.5]{BHHMS2}). Let  $w\in \sW_{\widetilde{P}}$ (cf. (\ref{wtildP})), hence  $w \widetilde{P} w^{-1}$ is a Zariski closed subgroup of $^w P$ (cf. \cite[Lem. 2.3.1.7]{BHHMS2}). We define $w(\Phi)$ to be the (bijective) map from  the set of subrepresentations  of $\pi$ to the set of good subrepresentations (with respect to $^wP$) of $L^{\otimes}|_{(w \widetilde{P} w^{-1})^{\Gal(L/\Q_p)}}$ sending $\pi$ to $w (\Phi(\pi))$, where $w(\Phi(\pi))(g)=\Phi(\pi)(w^{-1} g w)$ for $g\in w \widetilde{P} w^{-1}$.

\paragraph{Compatible with $\widetilde{P}$}Now we define the notion of compatibility with $\widetilde{P}$. Roughly speaking, this amounts to saying that $\pi$ admits a filtration, indexed and ordered by good subrepresentations of $L^{\otimes}|_{\widetilde{P}^{\Gal(K/\Q_p)}}$ whose graded pieces have similar  symmetric and ``parabolic" structure as good subquotients of $L^{\otimes}|_{\widetilde{P}^{\Gal(K/\Q_p)}}$. 
Precisely, we call $\pi$ is \textit{compatible with $\widetilde{P}$} if there exists a bijection $\Phi$ of partially ordered sets from  the set of subrepresentations  of $\pi$ to the set of good subrepresentations of $L^{\otimes}|_{\widetilde{P}^{\Gal(L/\Q_p)}}$ (ordered by inclusions)  such that for any $w_{\widetilde{P}}\in \sW_{\widetilde{P}}$, any standard parabolic subgroup $Q$ containing $^{w_{\widetilde{P}}} P$ and any isotypic component $C_Q$ of $L^{\otimes}|_{Z_{M_Q}}$, writing $M_{P(C_Q)}=M_1 \times \cdots \times M_d$ with $M_i \cong \GL_{n_i}$, the followings hold:

(1) Any $w_{\widetilde{P}}(\Phi)^{-1}(C_Q)$ has the  form (recalling $\theta^{P(C_Q)}$ here is a character $M_{P(C_Q)}(K) \ra K^{\times}$) $$w_{\widetilde{P}}(\Phi)^{-1}(C_Q)\cong \big(\Ind_{P(C_Q)^-(K)}^{\GL_n(K)} \pi(C_Q) \otimes_E \varepsilon^{-1} \circ \theta^{P(C_Q)}\big)^{\Q_p-\an}$$
where $\pi(C_Q)$ is a $M_{P(C_Q)}$-representation of the form $\pi(C_Q)\cong \pi_1(C_Q) \otimes \cdots \otimes \pi_d(C_Q)$ for some finite length admissible locally $\Q_p$-analytic representations $\pi_i(C_Q)$ of $M_k(K)$ over $E$. 

(2) For any $w\in \sW$ such that $w(S(P(C_Q)))\subset S$, let $\pi(C_Q)$ be as in (1), and  $w(\pi(C_Q))$ be the representation of $M_{^{w} P(C_Q)}(K)=wM_{P(C_Q)}(K)w^{-1}$ defined by $w(\pi(C_Q))(g):=\pi(C_Q)(w^{-1} g w)$. Then $\pi(w(C_Q))\cong w(\pi(C_Q))$ where $\pi(w(C_Q))$ is given as in (i) for the isotypic component $w(C_Q)$ of $L^{\otimes}|_{Z_{M_{^w P(C_Q)}}}$.

(3) Let $w_Q\in \sW(C_Q)$. As in (\ref{wCP}),  there exist an isotypic component $C_{w_Q,i}$ of $L_i^{\otimes}|_{Z_{M_{(^{w_Q} Q)_i}}}$ as algebraic representation of $M_{(^{w_Q} Q)_i}^{\Gal(K/\Q_p)}$ such that 
\begin{equation*}
w_Q(C_Q)=\otimes_{i=1}^d (C_{w_Q,i} \otimes (\underbrace{(\theta^{P(C_P)})_i \otimes \cdots \otimes (\theta^{P(C_P)})_i}_{\Gal(K/\Q_p)}))
\end{equation*}
as algebraic representation of $M_{^{w_Q} Q}^{\Gal(K/\Q_p)}=\prod_i M_{(^{w_Q} Q)_i}^{\Gal(K/\Q_p)}$. The bijection $\Phi$ satisfies that the restriction of $w_{\widetilde{P}}(\Phi)$ to the set of $w_{\widetilde{P}}(\Phi)^{-1}(C_Q)$ comes from $d$ bijections $w_{\widetilde{P}}(\Phi)_{w_Q,i}$ of partially ordered sets between the set of $M_i(K)$-subrepresentations of $\pi_i(C_Q)$ (see (1)) and the set of good subrepresentations of $C_{w_Q,i}|_{(^{w_Qw_{\widetilde{P}}} \widetilde{P})_{Q,i}^{\Gal(K/\Q_p)}}$ in the following sense (noting $^{w_Q w_{\widetilde{P}}} \widetilde{P}\subset {^{w_Q w_{\widetilde{P}}}P} \subset {^{w_Q} Q} \subset P(C_Q)$):  for any good subquotient $C_Q'$ of $C_Q|_{(w_{\widetilde{P}}\widetilde{P} (w_{\widetilde{P}})^{-1})^{\Gal(K/\Q_p)}}$, $w_{\widetilde{P}}(\Phi)^{-1}(C_Q')$ has the form
\begin{equation*}
w_{\widetilde{P}}(\Phi)^{-1}(C_Q')\cong \big(\Ind_{P(C_Q)^-(K)}^{\GL_n(K)} (\pi_1' \otimes \cdots \otimes \pi_d') \otimes_E (\varepsilon^{-1} \circ z^{\theta^{P(C_Q)}})\big)^{\Q_p-\an}
\end{equation*}
with $\pi_i'\in \cS_i$ and $C_Q'$ corresponds to the following algebraic representation of $(^{w_Qw_{\widetilde{P}}} \widetilde{P})_Q^{\Gal(K/\Q_p)}=\prod_{i=1}^d (^{w_Q w_{\widetilde{P}}} \widetilde{P})_{Q,i}^{\Gal(K/\Q_p)}$:
\begin{equation*}
\otimes_{i=1}^d \big(w(\Phi)_{w_Q,i}(\pi_i') \otimes \underbrace{(\theta^{P(C_Q)})_i \otimes \cdots \otimes (\theta^{P(C_Q)})_i}_{\Gal(K/\Q_p)}\big).
\end{equation*}

(4) For each isotypic component $C_P$ of $L^{\otimes}|_{Z_{M_P}}$, the $M_{P(C_P)}$-representation $\pi(C_P)$ is topologically irreducible and supersingular, i.e. is not isomorphic to any subquotient of a parabolic induced representation for some proper parabolic subgroups of $M_{P(C_P)}$.

\paragraph{Compatible with $\Delta$} Let $\Delta$ be a $(\varphi, \Gamma)$-module of constant Sen weight $0$ of rank $n$ over $\cR_{K,E}$ with pairwise distinct irreducible constituents. Let $\sF$ be a minimal filtration of $\Delta$ and $P:=P_{\sF}$ be the standard parabolic subgroup associated to $\sF$. Let $\widetilde{P} \subset P$ be the Zariski closed subgroup associated to $C_{\sF}$ (cf. Lemma \ref{rootCF}).  We call $\pi$ is \textit{weakly compatible} with $\Delta$, if $\pi$ is compatible with $\widetilde{P}$ via a bijection $\Phi$ from  the set of subrepresentations  of $\pi$ to the set of  good subrepresentations of $L^{\otimes}|_{Z_{M_P}}$ which satisfies moreover the following properties. Let $w_{\widetilde{P}}$ and $C_Q$ be as in (1) of the precedent paragraph, and let $w_Q\in \sW(C_Q)$. For $i=1, \cdots, d$, we have: 

(1a)  If $(^{w_Q w_{\widetilde{P}}} \widetilde{P})_{Q,i}=M_i$, there exists hence $i_0\in \{1,\cdots, k\}$, such that $w_{Q}w_{\widetilde{P}}$ sends the simple roots of $(M_P)_{i_0}$ to the simple roots of $M_i$.  Then $\pi_i(C_Q)\cong \pi(\Delta_{i_0})$ (cf. Remark \ref{remsin} (3)).

(1b) If $K=\Q_p$, $n_i=2$, and $(^{w_Q w_{\widetilde{P}}} \widetilde{P})_{Q,i}\neq M_i$ (which has to be a Borel subgroup of $\GL_2\cong M_i$), writing  $s_i=\sum_{j=0}^{i-1} n_j$, we have $\rk \Delta_{(w_Qw_{\widetilde{P}} )^{-1}(s_i)}=\rk \Delta_{(w_Qw_{\widetilde{P}})^{-1} (s_i+1)}=1$ and there exists a (unique) subquotient $\Delta_{w_Q,i}$ of $\Delta$  with  irreducible constituents given by $\Delta_{(w_Qw_{\widetilde{P}} )^{-1}(s_i)}$ and  $\Delta_{(w_Qw_{\widetilde{P}} )^{-1}(s_i+1)}$. In this case, we assume $\pi_i(C_Q)\cong \pi(\Delta_{w_Q,i})$ (cf. Remark \ref{remsin} (1)). 

Note that by Lemma \ref{refine0}, the compatibility  conditions do not depend on the choice of the minimal filtration $\sF$. 

\begin{conjecture}\label{conjLG}The representation $\pi(\Delta)$ has finite length and is weakly compatible with $\Delta$.
\end{conjecture}
\begin{theorem}
Conjecture \ref{conjLG} holds for $\GL_2(\Q_p)$. 
\end{theorem}
\begin{proof}
By \cite{Ding14}, $\pi(\Delta)$ is no other than the  locally analytic representation associated to $\Delta$. The crystabelline case follows from the footnote in Theorem \ref{Tgl2qp1}.  By \cite{Colm16}, if $\Delta$ is isomorphic to a non-split extension $[\cR_E(\phi_1) \rule[2.5pt]{10pt}{0.5 pt}  \cR_E(\chi_2)]$, then $\pi(\Delta)$ is isomorphic to a non-split extension of the form $\big[(\Ind_{B^-}^{\GL_2} (\phi_1 \varepsilon^{-1}) \otimes \phi_2)^{\Q_p-\an} \rule[2.5pt]{10pt}{0.5 pt}\ (\Ind_{B^-}^{\GL_2} (\phi_2 \varepsilon^{-1}) \otimes \phi_1)^{\Q_p-\an}\big]$, hence (weakly) compatible with $\Delta$. If $\Delta$ is irreducible, by \cite[Thm.~2.16]{Colm18} $\pi(\Delta)$ is topologically irreducible, hence also compatible with $\Delta$. 
\end{proof}
\begin{remark}

(1)	The conjecture is inspired by \cite[Conj.~2.5.1]{BHHMS2}. In fact, if $\Delta$ and all its irreducible constituents $\Delta_i$ are \'etale, then $\pi(\Delta)$ (resp. $\pi(\Delta_i)$) is  the locally $\Q_p$-analytic vectors of the unitary Banach representation $\widehat{\pi}(\Delta)$ (resp. $\widehat{\pi}(\Delta_i)$). In this case, Conjecture \ref{conjLG} is compatible with (the Banach version of) \cite[Conj.~2.5.1]{BHHMS2}.

(2) For $\GL_2(\Q_p)$, one can also prove the finite length property of $\pi(\Delta)$ without using $(\varphi, \Gamma)$-modules. Indeed, by \cite[Thm.~1.5]{DPS} (or  Corollary \ref{C-CMdual}), $\dim \pi(D)^*\leq 1$ hence $\dim \pi(\Delta)^* \leq 1$. As $\pi(\Delta)$ can't have non-zero locally algebaric subquotients (by looking at the $\cZ$-action), $\pi(\Delta)$ has finite length. 


(3) 
By the compatibility condition (1), for an isotypic component $C_P$ of $L^{\otimes}|_{Z_{M_P}}$, we have 
\begin{equation}\label{Eparirr}
	\Phi^{-1}(C_P)\cong \big(\Ind_{P(C_P)^-(K)}^{\GL_n(K)} \pi(C_P) \otimes_E \varepsilon^{-1} \circ \theta^{P(C_P)}\big)^{\Q_p-\an}.
\end{equation}
As $\cZ_K$ acts on $\Phi^{-1}(C_P)$ via $\chi_{-\theta_K}$, we deduce $\cZ_{M_{P(C_P),K}}$ acts on $\pi(C_P)$ via $\chi_{M_{P(C_P)},-\theta_{P(C_P),K}}$ (e.g. by \cite[Thm.~2.6]{CaO}). This is compatible with condition (1a).  The conjecture also tacitly implies the parabolic induction in (\ref{Eparirr}) is topologically irreducible. One may expect this always holds:  if $\pi_{M_P}$ is a topologically irreducible locally $\Q_p$-anayltic representation of $M_P(K)$ on which $\cZ_{\fm_{\fp},K}$ acts by $\chi_{M_P,-\theta_K}$, then $(\Ind_{P^-(K)}^{\GL_n(K)} \pi_{M_P})^{\Q_p-\an}$ would be topologically irreducible. For example, it holds when $P=B$ by \cite{OS}.

(4)	For a subquotient $\Delta'$ of $\Delta$, it could happen that any $(\varphi, \Gamma)$-module $D$ with $D[\frac{1}{t}] \cong \Delta'[\frac{1}{t}]$ is not isocline. 
If it happens, it is not clear to me how to associate to $\Delta'$ a candidate $\pi(\Delta')$ (in contrary to Remark \ref{remsin} (2)) expect when $\rk \Delta'=2$ and $K=\Q_p$ (see \cite{Colm16}). So we only consider the compatibility   between $\pi(\Delta)$  and the irreducible constituents of $\Delta$ (in (1a)), and that  between $\pi(\Delta)$ and rank two subquotients of $\Delta$ when $K=\Q_p$ (in (1b)). This is one of the reasons that we call such conditions  weak  compatibility.

(5) The compatible conditions in \cite[Def.~2.4.2.7]{BHHMS2} are formulated using the  Colmez-Breuil functors (\cite{Br15}) from smooth mod $p$ representations of $\GL_n(K)$ to mod $p$ (pro-)$(\varphi, \Gamma)$-modules . One may expect an analogue of such functors in the locally analytic setting. 
\end{remark}

\begin{example}\label{ExFS}
We give some examples on the conjectural structure of $\pi(\Delta)$. The pattern is  exactly the same as the mod $p$ setting in \cite[\S~2.4]{BHHMS2}, and we refer to \textit{loc. cit.} for more examples and details.  For a smooth character $\phi$ of $T(K)$, put $\jmath(\phi):=\otimes_{i=1}^n \phi_i |\cdot|_K^{i-n}$.

(1) Let $n=2$ and suppose $\Delta$ is an extension of $\cR_{K,E}(\phi_2)$ by $\cR_{K,E}(\phi_1)$. In this case $P=B$, and $L^{\otimes}|_{T}$ has the form $C_0 \rule[2.5pt]{10pt}{0.5 pt} C_1 \rule[2.5pt]{10pt}{0.5 pt} \cdots \rule[2.5pt]{10pt}{0.5 pt} C_{d_K}$, where $C_i$ are isotypic components of $L^{\otimes}|_{T}$, $P(C_i)=\begin{cases}
	B & i=0, d_K \\
	\GL_2 & i=1, \cdots, d_K-1
\end{cases}$.
If $\Delta$ is non-split,  $\pi(\Delta)$ has the following form: 
\begin{equation*}
	\cF_{B^-}^{\GL_2}(L^-(\theta_K), \jmath(\phi_1 \otimes \phi_2)) =\pi_0\   
	\rule[2.5pt]{10pt}{0.5 pt}\  \pi_1 \  \rule[2.5pt]{10pt}{0.5 pt}\  \cdots \rule[2.5pt]{10pt}{0.5 pt} \  \pi_{d_K}=\cF_{B^-}^{\GL_2}(L^-(\theta_K),\jmath(\phi_2 \otimes \phi_1))
\end{equation*}
where $\pi_i=\Phi^{-1}(C_i)$.   
If $\Delta$ splits or equivalently, $\Delta$ is crystabelline, then $\pi(\Delta)$ has the form: 	$$\cF_{B^-}^{\GL_2}(L^-(\theta_K),\jmath(\phi_1 \otimes \phi_2)) \oplus \cF_{B^-}^{\GL_2}(L^-(\theta_K), \jmath(\phi_2 \otimes \phi_1)) \oplus \bigoplus_{i=1}^{d_K-1} \pi_i.$$
We may expect these $\pi_i$ only depend on $\Delta^{\sss}$.

(2) Let $n=3$, $K=\Q_p$ and  $\Delta$ be trianguline, given by an successive extension of $\cR_{K,E}(\phi_i)$. In this case, $L^{\otimes}|_T$ has $7$ isotypic components: there is a unique one $C_{\SSS}$ such that $P(C_{\SSS})=\GL_3$ and for the others $C_i$, $P(C_i)=B$. If  $\Delta$ admits a unique refinement, then $\pi(\Delta)$ has the form (we omit $L^-(\theta_K)$ in $\cF_{B^-}^{\GL_3}(-)$)
\begin{equation*}\footnotesize
	\begindc{\commdiag}[200]
	\obj(0,1)[a]{$\cF_{B^-}^{\GL_3}(\jmath(\phi))$}
	\obj(4,0)[b]{$\cF_{B^-}^{\GL_3}(\jmath(s_1(\phi)))$}
	\obj(4,2)[c]{$\cF_{B^-}^{\GL_3}(\jmath(s_2(\phi)))$}
	\obj(8,1)[d]{$\Phi^{-1}(C_{\SSS})$}
	\obj(12,0)[e]{$\cF_{B^-}^{\GL_3}(\jmath(s_1s_2(\phi)))$}
	\obj(12,2)[f]{$\cF_{B^-}^{\GL_3}(\jmath(s_2s_1(\phi)))$}	
	\obj(16,1)[g]{$\cF_{B^-}^{\GL_3}( \jmath(w_0\phi))$}
	\mor{a}{b}{}[+1,\solidline]
	\mor{a}{c}{}[+1,\solidline]
	\mor{a}{c}{}[+1,\solidline]
	\mor{c}{d}{}[+1,\solidline]
	\mor{b}{d}{}[+1,\solidline]
	\mor{d}{e}{}[+1,\solidline]
	\mor{d}{f}{}[+1,\solidline]
	\mor{f}{g}{}[+1,\solidline]
	\mor{e}{g}{}[+1,\solidline]
	\enddc.
\end{equation*}	
Another extreme case is when $\Delta\cong \oplus_{i=1}^3 \cR_{K,E}(\phi_i)$, or equivalently, $\Delta$ is crystabelline. In this case,  `$\pi(\Delta)$ has the form $ \oplus_{w\in \sW} \cF_{B^-}^{\GL_3}(L^-(\theta_K), \jmath(w(\phi))) \oplus \Phi^{-1}(C_{\SSS})$. Again one may expect $\Phi^{-1}(C_{\SSS})$ only depends on $\{\phi_i\}$.

(3) We finally give a non-trianguline example. Suppose $n=3$, $K=\Q_p$, and $\Delta$ is a non-split extension of a rank one $\cR_E(\phi)$ by an irreducible rank two module $\Delta_1$. Let $P_1=\begin{pmatrix}
	\GL_2 & * \\ 0 & \GL_1
\end{pmatrix}$, $P_2=\begin{pmatrix}
	\GL_1 & * \\ 0 & \GL_2
\end{pmatrix}$.Then $L^{\otimes}|_{P_1}$ has the form $C_1 \rule[2.5pt]{10pt}{0.5 pt} C_{\SSS} \rule[2.5pt]{10pt}{0.5 pt} C_2$, where $P(C_i)=P_i$, and $P(C_{\SSS})=\GL_3$. Hence $\pi(\Delta)$ should have the following structure:
\begin{equation*}
	\big(\Ind_{P_1^-(\Q_p)}^{\GL_3(\Q_p)} (\phi \varepsilon^{-2}) \otimes \pi(\Delta) \big)^{\an}\   \rule[2.5pt]{10pt}{0.5 pt}\  \Phi^{-1}(C_{\SSS}) \   \rule[2.5pt]{10pt}{0.5 pt}\  	\big(\Ind_{P_2^-(\Q_p)}^{\GL_3(\Q_p)} (\pi(\Delta) \otimes_E \varepsilon^{-1} \circ \dett) \otimes \phi\big)^{\an}.
\end{equation*} 
\end{example}
\subsection{Finite slope part}\label{S3.3}In this section, we consider the case where $\Delta$ is isomorphic to a successive extension of rank one $(\varphi, \Gamma)$-modules. In this case, Conjecture \ref{conjLG} implies that $\pi(\Delta)$ contains a subrepresentation, denoted by $\pi(\Delta)^{\fss}$, whose irreducible constituents are given by locally $\Q_p$-analytic principal series. We give an explicit description of  $\pi(\Delta)^{\fss}$, and show  that in many circumstances that the explicit $\pi(\Delta)^{\fss}$ is indeed a subrepresentaiton of $\pi(\Delta)$ (whose definition \textit{a priori} depends on the global setup and the choice of $\rho$). 


Let $\sF$ be a minimal filtration of $\Delta$, and assume the associate parabolic subgroup $P=B$. There exist \textit{smooth} characters $\phi_i: K^{\times} \ra E^{\times}$ such that $\gr_i \sF\cong \cR_{K,E}(\phi_i)$. Let $C:=C_{\sF}$ and $B_C$ be the associated Zariski closed subgroup of $B$. Put $\lambda_0:=d_K \sum_{\alpha\in S} \lambda_{\alpha}$. As in \cite[Def.~2.2.6]{BH}, a weight in $L^{\otimes}|_T$ is called \textit{ordinary} if it is equal to $w(\lambda_0)$ for some $w\in \sW$.  We have by \cite[Thm.~2.2.4]{BH}:
\begin{theorem}\label{ordweight}
The only weights that occur with multiplicity $1$ in $L^{\otimes}|_T$ are the ordinary weights.
\end{theorem}
\begin{proof}
The case for $K=\Q_p$ is proved in \cite[Thm.~2.2.4]{BH}. From which,  we easily deduce that for general $K$,  if a weight is not ordinary, then it has multiplicity strictly bigger than $1$. As the ordinary weights in $\Q_p$-case  are all extremal, i.e. a highest weight for a certain Borel subgroup, an ordinary weight for general $K$ is also a highest weight. Together with \cite[Lem.~2.2.3]{BH}, it is not difficult to see an ordinary weight (for general $K$) occur with multiplicity $1$. 
\end{proof}
We define $(L^{\otimes}|_{B_C^{\Gal(K/\Q_p)}})^{\ord}$ to be the maximal $B_C^{\Gal(K/\Q_p)}$-subrepresentation such that all its weights (restricted to $B_C$ via the diagonal map) are ordinary. We recall an explicit construction of $(L^{\otimes}|_{B_C^{\Gal(K/\Q_p)}})^{\ord}$ (\cite[\S~2.3, \S~2.4]{BH}).  For $w\in \sW_{B_C}$, let $I\subset w^{-1}(S) \cap C$ be a subset of pairwise orthogonal roots. Denote by $G_I$ the Levi subgroup containing $T$ with roots exactly $\pm I$. 
Then $G_I \cong T_{I^c}\times \prod_{\alpha\in I} \GL_2$  where  $T_{I^c}\cong (\bG_m)^{n-2|I|}$, and $B\cap G_I$ also decomposes as $T_{I^c}$ times the product of the induced Borel $B_{\alpha}$ in each $\GL_2$.  Similarly, $T\cong T_{I^c} \times \prod_{\alpha\in I} T_{\alpha}$ where $T_{\alpha}$ is the corresponding split torus in $\GL_2$. Put $L_I:=w(\lambda)|_{T_{I^c}} \otimes( \otimes_{\alpha} L_{\alpha})$ where $L_{\alpha}$ is the $B_{\alpha}$-representation defined as the unique non-split extension of $w(\lambda)|_{T_{\alpha}}$ by $(s_{\alpha} w)(\lambda)|_{T_{\alpha}}$.  If $I' \subset I$ (hence $I'$ also consists of pairwise orthogonal roots), then $L_{I'} \subset L_I$. Put
\begin{equation*}
L_w^{\ord}:=\varinjlim_I L_I.
\end{equation*}
The following theorem is due to Breuil-Herzig.
\begin{theorem}
We have $(L^{\otimes}|_{B_C^{\Gal(K/\Q_p)}})^{\ord}\cong \begin{cases}
	\oplus_{w\in \sW_{B_C}} L_w^{\ord} & K=\Q_p \\ 
	\oplus_{w \in \sW_{B_C}} d_K  w(\lambda) & K\neq \Q_p
\end{cases}.$
\end{theorem}
\begin{proof}
The $\Q_p$-case is \cite[Thm.~2.4.1]{BH}. Assume $K\neq \Q_p$. By the same argument of Step 1 of the proof of \cite[Thm.~2.4.1]{BH}, $\soc_{B_C^{\Gal(K/\Q_p)}} (L^{\otimes}|_{B_C^{\Gal(K/\Q_p)}})^{\ord} \cong   \oplus_{w \in \sW_{B_C}} d_K w(\lambda)$ (as $B_C$-representation). Let $e\in (L^{\otimes}|_{B_C^{\Gal(K/\Q_p)}})^{\ord}$. As $(L^{\otimes}|_{B_C^{\Gal(K/\Q_p)}})^{\ord}$ can be spanned by weight vectors (of ordinary weights) for $T$. We can and do assume $e$ is an weight vector for $T$, which,  by  Theorem \ref{ordweight}, has the form $\otimes_{i=1}^{d_K} e_{\sigma_i}$ where each $e_{\sigma_i}$ is a weight vector (unique up to scalars) of weight $w(\lambda)$ for $T_{\sigma_i}$ (with a common $w$ for all $\sigma_i$). Consider the $B_{C,\sigma_i}$-subrepesentation of $L^{\otimes}_{\sigma} $ generated by $e_{\sigma_i}$. If it is not equal to $E e_{\sigma_i}$, then one easily gets a non-zero weight vector (for $T$) in $(L^{\otimes}|_{B_C^{\Gal(K/\Q_p)}})^{\ord}$ of the weight of the form $(d_K-1) w(\lambda)+\lambda'$ with $\lambda'\neq w(\lambda)$, which is however not an ordinary weight.  The theorem follows.
\end{proof}
Let $\phi:=\otimes_{i=1}^n \phi_i$ be a refinement of $\Delta$. Put $\chi:=\phi (\varepsilon^{-1} \circ \theta): T(K) \ra E^{\times}$, hence $\chi=\jmath(\phi) z^{-\theta_K}$.  Recall for $w\in \sW_n$,  $w(\phi)=\otimes_{i=1}^n \phi_{w^{-1}(i)}$, and we put $w(\chi):=w(\phi) (\varepsilon^{-1}\circ \theta)=\jmath(w(\phi))z^{-\theta_K}$. By  \cite[Thm. (iv)]{OS}, we have:
\begin{lemma}\label{sindisc}
For $w\in \sW_n$,  $(\Ind_{B^-(K)}^{\GL_n(K)} w(\chi))^{\Q_p-\an}\cong \cF_{B^-}^{\GL_n}(L^-(\theta_K), \jmath(w(\phi)))$ is topologically irreducible and pairwise distinct. 
\end{lemma}
For $K\neq \Q_p$, for $w\in \sW_{B_C}$, we set $\pi(\Delta, \phi)_w^{\fss}:=(\Ind_{B^-(K)}^{\GL_n(K)} w(\chi))^{\Q_p-\an}$ and $\pi(\Delta,\phi)^{\fss}:=\oplus_{w \in \sW_{B_C}} \pi(\Delta,\phi)_w^{\fss}$ (which is hence semi-simple).  

For $K=\Q_p$.
Let $w\in \sW_{B_C}$, and $I\subset w^{-1}(S) \cap C$ which is a subset of pairwise orthogonal coroots. We put 
\begin{equation}
\pi(\Delta,\phi)_{G_I}:=\big((w(\chi)) (\varepsilon^{-1} \circ \theta)\big)|_{T_{I^c}} \otimes_E \big(\widehat{\otimes}_{\alpha\in I} \pi(\Delta_{w,\alpha})\big)
\end{equation}
where $\Delta_{w,\alpha}$ is the rank $2$ subquotient of $\Delta$ with irreducible constituents given by $\Delta_{w^{-1}(i)}$ and $\Delta_{w^{-1}(i+1)}$, with $\alpha=e_i-e_{i+1}$, and $\pi(\Delta_{w,\alpha})$ is the associated locally analytic representation of $\GL_2(\Q_p)$. Note that $\pi(\Delta_{w,\alpha})$ is the unique non-split extension of the form 
\begin{equation*}
(\Ind_{B^-(\Q_p)}^{\GL_2(\Q_p)} (\phi_{w^{-1}(i)}\varepsilon^{-1}) \otimes \phi_{w^{-1}(i+1)})^{\Q_p-\an}\   \rule[2.5pt]{10pt}{0.5 pt}\  (\Ind_{B^-(\Q_p)}^{\GL_2(\Q_p)} (\phi_{w^{-1}(i+1)} \varepsilon^{-1}) \otimes \phi_{w^{-1}(i)})^{\Q_p-\an}. 
\end{equation*} 
Set $\pi(\Delta,\phi)_I:=(\Ind_{B^-(\Q_p)G_I(\Q_p)}^{\GL_n(\Q_p)} \pi(\Delta, \phi)_{G_I})^{\Q_p-\an}$,  $\pi(\Delta,\phi)_w^{\fss}:=\varinjlim_I \pi(\Delta,\phi)_I$, and $\pi(\Delta,\phi)^{\fss}:=\oplus_{w\in \sW_{B_C}} \pi(\Delta,\phi)_w^{\fss}$. 
By Lemma \ref{reftri}, $w\in \sW_{B_C}$ is equivalent to that $w(\phi)$ is a refinement of $\Delta$. By definition, it is straightforward to see $\pi(\Delta,w(\phi))_1^{\fss}\cong \pi(\Delta, \phi)_w^{\fss}$, which we also denote  by $\pi(\Delta)_{w(\phi)}^{\fss}$.  Thus  we have $\pi(\Delta)^{\fss}\cong \oplus_{\phi'} \pi(\Delta)_{\phi'}^{\fss}$, with $\phi'$ running over the refinements of $\Delta$. 
The following conjecture is a direct consequence of Conjecture \ref{conjLG}. 
\begin{conjecture}\label{conjLGfs}
Let $D$ be a trianguline \'etale $(\varphi, \Gamma)$-module of Sen weight $\textbf{h}$ such that $D[\frac{1}{t}]\cong \Delta[\frac{1}{t}]$. Then  $\pi(\Delta)^{\fss}$ is a subrepresentation of $T_{\lambda}^{-\theta_K} \pi(D)$. 
\end{conjecture}
In the next sections, we collect some results towards the conjecture. We end this section by some examples, and we invite the reader to compare them with \cite[Ex.~2.1.5, 2.2.2]{Br12}.
\begin{example}\label{Efs}
(1) Suppose $\Delta$ is crystabelline, hence $\Delta\cong \oplus_{i=1}^n \cR_{K,E}(\phi_i)$ and $B_C=T$. We have $\pi(\Delta)^{\fss}\cong \oplus_{w\in \sW_n}(\Ind_{B^-(K)}^{\GL_n(K)} w(\chi))^{\Q_p-\an}$.

(2) Suppose $K\neq \Q_p$. In this case $\pi(\Delta)^{\fss}$ is semi-simple and isomorphic to 
\begin{equation*}
	\oplus_{\phi} \cF_{B^-}^{\GL_n}(L^-(\theta_K), \jmath(\phi))
\end{equation*}
where $\phi$ runs though refinements of $\Delta$. 

(3) Suppose $n=3$, $K=\Q_p$ and fix a refinement $\phi$ of $\Delta$. If $B_C=B$, then $\pi(\Delta)^{\fss}$ has the following form 
\begin{equation*}
	\begindc{\commdiag}[200]
	\obj(0,1)[a]{$\cF_{B^-}^{\GL_3}(L^-(\theta_K), \jmath(\phi))$}
	\obj(6,2)[b]{$\cF_{B^-}^{\GL_3}\big(L^-(\theta_K),\jmath(s_1(\phi))\big)$}
	\obj(6,0)[c]{$\cF_{B^-}^{\GL_3}\big(L^-(\theta_K), \jmath(s_2(\phi))\big)$}
	\mor{a}{b}{}[+1,\solidline]
	\mor{a}{c}{}[+1,\solidline]
	\enddc.
\end{equation*}	
If $B_C=\begin{pmatrix}
	* & * & 0 \\ 0 & * & 0 \\ 0 & 0 & *
\end{pmatrix}$, then $\pi(\Delta)^{\fss}$ has the form (noting $\Delta$ has $3$ refinements in this case)
\begin{eqnarray*}
	&&\Big(\cF_{B^-}^{\GL_3}(L^-(\theta_K), \jmath(\phi)) \ \rule[2.5pt]{10pt}{0.5 pt} \  \cF_{B^-}^{\GL_3}\big(L^-(\theta_K),\jmath(s_1(\phi))\big) \Big) \oplus \cF_{B^-}^{\GL_3}\big(L^-(\theta_K), \jmath(s_2(\phi))\big)  \\ &\oplus& \Big(\cF_{B^-}^{\GL_3}\big(L^-(\theta_K), \jmath(s_2s_1(\phi))\big) \ \rule[2.5pt]{10pt}{0.5 pt}\  \cF_{B^-}^{\GL_3}\big(L^-(\theta_K), \jmath(s_2s_1s_2(\phi))\big)\Big).
\end{eqnarray*}

(4) Suppose $n=4$, $K=\Q_p$ and fix a refinement $\phi$. Suppose $B_C=B$, then $\pi(\Delta)^{\fss}$ has the following form 
\begin{equation*}
	\begindc{\commdiag}[400]
	\obj(0,1)[a]{$\cF_{B^-}^{\GL_4}(L^-(\theta_K), \jmath(\phi))$}
	\obj(4,2)[b]{$\cF_{B^-}^{\GL_4}\big(L^-(\theta_K),\jmath(s_1(\phi))\big)$}
	\obj(4,0)[c]{$\cF_{B^-}^{\GL_4}\big(L^-(\theta_K), \jmath(s_3(\phi))\big)$}
	\obj(4,1)[d]{$\cF_{B^-}^{\GL_4}\big(L^-(\theta_K), \jmath(s_2(\phi))\big)$}
	\obj(8,1)[e]{$\cF_{B^-}^{\GL_4}\big(L^-(\theta_K), \jmath(s_1s_3(\phi))\big)$}
	\mor{a}{b}{}[+1,\solidline]
	\mor{a}{c}{}[+1,\solidline]
	\mor{a}{d}{}[+1,\solidline]
	\mor{b}{e}{}[+1,\solidline]
	\mor{c}{e}{}[+1,\solidline]
	\enddc.
\end{equation*}	

\end{example}

\subsubsection{Generic trianguline case}\label{S331}Keep assuming $\Delta$ is a trianguline $(\varphi, \Gamma)$-module of constant weight $0$ of rank $n$ over $\cR_{K,E}$. 
We call $\Delta$ \textit{generic} if $\phi_i \phi_j^{-1}\notin \{1, q_K^{\pm 1}\}$ for $i \neq j$. 

Let $D$ be as in Conjecture \ref{conjLGfs}, $\rho$ be the associated $\Gal_K$-representation. We assume moreover the Jacquet-Emerton module $J_B(\pi(D))\neq 0$ (cf. \cite{Em11}), which is equivalent to that $D$ appears in the patched eigenvariety of Breuil-Hellmann-Schraen \cite{BHS1}. As we deal with a single $p$-adic field $K$, we use the setting of \cite[\S~4.1]{Ding7} (that is a slight variant of \cite{BHS1}) and let $\cE$ be the corresponding patched eigenvariety (that was denoted by $X_{\wp}(\overline{\rho})$ in \cite[\S~4.1]{Ding7}, $\overline{\rho}$ being a mod $p$ reduction of $\rho$). A refinement $\phi$ of $D$ is called appearing  on the patched eigenvariety if there exists an algebraic character $z^{\mu}$ of $T(K)$ such that $(\rho, \jmath(\phi) \delta_B z^{\mu})\in \cE$.  Let $X_{\tri}(\overline{\rho})$ be the trianguline variety of \cite[\S~2.2]{BHS2} (where we drop the framing ``$\square$" in the notation). As $\phi$ is a refinement of $D$ (and is generic), by the construction of $X_{\tri}(\overline{\rho})$ (cf. \cite[\S~2.2]{BHS1} and Remark \ref{Rrefine}), there exists $w\in \sW_K$ such that $(\rho, \phi z^{w(\textbf{h})})$ lies in $X_{\tri}(\overline{\rho})$.
We prove the following theorem in the section.
\begin{theorem}\label{thmlgfs}
Suppose all the refinements of $D$ appear on the patched eigenvariety, then Conjecture \ref{conjLGfs} holds, i.e. $\pi(\Delta)^{\fss} \hookrightarrow \big(\pi(D) \otimes_E L(-w_0(\textbf{h}))\big)[\cZ_K=\chi_{-\theta_K}]$.
\end{theorem}
\begin{remark}Note that if $D$ is crystabelline, then $\Delta\cong \oplus_{i=1}^n \Delta_i$ (and $B_C=T$). In this case, by \cite{BHS2} \cite{BHS3} (for strictly dominant $\textbf{h}$, or equivalently for regular $\lambda$),  \cite{Wu21} \cite{Wu22} (for non-regular $\lambda$), the second assumption is equivalent to $J_B(\pi(D))\neq 0$.
\end{remark}
We have $\jmath(\phi)\delta_B=\otimes_{i=1}^n \phi_i |\cdot|_K^{1-i}$. For $w\in \sW_K$, denote by $\delta_w:= (\otimes_{i=1}^n \phi_i) z^{w(\textbf{h})}$, and $\chi_w:=\delta_w \delta_B (\varepsilon^{-1} \circ \theta) = z^{w\cdot \lambda} \jmath(\phi)\delta_B$. By assumption and \cite[Prop.~2.9, Thm.~3.21]{BHS1}, there exists $w\in \sW_K$ with maximal length such that $(\rho, \chi_{w})\in \cE$.\footnote{One may expect that   $\delta_{w}$ is a trianguline parameter of $D_{\rig}(\rho)$. But we don't need this in the paper.} 
In particular, for $w'>w$, $(\rho,\chi_{w'})\notin \cE$. 
By definition, the point $(\rho, \chi_w)$ gives rise to an injection of $T(K)$-representation
\begin{equation*}
z^{w\cdot \lambda} \jmath(\phi)\delta_B=	\chi_{w} \hooklongrightarrow J_B(\pi(D)),
\end{equation*}
which by Breuil's adjunction formula (\cite[Thm.~4.3]{Br13I}), induces a non-zero $\GL_n(K)$-equivariant map (where $(-)^{\vee}$ is the dual in the BGG category $\co$)
\begin{equation*}
\cF_{B^-}^{\GL_n}\big(M^-(-w\cdot \lambda)^{\vee}, \jmath(\phi)\big) \lra \pi(D). 
\end{equation*}
By the structure of Orlik-Strauch representations \cite[Thm.]{OS} and the fact $\phi$ is generic, the map factors though an injection 
\begin{equation}\label{socinj0}
\cF_{B^-}^{\GL_n}(L^-(-w'\cdot \lambda), \jmath(\phi)) \hooklongrightarrow \pi(D),
\end{equation}
for certain $w'\geq w$.
However, by \cite[Thm.~4.3]{Br13II}, this implies $\chi_w'\hookrightarrow J_B(\pi(D))$ hence $(\rho, \chi_{w'})\in \cE$. By (the maximal length) assumption on $w$,  we have $w'=w$. Now consider the representation $(\Ind_{B^-(K)}^{\GL_n(K)} \delta_w(\varepsilon^{-1}\circ \theta))^{\Q_p-\an}\cong \cF_{B^-}^{\GL_n}(M^-(-w\cdot \lambda),\jmath(\phi))$. By \cite[Cor.~3.3]{Br13I},  $\soc_{\GL_n(K)}\big((\Ind_{B^-(K)}^{\GL_n(K)} \delta_w(\varepsilon^{-1}\circ \theta))^{\Q_p-\an}\big)=
\cF_{B^-}^{\GL_n}(L^-(-w\cdot \lambda), \jmath(\phi))$. By \cite[Thm.]{OS}, all the \textit{other} irreducible constituents of $(\Ind_{B^-(K)}^{\GL_n(K)} \delta_w(\varepsilon^{-1}\circ \theta))^{\Q_p-\an}$ has the form $\cF_{B^-}^{\GL_n}(L^-(-w'\cdot \lambda), \jmath(\phi)))$ for $w'>w$, which hence can not inject into $\pi(D)$ (by assumption on $w$). 
By \cite[Prop.~4.8]{BH2} (see also \'Etale 1 in \cite[\S~6.4]{Br16}), we deduce
\begin{equation}\label{ancont1}
\Hom_{\GL_n(\Q_p)}\big((\Ind_{B^-(K)}^{\GL_n(K)} \delta_w(\varepsilon^{-1}\circ \theta))^{\Q_p-\an}, \pi(D)\big)\xlongrightarrow{\sim} 
\Hom_{\GL_n(\Q_p)}\big(\cF_{B^-}^{\GL_n}(L^-(-w\cdot \lambda), \jmath(\phi)), \pi(D)\big).
\end{equation}
The injection (\ref{socinj0}) (noting $w'=w$) induces hence an injection
\begin{equation}\label{Efs1}
(\Ind_{B^-(K)}^{\GL_n(K)} \delta_w(\varepsilon^{-1}\circ \theta))^{\Q_p-\an} \hooklongrightarrow \pi(D).
\end{equation}
Applying $T_{\lambda}^{-\theta_K}$ on both sides and using \cite[Thm.~4.1.12]{JLS}, we get
\begin{equation*}
\cF_{B^-}^{\GL_n}(L^-(\theta_K), \jmath(\phi))\cong (\Ind_{B^-(K)}^{\GL_n(K)} \phi (\varepsilon^{-1} \circ  \theta))^{\Q_p-\an} \hooklongrightarrow \pi(\Delta). 
\end{equation*}
By assumption in the theorem (and using Lemma \ref{sindisc}), we see
\begin{equation}\label{fssocle}
\soc_{\GL_n(K)} \pi(\Delta)^{\fss} \cong \oplus_{\phi}(\Ind_{B^-(K)}^{\GL_n(K)} \phi (\varepsilon^{-1} \circ  \theta))^{\Q_p-\an} \hooklongrightarrow \pi(\Delta)
\end{equation}
where $\phi$ run over refinements of $D$.
This finishes the proof for  $K\neq \Q_p$ (noting $\pi(\Delta)^{\fss}$ is semi-simple in this case).

To prove the $K=\Q_p$-case, we first construct some representations. It is sufficient to show $\pi(\Delta)_{\phi}^{\fss}\hookrightarrow \pi(\Delta)$ for each refinement $\phi$. We fix $\phi$, and let $w$ be as above. Let  $I\subset S$ be as in the discussion below Theorem \ref{ordweight} (consisting of pairwise orthogonal roots). 
\begin{lemma}\label{weyl0}
We have $w\geq \prod_{\alpha\in I} s_{\alpha}=:s_I$. 
\end{lemma}
\begin{proof}It suffices to show each $s_{\alpha}$ appears in $w$ for $\alpha\in I$. 
If $s_{\alpha}$ does not appear in $w$, then $w\cdot \lambda$ is $P_{\alpha}$-dominant. We have then (cf. \cite[Thm.]{OS}) 
\begin{multline*}
	\cF_{B^-}^{\GL_n}(L^-(-w\cdot \lambda), \jmath(\phi))\cong \cF_{P_{\alpha}^-}^{\GL_n}\big(L^-(-w\cdot \lambda), (\Ind_{B^-(\Q_p) \cap L_{\alpha}(\Q_p)}^{L_{\alpha}}(\Q_p)\jmath(\phi))^{\infty}\big)\\
	\cong \cF_{B^-}^{\GL_n}(L^-(-w\cdot \lambda), \jmath(s_{\alpha}(\phi))).
\end{multline*}
The injection (\ref{socinj0}) then induces $z^{w\cdot \lambda}\jmath(s_{\alpha}\phi) \delta_B \hookrightarrow J_B(\pi(D))$. By the global triangulation theory \cite{KPX} \cite{Liu}, this implies $s_{\alpha}(\phi)$ is a refinement of $D$, contradicting $\alpha\in I$.
\end{proof}
For $\alpha=e_i-e_{i+1}\in I$, as $\phi$ is generic, there exists a unique non-split extension $D_{i,i+1}$ of $\cR_E(\delta_{i+1})$ by $\cR_E(\delta_i)$ where  $\delta_{i+1}=\phi_{i+1} z^{h_{i}}$ and $\delta_i=\phi_i z^{h_{i+1}}$ (recalling $h_i\geq h_{i+1}$).  Let $\pi(D_{i,i+1})$ be the locally analytic $\GL_2(\Q_p)$-representation corresponding to $D_{i,i+1}$, which has the following form
\begin{equation}
\label{GL20}(\Ind_{B^-(\Q_p)}^{\GL_2(\Q_p)} \phi_i z^{h_{i+1}} \varepsilon^{-1} \otimes \phi_{i+1} z^{h_i})^{\Q_p-\an} - (\Ind_{B^-(\Q_p)}^{\GL_2(\Q_p)} \phi_{i+1} z^{h_i}\varepsilon^{-1}  \otimes \phi_{i} z^{h_{i+1}})^{\Q_p-\an},
\end{equation}
where the extension is  the unique non-split one. 
Consider the parabolic induction (where $P_I=BG_I$):
\begin{equation*}
\pi(\phi,\lambda)_I:=	\Big(\Ind_{B^-(\Q_p) G_I(\Q_p)}^{\GL_n(\Q_p)} \big(\delta|_{T_{I^c}} \otimes (\widehat{\otimes}_{\alpha\in I} \pi(D_{i,i+1}))\big)\otimes_E (\varepsilon^{-1} \circ \theta^{P_I})\Big)^{\Q_p-\an}.
\end{equation*}
We have  (by (\ref{GL20}))
\begin{multline*}
\cF_{B^-}^{\GL_n}(M^-(-s_I\cdot \lambda), \jmath(\phi))\cong (\Ind_{B^-(\Q_p)}^{\GL_n(\Q_p)}\delta_{s_I} (\varepsilon^{-1} \circ \theta))^{\Q_p-\an}\\ \cong \Big(\Ind_{B^-(\Q_p) G_I(\Q_p)}^{\GL_n(\Q_p)} \big(\delta|_{T_{I^c}}\otimes (\widehat{\otimes}_{\alpha\in I}(\Ind_{B^-(\Q_p)}^{\GL_2(\Q_p)} (\phi_i z^{h_{i+1}}\varepsilon^{-1})\otimes \phi_{i+1} z^{h_i})^{\Q_p-\an})\big)\otimes_E (\varepsilon^{-1} \circ \theta^{P_I}) \Big)^{\Q_p-\an} \\
\hooklongrightarrow \pi(\phi,\lambda)_I.
\end{multline*}
As $w \geq s_I$, we have a natural surjection (induced by $M^-(-w \cdot \lambda) \hookrightarrow M^-(-s_I \cdot \lambda)$).
\begin{equation*}
\pr: (\Ind_{B^-(\Q_p)}^{\GL_n(\Q_p)}\delta_{s_I} (\varepsilon^{-1} \circ \theta))^{\Q_p-\an} \twoheadlongrightarrow (\Ind_{B^-(\Q_p)}^{\GL_n(\Q_p)}\delta_{w} (\varepsilon^{-1} \circ \theta))^{\Q_p-\an}.
\end{equation*}
Let $\pi(\phi,w,\lambda)_I$ be the (unique) quotient of  $\pi(\phi,\lambda)_I/\Fer \pr$ such that the composition 
\begin{equation*}
(\Ind_{B^-(\Q_p)}^{\GL_n(\Q_p)}\delta_{w} (\varepsilon^{-1} \circ \theta))^{\Q_p-\an} \hooklongrightarrow \pi(\phi,\lambda)_I/\Fer \pr \lra \pi(\phi,w,\lambda)_I
\end{equation*}
is injective and induces an isomorphism on socle. Note by \cite[Cor.~3.3]{Br13I},  $\soc_{\GL_n(\Q_{p})}(\Ind_{B^-(\Q_p)}^{\GL_2(\Q_p)} \delta_{w}(\varepsilon^{-1} \circ \theta))^{\Q_p-\an} \cong \cF_{B^-}^{\GL_n}(L^-(-w\cdot \lambda), \jmath(\phi))$. 

\begin{proposition}\label{tranfs1}
We have $T_{\lambda}^{-\theta}\pi(\phi,\lambda)_I\cong \pi(\Delta)_I^{\fss}$, and 	the natural projection $\pi(\phi,\lambda)_I \twoheadrightarrow \pi(\phi,w,\lambda)_I$ induces an isomorphism 
\begin{equation*}
	T_{\lambda}^{-\theta} \pi(\phi, \lambda) \xlongrightarrow{\sim} T_{\lambda}^{-\theta} \pi(\phi,w,\lambda)_I.
\end{equation*}
\end{proposition}
\begin{proof}
We have
\begin{multline*}
	T_{\lambda}^{-\theta}\pi(\phi,\lambda)_I \cong 	\Big(\Ind_{B^-(\Q_p) G_I(\Q_p)}^{\GL_n(\Q_p)} T_{s_I \cdot \lambda}^{-\theta}\Big( \big(\delta|_{T_{I^c}} \otimes (\widehat{\otimes}_{\alpha\in I} \pi(D_{i,i+1}))\big)\otimes_E (\varepsilon^{-1} \circ \theta^{P_I})\Big) \Big)^{\Q_p-\an} \\
	\cong 
	\Big(\Ind_{B^-(\Q_p) G_I(\Q_p)}^{\GL_n(\Q_p)} \big(\phi|_{T_{I^c}} \otimes (\widehat{\otimes}_{\alpha\in I}\pi(\Delta_{i,i+1}))\big)\otimes_E (\varepsilon^{-1} \circ \theta^{P_I}) \Big)^{\Q_p-\an}=\pi(\Delta)_I^{\fss},
\end{multline*}
where the first isomorphism follows from Proposition \ref{tranInd}, and the second from the isomorphism $T_{(h_{i+1}+i-1, h_i+i)}^{(-1,0)} \pi(D_{i,i+1})\cong \pi(\Delta_{i,i+1})$ by \cite{Ding14}.  
Hence  $\soc_{\GL_n(\Q_p)} T_{\lambda}^{-\theta}\pi(\phi,\lambda)_I\cong (\Ind_{B^-(\Q_p)}^{\GL_n(\Q_p)} \phi (\varepsilon^{-1} \circ \theta))^{\Q_p-\an}$ which has multiplicity one as irreducible constituent.
The projection $\pi(\phi,\lambda)_I \twoheadrightarrow \pi(\phi,w,\lambda)_I$ induces a surjective map
\begin{equation}\label{compare0}
	T_{\lambda}^{-\theta}\pi(\phi,\lambda)_I  \twoheadlongrightarrow T_{\lambda}^{-\theta} \pi(\phi,w,\lambda)_I.
\end{equation}
As we have  (the second isomorphism following again from Proposition \ref{tranInd}, or \cite[Thm.~4.1.12]{JLS})
\begin{multline*}
	\soc_{\GL_n(\Q_p)} T_{\lambda}^{-\theta}\pi(\phi,\lambda)_I \cong (\Ind_{B^-(\Q_p)}^{\GL_n} \phi (\varepsilon^{-1} \circ \theta))^{\Q_p-\an}\cong T_{\lambda}^{-\theta} (\Ind_{B^-(\Q_p)}^{\GL_n(\Q_p)} \delta_{w} (\varepsilon^{-1} \circ \theta))^{\Q_p-\an}\\
	\hooklongrightarrow T_{\lambda}^{-\theta} \pi(\phi,w,\lambda)_I
\end{multline*} 
(\ref{compare0}) has to be an isomorphism. The proposition follows. 
\end{proof}
We get the following corollary, being of interest in its own right. 
\begin{corollary}
For any $J \subset I$, the representation $\cF_{B^-}^{\GL_n}(L^-(-w_0 \cdot \lambda), \jmath(s_J(\phi)))$ appears as irreducible constituent in $ \pi(\phi, w, \lambda)_I$.
\end{corollary}
\begin{proof}
By Remark \ref{transzero} and \cite[Thm.~4.1.12]{JLS}, the only irreducible constituent of $\pi(\phi, \lambda)_I$ that is translated to $\cF_{B^-}^{\GL_n}(L^-(\theta), \jmath(s_J(\phi)))$ under $T_{\lambda}^{-\theta}$ is  $\cF_{B^-}^{\GL_n}(L^-(-w_0 \cdot \lambda), \jmath(s_J(\phi)))$ and has multiplicity one. As $\cF_{B^-}^{\GL_n}(L^-(\theta), \jmath(s_J(\phi)))$ appears in $T_{\lambda}^{-\theta} \pi(\phi,w,\lambda)_I$ by Proposition \ref{tranfs1}, $\cF_{B^-}^{\GL_n}(L^-(-w_0 \cdot \lambda), \jmath(s_J(\phi)))$ has to appear in $ \pi(\phi, w, \lambda)_I$.
\end{proof}
In general, for $J\subset I$, we have
\begin{multline}\label{compJI}
\pi(\phi,\lambda)_J\cong \Big(\Ind_{B^-(\Q_p) G_J(\Q_p)}^{\GL_n(\Q_p)} \big(\delta|_{T_{J^c}} \otimes (\widehat{\otimes}_{\alpha\in J}\pi(D_{i,i+1}))\big)\otimes_E \varepsilon^{-1} \circ \theta^{P_J}\Big)^{\Q_p-\an}\\
\cong \Big(\Ind_{B^- G_I}^{\GL_n} \big(\delta|_{T_{I^c}} \otimes (\widehat{\otimes}_{\alpha\in J}\pi(D_{i,i+1})) \otimes (\widehat{\otimes}_{\alpha\in I\setminus J} (\Ind_{B^-}^{\GL_2} (\phi_i z^{h_i}\varepsilon^{-1}) \otimes (\phi_{i+1}z^{h_{i+1}}))^{\an})  \big)\otimes \varepsilon^{-1} \circ \theta^{P_I}\Big)^{\an}\\
\twoheadlongrightarrow 
\Big(\Ind_{B^- G_I}^{\GL_n} \big(\delta|_{T_{I^c}} \otimes (\widehat{\otimes}_{\alpha\in J}\pi(D_{i,i+1})) \otimes (\widehat{\otimes}_{\alpha\in I\setminus J} (\Ind_{B^-}^{\GL_2} (\phi_i z^{h_{i+1}}\varepsilon^{-1})\otimes (\phi_{i+1}z^{h_{i}}))^{\an})  \big)\otimes \varepsilon^{-1} \circ \theta^{P_I}\Big)^{\an}
\\
\hooklongrightarrow \pi(\phi,\lambda)_I,
\end{multline}
where the third map is induced by the natural surjection $(\Ind_{B^-}^{\GL_2} (\phi_i z^{h_i}\varepsilon^{-1}) \otimes (\phi_{i+1}z^{h_{i+1}}))^{\an} \twoheadrightarrow (\Ind_{B^-}^{\GL_2} (\phi_i z^{h_{i+1}}\varepsilon^{-1}) \otimes \phi_{i+1}z^{h_{i}})^{\an}$. By  construction, it is easy to see (\ref{compJI}) induces  a map $\pi(\phi,w,\lambda)_J\hookrightarrow \pi(\phi,w,\lambda)_I$, which restricts to  an isomorphism on the socle hence is injective.

For  $\emptyset \neq J\subset I$, by assumption,  $s_J(\phi)$ is not a refinement of $D$. The following lemma is hence an easy consequence of the global triangulation theory.
\begin{lemma}\label{norefi}
For $\emptyset \neq J\subset I$, and any algebraic character $z^{\mu}$ of $T(K)$, $(\rho, \jmath(s_J(\phi))z^{\mu})\notin \cE$. Consequently,
\begin{equation*}
	\Hom_{T(K)}(\jmath(s_J(\phi))\delta_B z^{\mu}, J_B(\pi(D)))=0.
\end{equation*}
\end{lemma}
\begin{proposition}
The restriction map
\begin{equation*}
	\Hom_{\GL_n(\Q_p)}\big(\pi(\phi,w,\lambda)_I, \pi(D)\big) \lra \Hom_{\GL_n(\Q_p)}\big(\cF_{B^-}^{\GL_n}(L^-(-w\cdot \lambda), \jmath(\phi)), \pi(D)\big)
\end{equation*}
is a bijection.
\end{proposition}
\begin{proof}
By (\ref{ancont1}), it suffices to show 
\begin{equation*}
	\Hom_{\GL_n(\Q_p)}\big(\pi(\phi,w,\lambda)_I, \pi(D)\big) \xlongrightarrow{\sim} \Hom_{\GL_n(\Q_p)}\big((\Ind_{B^-(K)}^{\GL_n(K)} \delta_w(\varepsilon^{-1}\circ \theta))^{\Q_p-\an}, \pi(D)\big). 
\end{equation*}
But this follows by \cite[Prop.~4.8]{BH2} (noting the condition (iii) is satisfied by Lemma \ref{norefi}).
\end{proof}
By the proposition, the injection uniquely extends to an injection $ \pi(\phi,w,\lambda)_I \hookrightarrow \pi(D)$. By varying $I$ varying and using  the discussion below (\ref{compJI}), it uniquely extends to an injection $\varinjlim_I \pi(\phi,w,\lambda)_I \hookrightarrow \pi(D)$. Finally, by applying $T_{\lambda}^{-\theta}$ and Proposition \ref{tranfs1}, the injection (\ref{fssocle})  extends to an injection $\pi(\Delta)_{\phi}^{\fss}\hookrightarrow \pi(\Delta)$. The theorem follows.

\begin{remark}Although the element $w$ is rather auxiliary in the proof, it is an important invariant of $D_{\rig}(\rho)$. For example, when $D$ is crystabelline, $w$ reflects the relative position of the Hodge filtration and Weil filtration (associated to $\phi$) (cf. \cite{Br13I} \cite{BHS3} \cite{Wu21}). As the information on Hodge filtration of $\rho$ is lost in $\Delta$, it should  not be surprising that the information of $w$ disappears in $\pi(\Delta)$. 
\end{remark}
\subsubsection{Steinberg case}
\begin{theorem}\label{lgSt}
Suppose $\textbf{h}$ is strictly dominant,  $D$ is semi-stable non-crystalline with the monodromy $N^{n-1}\neq 0$ (on $D_{\st}(D)$), $D$ is non-critical, and $D$ appears on the patched eigenvariety, then $\pi(\Delta)^{\fss} \hookrightarrow T_{\lambda}^{-\theta_K} \pi(D)$.
\end{theorem}
Note that by the assumption,  $\Delta$ admits a unique refinement $\sF$. In fact, there exists a smooth character $\eta$ of $K^{\times}$ such that  $\Delta$ is a successive non-split  extension of $\cR_{K,E}(\eta |\cdot|_K^{n-i})$. The set $C:=C_{\sF}$ is equal to $R^+$. By definition $\soc_{\GL_n(K)} \pi(\Delta)^{\fss}\cong (\Ind_{B^-(K)}^{\GL_n(K)}   z^{-\theta_K}\eta\circ \dett)^{\Q_p-\an}\cong  (\Ind_{B^-(K)}^{\GL_n(K)}  z^{-\theta_K})^{\Q_p-\an} \otimes_E\eta\circ \dett$. We have as in \cite[Lem.~4.6]{Ding7}:
\begin{lemma}\label{sppt}
Let $\chi$ be a locally $\Q_p$-analytic character of $T(K)$, then $\Hom_{T(K)}(\chi, J_B(\pi(D)))\neq 0$ if and only if $\chi=z^{\lambda} \delta_B (\eta\circ \dett)$. 
\end{lemma}
Let $\St^{\infty}:=(\Ind_{B^-(K)}^{\GL_n(K)} 1)^{\infty}/\sum_{P\supsetneq B} (\Ind_{P^-(K)}^{\GL_n(K)} 1)^{\infty}$, $\St^{\infty}(\lambda):=\St^{\infty}\otimes_E L(\lambda)$,
$$\St(\lambda)^{\an}:=(\Ind_{B^-(K)}^{\GL_n(K)} z^{\lambda})^{\Q_p-\an}/\sum_{P\supsetneq B} (\Ind_{P^-(K)}^{\GL_n(K)} L(\lambda)_P)^{\Q_p-\an},$$ 
where $L(\lambda)_P$ is the representation of $M_P(K)$ of weight $\lambda$. It is clear $\St^{\infty}(\lambda)\hookrightarrow \St^{\an}(\lambda)$. 
\begin{lemma}\footnote{I thank Zicheng Qian for drawing my attention to the continuous Steinberg representation.} \label{socSt} 
We have $\soc_{\GL_n(K)} \St^{\an}(\lambda) \cong \St^{\infty}(\lambda)$.
\end{lemma}
\begin{proof} Denote by $\ul{0}$ the zero weight of $\ft_K$. 
As the translation $T_{\lambda}^{\ul{0}}$ induces  an equivalence of categories (cf. \cite[Thm.~3.2.1]{JLS}), we reduce to the case where $\lambda=\ul{0}$. We write $\St^{\an}:=\St^{\an}(\ul{0})$ and $\St^{\infty}:=\St^{\infty}(\ul{0})$. Put ($(\Ind-)^{\cC_0}$ denoting the continuous induction) $$\St^{\cont} :=(\Ind_{B^-(K)}^{\GL_n(K)} 1)^{\cC_0}/\sum_{P\supsetneq B} (\Ind_{P^-(K)}^{\GL_n(K)} 1)^{\cC_0},$$ which is a unitary Banach representation with the supreme norm. It is clear that $\St^{\an} \hookrightarrow \St^{\cont}$. It is sufficient to show that any irreducible constituent $\cF_{P^-}^{\GL_n}(L^-(w\cdot \ul{0}), \pi_P)$ of $\St^{\an}$ \big(or more generally, of $(\Ind_{B^-(K)}^{\GL_n(K)} 1)^{\Q_p-\an}$\big) with $w\neq 1$ does not admit a $\GL_n(K)$-invariant lattice where $\pi_P$ is a certain subquotient of $(\Ind_{B^-(K)\cap L_P(K)}^{L_P(K)} 1)^{\infty}$. As the central character of $\pi_P$ is trivial, this follows easily from \cite[Cor.\ 3.5]{Br13I}.
\end{proof}
By Lemma \ref{sppt}, we have an injection $z^{\lambda}\delta_B (\eta\circ \dett) \hookrightarrow J_B(\pi(D))$, which induces an injection (e.g. see \cite[Prop.~4.7]{Ding7})
\begin{equation}
\label{stsm}\St^{\infty}(\lambda) \otimes_E \eta \circ \dett \hooklongrightarrow  \pi(D).
\end{equation}
Again using Lemma \ref{sppt} and similar argument as for the generic case (using \cite[Prop.~4.8]{BH2}), the embedding (\ref{stsm}) extends uniquely to an embedding 
\begin{equation}\label{stan}
\St^{\an}(\lambda)\otimes_E \eta\circ \dett  \hooklongrightarrow \pi(D)
\end{equation}
\begin{lemma}
We have $T^{-\theta_K}_\lambda \St^{\an}(\lambda)\cong (\Ind_{B^-(K)}^{\GL_n(K)} z^{-\theta_K})^{\Q_p-\an}$. 
\end{lemma}
\begin{proof}
By \cite[Thm.~4.1.12]{JLS}, $T^{-\theta_K}_\lambda (\Ind_{B^-(K)}^{\GL_n(K)}z^\lambda)^{\Q_p-\an}\cong (\Ind_{B^-(K)}^{\GL_n(K)} z^{-\theta_K})^{\Q_p-\an}$. It suffices to show for $P\supsetneq B$, $T^{-\theta_K}_\lambda (\Ind_{P^-(K)}^{\GL_n(K)} L(\lambda)_P)^{\Q_p-\an}=0$ . This follows from Proposition \ref{tranInd} and the fact $T_\lambda^{-\theta_K} L(\lambda)_P=0$ (cf. Remark \ref{transzero}). 
\end{proof}
By the lemma, the injection (\ref{stan}) induces $$ \soc_{\GL_n(K)} \pi(\Delta)^{\fss}\cong (\Ind_{B^-(K)}^{\GL_n(K)} z^{-\theta_K})^{\Q_p-\an} \otimes_E (\eta\circ \dett) \hooklongrightarrow T_{\lambda}^{-\theta_K} \pi(D)=\pi(\Delta).$$ This finishes the proof of Theorem \ref{lgSt} when $K\neq \Q_p$. 

Assume henceforth $K=\Q_p$, we first construct a representation $\pi(D)^{\fss}$ in a similar way as $\pi(\Delta)^{\fss}$. For $i=0, \cdots, n-1$, let $D_{i,i+1}$ be the subquotient of $D$, which is an extension of $\cR_E(\eta |\cdot |^{n-i-1}z^{h_{i+1}})$ by $\cR_E(\eta |\cdot |^{n-i} z^{h_i})$. Let $\pi(D_{i,i+1})$ be the locally analytic $\GL_2(\Q_p)$-representation associated to $D_{i,i+1}$. Recall that  $\pi(D_{i,i+1})$ is an extension of the form (where $\lambda_{\alpha}=(\lambda_i,\lambda_{i+1})$ for $\alpha=e_i-e_{i+1}\in S$)
\begin{equation}\label{rep2}
\eta\circ \dett \otimes_E [\St^{\an}_2(\lambda_{\alpha})\  \rule[2.5pt]{10pt}{0.5 pt}\  L(\lambda_{\alpha}) \  \rule[2.5pt]{10pt}{0.5 pt}\ (\Ind_{B^-}^{\GL_2} z^{\lambda_{i+1}-1}|\cdot|^{-1} \otimes z^{\lambda_i+1}|\cdot|)^{\an}].
\end{equation} 
Let $I\subset S$ be a subset of pairwise orthogonal roots. Consider 
\begin{equation}\label{rep3}
\pi_I:=	\Big(\Ind_{P_I^-(\Q_p)}^{\GL_n(\Q_p)}( \widehat{\otimes}_{\alpha\in I}\pi(D_{i,i+1})) \otimes (\otimes_{e_i-e_{i+1}\notin I} \eta |\cdot|^{n-i}z^{h_{i}}) \otimes_E (\varepsilon^{-1}\circ \theta^{P_I})\Big)^{\Q_p-\an}.
\end{equation}
By \cite[Thm.]{OS}, it is easy to see $\St^{\infty}(\lambda) \otimes_E (\eta \circ \dett)$ has multiplicity one in $\pi_I$. Let  $\pi(D)_I^{\fss}$ be the quotient of $\pi_I$ 
of socle $\St^{\infty}(\lambda)\otimes_E (\eta \circ \dett)$. As $\St^{\an}(\lambda) \otimes (\eta \circ \dett)$ is a quotient of 
\begin{equation*}
\eta \circ \dett \otimes_E  \Big(\Ind_{P_I^-(\Q_p)}^{\GL_n(\Q_p)}( \widehat{\otimes}_{\alpha\in I}\St^{\an}(\lambda_{\alpha})  )) \otimes (\otimes_{e_i-e_{i+1}\notin I} |\cdot|_K^{n-i}z^{h_{i}}) \otimes_E \varepsilon^{-1}\circ \theta^{P_I}\Big)^{\Q_p-\an},
\end{equation*}
the latter being a subrepresentation of $\pi_I$ (using (\ref{rep2})) and $\soc_{\GL_n(K)} \St^{\an}(\lambda) \otimes (\eta \circ \dett)\cong \St^{\infty}(\lambda)  \otimes_E (\eta \circ \dett)$, we deduce  $\St^{\an}(\lambda) \otimes (\eta \circ \dett) \hookrightarrow \pi(D)_I^{\fss}$. 

For $J\subset I$, and $e_i-e_{i+1}\in I \setminus J$, the natural composition 
\begin{equation*}(\Ind_{B^-}^{\GL_2}( \eta |\cdot|_K^{n-i} z^{h_i} \otimes \eta|\cdot|_K^{n-i-1} z^{h_{i+1}}) (\varepsilon^{-1} \circ \theta))^{\an} \twoheadlongrightarrow \eta \circ \dett \otimes_E \St_2^{\an} \hooklongrightarrow \pi(D_{i,i+1})\end{equation*}
induces a natural map $\pi_J \ra \pi_I$ hence $\pi(D)_J^{\fss} \ra \pi(D)_I^{\fss}$ which has to be injective, as the both have socle $\St^{\infty}(\lambda) \otimes_E \eta \circ \dett$ with multiplicity one  (as irreducible constituent). 
We define finally $\pi(D)^{\fss}:=\varinjlim_I \pi(D)_I^{\fss}$, with $I$ running through subsets of $S$ of pairwise orthogonal roots.
\begin{lemma}\label{transt1}
We have $T_{\lambda}^{-\theta} \pi(D)^{\fss}\cong \pi(\Delta)^{\fss}$.
\end{lemma}
\begin{proof}
Let $I\subset S$ be a subset of pairwise orthogonal roots.	By Proposition \ref{tranInd} and \cite{Ding14}, $T_{\lambda}^{-\theta} \pi_I \cong \pi(\Delta)^{\fss}_I$. As $\soc_{\GL_n(\Q_p)} \pi(\Delta)^{\fss}_I \cong (\Ind_{B^-(\Q_p)}^{\GL_n(\Q_p)} z^{-\theta})^{\Q_p-\an} \otimes \eta \circ \dett$, which has multiplicity one in $\pi(\Delta)^{\fss}$. 
The quotient map $\pi_I \ra \pi(D)_I^{\fss}$ induces $\pi(\Delta)^{\fss}_I \twoheadrightarrow T_{\lambda}^{-\theta} \pi(D)_I$ which has to be an isomorphism as the both have the same socle (with multiplicity one as irreducible constituent). By taking direct limit, the lemma follows. 
\end{proof}

Recall for each $D_{i}^{i+1}$, one can associate a Fontaine-Mazur $\cL$-invariant $\cL_i\in E$. For each $\cL_i$, one can associate as in \cite{Ding7} a locally analytic representation $\St^{\an}(\lambda, \cL_i)$ which is isomorphic to an extension of $v_i^{\infty} \otimes_E L(\lambda)$ by $\St^{\an}(\lambda)$, where $v_i^{\infty}=(\Ind_{P_i^-(\Q_p)}^{\GL_n(\Q_p)} 1)^{\infty}/\sum_{Q^-\supsetneq P_i^-} (\Ind_{Q^-(\Q_p)}^{\GL_n(\Q_p)} 1)^{\infty}$ is a generalised Steinberg representation, $P_i$ being the minimal parabolic subgroup associated to $i$. 
Let $\St^{\an}(\lambda, \ul{\cL}) \otimes_E \eta \circ \dett:=\bigoplus_{\St^{\an}(\lambda)}^{i=1,\cdots n-1}\St^{\an}(\lambda, \cL_i) \otimes_E \eta \circ \dett$.  As the subrepresentation $\eta \circ \dett \otimes [\St_2^{\an}(\lambda_{\alpha})-L(\lambda_{\alpha})]$ of $\pi(D_{i,i+1})$ is just $\St^{\an}(\lambda_{\alpha}, \cL_i)$ (for $\GL_2(\Q_p)$), it is not difficult to see (e.g from the explicit construction in \cite[Rem.~2.18]{Ding7}) there is a natural injection $\St^{\an}(\lambda, \cL_i) \otimes_E \eta \circ \dett \hookrightarrow \pi(D)_{\{e_i-e_{i+1}\}}^{\fss}$ hence $\St^{\an}(\lambda, \ul{\cL}) \otimes_E \eta \circ \dett \hookrightarrow \pi(D)^{\fss}$. By \cite[Thm.~1.2]{Ding7} and Lemma \ref{socSt}, the injection (\ref{stsm}) uniquely extends to an injection 
\begin{equation}\label{stL}\St^{\an}(\lambda, \ul{\cL}) \otimes_E \eta \circ \dett  \hooklongrightarrow \pi(D)
\end{equation}
By  an easy variant of the  argument  in \cite[\S~6.4]{Br16} (see also \cite[Prop.~7.4]{Qian2}), the injection (\ref{stL}) extends uniquely to an injection $\pi(D)^{\fss} \hookrightarrow \pi(D)$. Indeed,   for the irreducible constituents of $\pi(D)^{\fss}/ (\St^{\an}(\lambda,\ul{\cL})\otimes_E \eta \circ \dett)$, we don't need the argument in \'Etape 3 of \cite[\S~6.4]{Br16}, hence the extension of the injection is unique. 
Applying $T_{\lambda}^{-\theta}$ and using  Lemma \ref{transt1}, the theorem follows.
\begin{remark}
Note that the simple $\cL$-invariants of $D$ play an auxiliary role in the proof, and they are actually invisible in $\pi(\Delta)^{\fss}$.
\end{remark}

\section{Wall-crossing and Hodge filtration}
Keep assuming $D$ is an \'etale $(\varphi, \Gamma)$-module of rank $n$ over $\cR_{K,E}$ of integral Sen weights $\textbf{h}$. We study the interplay between $\pi(D)$ and its wall-crossing $T_{-\theta_K}^{\lambda} T_{\lambda}^{-\theta_K} \pi(D)\cong T_{-\theta_K}^{\lambda} \pi(\Delta)$, in particular its possible relation  with the Hodge filtrations of $D$ when $D$ is de Rham.

\subsection{Wall-crossings of $\pi(D)$}
Applying the wall-crossing functors to $\pi(D)$, we obtain various locally $\Q_p$-analytic representations of $\GL_n(K)$. We make some conjectures and speculations on these representations. As we find it more convenient to work on the dual side,  throughout the section, we mainly consider the translations of $\pi(D)^*$ (see Remark \ref{socSt} (1) for a related discussion).

By Lemma \ref{1dual}, $\pi(\Delta)^*\cong T_{\lambda^*}^{\theta_K^*} \pi(D)^*$. There are natural maps of functors $\iota: \id \ra T_{\theta_K^*}^{\lambda^*} T_{\lambda^*}^{\theta_K^*}$ and $\kappa: T_{\theta_K^*}^{\lambda^*} T_{\lambda^*}^{\theta_K^*} \ra \id$ induced by $E \hookrightarrow L(-w_0(\textbf{h})) \otimes_E L(-w_0(\textbf{h}))^{\vee}$ and $L(-w_0(\textbf{h})) \otimes_E L(-w_0(\textbf{h}))^{\vee} \twoheadrightarrow E$ respectively. 
\begin{conjecture}
\label{conj2}
(1)		The natural map $\iota: \pi(D)^* \ra (T_{\theta_K^*}^{\lambda^*} \pi(\Delta)^*)[\cZ_K=\chi_{\lambda}]=:\pi(\Delta,\lambda)^*$ is injective.

(2) $\pi(\Delta, \lambda)^*$ and the image $\pi_0(\Delta, \lambda)^*$ of $\kappa: T_{\theta_K^*}^{\lambda^*} \pi(\Delta)^* \ra \pi(D)^*$ only depend on $\Delta$ and $\lambda$.

(3) Assume that $\textbf{h}$ is strictly dominant.  The followings are equivalent:

(i) $\dim_E D_{\dR}(\rho)_{\sigma}=1$ for all $\sigma\in \Sigma_K$, 

(ii) $\pi_0(\Delta, \lambda)^* \xrightarrow{\sim} \pi(D)^*$,

(iii) $\pi(D)^* \xrightarrow{\sim} \pi(\Delta, \lambda)^*$. 
\end{conjecture}
\begin{remark} \label{Rwc}	
(1) We can also construct $\pi(\Delta, \lambda)=(\pi(\Delta, \lambda)^*)^*$ and $\pi_0(\Delta, \lambda)=(\pi_0(\Delta, \lambda)^*)^*$ from the translation of $\pi(\Delta)$. Indeed, we have similar maps $\iota': \pi(D)\ra T_{-\theta_K}^{\lambda} \pi(\Delta)$ and $\kappa': T_{-\theta_K}^{\lambda} \pi(\Delta) \ra \pi(D)$. Using the argument in the proof of Lemma \ref{1dual},  $\iota'$ (resp. $\kappa'$) coincides with the dual of $\kappa$ (resp. of $\iota$). Hence $\iota'$ factors through the quotient $\pi_0(\Delta, \lambda)$ of $\pi(D)$, and $\kappa'$ factors though its $\chi_{\lambda}$-covariant quotient $\pi(\Delta, \lambda)$. Conjecture \ref{conj2} (1) is equivalent to that $\kappa'$ is surjective. 

(2) Conjecture \ref{conj2} (1) indicates that $\pi(\Delta, \lambda)$ is a ``universal" object (depending only on $\Delta$ and $\lambda$). Namely, for any \'etale $(\varphi, \Gamma)$-module $D'$ of Sen weights $\textbf{h}$  with $D'[\frac{1}{t}]\cong \Delta[\frac{1}{t}]$, $\pi(D')$ should be a quotient of $\pi(\Delta, \lambda)$. 

(3) Suppose $\Delta$ is de Rham, and $\textbf{h}$ is strictly dominant (hence $\lambda$ is dominant). Let $\pi_{\infty}(\Delta)$ be the smooth $\GL_n(K)$-representation associated to $\Delta$ via the classical local Langlands correspondence. We conjecture that there exists $r\geq 1$ such that 
\begin{equation*}
	\pi(\Delta, \lambda)^{\lalg} \cong (\pi_{\infty}(\Delta) \otimes_E L(\lambda))^{\oplus r}.
\end{equation*}

(4)  Suppose $\Delta$ is de Rham, and $\textbf{h}$ is strictly dominant (hence $\lambda$ is dominant). The quotient map $\kappa': \pi(\Delta, \lambda) \twoheadrightarrow \pi(D)$ should decode the information of Hodge filtrations of $D$. Note that when $n>2$ or $K\neq \Q_p$, the kernel of $\kappa'$ should  also carry some information of Hodge filtrations of $\rho$ (see discussions in \S~\ref{secpirho}). Similarly, the cokernel $\pi(D)/\pi_0(\Delta, \lambda)$ should  carry some piece of information on Hodge filtrations,  which  together with its extension group by $\pi_0(\Delta, \lambda)$, would reveal the full information of $\rho$.

(5) There should exist certain Schneider-Teitelbaum dualities between $\pi(\Delta, \lambda)^*$ and $\pi_0(\Delta, \lambda)^*$ (see for example Conjecture \ref{conjdim}, Theorem \ref{TtubeGL2}). 

(6) Conjecture \ref{conj2} (3) is closely related to the ``local avatar" of the Fontaine-Mazur conjecture: $\rho$ is de Rham of distinct Hodge-Tate weights if and only if $\pi(D)$ has non-zero locally algebraic vectors. Indeed, when $n=2$ and $K=\Q_p$, $\Coker \kappa$ is exactly the dual of the locally algebraic subrepresentation of $\pi(D)$ (see Proposition \ref{PgIfini}) hence (ii) is equivalent to that $\pi(D)$ does not have non-zero locally algebraic vectors.



(7) Assume $D$ is generic crystabelline and $\lambda$ is dominant. Let $\pi(D)^{\fss} \subset  \pi(D)$ be its finite slope part given  in \cite[\S~5]{BH2}. By the proof of Theorem \ref{thmlgfs}, $\pi(\Delta)^{\fss}$ is no other than $T_{\lambda}^{-\theta_K} \pi(D)^{\fss}$. We have hence natural commutative diagrams ($\kappa'$, $\iota'$ being as in the above (1), and $\kappa'_{\fss}$, $\iota'_{\fss}$ being defined in a similar way with $\pi(D)$ replaced by $\pi(D)^{\fss}$)
\begin{equation*}
	\begin{CD}
		T_{-\theta_K}^{\lambda} \pi(\Delta)^{\fss} 	 @>>> T_{-\theta_K}^{\lambda} \pi(\Delta)  @. \ \ \ \ \ \ \  @.  \pi(D)^{\fss} @>>> \pi(D) \\
		@V\kappa'_{\fss}VV @V\kappa' VV @.   @V \iota'_{\fss} VV @V\iota' VV \\
		\pi(D)^{\fss} @>>> \pi(D) @. \ \ \ \ \ \ \  @. T_{-\theta_K}^{\lambda} \pi(\Delta)^{\fss}  @>>> T_{-\theta_K}^{\lambda} \pi(\Delta)
	\end{CD}
\end{equation*}
By results in \cite{Hum08} (on the wall-crossing of Verma modules) and \cite[Thm.~4.1.2]{JLS}, we have an isomorphism $\rm{cosoc}_{\GL_n(K)}(T_{-\theta_K}^{\lambda} \pi(\Delta)^{\fss}) \cong \oplus_{\phi} \cF_{B^-}^{\GL_n}(L^-(-w_0 \cdot \lambda), \jmath(\phi))$ and  $$T_{-\theta_K}^{\lambda} \pi(\Delta)^{\fss}\twoheadrightarrow \oplus_{\phi} \cF_{B^-}^{\GL_n}(M^-(-\lambda), \jmath(\phi))\twoheadrightarrow  \oplus_{\phi} \cF_{B^-}^{\GL_n}(L^-(-w_0 \cdot \lambda), \jmath(\phi)).$$ We also have  $T_{-\theta_K}^{\lambda} \pi(\Delta)^{\fss}[\cZ_K=\chi_{\lambda}]\cong \oplus_{\phi}  \cF_{B^-}^{\GL_n}(M^-(-\lambda)^{\vee}, \jmath(\phi))$ and $\soc_{\GL_n(K)} T_{-\theta_K}^{\lambda} \pi(\Delta)^{\fss} \cong \cF_{B^-}^{\GL_n}(L^-(-w_0\cdot \lambda), \jmath(\phi))$.
Moreover, 	it is not difficult to see $\kappa_{\fss}'$ factors though 
\begin{equation*}
	\oplus_{\phi} \cF_{B^-}^{\GL_n}(M^-(-\lambda), \jmath(\phi)) \twoheadlongrightarrow \oplus_{\phi}  \cF_{B^-}^{\GL_n}(M^-(-w_{\phi} \cdot  \lambda), \jmath(\phi)) \hooklongrightarrow \pi(D)^{\fss}
\end{equation*}
where $w_{\phi}\in \sW_K$ is the maximal length element such that  $\cF_{B^-}^{\GL_n}(L^-(-w_{\phi} \cdot  \lambda), \jmath(\phi))\hookrightarrow \pi(D)$.
In particular, $\kappa_{\fss}'$ is surjective if and only if $\soc_{\GL_n(K)} \pi(D)^{\fss}\cong \pi(D)^{\lalg}$, in contrast to that $\kappa'$ is conjectured to be always surjective.  Indeed, the composition $T_{-\theta_K}^{\lambda} \pi(\Delta)^{\fss} 	 \hookrightarrow T_{-\theta_K}^{\lambda} \xrightarrow{\kappa'} \pi(D)$ already carries the information on $w_{\phi}$. This also suggests that the translation of the ``supersingular" constituents could have Orlik-Strauch representations as constituents when $K\neq \Q_p$ or $n>2$.  For $\iota_{\fss}'$, one can show that  $\Ima(\iota_{\fss}')\cong \oplus_{\phi}  \cF_{B^-}^{\GL_n}(L^-(-w_0 \cdot \lambda), \jmath(\phi))$. 

(8) Finally remark that we may use various  wall-crossing functors (for the  walls in the Weyl chamber)  to obtain more representations. We will study these for $\GL_2(K)$ in the next section. 
\end{remark}
\begin{theorem}\label{TGL2}
The conjecture holds for $\GL_2(\Q_p)$. 
\end{theorem}
\begin{proof}
The theorem follows from results in \cite{Ding14}. We include a proof for the reader's convenience.
The case where $\lambda$ is itself singular is trivial, we assume the weight $\lambda$ is regular.  Twisting $D$ by a certain character, we can and do assume $\textbf{h}=(k,0)$ with $k \in \Z_{\geq 1}$.  Let $\fc=\fh^2-2\fh+4u^+u^-\in \text{U}(\gl_2)$ be the Casimir element, and $\fz=\begin{pmatrix} 1 & 0 \\ 0 & 1\end{pmatrix}$. Recall  $D$ is equipped with a natural $\gl_2$-action with $\fz=k-1$ and $\fc=k^2-1$ (cf. \cite[\S~3.2.1]{Colm16}, see also \cite[Prop.~2.13]{Ding14}). We have $T_{\lambda}^{-\theta}=(-\otimes_EV_k)[\fc=-1] \otimes_E z^{-k} \circ \dett$.  By \cite[Prop.~2.19 (1) ]{Ding14},  $T_{\lambda}^{-\theta}D \cong \Delta$.  
By Lemma \ref{traexseq}, we have an exact sequence
\begin{equation}
	0 \lra \Delta^{\sharp}  \lra T_{-\theta}^{\lambda}\Delta \lra \Delta^{\flat} \lra 0,
\end{equation}
where $\Delta^{\sharp}:=(T_{-\theta}^{\lambda}\Delta)[\fc=k^2-1]$ (hence $\iota$ factors through $D \hookrightarrow \Delta^{\sharp}$), and where $\kappa: T_{-\theta}^{\lambda}\Delta \ra D$ factors through an injection $\Delta^{\flat} \hookrightarrow D$.  Note that $\Delta^{\sharp}$ and $\Delta^{\flat}$ are all $(\varphi, \Gamma)$-modules of rank $2$ (which we can precisely describe as in \cite[Prop.~2.19]{Ding14}). Let $\delta_D$ be the continuous character such that $\cR_E(\delta_D \varepsilon)\cong \wedge^2 D$. By the discussion above \cite[Prop.~3.1]{Ding14}, the involution $\omega_D$ on $D^{\psi=0}$ (cf. \cite[Thm.~0.1 \& \S~4.3.1]{Colm16}) induces an involution on $(T_{-\theta}^{\lambda} T_{\lambda}^{-\theta} D)^{\psi=0}$ hence on $(\Delta^{\sharp})^{\psi=0}$ and $(\Delta^{\flat})^{\psi=0}$. Using these involutions, we have a $\GL_2(\Q_p)$-equivariant  exact sequence (cf. \cite[\S~3.1]{Ding14}):
\begin{equation*}
	0 \lra \Delta^{\sharp} \boxtimes_{\delta_D} \bP^1(\Q_p) \lra (T_{-\theta}^{\lambda} \Delta) \boxtimes_{\delta_D} \bP^1(\Q_p) \lra \Delta^{\flat} \boxtimes_{\delta_D} \bP^1(\Q_p) \lra 0.
\end{equation*} We have $(T_{-\theta}^{\lambda} \Delta) \boxtimes_{\delta_D} \bP^1(\Q_p)\cong T_{-\theta}^{\lambda}T_{\lambda}^{-\theta} (D \boxtimes_{\delta_D} \bP^1(\Q_p))$. By similar arguments as in \cite[Thm.~3.4]{Ding14}, we have 
\begin{equation*}\footnotesize
	\begin{CD} 
		0 @>>>  \pi(\Delta, \lambda)^* \otimes_E \delta_D\circ \dett @>>> T_{-\theta}^{\lambda}T_{\lambda}^{-\theta} (\pi(D)^*\otimes_E \delta_D\circ \dett) @>>> \pi_0(\Delta, \lambda)^*\otimes_E \delta_D\circ \dett @>>> 0 \\
		@. @VVV @VVV @VVV @. \\
		0 @>>> \Delta^{\sharp} \boxtimes_{\delta_D} \bP^1(\Q_p)  @>>>  (T_{-\theta_K}^{\lambda} \Delta) \boxtimes_{\delta_D} \bP^1(\Q_p)  @>>>  \Delta^{\flat} \boxtimes_{\delta_D} \bP^1(\Q_p) @>>> 0 
	\end{CD}.
\end{equation*}
And the maps $\iota$, $\kappa$ for $\pi(D)^*$ are compatible with the respective maps for $D$.
As $D$ does not have non-zero $\text{U}(\gl_2)$-finite vectors (using $X$ is invertible in $\cR_E$), by Lemma \ref{injlemm}, $\iota: D \ra T_{-\theta}^{\lambda} T_{\lambda}^{-\theta} D$ is injective. Hence $D \boxtimes_{\delta_D} \bP^1(\Q_p) \hookrightarrow \Delta^{\sharp} \boxtimes_{\delta_D} \bP^1(\Q_p)$, and (1) follows. By \cite[Prop. 2.19]{Ding14}, it is easy to see $\Delta^{\flat}$ just depends on $\Delta$ and $\lambda$, so does $\pi_0(\Delta, \lambda)^*$. Finally, by \cite[Prop. 2.19 (3) (4)]{Ding14}, $D$ is not de Rham if and only if  $\Delta^{\flat} \cong D$ if and only if  $\Delta^{\sharp}\cong D$. (3) follows. 
\end{proof}
\begin{remark}
(1)	By Lemma \ref{injlemm}, the statement in (1) for $\GL_2(\Q_p)$ is a consequence of the following:
\begin{equation}\label{Enoalg}
	\text{ $\pi(D)^*$ does not have non-zero $\text{U}(\ug)$-finite vectors.}
\end{equation} For this, we (tacitly) used in the proof the fact  that $\pi(D)^*\otimes_E \delta_D \circ \dett $ injects into $D \boxtimes_{\delta_D} \bP^1(\Q_p)$. However, (\ref{Enoalg}) is also a direct consequence that $\pi(D)^*$, as coadmissible $\cD(H,E)$-module, is pure of dimension $1$. Later, we will use a similar argument to prove (1) for $\GL_2(\Q_{p^2})$ (under mild assumptions).

(2) Certain parts of the statement in  (3) can also be dealt with  using global method (without using $(\varphi, \Gamma)$-modules), see Remark \ref{Rgl2Qp2}.

(3) Our proof of the statement in (2) crucially relies on Colmez's construction of $\pi(D)$ via $(\varphi, \Gamma)$-modules. However, when $\pi(D)$ appears in the completed cohomology of modular curves (with trivial coefficient), then $T_{-\theta}^{\lambda} T_{\lambda}^{-\theta}\pi(D)$ appears in the completed cohomology with non-trivial coefficients (cf. \cite[Thm.~2.2.17]{Em04}).  In this setting,  one may expect a geometric proof of (2) via Pan's approach (cf. \cite{Pan3} \cite{Pan4}).
\end{remark}

\subsection{Hodge filtration hypercubes for $\GL_2(K)$}We consider the case of $\GL_2(K)$. \label{Shfh}
Applying various wall-crossing functors to $\pi(D)$, we construct two hypercubes consisting of locally $\Q_p$-analytic representations of $\GL_2(K)$. We discuss their possible relation with the Hodge filtration of $\rho$. We show some properties of the hypercubes, notably for $K=\Q_{p^2}$. We then use the hypercubes to investigate the internal structure of $\pi(D)$. 
\subsubsection{Formal constructions}\label{S4.2.1}Let $\lambda$ be a dominant integral weight of $\ft_K$. 
We use some functors on $\Mod(\text{U}(\ug_K)_{\chi_{\lambda}})$ to factorize the maps $\kappa$ and $\iota$ in the precedent section.  

We introduce some notation. Throughout  this section, let $\ug:=\gl_2$.  For $I\subset \Sigma_K$, and a Lie algebra $A$ over $K$, denote by $A_I:=\prod_{\sigma\in I} A \otimes_{K, \sigma} E$. When $I$ is a singleton, we use its element to denote it. 
We denote by $\lambda_I:=(\lambda_{\sigma})_{\sigma \in I}$, and $\theta_I=(1,0)_{\sigma\in I}$, as weights of $\ft_I$. For an integral dominant weight $\mu_I$ of $\ft_I$, we denote by $L_I(\mu_I)$ the algebraic representation of $\ug_{I}$ of weight $\mu_I$. Hence $L_I(\mu_I)\cong \otimes_{\sigma\in I} L_{\sigma}(\mu_{\sigma})$. Let $s_I:=\prod_{\sigma\in I} s_{\sigma}\in \sW_K$.

Let $\cZ_{\sigma}$ be the centre of $\text{U}(\ug_{\sigma})$, hence $\cZ_I:= \otimes_{\sigma\in I} \cZ_{\sigma}$ is the centre of $\text{U}(\ug_{I})$. As before, we use $\chi_{\lambda_I}$, $\chi_{-\theta_I}$ to denote the central characters (of $\cZ_I$) associated to $\lambda_I$, $-\theta_I$ respectively. Consider the translation functor  $T_{\lambda_I}^{-\theta_I}: \Mod(\text{U}(\ug_{I})_{\chi_{\lambda_I}}) \ra \Mod(\text{U}(\ug_{I})_{\chi_{-\theta_I}})$, which we also view as a functor on $\Mod(\text{U}(\gl_{2,J})_{\chi_{\lambda_J}})$ for $J\supset I$. Using the decomposition $\cZ_I\cong \otimes_{\sigma\in I} \cZ_{\sigma}$ and $L(-\theta_I-s_I\cdot \lambda_I)\cong \otimes_{\sigma\in I} L_{\sigma}(-\theta_{\sigma}-s_{\sigma}\cdot \lambda_{\sigma})$, it is clear that $\{T_{\lambda_{\sigma}}^{-\theta_{\sigma}}\}_{\sigma \in I}$ commute, and  $T_{\lambda_I}^{-\theta_I}=\prod_{\sigma\in I} T_{\lambda_{\sigma}}^{-\theta_{\sigma}}$. In fact,  by similar reasoning, the translation functors for different embeddings commute. 
Put $\Theta_I:=T_{-\theta_I}^{\lambda_I} T_{\lambda_I}^{-\theta_I}$, which we view as a functor $\Mod(\text{U}(\ug_K)_{\chi_{\lambda}}) \ra \Mod(\text{U}(\ug_K)_{\chi_{\lambda}})$. Then $\Theta_{\sigma}$, for $\sigma\in K$ commute, and we have $\Theta_I=\prod_{\sigma\in K} \Theta_{\sigma}$. Finally,  it is easy to see the wall-crossing functor $\Theta_I$ does not change if $-\theta_I$ is replaced  by $-\theta_I^*$.

There are natural maps $\iota_I: \id \ra \Theta_I$, $\kappa_I: \Theta_I \ra \id$, induced by $\iota_I: E \hookrightarrow L_I(-\theta_I-s_I\cdot \lambda_I) \otimes_E L_I(-\theta_I- s_I \cdot \lambda_I)^{\vee}$, and $\kappa_I: L_I(-\theta_I-s_I \cdot \lambda_I) \otimes_E L_I(-\theta_I- s_I\cdot\lambda_I)^{\vee} \ra E$ respectively. Denote by $\vartheta_I^+:=\Theta_I(-)[\cZ_K=\chi_{\lambda}]\hookrightarrow \Theta_I$, which is left exact. Note $\id \ra \Theta_I$ factors through $\id \ra \vartheta_I^+$. Denote by $\vartheta_I^-$ the functor sending $M$ to $\Ima[\Theta_I(M) \ra M]$.


\begin{lemma}\label{sigmataucomm}
(1) $\vartheta_{\sigma}^+ \vartheta_{\tau}^+=\vartheta_{\tau}^+ \vartheta_{\sigma}^+$ and $\vartheta_{\sigma}^+ \vartheta_{\sigma}^+=\vartheta_{\sigma}^+$.

(2) $\vartheta_{\sigma}^-\vartheta_{\tau}^-=\vartheta_{\tau}^-\vartheta_{\sigma}^-$ and $\vartheta_{\sigma}^-\vartheta_{\sigma}^-=\vartheta_{\sigma}^-$.

(3) For $I\subset \Sigma_K$, we have $\vartheta_I^+=\prod_{\sigma\in I} \vartheta_{\sigma}^+$ and $\vartheta_I^-=\prod_{\sigma\in I} \vartheta_{\sigma}^-$.  
\end{lemma}
\begin{proof}
As $\Theta_{\sigma}\Theta_{\tau}=\Theta_{\tau} \Theta_{\sigma}$ and $\cZ\cong \otimes_{\sigma'} \cZ_{\sigma'}$, the first part of (1) and (3) easily follow. The second part of (1) and (2) follow from Lemma  \ref{lemsharp} (1) (2) respectively. For $M\in \Mod(\text{U}(\ug_K)_{\chi_{\lambda}})$, we have a commutative diagram with surjective vertical maps
\begin{equation*}
	\begin{CD}
		\Theta_{\tau} \Theta_{\sigma} M @> \kappa_{\tau} >> \Theta_{\sigma} M \\
		@VVV @V\kappa_{\sigma} VV \\
		\Theta_{\tau} \vartheta_{\sigma}^- M @>>> \vartheta_{\sigma}^-M
	\end{CD}
\end{equation*}
hence $\vartheta^-_{\{\sigma, \tau\}}=\vartheta^-_{\tau}\vartheta^-_{\sigma}$. The rest part of (2) and (3) follow.  
\end{proof}We put $\nabla_{\emptyset}^{\pm}=\vartheta_{\emptyset}^{\pm}:=\id$. For $I\subset \Sigma_K$, we denote by 
\begin{equation*}
\nabla_I^+:=\Coker\big(\sum_{I'\subsetneq I} \vartheta_{I'}^+ \ra \vartheta_I^+\big)=\Coker\big(\sum_{\sigma \in I} \vartheta_{I\setminus \{\sigma\}}^+ \ra \vartheta_I^+\big),\ 
\nabla_I^-:=\Coker\big(\sum_{\sigma \in I} \vartheta_{\sigma}^- \ra \id\big).
\end{equation*}
\begin{lemma}\label{nabla-}
We have $\nabla_I^-=\prod_{\sigma\in I} \nabla_\sigma^-$.
\end{lemma}
\begin{proof}
For $\sigma\in I$, we have a commutative diagram
\begin{equation*}
	\begin{CD}
		\Theta_{\sigma} (\sum_{\tau\in I\setminus \{\sigma\}} \vartheta_{\tau}^-)@>>> \Theta_{\sigma} @>>> \Theta_{\sigma} \nabla_{I\setminus \{\sigma\}}^- @>>> 0 \\
		@VVV @VVV @VVV @. \\
		\sum_{\tau\in I\setminus \{\sigma\}} \vartheta_{\tau}^- @>>>  \id @>>>  \nabla_{I\setminus \{\sigma\}}^- @>>> 0 \\
		@VVV @VVV @VVV @.\\
		\nabla_{\sigma}^- (\sum_{\tau\in I\setminus \{\sigma\}} \vartheta_{\tau}^-) @>>> \nabla_{\sigma}^- @>>> \nabla_{\sigma}^-(\nabla_{I\setminus \{\sigma\}}^-) @>>> 0
	\end{CD}.
\end{equation*}
The lemma follows.
\end{proof}
\begin{lemma}\label{lmbox0} Let $M\in \Mod(\text{U}(\ug_K)_{\chi_{\lambda}})$ and $\emptyset \neq I \subset \Sigma_K$.  

(1)  $\Fer[M \ra \vartheta_I^+M]$ is generated by $\text{U}(\ug_{\sigma})$-finite vectors for $\sigma\in I$. 

(2) $\nabla_I^+ M$ are $\nabla_I^- M$  are generated by $\text{U}(\ug_{I})$-finite vectors.
\end{lemma}
\begin{proof}
The case $I=\{\sigma\}$ follows from Lemma \ref{injlemm}. (1) follows by factorising $M \ra \vartheta_I^+M$ as
\begin{equation*}
	M \ra \vartheta_{\sigma_1}^+ M \ra \vartheta_{\sigma_2}^+ \vartheta_{\sigma_1}^+ M \ra \cdots \ra \vartheta_I^+M.
\end{equation*}
(2) follows from the singleton case together with the  surjective maps $\vartheta_I^+ M/\vartheta_{I\setminus \{\sigma\}}^+ M \twoheadrightarrow \nabla_I^+ M$ and $M/\vartheta^-_{\sigma} M\twoheadrightarrow \nabla_I^- M$ for $\sigma\in I$. 
\end{proof}

We use the maps $\big\{\nabla_{I}^+ \vartheta_{J}^+ \ra \nabla_{I \cup \{\sigma\}}^+ \vartheta_{J \setminus \{\sigma\} }^+\big\}_{\substack{I\cap J=\emptyset \\ \sigma\in J}}$ (induced by $\vartheta^+_{\sigma} \ra \nabla_{\sigma}^+$) and the maps $\big\{ \nabla_{I}^+ \vartheta_{J}^+ \ra \nabla_{I}^+ \vartheta_{J \cup \{\sigma\}}^+\big\}_{\substack{I\cap J=\emptyset \\ \sigma \notin I\cup J}} $ (induced by $\id \ra \vartheta_{\sigma}^+$) to build a $d_K$-dimensional hypercube, denoted by $\boxdot^+$ such that the centres of its $d$-dimensional faces (that are $d$-dimensional hypercubes) are  given  by $\{\nabla_{I}^+ \vartheta_{J}^+\}$ with $\#J=d$ (and $I\cap J=\emptyset$). For example, the centre of the entire hypercube is $\vartheta_{\Sigma_K}^+$, the $2^{(n-1)d_K}$ vertexes are given by $\{\nabla_{I}^+ \}_{I\subset \Sigma_K}$, and the centres of the $1$-dimensional faces are $\{\nabla_{I}^+ \vartheta_{\sigma}^+\}_{\substack{\sigma\in \Sigma_K\\ I \subset \Sigma_K\setminus \{\sigma\}}}$... Note that for $I\cap J =\emptyset$, and $\sigma\notin I\cup J$, the following sequence is exact:
\begin{equation}\label{edge+}
\nabla_I^+ \vartheta_J^+ \lra \nabla_I^+ \vartheta_{J\cup \{\sigma\}}^+ \lra \nabla_{I\cup \{\sigma\}}^+ \vartheta_{J}^+ \lra 0.
\end{equation}
Similarly, the maps 
$\big\{\nabla_{I}^- \vartheta_{J}^- \ra \nabla_{I}^- \vartheta_{J\setminus \{\sigma\}}^-\big\}_{\substack{I\cap J=\emptyset \\ \sigma\in J}}$ (induced by $\vartheta_{\sigma}^- \ra \id$) and the maps $\big\{  \nabla_{I}^- \vartheta_{J}^- \ra \nabla_{I \cup \{\sigma\}}^- \vartheta_{J}^-\big\}_{\substack{ I \cap J =\emptyset\\ \sigma \in \Sigma_K \setminus (I\cup J)}}$ (induced by the natural quotient maps)  
form an $d_K$-dimenisonal hypercube, denoted by $\boxdot^-$ such that the centres of its $d$-dimensional faces are  given  by $\{\nabla_{I}^- \vartheta_{J}^-\}$ with $\#(I\cup J)=2d_K-d$ (and $I\cap J=\emptyset$). For example, the centre of the entire hypercube is $\id$, the $2^{(n-1)d_K}$ vertexes are given by $\{\nabla_{I}^-\vartheta_{\Sigma_K \setminus I}^- \}_{I\subset\Sigma_K}$, and the centres of the $1$-dimensional faces are $\{\nabla_{I}^-\vartheta_{J}^-\}_{\substack{\sigma\in \Sigma_K\\ I \cup J =\Sigma_K\setminus \{\sigma\}}}$... For $I\cap J=\emptyset$, $\sigma \in J$, the following sequence is exact
\begin{equation}\label{egdge-}
\nabla_I^- \vartheta_J^- \lra \nabla_I^- \vartheta_{J\setminus \{\sigma\}}^- \lra \nabla_{I \cup \{\sigma\}}^- \vartheta_J^- \lra 0.
\end{equation}
\begin{remark}
In general, the functors $\nabla_I^+$ and $\vartheta_J^+$ (resp. $\nabla_I^-$ and $\vartheta_J^-$) may not commute. However, one may expect they commute when applied to $\pi(D)^*$. See Conjecture \ref{conjGL2} (1). 
\end{remark}
\begin{example}
We give an example for $d_K=3$ to illustrate the above construction. We use $123$ to label the embeddings in $\Sigma_K$. Then the cube $\boxdot^+$ has the form:
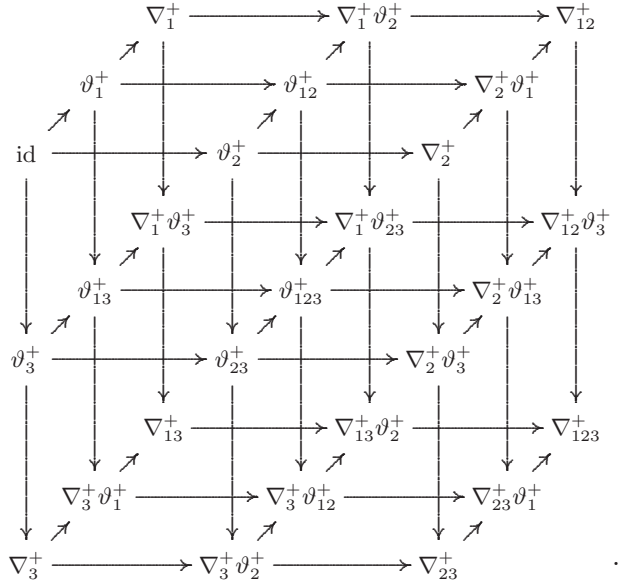
\begin{equation*}\footnotesize\begindc{\commdiag}[260]
	\obj(0,0)[a]{$\nabla_3^+$}
	\obj(3,0)[b]{$\nabla_3^+\vartheta_2^+$}
	\obj(6,0)[c]{$\nabla_{23}^+$}
	\obj(0,3)[d]{$\vartheta_3^+$}
	\obj(0,6)[e]{$\id$}
	\obj(3,3)[f]{$\vartheta_{23}^+$}
	\obj(3,6)[g]{$\vartheta_2^+$}
	\obj(6,3)[h]{$\nabla_2^+ \vartheta_3^+$}
	\obj(6,6)[i]{$\nabla_2^+ $}
	\obj(1,1)[a1]{$\nabla_3^+ \vartheta_1^+$}
	\obj(4,1)[b1]{$\nabla_3^+\vartheta_{12}^+$}
	\obj(7,1)[c1]{$\nabla_{23}^+ \vartheta_1^+$}
	\obj(1,4)[d1]{$\vartheta_{13}^+$}
	\obj(1,7)[e1]{$\vartheta_1^+$}
	\obj(4,4)[f1]{$\vartheta_{123}^+$}
	\obj(4,7)[g1]{$\vartheta_{12}^+$}
	\obj(7,4)[h1]{$\nabla_2^+\vartheta_{13}^+$}
	\obj(7,7)[i1]{$\nabla_2^+ \vartheta_1^+$}
	\obj(2,2)[a2]{$\nabla_{13}^+$}
	\obj(5,2)[b2]{$\nabla_{13}^+\vartheta_2^+$}
	\obj(8,2)[c2]{$\nabla_{123}^+$}
	\obj(2,5)[d2]{$\nabla_1^+ \vartheta_3^+$}
	\obj(2,8)[e2]{$\nabla_1^+ $}
	\obj(5,5)[f2]{$\nabla_1^+\vartheta_{23}^+$}
	\obj(5,8)[g2]{$\nabla_1^+ \vartheta_2^+$}
	\obj(8,5)[h2]{$\nabla_{12}^+ \vartheta_3^+$}
	\obj(8,8)[i2]{$\nabla_{12}^+ $}
	\mor{a}{b}{}
	\mor{d}{a}{}
	\mor{b}{c}{}
	\mor{f}{b}{}
	\mor{h}{c}{}
	\mor{e}{d}{}
	\mor{d}{f}{}
	\mor{e}{g}{}
	\mor{g}{f}{}
	\mor{f}{h}{}
	\mor{g}{i}{}
	\mor{i}{h}{}
	\mor{a1}{b1}{}
	\mor{d1}{a1}{}
	\mor{b1}{c1}{}
	\mor{f1}{b1}{}
	\mor{h1}{c1}{}
	\mor{e1}{d1}{}
	\mor{d1}{f1}{}
	\mor{e1}{g1}{}
	\mor{g1}{f1}{}
	\mor{f1}{h1}{}
	\mor{g1}{i1}{}
	\mor{i1}{h1}{}
	\mor{a2}{b2}{}
	\mor{d2}{a2}{}
	\mor{b2}{c2}{}
	\mor{f2}{b2}{}
	\mor{h2}{c2}{}
	\mor{e2}{d2}{}
	\mor{d2}{f2}{}
	\mor{e2}{g2}{}
	\mor{g2}{f2}{}
	\mor{f2}{h2}{}
	\mor{g2}{i2}{}
	\mor{i2}{h2}{}
	\mor{a}{a1}{}
	\mor{b}{b1}{}
	\mor{c}{c1}{}
	\mor{d}{d1}{}
	\mor{e}{e1}{}
	\mor{f}{f1}{}
	\mor{g}{g1}{}
	\mor{h}{h1}{}
	\mor{i}{i1}{}
	\mor{a1}{a2}{}
	\mor{b1}{b2}{}
	\mor{c1}{c2}{}
	\mor{d1}{d2}{}
	\mor{e1}{e2}{}
	\mor{f1}{f2}{}
	\mor{g1}{g2}{}
	\mor{h1}{h2}{}
	\mor{i1}{i2}{}
	\enddc.\end{equation*}
$\boxdot^-$ has the following form:
\begin{equation*}\footnotesize\begindc{\commdiag}[260]
	\obj(0,0)[a]{$\nabla_3^-\vartheta_{12}^-$}
	\obj(3,0)[b]{$\nabla_3^- \vartheta_{1}^-$}
	\obj(6,0)[c]{$\nabla_{23}^- \vartheta_{1}^-$}
	\obj(0,3)[d]{$\vartheta_{12}^-$}
	\obj(0,6)[e]{$\vartheta_{123}^-$}
	\obj(3,3)[f]{$\vartheta_{1}^-$}
	\obj(3,6)[g]{$\vartheta_{13}^-$}
	\obj(6,3)[h]{$\nabla_2^- \vartheta_{1}^-$}
	\obj(6,6)[i]{$\nabla_2^- \vartheta_{13}^-$}
	\obj(1,1)[a1]{$\nabla_3^- \vartheta_{2}^-$}
	\obj(4,1)[b1]{$\nabla_3^-$}
	\obj(7,1)[c1]{$\nabla_{23}^- $}
	\obj(1,4)[d1]{$\vartheta_{2}^-$}
	\obj(1,7)[e1]{$\vartheta_{23}^-$}
	\obj(4,4)[f1]{$\id$}
	\obj(4,7)[g1]{$\vartheta_{3}^-$}
	\obj(7,4)[h1]{$\nabla_2^-$}
	\obj(7,7)[i1]{$\nabla_2^- \vartheta_{3}^-$}
	\obj(2,2)[a2]{$\nabla_{13}^-\vartheta_{2}^-$}
	\obj(5,2)[b2]{$\nabla_{13}^-$}
	\obj(8,2)[c2]{$\nabla_{123}^- $}
	\obj(2,5)[d2]{$\nabla_1^- \vartheta_{2}^-$}
	\obj(2,8)[e2]{$\nabla_1^- \vartheta_{13}^-$}
	\obj(5,5)[f2]{$\nabla_1^-$}
	\obj(5,8)[g2]{$\nabla_1^- \vartheta_{3}^-$}
	\obj(8,5)[h2]{$\nabla_{12}^- $}
	\obj(8,8)[i2]{$\nabla_{12}^- \vartheta_{3}^-$}
	\mor{a}{b}{}
	\mor{d}{a}{}
	\mor{b}{c}{}
	\mor{f}{b}{}
	\mor{h}{c}{}
	\mor{e}{d}{}
	\mor{d}{f}{}
	\mor{e}{g}{}
	\mor{g}{f}{}
	\mor{f}{h}{}
	\mor{g}{i}{}
	\mor{i}{h}{}
	\mor{a1}{b1}{}
	\mor{d1}{a1}{}
	\mor{b1}{c1}{}
	\mor{f1}{b1}{}
	\mor{h1}{c1}{}
	\mor{e1}{d1}{}
	\mor{d1}{f1}{}
	\mor{e1}{g1}{}
	\mor{g1}{f1}{}
	\mor{f1}{h1}{}
	\mor{g1}{i1}{}
	\mor{i1}{h1}{}
	\mor{a2}{b2}{}
	\mor{d2}{a2}{}
	\mor{b2}{c2}{}
	\mor{f2}{b2}{}
	\mor{h2}{c2}{}
	\mor{e2}{d2}{}
	\mor{d2}{f2}{}
	\mor{e2}{g2}{}
	\mor{g2}{f2}{}
	\mor{f2}{h2}{}
	\mor{g2}{i2}{}
	\mor{i2}{h2}{}
	\mor{a}{a1}{}
	\mor{b}{b1}{}
	\mor{c}{c1}{}
	\mor{d}{d1}{}
	\mor{e}{e1}{}
	\mor{f}{f1}{}
	\mor{g}{g1}{}
	\mor{h}{h1}{}
	\mor{i}{i1}{}
	\mor{a1}{a2}{}
	\mor{b1}{b2}{}
	\mor{c1}{c2}{}
	\mor{d1}{d2}{}
	\mor{e1}{e2}{}
	\mor{f1}{f2}{}
	\mor{g1}{g2}{}
	\mor{h1}{h2}{}
	\mor{i1}{i2}{}
	\enddc.\end{equation*}
\end{example}

\begin{example}
(1) We apply the functors to the Verma modules $M(\lambda)\cong \otimes_{\sigma\in \Sigma_K} M_{\sigma}(\lambda_{\sigma})$ and $M(s_{\Sigma_K} \cdot \lambda)$. For $\sigma \in \Sigma_K$, denote by $P_{\sigma}(s_{\sigma} \cdot \lambda_{\sigma})$ the projective envelop of $L_{\sigma}(s_{\sigma}\cdot \lambda_{\sigma})$ in the BGG category $\co_{\ub_{\sigma}}$, which is a non-split  extension of $L_{\sigma}(s_{\sigma} \cdot \lambda)$ by $M_{\sigma}(\lambda_{\sigma})$. Denote by $P_I(s_I\cdot \lambda_I):=\otimes_{\sigma\in I} P_{\sigma}(s_{\sigma} \cdot \lambda_{\sigma})$ which is actually the projective envelop of $L_I(s_I\cdot \lambda_I)$ in $\co_{\ub_I}$.  We have 
\begin{equation}\label{sigtran}
	\Theta_{\sigma} M_{\sigma}(\lambda_{\sigma})\cong \Theta_{\sigma} L_{\sigma}(s_{\sigma}\cdot \lambda_{\sigma})\cong P_{\sigma}(s_{\sigma} \cdot \lambda_{\sigma}),
\end{equation}
hence $	\Theta_{\sigma} L_{\sigma}(\lambda_{\sigma})=0$.  We then deduce
\begin{equation*}\nabla_J^+ \vartheta_I^+ M(\lambda)\cong \begin{cases}
		M(\lambda) & J=\emptyset \\
		0 & \text{otherwise}
	\end{cases}
\end{equation*} 
so $\boxdot^+(M(\lambda))$ just has $M(\lambda)$ in one corner and $0$ in other parts. We also have
\begin{equation*}
	\nabla_J^- \vartheta_I^- M(\lambda)\cong \vartheta_I^- \nabla_J^- M(\lambda) \cong L_I(s_I\cdot \lambda_I) \otimes_E L_J(\lambda_J) \otimes_E M_{\Sigma_K\setminus (I\cup J)}(\lambda_{\Sigma_K \setminus (I\cup J)}),
\end{equation*}
and the sequence (\ref{egdge-}) is just the following natural exact sequence tensor with $L_I(s_I\cdot \lambda_I) \otimes_E L_{J\setminus \{\sigma\}}(\lambda_{J\setminus \{\sigma\}}) \otimes_E M_{\Sigma_K\setminus (I\cup J)}(\lambda_{\Sigma_K \setminus (I\cup J)})$
\begin{equation}\label{canoext1}
	0 \lra L_{\sigma}(s_{\sigma} \cdot \lambda_{\sigma}) \lra M_{\sigma}(\lambda_{\sigma}) \lra L_{\sigma}(\lambda_{\sigma}) \lra 0. 
\end{equation}
For $M(s_{\Sigma_K} \cdot \lambda)$, we have 
\begin{equation*}\nabla_J^- \vartheta_I^- M(s_{\Sigma_K} \cdot \lambda)\cong \begin{cases}
		M(s_{\Sigma_K} \cdot \lambda) & J=\emptyset \\
		0 & \text{otherwise}
	\end{cases}
\end{equation*}
so now $\boxdot^-(M(s_{\Sigma_K} \cdot \lambda))$ has $M(s_{\Sigma_K} \cdot \lambda)$ in one corner and $0$ elsewhere. Meanwhile,
\begin{equation*}
	\nabla_J^+ \vartheta_I^+ M(s_{\Sigma_K} \cdot \lambda)\cong \vartheta_I^+ \nabla_J^+M(s_{\Sigma_K} \cdot \lambda) \cong M_I(\lambda_I) \otimes L_J(\lambda_J) \otimes M_{\Sigma_K\setminus (I\cup J)}(s_{\Sigma_K \setminus (I\cup J)} \cdot \lambda_{\Sigma_K \setminus (I\cup J)}),
\end{equation*}
and the sequence (\ref{edge+}) in this case is induced from (\ref{canoext1}) by tensoring with the other terms.

(2) Let $M:=\big((\Ind_{B^-(K)}^{\GL_2(K)} \jmath(\phi) z^{\lambda})^{\Q_p-\an}\big)^*$ where $\phi=\phi_1\otimes \phi_2$ satisfies $\phi_1\phi_2^{-1}\neq 1, \unr_K(q_K)^{\pm 1}$.   By the discussion in (1) and \cite{OS}, we see the hypercube $\boxdot^-(M)$ is no other than the dual of the  hypercube considered in \cite[\S~4]{Br}.
\end{example}
Let $\mu$ be another dominant weight, in particular, $\lambda$ and $\mu$ lies in the same Weyl chamber. 
Recall the functor $T_{\lambda}^{\mu}$ induces an equivalence of categories with inverse given by $T_{\mu}^{\lambda}$. To distinguish the functors, we add $\lambda$ (resp. $\mu$) in the subscript for the above functors to emphasize that they are functors on $\Mod(\text{U}(\ug_K)_{\chi_{\lambda}})$ (resp. $\Mod(\text{U}(\ug_K)_{\chi_{\mu}})$).
\begin{lemma}
We have $T_{\mu}^{\lambda}(\nabla_{I,\mu}^{\pm} \vartheta_{J,\mu}^{\pm}) T_{\lambda}^{\mu}=\nabla_{I,\lambda}^{\pm} \vartheta_{J,\lambda}^{\pm}$. 
\end{lemma} 
\begin{proof}
It is clear that $T_{\mu}^{\lambda} \Theta_{I, \mu} T_{\lambda}^{\mu}=\Theta_{I,\lambda}$. Together with $T_{\mu}^{\lambda} T_{\lambda}^{\mu}=\id$, we see $T_{\mu}^{\lambda} \vartheta_{I,\mu}^{\pm} T_{\lambda}^{\mu}=\vartheta_{I,\lambda}^{\pm}$ and $T_{\mu}^{\lambda} \nabla_{I,\mu}^{\pm} T_{\lambda}^{\mu}=\nabla_{I,\lambda}^{\pm}$ for $I\subset \Sigma_K$. Finally $T_{\mu}^{\lambda}(\nabla_{I,\mu}^{\pm} \vartheta_{J,\mu}^{\pm}) T_{\lambda}^{\mu}=(T_{\mu}^{\lambda}\nabla_{I,\mu}^{\pm} T_{\lambda}^{\mu})(T_{\mu}^{\lambda}\vartheta_{J,\mu}^{\pm}T_{\lambda}^{\mu})=\nabla_{I,\lambda}^{\pm} \vartheta_{J,\lambda}^{\pm}$.
\end{proof}
\subsubsection{$\boxdot^+(\pi(D)^*)$ and $\boxdot^-(\pi(D)^*)$}Let $D$ be an \'etale $(\varphi, \Gamma)$-module of rank $2$ over $\cR_{K,E}$ of Sen weights $\textbf{h}$. Assume $\textbf{h}$ is strictly dominant hence $\lambda$ is dominant. 
We apply the above construction to $\pi(D)^*$, and denote by $\boxdot^+(\pi(D)^*)$ and $\boxdot^-(\pi(D)^*)$ the resulting diagrams (of the strong dual of admissible locally analytic representations). We show some properties of $\boxdot^{\pm}(\pi(D)^*)$. We also propose some conjectures.

\begin{lemma}\label{noalg2}
Let $V$ be a locally $\Q_p$-analytic representation of $\GL_2(K)$ on space of compact type, and suppose $\cZ_{\sigma}$ acts on $V$ via $\chi_{\lambda_{\sigma}}$. Then $T_{-\theta_{\sigma}}^{\lambda_{\sigma}} T_{\lambda_{\sigma}}^{-\theta_{\sigma}} V$ does not have non-zero $\text{U}(\ug_{\sigma})$-finite subrepresentations or $\text{U}(\ug_{\sigma})$-finite quotient representations.
\end{lemma}
\begin{proof}
The statement for subrepresentations follows from Lemma \ref{lmnoalg}. We have by Lemma \ref{1dual} $T_{-\theta_{\sigma}}^{\lambda_{\sigma}} T_{\lambda_{\sigma}}^{-\theta_{\sigma}} V\cong \big(T_{-\theta_{\sigma}^*}^{\lambda_{\sigma}^*} T_{\lambda_{\sigma}^*}^{-\theta_{\sigma}^*} V^*\big)^*$. As $T_{-\theta_{\sigma}^*}^{\lambda_{\sigma}^*} T_{\lambda_{\sigma}^*}^{-\theta_{\sigma}^*} V^*$ does not have non-zero $\text{U}(\ug_{\sigma})$-finite vectors (Lemma \ref{lmnoalg}),  $T_{-\theta_{\sigma}}^{\lambda_{\sigma}} T_{\lambda_{\sigma}}^{-\theta_{\sigma}} V$ does not have non-zero $\text{U}(\ug_{\sigma})$-finite quotient representations.
\end{proof}
\begin{proposition}\label{PgIfini}Let $V$ be a locally $\Q_p$-analytic representation of $\GL_2(K)$ on space of compact type on which $\cZ_K$ acts by $\chi_{\lambda}$, and let $I, J\subset \Sigma_K$, $I\cap J=\emptyset$. Then 
$(\nabla_I^-(\vartheta_J^- V^*))^* $ is the $\text{U}(\ug_I)$-finite subrepresentation of $(\vartheta_J^-V^*)^*$.
\end{proposition}
\begin{proof}
By Lemma \ref{lmbox0},	$\nabla_I^- \vartheta_J^- V^*$ is  generated by $\text{U}(\ug_I)$-finite vectors. By Lemma \ref{nabla-},  $(\nabla_I^-(\vartheta_J^- V^*))^*=\cap_{\sigma\in I}(\nabla_{\sigma}^-(\vartheta_J^- V^*))^*$. It suffices to show the statement for $I=\{\sigma\}$.  The natural  surjective map $T_{-\theta_{\sigma}}^{\lambda_{\sigma}}T_{\lambda_{\sigma}}^{-\theta_{\sigma}}(\vartheta_J^- V^*) \ra \vartheta_{\sigma}^- (\vartheta_J^- V^*)$ induces  $(\vartheta_{\sigma}^-(\vartheta_J^- V^*))^* \hookrightarrow T_{-\theta_{\sigma}}^{\lambda_{\sigma}}T_{\lambda_{\sigma}}^{-\theta_{\sigma}}\big((\vartheta_J^- V^*)^*\big)$. As the latter doe not have non-zero $\text{U}(\ug_I)$-finite subrepresentations by Lemma \ref{noalg2}, neither does $(\vartheta_{\sigma}^-(\vartheta_J^- V^*))^* $. Together with the tautological  exact sequence
\begin{equation*}
	0 \lra (\nabla_{\sigma}^-(\vartheta_J^- V^*))^*  \lra (\vartheta_J^- V^*)^* \lra (\vartheta_{\sigma}^-(\vartheta_J^- V^*))^*  \ra 0
\end{equation*}
we deduce any $\text{U}(\ug_{\sigma})$-finite vector of $(\vartheta_J^- V^*)^*$ is contained in $(\vartheta_{\sigma}^-(\vartheta_J^- V^*))^* $. The proposition follows.
\end{proof}
By  Lemma \ref{lmbox0} and Proposition \ref{PgIfini}, we have:
\begin{corollary} \label{Ugfini}Let $I, J\subset \Sigma_K$, $I\cap J=\emptyset$.

(1) $\nabla_I^+ \vartheta_J^+ \pi(D)^*$ and $\nabla_I^- \vartheta_J^- \pi(D)^*$ are $\text{U}(\ug_I)$-finite. In particular, $(\nabla_{\Sigma_K}^{\pm}  \pi(D)^*)^*$ are locally algebraic representations.

(2) $(\nabla_I^- \vartheta_J^- \pi(D)^*)^*$ is the $\text{U}(\ug_I)$-finite subrepresentation of $(\vartheta_J^- \pi(D)^*)^*$. In particular, we have $\nabla_{\Sigma_K}^- \pi(D)^* \cong \pi(D)^{\lalg, *}$. 
\end{corollary}
\begin{proposition}\label{pdR}
Let $\sigma\in \Sigma_K$, and suppose $\nabla_{\Sigma_K \setminus \{\sigma\}}^- \pi(D)^* \neq 0$. Then $D$ is $\Sigma_K \setminus \{\sigma\}$-de Rham, i.e. $\dim_E D_{\dR}(D)_{\tau}=2$ for all $\tau \neq \sigma$. 
\end{proposition}
\begin{proof}
Recall that $\pi(D)=\Pi_{\infty}^{R_{\infty}-\an}[\fm]$, where $\Pi_{\infty}$ is the patched Banach representation of \cite{CEGGPS1}, equipped with an action of the patched Galois deformation ring $R_{\infty}$, $(-)^{R_{\infty}-\an}$ means the subspace of locally $R_{\infty}$-analytic vectors  in the sense of \cite[\S~3.1]{BHS1}, and $\fm$ is a maximal ideal of $R_{\infty}[1/p]$ corresponding to $\rho$. We refer to the references for the precise definitions. Let $R_{\infty}^{\rig}$ be the global section of the generic fibre $X_{\infty}:=(\Spf R_{\infty})^{\rig}$ (that is a Fr\'echet-Stein algebra). By an easy variant of  \cite[Thm.~5.2.4]{Bel15}, the points on $X_{\infty}$ whose associated $\Gal_K$-representation is $\Sigma_K\setminus \{\sigma\}$-de Rham of weights $\textbf{h}_{\Sigma_K \setminus \{\sigma\}}$ form a Zariski-closed subspace $X_{\infty}(\textbf{h}_{\Sigma_K \setminus \{\sigma\}})$ of  $X_{\infty}$. Let $R_{\infty}^{\rig}(\textbf{h}_{\Sigma_K \setminus \{\sigma\}})$ be its global sections, being a quotient of $R_{\infty}^{\rig}$. 

Consider $\Pi_{\infty}^{R_{\infty}-\an}(\lambda^{\sigma}):=\big(\Pi_{\infty}^{R_{\infty}-\an} \otimes_E  L_{\Sigma_K \setminus \{\sigma\}}(\lambda_{\Sigma_K \setminus \{\sigma\}})^{\vee}\big)^{\sigma-\la}$. By the same arguments of \cite[Thm.~8.4 and \S~7]{DPS1}, the locally $\sigma$-algebraic vectors (for the $\GL_2(K)$-action) $\Pi_{\infty}^{R_{\infty}-\an}(\lambda^{\sigma})^{\lalg}$ are dense in  $\Pi_{\infty}^{R_{\infty}-\an}(\lambda^{\sigma})$. As the $R_{\infty}^{\rig}$-action on $\Pi_{\infty}^{R_{\infty}-\an}(\lambda^{\sigma})^{\lalg}$ factors through $R_{\infty}^{\rig}(\textbf{h}_{\Sigma_K \setminus \{\sigma\}})$, so does  its action on $\Pi_{\infty}^{R_{\infty}-\an}(\lambda^{\sigma})$. By assumption, $\nabla_{\Sigma_K \setminus \{\sigma\}}^- \pi(D)^*\neq 0$ on which $\cZ_K$ acts by $\chi_{\lambda^*}$, hence $\big(\pi(D) \otimes_E L_{\Sigma_K \setminus \{\sigma\}}(\lambda_{\Sigma_K \setminus \{\sigma\}})\big)^{\sigma-\la}\neq 0$. In particular, $\Pi_{\infty}^{R_{\infty}-\an}(\lambda^{\sigma})[\fm]\neq 0$. The proposition follows.
\end{proof}
\begin{remark}\label{RpdR}
When $\pi(D)$ comes from the completed cohomology of unitary Shimura modular curves, a similar statement is also obtained by \cite{QS} generalizing  Pan's geometric method.
\end{remark}
We discuss the Galois data in $\boxdot^+ (\pi(D)^*)$ and $\boxdot^- (\pi(D)^*)$. First, by the same argument as for Lemma \ref{pDE0}, we have 
\begin{lemma}
For $I\subset \Sigma_K$, there exists a unique $(\varphi, \Gamma)$-module $D_I$ of rank $2$ over $\cR_{K,E}$ such that $D_I[\frac{1}{t}]\cong  D[\frac{1}{t}]$ and the Sen $\sigma$-weights of $D_I$ are $\textbf{h}_{\sigma}$  if $\sigma \in I$, and $(0,0)$ if $\sigma \notin I$.
\end{lemma}
\begin{remark}
Suppose $D$ is de Rham. Passing from $D$ to $D_I$, we lose exactly the information of Hodge filtrations at $\sigma\in I$ of $D$.
%
\end{remark}
For $\sigma\in \Sigma_K$, let $t_{\sigma} \in \cR_{K,E}$ be the element defined in \cite[Notation~6.27]{KPX}. For $I\subset \Sigma_K$ and $\sigma\in \Sigma_K \setminus I$, we have two natural exact sequences 
\begin{equation*}
0 \lra t_{\sigma}^{-h_{\sigma, 2}}D_{I\cup \{\sigma\}} \lra D_I \lra \cR_{K,E}/t_{\sigma}^{h_{\sigma, 1}-h_{\sigma,2}} \lra 0,
\end{equation*}
\begin{equation*}
0 \lra t_{\sigma}^{h_{\sigma, 1}}   D_I \lra D_{I\cup \{\sigma\}} \lra t_{\sigma}^{h_{\sigma,2}}\cR_{K,E}/t_{\sigma}^{h_{\sigma,1}} \cR_{K,E}  \lra 0.
\end{equation*}
\begin{conjecture}[Hodge filtration hypercubes]\label{conjGL2}
(1)	For $I, J \subset \Sigma_K$, $I \cap J=\emptyset$, $\nabla_I^+ \vartheta_J^+ \pi(D)^*\cong \vartheta_J^+ \nabla_I^+ \pi(D)^*$ and $\nabla_I^- \vartheta_J^- \pi(D)^*\cong \vartheta_J^- \nabla_I^- \pi(D)^*$. Moreover let $\sigma \in \Sigma_K \setminus (I\cup J)$ (resp. $\sigma \in J$) the sequence (\ref{edge+}) (resp. (\ref{egdge-})) applied to $\pi(D)^*$ coincides with
\begin{equation*}
	0 \lra \nabla_I^+ \vartheta_J^+ \pi(D)^* \lra \vartheta_{\sigma}^+(\nabla_I^+ \vartheta_J^+ \pi(D)^*) \lra \nabla_{\sigma}^+(\nabla_I^+ \vartheta_J^+ \pi(D)^*) \lra 0
\end{equation*}
\begin{equation*}
	\big(\text{resp. } 0 \lra \vartheta_{\sigma}^-(\nabla_I^- \vartheta_{J\setminus \{\sigma\}}^- \pi(D)^*) \lra \nabla_I^- \vartheta_{J\setminus \{\sigma\}}^- \pi(D)^* \lra \nabla_{\sigma}^-(\nabla_I^- \vartheta_{J\setminus \{\sigma\}}^- \pi(D)^*) \lra 0 \big). 
\end{equation*}

(2) For $I, J \subset \Sigma_K$, $I \cap J=\emptyset$, the ($\text{U}(\ug_I)$-finite) representations	$\nabla_I^+\vartheta_J^+\pi(D)^*$ and $\nabla_I^- \vartheta_J^- \pi(D)^*$ only depend on $D_{\Sigma_K\setminus (I\cup J)}$, which we denote  respectively by $\pi^{\pm}(D_{\Sigma_K \setminus (I\cup J)}, \Sigma_K \setminus I, \lambda)^*$. Moreover, $\pi^{\pm}(D_{\Sigma_K \setminus (I\cup J)}, \Sigma_K \setminus I, \lambda)^*\neq 0$ if and only if $D$ is $I$-de Rham. 

(3) Let $\sigma \in \Sigma_K$, $I, J \subset \Sigma_K \setminus \{\sigma\}$ (resp. $I, J\subset \Sigma_K$, $\sigma\in J$), $I\cap J=\emptyset$. Suppose $D$ is $\sigma$-de Rham.  For subsets $I', J'\subset \Sigma_K$, $I'\cap J'=\emptyset$,  put $$e_{I',J'}:=\begin{cases}
	\#(\Sigma_K \setminus I')+1 & \text{$D$ is crystabelline, and $D_{\Sigma_K \setminus (I'\cup J')}$ is split}
	\\
	1 & \text{otherwise.}
\end{cases} $$
There is a natural isomorphism
\begin{multline*}
	\Hom_{\GL_2(K)}\Big(\pi^+(D_{\Sigma_K \setminus (I \cup J \cup \{\sigma\})},  \Sigma_K \setminus I, \lambda)^*, \pi^+(D_{\Sigma_K \setminus (I \cup J \cup \{\sigma\})},\Sigma_K \setminus (I\cup \{\sigma\}), \lambda)^*\Big) \\ \xlongrightarrow{\sim}\Hom_{(\varphi, \Gamma)}\big(D_{\Sigma_K \setminus (I\cup J \cup \{\sigma\})}, \cR_{K,E}/t_{\sigma}^{h_{\sigma,1}-h_{\sigma,2}} \big)^{\oplus e_{I\cup \{\sigma\}, J}}
\end{multline*}
\begin{multline*}
	\Big(\text{resp. }\Ext^1_{\GL_2(K)}\big(\pi^-(D_{\Sigma_K\setminus (I\cup J)}, \Sigma_K \setminus (I \cup \{\sigma\}), \lambda)^*, \pi^-(D_{\Sigma_K \setminus (I\cup J)}, \Sigma_K \setminus I, \lambda)^*\big)\\ \xlongrightarrow{\sim} \Ext^1_{(\varphi, \Gamma)}\big(t_{\sigma}^{h_{\sigma,2}}\cR_{K,E}/t_{\sigma}^{h_{\sigma,1}}, t_{\sigma}^{h_{\sigma, 1}}D_{\Sigma_K \setminus (I\cup J)}\big)^{\oplus e_{I \cup\{\sigma\}, J\setminus \{\sigma\}} }\Big)
\end{multline*}
such that the class $[\nabla_I^+\vartheta_J^+ \pi(D')^*]$ \big(resp. $[ \nabla_I^- \vartheta_{J\setminus \{\sigma\}}^- \pi(D')^*]$\big) is sent to $[t_{\sigma}^{-h_{\sigma,2}} D_{\Sigma_K \setminus (I\cup J)}']^{\oplus e_{I \cup \{\sigma\}, J}}$ \big(resp.  to $[D_{\Sigma_K\setminus (I\cup J \setminus \{\sigma\})}']^{\oplus e_{I \cup \{\sigma\},J\setminus \{\sigma\}} }$\big) for any $\sigma$-de Rham rank two $(\varphi, \Gamma)$-module $D'$ of weight $\textbf{h}$ with $D'_{\Sigma_K \setminus (I\cup J \cup\{\sigma\})}\cong D_{\Sigma_K \setminus (I\cup J \cup\{\sigma\})}$ (resp. with $D'_{\Sigma_K \setminus (I\cup J)}\cong D_{\Sigma_K \setminus (I\cup J)}$). 
\end{conjecture}
\begin{remark}\label{ReconjGL2}
(1) The label ``$\Sigma_K \setminus I$" in  $\pi^{\pm} (D_{\Sigma_K \setminus (I \cup J)}, \Sigma\setminus I, \lambda)^*$ signifies the embeddings $\sigma$, for which the  representation is ``genuinely analytic", i.e. not $\text{U}(\ug_{\sigma})$-finite. The number $e_{I, J}$ should be equal to the number of direct summands in $\pi^{\pm}(D_{\Sigma_K \setminus (I \cup J)}, \Sigma\setminus I, \lambda)^*$.  

(2) Let $\sigma\in \Sigma_K$ and assume $\textbf{h}_{\sigma}$ is strictly dominant. Conjecture \ref{conjGL2} (1)(2) implies the followings are equivalent (which refines Conjecture \ref{conj2} (3)): (i) $\dim D_{\dR}(D)_{\sigma}=1$, (ii) $\pi(D)^* \xrightarrow{\sim} \vartheta_{\sigma}^+ \pi(D)^*$, (iii) $ \vartheta_{\sigma}^- \pi(D)^*\xrightarrow{\sim}\pi(D)^*$.

(3) Suppose $D$ is the de Rham. For $\sigma\in \Sigma_K$, the conjectural exact sequence 
\begin{equation*}
	0 \lra \pi^-(\Delta, \sigma, \lambda)^* \lra \pi^-(D_{\sigma}, \sigma, \lambda)^* \lra \pi^-(\Delta, \emptyset, \lambda)^* \lra 0
\end{equation*}
should be (the dual of) the extension considered in \cite[Conj.~1.1]{Br16} for $\GL_2(K)$ (noting $\pi^-(\Delta, \emptyset, \lambda)\cong \pi_{\infty}(\Delta) \otimes_E L(\lambda)$).

(4)  Conjecture \ref{conjGL2} (3) is formulated using  $\Hom$ and $\Ext^1$ of $(\varphi, \Gamma)$-modules. But we also have an equivalent version in terms of Hodge filtrations. In fact, let $I$, $J$, and $\sigma$ be as in (3), by similar arguments as in \cite[Lem.~5.1.1, Prop.~5.1.2]{BD2},  there is a natural isomorphism of $E$-vector spaces\footnote{For the first isomorphism, we also use a natural isomorphism $D_{\dR,\sigma}(D_{\Sigma_K \setminus (I\cup J \cup \{\sigma\})})\cong D_{\dR,\sigma}(D_{\Sigma_K \setminus (I\cup J \cup \{\sigma\})})^{\vee}$, see \cite[Rem.~1.2]{Ding12}.}
\begin{equation*}
	\Hom_{(\varphi, \Gamma)}\big(D_{\Sigma_K \setminus (I\cup J \cup \{\sigma\})}, \cR_{K,E}/t_{\sigma}^{h_{\sigma,1}-h_{\sigma,2}} \big) \xlongrightarrow{\sim} D_{\dR, \sigma}\big(D_{\Sigma_K\setminus (I\cup J \cup \{\sigma\})}\big)
\end{equation*}
\begin{equation*}
	\Big(\text{resp. }  \Ext^1_{(\varphi, \Gamma)}\big(t_{\sigma}^{h_{\sigma,2}}\cR_{K,E}/t_{\sigma}^{h_{\sigma,1}}, t_{\sigma}^{h_{\sigma, 1}}D_{\Sigma_K \setminus (I\cup J)}\big) \xlongrightarrow{\sim} D_{\dR,\sigma}\big(D_{\Sigma_K \setminus (I\cup J)}\big)\Big)
\end{equation*}
which sends the $E$-line $[f]$ (resp. the $E$-line $[D']$) to $\Fil^{0} D_{\dR, \sigma}(\Fer f)\subset D_{\dR, \sigma}\big(D_{\Sigma_K\setminus (I\cup J \cup \{\sigma\})}\big)$ \big(resp. to $E$-line $\Fil^{-h_{\sigma_2}} D_{\dR, \sigma}(D') \subset D_{\dR,\sigma}(D')\cong D_{\dR,\sigma}\big(D_{\Sigma_K \setminus (I\cup J)}\big)$\big). Combining these isomorphisms with those in Conjecture \ref{conjGL2} (3), we then get a version on Hodge filtrations.


\end{remark}
\begin{theorem}
Conjecture \ref{conjGL2} holds for $\GL_2(\Q_p)$. 
\end{theorem}
\begin{proof}We let $\sigma$ be the unique embedding. 
As one of $I$, $J$ has to be empty, (1) is clear. For (2), the case where $I=\emptyset$ was already dealt with in Theorem \ref{TGL2}.  We know $D$ is de Rham if and only if $\pi(D)^{\lalg,*}=\nabla_{\sigma}^- \pi(D)^* \neq 0$. If $D$ is not de Rham, by Theorem \ref{TGL2}, $\vartheta_{\sigma}^{\pm} \pi(D)^*\cong \pi(D)^*$, hence $\nabla_{\sigma}^{\pm} \pi(D)^*=0$.  Suppose $D$ is de Rham, by \cite[Thm.~3.6, Rem.~3.7]{Ding14} and \cite[Thm.~0.6]{Colm18}, the exact sequences  $\boxdot^{\pm}(\pi(D)^*)$ are respectively given by 
\begin{equation}\label{EGL2qp1}
	0 \lra \pi(D)^* \lra \pi(\Delta, \lambda)^* \lra (\pi_{\infty}(\Delta)  \otimes_E L(\lambda))^* \lra 0,
\end{equation}
\begin{equation}\label{EGL2qp2}
	0 \lra \pi_0(\Delta, \lambda)^* \lra \pi(D)^* \lra (\pi_{\infty}(\Delta) \otimes_E L(\lambda))^* \lra 0.
\end{equation}
Let $F$ be the functor  defined in \cite{BD2} (see \cite[\S~2]{Ding12} for a quick summary in the case of $\GL_2(\Q_p)$). By \cite[Thm.~5.4.2]{BD2} \cite[Thm.~2.1, Cor.~2.4]{Ding12}, we see $F$ induces the  isomorphisms in (3). 
\end{proof}
\begin{remark} (1) By applying the functor $F$ of \cite{BD2} to (\ref{EGL2qp1}) and (\ref{EGL2qp2}), we can actually obtain two exact sequences of $(\varphi, \Gamma)$-modules:
\begin{equation*}
	0 \lra t^{-h_2} D \lra \Delta \lra \cR_E/t^{h_1-h_2} \lra 0,
\end{equation*}
\begin{equation*}
	0 \lra t^{h_1} \Delta \lra D \lra t^{h_2} \cR_E/t^{h_1} \cR_E \lra 0,
\end{equation*}
which are no other than $\boxdot^{\pm} D$ by \cite[Prop.~2.9]{Ding14}. It is then natural to  expect a multi-variable $(\varphi, \Gamma)$-module avatar of $\boxdot^{\pm} \pi(D)^*$. See \cite{BHHMS3} for multi-variable $(\varphi, \Gamma)$-modules in  the mod $p$ setting. 

(2) The sequences (\ref{EGL2qp1}) (\ref{EGL2qp2}) admit  geometric realizations, see \cite{DLB} (for the de Rham non-trianguline case) and \cite[\S~7.3]{Pan4} (for the general case). We expect the hypercubes $\boxdot^{\pm}(\pi(D)^*)$ also admits a geometric realization.   

\end{remark}
We study representations in $\boxdot^{\pm}(\pi(D)^*)$. For a $\cD(\GL_2(K),E)$-module $M$, that is coadmissible as $\cD(H,E)$-module for some compact open subgroup $H$ of $\GL_2(K)$, we put $\EE^i(M):=\Ext^i_{\cD(\GL_2(K),E)}(M, \cD_c(\GL_2(K),E))$ to be the $i$-th Schneider-Teitelbaum dual of $M$ (\cite{ST-dual}).  The following hypothesis will be crucially used (recalling $\delta_D$ is the central character of $\pi(D)$). By Corollary \ref{C-CMdual}, the hypothesis holds when $K$ is unramified over $\Q_p$ under  some mild hypothesis. 
\begin{hypothesis}\label{hypoCM}
Suppose $\pi(D)^*$ is Cohen-Macaulay of dimension $d_K$, and $\pi(D)^*$ is essentially self-dual, i.e. $E^{3d_K} \pi(D)^* \cong \pi(D)^* \otimes_E \delta_D\circ \dett$.
\end{hypothesis}
The following conjecture generalizes Hypothesis \ref{hypoCM}.
\begin{conjecture}\label{conjdim}
Let $I, J \subset \Sigma_K$, $I\cap J=\emptyset$, then $\nabla_I^{\pm} \vartheta_J^{\pm}  \pi(D)^*$ is zero or Cohen-Macaulay of dimension  $(d_K-\#I)$. Moreover, $E^{3d_K+\#I}(\nabla_I^{\pm} \vartheta_J^{\pm}  \pi(D)^*)\cong \nabla_I^{\mp} \vartheta_J^{\mp}  \pi(D)^* \otimes_E \delta_D$. 
\end{conjecture}
\begin{lemma}\label{iotakappadual}
Let $M$ be a $\cD(\GL_2(K),E)$-module on which $\cZ_K$ acts via $\chi_{\lambda^*}$, then the map $\EE^i(\Theta_I M) \ra \EE^i(M)$ induced by $\iota: M \ra \Theta_I M$ coincides with $\EE^i(\Theta_I M)\cong \Theta_I \EE^i(M) \xrightarrow{\kappa} \EE^i M$.  
\end{lemma}
\begin{proof}
As $\Theta_I$ is exact and preserves projective objects, it suffices to show the statement for $i=0$. By (the proof of) Proposition \ref{STdualtran}, it suffices to show that for an algebraic representation $V$ of $\GL_2(K)$, and  $f\in \Hom_{\cD(\GL_2(K),E)}(M, \cD_c(\GL_2(K),E) \otimes_E V \otimes V^{\vee})$, the induced map (the last map induced by $W \otimes_E W^{\vee} \ra E$ for $W=V \otimes_E V^{\vee}$.
\begin{equation*}
	M \ra M\otimes_E V \otimes_E V^{\vee} \xrightarrow{f \otimes \id} \big(\cD_c(\GL_2(K),E) \otimes_E V \otimes _E V^{\vee}\big) \otimes_E (V \otimes_E V^{\vee}) \ra \cD_c(\GL_2(K),E)
\end{equation*}   
coincides with $M \xrightarrow{f} \cD_c(\GL_2(K),E) \otimes_E V \otimes V^{\vee} \ra  \cD_c(\GL_2(K),E)$. It is straightforward to check it (using $W \otimes_E W^{\vee} \ra E$ coincides with $(V\otimes_EV^{\vee})\otimes_E (V\otimes_E V^{\vee}) \ra E \otimes_E E \cong E$).  
\end{proof}
\begin{lemma}\label{Linj1}
Let $M$ be a $\cD(\GL_2(K),E)$-module on which $\cZ_K$ acts by $\chi_{\lambda^*}$. Suppose $M$ is coadmissible over $\cD(H,E)$ and Cohen-Macaulay of  dimension $d$. Then the followings are equivalent:

(1) $\dim \nabla_{\sigma}^- \EE^d(M) \leq d-1$, 

(2) $\iota_{\sigma}: M \ra \vartheta_{\sigma}^+ M$ is injective. 
\end{lemma}\begin{proof}
By taking duals and using $\Theta_{\sigma}(M)$ is Cohen-Macaulay (Corollary \ref{transpure}), the exact sequence $0 \ra \vartheta_{\sigma}^+ \EE^d(M) \ra \Theta_{\sigma} \EE^d(M) \ra \vartheta_{\sigma}^- \EE^d(M) \ra 0$ induces
\begin{equation*}
	0 \ra \EE^d(\vartheta_{\sigma}^- \EE^d(M)) \ra \EE^d(\Theta_{\sigma} \EE^d(M)) \ra \EE^d(\vartheta_{\sigma}^+ \EE^d(M)) \ra E^{d+1}(\vartheta_{\sigma}^- \EE^d(M)) \ra 0.
\end{equation*}Consider $0 \ra \vartheta_{\sigma}^- \EE^d (M) \ra \EE^d(M) \ra \nabla_{\sigma}^- \EE^d (M)\ra 0$. We see (1) is equivalent to that the induced map $\EE^d(\EE^d(M)) \ra \EE^d(\vartheta_{\sigma}^- \EE^d(M))$ is injective. However, by Lemma \ref{iotakappadual}, the composition $\EE^d(\EE^d(M)) \ra \EE^d(\vartheta_{\sigma}^- \EE^d(M))\hookrightarrow \EE^d(\Theta_{\sigma} \EE^d(M))$ coincides with $M \ra \vartheta_{\sigma}^+$. The lemma follows.
\end{proof}
Applying Lemma \ref{Linj1} to $M=\pi(D)^*$, we get:
\begin{proposition}\label{Pinj1} 
Assume Hypothesis \ref{hypoCM}. For $\sigma\in \Sigma_K$, the followings are equivalent:

(1) $\dim \nabla_{\sigma}^- \pi(D)^* \leq d_K-1$;

(2) $\iota_{\sigma}: \pi(D)^* \ra \vartheta_{\sigma}^+ \pi(D)^*$ is injective.
\end{proposition}
\begin{remark}
When $K=\Q_p$, $\nabla^- \pi(D)^*$ is locally algebraic hence $\dim \nabla^- \pi(D)^*=0$ (cf. \cite{ST03}). The proposition then gives an alternative proof of the injectivity of $\iota$.
\end{remark}
Let $\sigma \in \Sigma_K$. The following proposition is due to Dospinescu-Schraen-Pa{\v{s}}k{\=u}nas (\cite{DPS}).
\begin{proposition}\label{Psigmaan}Let $H$ be a uniform open subgroup of $\GL_2(K)$,  $\bM$ be a finitely generated $E[[H]]:=\co_E[[H]] \otimes_{\co_E} E$-module on which $Z_H:=Z(K)\cap H$ acts by a certain character. Let $M:=\bM \otimes_{E[[H]]} \cD(H,E)$, and $N$ be a subquotient of $M$. Suppose $\dim \bM < 3d_K$, $\cZ_K$ acts on $N$ by a  certain character $\chi$ and $N$ is $\text{U}(\ug_{\Sigma_K \setminus \{\sigma\}})$-finite. Then $\dim N\leq 1$. 
\end{proposition}
\begin{proof}
As $N$ is $\text{U}(\ug_{\Sigma_K \setminus \{\sigma\}})$-finite on which $\cZ_K$ acts by $\chi_{\mu}$, there exists an irreducible algebraic representation $V$ of $\GL_2(K)$ such that $N\cong (N \otimes_E V^{\vee})^{\ug_{\Sigma_K \setminus \{\sigma\}}} \otimes_E V$. Replacing $\bM$ by $\bM \otimes_E V^{\vee}$ and   $N$ by $(N \otimes_E V^{\vee})^{\ug_{\Sigma_K \setminus \{\sigma\}}}$ and using Proposition \ref{STdualtran} (and a similar version for $E[[H]]$-modules), we reduce to the case where $N$ is locally $\sigma$-analytic. Let $H_1:=H\cap \SL_2(K)$, by Lemma \ref{Lextravaria}, $\dim_{E[[H_1]]} \bM<3d_K$ hence $\bM$ is a torsion module over $E[[H_1]]$ (\cite[Lem.~3.9]{DPS}). By \cite[Prop.~6.14]{DPS} (and the proof), it suffices to show $\dim_{\cD(H_1,E)} N\leq 1$. Suppose $\dim_{\cD(H_1,E)} N\geq 2$. As $N$ is locally $\sigma$-analytic, by Corollary \ref{CJandual}, $\dim_{\cD_{\sigma}(H_1,E)} N \geq 2$. Let $\chi_{\sigma}:=\chi|_{\cZ_{\sigma}}$. By \cite[Cor.~6.11]{DPS} (and shrinking $H$ if needed), we deduce 
\begin{equation}
	\label{Enonzero}\Hom_{\cD(H_1,E)}(N, \cD_{\sigma}(H_1,E)_{\chi_{\sigma}})\neq 0,
\end{equation} where $\cD_{\sigma}(H_1, E)_{\chi_{\sigma}}=\cD_{\sigma}(H_1,E) \otimes_{\cZ_{\sigma}, \chi_{\sigma}} E$. However, $\bM$ is torsion over $E[[H_1]]$, hence $N$ is annihilated by a certain \textit{non-zero} element in $E[[H_1]]$, contradicting (\ref{Enonzero}) by the proof of \cite[Prop.~6.15]{DPS} (using \cite[Thm.~4.1, 5.1]{DPS},  see in particular  the proof of \cite[Lem.~6.16]{DPS}). The proposition follows.
\end{proof}
\begin{corollary}\label{Csigmaan}
Assume Hypothesis \ref{hypoCM}. Let $\sigma\in \Sigma_K$, then $\dim \nabla^{\pm}_{\Sigma_K \setminus \{\sigma\}} \pi(D)^* \leq 1$.
\end{corollary}
\begin{proof}
Apply Proposition \ref{Psigmaan} to $\bM=\big(\widehat{\pi}(D) \otimes_E L(\textbf{h})^{\vee} \otimes_E L(\textbf{h})\big)^*$ (which is also Cohen-Macaulay of dimension $d_K$ by similar arguments as in Proposition \ref{STdualtran}).
\end{proof}

\subsubsection{$\boxdot^{\pm}(\pi(D)^*)$ for $\GL_2(K)$ with $[K:\Q_p]=2$}\label{S423}
Throughout the section, we assume $d_K=2$ and Hypothesis \ref{hypoCM}. We study the dimensions and dualities of the representations in $\boxdot^{\pm}\pi(\rho)^*$. Note that  when $K=\Q_{p^2}$, it is proved in  Appendix \ref{AppCM} that Hypothesis \ref{hypoCM} holds under mild assumptions.  

\begin{theorem}\label{Tinj}
(1) For $\sigma \in \Sigma_K$, the map $\iota_{\sigma}: \pi(D)^* \ra \vartheta_{\sigma}^+ \pi(D)^*$ is injective.

(2) The map $\iota: \pi(D)^* \ra \vartheta_{\Sigma_K}^+ \pi(D)^*$ is injective. 
\end{theorem}
\begin{proof}
(1) follows from Corollary \ref{Csigmaan} and Proposition \ref{Pinj1} (or a similar argument as below). Let $\tau\in \Sigma_K$, $\tau\neq \sigma$. The map $\iota$ factors as $\pi(D)^* \xrightarrow{\iota_{\sigma}} \vartheta_{\sigma}^+ \pi(D)^* \xrightarrow{\iota_{\tau}} \vartheta_{\tau}^+ \vartheta_{\sigma}^+ \pi(D)^*$. By Proposition \ref{Psigmaan}, any $\text{U}(\ug_{\tau})$-finite subquotient of $\pi(D)^* \otimes_E L(\textbf{h}_{\sigma})^{\vee} \otimes_E L(\textbf{h}_{\sigma})$ (hence of $\Theta_{\sigma} \pi(D)^*$ and of $\vartheta_{\sigma}^+ \pi(D)^*$) has dimension no bigger than $1$. However, as $\vartheta_{\sigma}^+ \pi(D)^* \hookrightarrow \Theta_{\sigma}^+ \pi(D)^*$ the latter being Cohen-Macaulay of dimension $d_K$, $\vartheta_{\sigma}^+ \pi(D)^*$ is pure of dimension $d_K=2$.  We deduce  $\vartheta_{\sigma}^+ \pi(D)^*$ can not have non-zero $\text{U}(\ug_{\tau})$-finite subrepresentation. By Lemma \ref{lmbox0}  (1), $ \vartheta_{\sigma}^+ \pi(D)^* \xrightarrow{\iota_{\tau}} \vartheta_{\tau}^+ \vartheta_{\sigma}^+ \pi(D)^*$ is injective. (2) follows. 
\end{proof}
We will frequently use the following lemma.
\begin{lemma}\label{Ldim2}
Let $M$, $N$ be two $\cD(\GL_2(K),E)$-module on which $\cZ_{K}$ acts by $\chi_{\lambda^*}$.  Suppose $M$ and $N$ are admissible over $\cD(H,E)$ for a compact open subgroup $H$ of $\GL_2(K)$, and are Cohen-Macaulay of grade $j\geq 3d_K=6$ (or equivalently of dimension $\leq 2$) such that $\EE^j(M)\cong N$ and $\EE^j(N)\cong M$. Let $\sigma\in \Sigma_K$, suppose  $\dim(\nabla_{\sigma}^{\pm} M)\leq 4d_K-j-1$ and $\dim (\nabla_{\sigma}^{\pm} N)\leq 4d_K-j-1$, then we have (the same holds with $M$ and $N$ exchanged):

(1) $\dim \vartheta_{\sigma}^{\pm} M=\dim \vartheta_{\sigma}^{\pm} N=4d_K-j$.

(2)	$\EE^j (\vartheta_{\sigma}^- M) \cong \vartheta_{\sigma}^+ N$ and $E^{j+1} (\nabla_{\sigma}^- M) \cong \nabla_{\sigma}^+ N$.

(3) $\vartheta_{\sigma}^+ M$ is Cohen-Macaulay, and $\nabla_{\sigma}^+ M$ is  zero or Cohen-Macaulay of dimension $4d_K-j-1$. Moreover $E^{j+1}(\vartheta^-_{\sigma} M)\cong E^{j+2}(\nabla^-_{\sigma}M)$, and there are  exact sequences
\begin{equation*}
	0 \lra \vartheta_{\sigma}^- N \lra E^{j}(\vartheta_{\sigma}^+ M) \lra E^{j+2}(\nabla^-_{\sigma} M) \lra 0,
\end{equation*}
\begin{equation*}
	0 \lra E^{j+2}(\nabla^-_{\sigma} M) \lra \nabla^-_{\sigma} N \lra E^{j+1}(\nabla_{\sigma}^+ M) \lra 0.
\end{equation*}
In particular, if $\nabla_{\sigma}^- N\neq 0$, then $E^{j+2}(\nabla^-_{\sigma} M)=0$ is equivalent to $\nabla_{\sigma}^- N$ is pure of dimension $4d_K-j-1$. If the latter holds, then $\vartheta_{\sigma}^- M$ is Cohen-Macaulay with $\EE^j(\vartheta_{\sigma}^- M)\cong \vartheta_{\sigma}^+ N$, and $\nabla_{\sigma}^- M$ is zero or Cohen-Macaulay of dimension $4d_K-j-1$ with $E^{j+1}(\nabla_{\sigma}^- M)\cong \nabla_{\sigma}^+ N$.
\end{lemma}
\begin{proof} Note first $E^{j+3}(M')=0$ for any admissible $\cD(H,E)$-module $M'$ (as $j+3\geq 4d_K+1$).  As $\dim(\nabla_{\sigma}^{\pm} M)\leq 4d_K-j-1$, the exact sequence (by Lemma \ref{Linj1}) $0 \ra M \ra \vartheta_{\sigma}^+ M \ra \nabla_{\sigma}^+ M \ra 0$ (resp. $0 \ra \vartheta_{\sigma}^- M \ra M \ra \nabla_{\sigma}^- M \ra 0$) induces by taking duals:
\begin{equation}\label{EdualM1}
	0 \ra \EE^j(\vartheta^+ M) \ra \EE^j(M) \ra E^{j+1}(\nabla_{\sigma}^+ M) \ra E^{j+1} (\vartheta_{\sigma}^+M) \ra 0, 
\end{equation}
\begin{equation}
	\label{EdualM2}
	\text{\big(resp. } 0 \ra \EE^j(M) \ra \EE^j(\vartheta_{\sigma}^- M) \ra E^{j+1}(\nabla_{\sigma}^- M)\ra 0 \text{\big), }
\end{equation}
and $E^{j+2}(\nabla_{\sigma}^+ M)\xrightarrow{\sim} E^{j+2}(\vartheta_{\sigma}^+ M)$ (resp. $E^{j+1}(\vartheta_{\sigma}^- M) \xrightarrow{\sim} E^{j+2}(\nabla_{\sigma}^- M)$, $E^{j+2}(\vartheta_{\sigma}^- M)=0$). Taking dual to the exact sequence $0 \ra  \vartheta_{\sigma}^+ M \ra \Theta_{\sigma} M \ra \vartheta_{\sigma}^- M \ra 0$ and using Corollary \ref{transpure}, we get 
\begin{equation}\label{EdualM3}
	0 \ra \EE^j(\vartheta_{\sigma}^-M) \ra \EE^j(\Theta_{\sigma} M) \ra E^{j}(\vartheta_{\sigma}^+ M) \ra E^{j+1}(\vartheta_{\sigma}^- M) \ra 0,
\end{equation}
and $E^{j+1}(\vartheta_{\sigma}^+ M)\xrightarrow{\sim} E^{j+2}(\vartheta_{\sigma}^- M)$, $E^{j+2}(\vartheta_{\sigma}^+ M)=0$. Putting these together we see $E^{j+1}(\vartheta_{\sigma}^+ M)=E^{j+2}(\vartheta_{\sigma}^+M)=E^{j+2}(\nabla_{\sigma}^+ M)=0$, hence $\vartheta_{\sigma}^+ M$ and $\nabla^+_{\sigma} M$ are Cohen-Macaulay. The same discussion holds with $M$ replaced by $N$. We obtain (1) and the first part of (3).

Now we use the duality between $M$ and $N$. Compare (\ref{EdualM1}) with $0 \ra \nabla_{\sigma}^- N \ra N \ra \nabla_{\sigma}^- N \ra 0$. As $E^{j+1}(\nabla_{\sigma}^+ M)$ is $\text{U}(\ug_{\sigma})$-finite, there is an injection $\nabla_{\sigma}^- N \hookrightarrow \EE^j(\vartheta_{\sigma}^+ M)$. By Lemma \ref{iotakappadual}, the composition $\EE^j(\Theta_{\sigma} M) \ra \EE^j(\vartheta_{\sigma}^+ M) \hookrightarrow \EE^j(M)$ coincides with $\Theta_{\sigma} N \ra N$. Identifying their kernel and image, and using (\ref{EdualM1}) (\ref{EdualM3}), we get $\EE^j(\vartheta_{\sigma}^- M) \xrightarrow{\sim} \vartheta_{\sigma}^+ N$ and  the exact sequences in (3). Finally comparing (\ref{EdualM2}) with $0 \ra N \ra \vartheta_{\sigma}^+ N \ra \nabla_{\sigma}^+ N \ra 0$ and using $\EE^j(\vartheta_{\sigma}^- M) \xrightarrow{\sim} \vartheta_{\sigma}^+ N$, we see $E^{j+1}(\nabla_{\sigma}^- M)\cong \nabla_{\sigma}^+ N$.
\end{proof}

We deduce the following theorem towards Conjecture \ref{conj2} (3):
\begin{theorem}\label{TndR}
(1) Let $\sigma\in \Sigma_K$ and suppose $\dim_E D_{\dR}(D)_{\sigma}=1$. Then $\vartheta_{\sigma}^- \pi(D)^* \xrightarrow{\sim} \pi(D)^* \xrightarrow{\sim} \vartheta_{\sigma}^+ \pi(D)^*$. Consequently if $\dim_E D_{\dR}(D)_{\sigma}=\dim_E D_{\dR}(D)_{\tau}=1$, then $\vartheta_{\Sigma_K}^- \pi(D)^* \xrightarrow{\sim} \pi(D)^* \xrightarrow{\sim} \vartheta_{\Sigma_K}^+ \pi(D)^*$.

(2) For $\sigma\in \Sigma_K$, if $\vartheta_{\sigma}^- \pi(D)^* \xrightarrow{\sim} \pi(D)^*$, then $\pi(D)^* \xrightarrow{\sim} \vartheta^+_{\sigma} \pi(D)^*$. Consequently if $\vartheta_{\Sigma_K}^- \pi(D)^* \xrightarrow{\sim} \pi(D)^*$ then $\pi(D)^* \xrightarrow{\sim} \vartheta^+_{\Sigma_K} \pi(D)^*$.
\end{theorem}
\begin{proof}
(2) follows from Lemma \ref{Ldim2} (2) (applied to $M=\pi(D)^*$ and $N=\pi(D)^* \otimes_E \delta_D$). Together with Proposition \ref{pdR}, (1) also follows.
\end{proof}
\begin{remark}\label{Rgl2Qp2}
The same argument gives an alternative proof of (i) $\Rightarrow$ (ii) $\Leftrightarrow$ (iii) in Conjecture \ref{conj2} (3) for $\GL_2(\Q_p)$ (without using $(\varphi, \Gamma)$-modules).
\end{remark}
In the rest of the section, we furthermore assume the following hypothesis.
\begin{hypothesis}\label{newhypo}
Assume $\nabla_{\sigma}^- \pi(D)^*$ is pure for $\sigma\in \Sigma_K$.
\end{hypothesis}
As $\dim \nabla_{\sigma}^- \pi(D)^* \leq 1$, by \cite[Cor.~9.1]{AW}, the hypothesis is equivalent to that $\nabla_{\sigma}^- \pi(D)^*$ does not have non-zero locally algebraic sub. When $\pi(D)$ is cut off from the completed cohomology of unitary Shimura curves and $\pi(D)^{\lalg}\neq 0$, it is showed in \cite{QS} \cite{Su1} that Hypothesis \ref{newhypo} holds when $D$ is de Rham (see also Hypothesis \ref{sigma} and Remark \ref{RemQS}). Under Hypothesis \ref{newhypo}, we have by Lemma \ref{Ldim2} (applied to $M=\pi(D)^*$):
\begin{proposition}\label{Pdual2}For $\sigma\in \Sigma_K$. 

(1) $\vartheta_{\sigma}^{\pm} \pi(D)^*$ are Cohen-Macaulay of dimension $d_K$, and $E^{3d_K}(\vartheta_{\sigma}^{\pm} \pi(D)^*) \cong \vartheta_{\sigma}^{\mp} \pi(D)^* \otimes_E \delta_D$. 

(2) $\nabla_{\sigma}^{\pm}  \pi(D)^*$ are zero or Cohen-Macaulay of dimension $d_K-1$, and $E^{3d_K+1}\nabla_{\sigma}^{\pm} \pi(D)^*)\cong E^{3d_K+1} \nabla_{\sigma}^{\mp} \pi(D)^*$.
\end{proposition}
Write $\Sigma_K=\{\sigma, \tau\}$.   By  Proposition \ref{Psigmaan}, any $\text{U}(\ug_{\tau})$-finite subquotient of $\Theta_{\Sigma_K} \pi(D)^*$ has dimension no bigger than 1. This, together with Proposition \ref{Pdual2}, allow to apply Lemma \ref{Ldim2} (for $\tau$) to $\vartheta_{\sigma}^{\pm} \pi(D)^*$. We obtain:
\begin{proposition}\label{Pthetasigmatau}
(1) $\dim \vartheta_{\tau}^{?} \vartheta_{\sigma}^{??} \pi(D)^*=d_K$ for $?, ??\in \{+,-\}$.

(2) $E^{3d_K} (\vartheta_{\tau}^- \vartheta_{\sigma}^{\pm} \pi(D)^*)\cong \vartheta_{\tau}^+ \vartheta_{\sigma}^{\mp} \pi(D)^* \otimes_E \delta_D$ and $E^{3d_K+1}(\nabla_{\tau}^- \vartheta_{\sigma}^{\pm} \pi(D)^*) \cong \nabla_{\tau}^+\vartheta_{\sigma}^{\mp} \pi(D)^* \otimes_E \delta_D$.

(3) $\vartheta_{\tau}^+ \vartheta_{\sigma}^{\pm} \pi(D)^*$ is Cohen-Macaulay, and $\nabla_{\tau}^+  \vartheta_{\sigma}^{\pm} \pi(D)^*$ is zero or Cohen-Macaulay of dimension $d_K-1=1$. Moreover $E^{3d_K+1}(\vartheta_{\tau}^- \vartheta_{\sigma}^{\pm} \pi(D)^*)\cong E^{3d_K+2}(\nabla_{\tau}^- \vartheta_{\sigma}^{\mp} \pi(D)^*)$ and we have exact sequences
\begin{equation*}
	0  \lra \vartheta_{\tau}^- \vartheta_{\sigma}^{\pm} \pi(D)^* \otimes_E \delta_D \lra E^{3d_K}(\vartheta_{\tau}^+ \vartheta_{\sigma}^{\mp} \pi(D)^*) \lra E^{3d_K+2}(\nabla_{\tau}^- \vartheta_{\sigma}^{\mp} \pi(D)^*) \lra 0,
\end{equation*}
\begin{equation*}
	0 \lra E^{3d_K+2}(\nabla_{\tau}^- \vartheta_{\sigma}^{\mp} \pi(D)^*) \lra \nabla_{\tau}^- \vartheta_{\sigma}^{\pm} \pi(D)^* \otimes_E \delta_D \lra E^{3d_K+1}(\nabla_{\tau}^+ \vartheta_{\sigma}^{\mp} \pi(D)^*) \lra 0.
\end{equation*}
\end{proposition}
Similarly, we apply Lemma \ref{Ldim2} to $\nabla_{\sigma}^{\pm}(\pi(D)^*)$. One difference from the precedent proposition  is that  the last part in Lemma \ref{Ldim2} (3) now automatically holds  (as $j>3d_K$). 
\begin{proposition}\label{Pnablasigmatau+} 

(1) $ \vartheta_{\tau}^{?} \nabla_{\sigma}^{??} \pi(D)^*$ is Cohen-Macaulay of dimension $1$ for $?, ??\in \{+,-\}$. Moreover,  $E^{3d_K+1}(\vartheta_{\tau}^{+} \nabla_{\sigma}^{\pm} \pi(D)^*)\cong \vartheta_{\tau}^- \nabla_{\sigma}^{\mp} \pi(D)^* \otimes_E \delta_D$ and $E^{3d_K+1}(\vartheta_{\tau}^- \nabla_{\sigma}^{\pm} \pi(D)^*) \cong \vartheta_{\tau}^+ \nabla_{\sigma}^{\mp} \pi(D)^*$.

(2) $ \nabla_{\tau}^{?} \nabla_{\sigma}^{??} \pi(D)^*$ is zero or Cohen-Macaulay of dimension $0$ for $?, ??\in \{+,-\}$.  Moreover, $E^{3d_K+1}(\vartheta_{\tau}^{+} \nabla_{\sigma}^{\pm} \pi(D)^*)\cong \vartheta_{\tau}^- \nabla_{\sigma}^{\mp} \pi(D)^* \otimes_E \delta_D$ and $E^{3d_K+1}(\vartheta_{\tau}^- \nabla_{\sigma}^{\pm} \pi(D)^*) \cong \vartheta_{\tau}^+ \nabla_{\sigma}^{\mp} \pi(D)^*$.
\end{proposition}
Now we discuss the commutativity between ``$\sigma$-operators" and ``$\tau$-operators".
\begin{proposition}\label{sigmatau+}
(1) We have $\nabla_{\sigma}^+ \vartheta_{\tau}^+ \pi(D)^* \xrightarrow{\sim} \vartheta_{\tau}^+(\nabla_{\sigma}^+ \pi(D)^*)$ and    $\nabla_{\tau}^+(\nabla_{\sigma}^+ \pi(D)^*) \xrightarrow{\sim} \nabla_{\Sigma_K}^+ \pi(D)^*$. 

(2) We have $\nabla_{\sigma}^+ \pi(D)^* \hookrightarrow \vartheta_{\tau}^+ \nabla_{\sigma}^+ \pi(D)^*$ and the exact sequence \big(obtained by applying $\nabla_{\sigma}^+$ to $0 \ra \pi(D)^* \ra \vartheta_{\tau}^+ \pi(D)^* \ra \nabla_{\tau}^+ \pi(D)^* \ra 0$\big)
\begin{equation*}
	\nabla_{\sigma}^+ \pi(D)^* \lra \nabla_{\sigma}^+ \vartheta_{\tau}^+\pi(D)^* \lra \nabla_{\Sigma_K}^+ \pi(D)^* \lra 0
\end{equation*} coincides with (the $0 \ra \id \ra \vartheta_{\tau}^+ \ra \nabla_{\tau}^+ \ra 0$ sequence applied to $\nabla_{\sigma}^+ \pi(D)^*$)
\begin{equation*}
	0 \lra \nabla_{\sigma}^+ \pi(D)^* \lra \vartheta_{\tau}^+(\nabla_{\sigma}^+ \pi(D)^*) \lra \nabla_{\tau}^+(\nabla_{\sigma}^+ \pi(D)^*) \lra 0.
\end{equation*}
\end{proposition}
\begin{proof}
We have a commutative diagram
\begin{equation}\label{Esigmatau1}
	\begin{CD}
		0 @>>> \Theta_{\tau}\pi(D)^* @>>> \Theta_{\tau}(\vartheta_{\sigma}^+ \pi(D)^*) @>>> \Theta_{\tau} (\nabla_{\sigma}^+ \pi(D)^*) @>>> 0 \\
		@. @VVV @VVV @VVV @. \\
		0 @>>> \pi(D)^* @>>> \vartheta_{\sigma}^+ \pi(D)^* @>>> \nabla_{\sigma}^+ \pi(D)^* @>>> 0.
	\end{CD}
\end{equation}
From which, we deduce an exact sequence
\begin{equation}\label{Et+s+}
	0 \ra \vartheta_{\tau}^+ \pi(D)^* \ra \vartheta_{\tau}^+ \vartheta_{\sigma}^+ \pi(D)^* \ra \vartheta_{\tau}^+ \nabla_{\sigma}^+ \pi(D)^* \xrightarrow{\delta} \nabla_{\tau}^- \pi(D)^* \ra \nabla_{\tau}^- \vartheta_{\sigma}^+ \pi(D)^* \ra \nabla_{\tau}^- \nabla_{\sigma}^+ \pi(D)^* \ra 0.
\end{equation}
As $\nabla_{\sigma}^+ \pi(D)^*$ is $\text{U}(\ug_{\sigma})$-finite, so is $\vartheta_{\tau}^+ \nabla_{\sigma}^+ \pi(D)^*$. While, $\nabla_{\tau}^- \pi(D)^*$ is $\text{U}(\ug_{\tau})$-finite. We see $\Ima(\delta)$ is locally algebraic hence zero dimensional. However $\nabla_{\tau}^- \pi(D)^*$ is zero or pure of dimension $1$. Hence $\delta=0$.  The first isomorphism in (1) follows. Using $\delta=0$, we have furthermore  a commutative diagram of exact sequences (which is just $\boxdot^+ \pi(D)^*$, using Proposition \ref{Pdual2} (2) for the right vertical sequence)
\begin{equation}\label{Ctausigma1}
	\begin{CD}@. 0 @. 0 @. 0 @. \\
		@. @VVV @VVV @VVV @. \\		
		0 @>>> \pi(D)^* @>>> \vartheta_{\sigma}^+ \pi(D)^* @>>> \nabla_{\sigma}^+ \pi(D)^* @>>> 0 \\ 
		@. @VVV @VVV @VVV @. \\
		0 @>>>  \vartheta_{\tau}^+ \pi(D)^* @>>> \vartheta_{\tau}^+ \vartheta_{\sigma}^+ \pi(D)^* @>>> \vartheta_{\tau}^+ \nabla_{\sigma}^+ \pi(D)^* @>>> 0 \\
		@. @VVV @VVV @VVV @. \\
		0 @>>> \nabla_{\tau}^+ \pi(D)^* @>>>  \nabla_{\tau}^+ \vartheta_{\sigma}^+ \pi(D)^* @>>> \nabla_{\tau}^+ \nabla_{\sigma}^+ \pi(D)^* @>>> 0
		\\ @. @VVV @VVV @VVV @. \\
		@. 0 @. 0 @. 0. @.
	\end{CD}
\end{equation}
From the diagram (and exchanging $\sigma$ and $\tau$ if needed),  the second part of (1) and (2) follow.
\end{proof}
Using $\delta=0$ in the proof, we also obtain a commutative diagram of exact sequences
\begin{equation}\label{D-+}
\begin{CD}@. 0 @. 0 @. 0 @. \\
	@. @VVV @VVV @VVV @. \\		
	0 @>>>  \vartheta_{\tau}^- \pi(D)^* @>>> \vartheta_{\tau}^- \vartheta_{\sigma}^+ \pi(D)^* @>>> \vartheta_{\tau}^- \nabla_{\sigma}^+ \pi(D)^* @>>> 0		 \\ 
	@. @VVV @VVV @VVV @. \\
	0 @>>> \pi(D)^* @>>> \vartheta_{\sigma}^+ \pi(D)^* @>>> \nabla_{\sigma}^+ \pi(D)^* @>>> 0	 \\
	@. @VVV @VVV @VVV @. \\
	0 @>>> \nabla_{\tau}^- \pi(D)^* @>>>  \nabla_{\tau}^- \vartheta_{\sigma}^+ \pi(D)^* @>>> \nabla_{\tau}^- \nabla_{\sigma}^+ \pi(D)^* @>>> 0
	\\ @. @VVV @VVV @VVV @. \\
	@. 0 @. 0 @. 0. @.
\end{CD}
\end{equation}
Taking dual to bottom horizontal exact sequences and using Proposition \ref{Pdual2} (2),  \ref{Pthetasigmatau} (1) and  \ref{Pnablasigmatau+} (2)  we get
\begin{equation}\label{Ecokalg}
0 \ra \nabla_{\tau}^+ \vartheta_{\sigma}^- \pi(D)^* \ra \nabla_{\tau}^+ \pi(D)^* \ra \nabla_{\tau}^+ \nabla_{\sigma}^- \pi(D)^* \ra E^{3d_K+2}(\nabla_{\tau}^- \vartheta_{\sigma}^+ \pi(D)^*) \otimes_E \delta_D^{-1} \ra 0.
\end{equation}
\begin{hypothesis}\label{Hnew2}
Assume $\nabla_{\sigma}^- \vartheta_{\tau}^- \pi(D)^*$ and $\nabla_{\tau}^- \vartheta_{\sigma}^- \pi(D)^*$ are zero or pure of dimension $1$.
\end{hypothesis}
The hypothesis is equivalent to that $\nabla_{\sigma}^- \vartheta_{\tau}^- \pi(D)^*$ and $\nabla_{\tau}^- \vartheta_{\sigma}^- \pi(D)^*$ do not have non-zero locally algebraic sub. In the next sequence, we will show the hypothesis holds for some crystabelline generic $D$ (cf. Proposition \ref{Ppurity}). We can finally prove:
\begin{theorem}\label{TtubeGL2}
Assume Hypothesis \ref{newhypo} and Hypothesis \ref{Hnew2}, then Conjecture \ref{conjdim} and Conjecture \ref{conjGL2} (1) hold.
\end{theorem}
\begin{proof}
As $\nabla_{\tau}^- \vartheta_{\sigma}^- \pi(D)^*$ is zero or pure of dimension $1$, by the  exact sequences in Proposition \ref{Pthetasigmatau} (3), $E^{3d_K+2}(\nabla_{\tau}^- \vartheta_{\sigma}^+ \pi(D)^*)=0$ (hence $\nabla_{\tau}^- \vartheta_{\sigma}^+ \pi(D)^*$ is Cohen-Macaulay), $E^{3d_K+1}(\nabla_{\tau}^+ \vartheta_{\sigma}^+ \pi(D)^*)\cong \nabla_{\tau}^- \vartheta_{\sigma}^- \pi(D)^* \otimes_{E} \delta_D$ and $E^{3d_K}(\vartheta_{\tau}^+ \vartheta_{\sigma}^- \pi(D)^*)\cong \vartheta_{\tau}^- \vartheta_{\sigma}^+ \pi(D)^* \otimes_E \delta_D$. Hence $\nabla_{\tau}^- \vartheta_{\sigma}^- \pi(D)^*$ (resp. $\vartheta_{\tau}^-\vartheta_{\sigma}^+ \pi(D)^*$) is zero or Cohen-Macaulay of dimension $1$ (resp. of dimension $d_K=2$).  This together with the exact sequences in Proposition \ref{Pthetasigmatau} (3) (again) imply $E^{3d_K}(\vartheta_{\tau}^+ \vartheta_{\sigma}^+\pi(D)^*) \cong \vartheta_{\tau}^- \vartheta^-_{\sigma} \pi(D)^* \otimes_E\delta_D$ and $E^{3d_K+1}(\nabla_{\tau}^+ \vartheta_{\sigma}^- \pi(D)^*) \cong \nabla_{\tau}^- \vartheta_{\sigma}^+ \pi(D)^* \otimes_E\delta_D$. With Proposition \ref{Pdual2},  \ref{Pthetasigmatau} and \ref{Pnablasigmatau+}, these (and the counterparts exchanging $\sigma$ and $\tau$) prove Conjecture \ref{conjdim}. 

The ``$+$" part of Conjecture \ref{conjGL2} (1) has been obtained in Proposition \ref{sigmatau+}. Consider the  commutative diagram
\begin{equation*}
	\begin{CD}
		0 @>>> \Theta_{\tau}\vartheta^-_{\sigma}\pi(D)^* @>>> \Theta_{\tau} \pi(D)^* @>>> \Theta_{\tau} (\nabla_{\sigma}^- \pi(D)^*) @>>> 0 \\
		@. @VVV @VVV @VVV @. \\
		0 @>>> \vartheta_{\sigma}^- \pi(D)^* @>>> \pi(D)^* @>>> \nabla_{\sigma}^- \pi(D)^* @>>> 0.
	\end{CD}
\end{equation*}
which induces
\begin{equation}\label{Et+s-}
	0 \ra \vartheta_{\tau}^+ \vartheta_{\sigma}^- \pi(D)^* \ra \vartheta_{\tau}^+ \pi(D)^* \ra \vartheta_{\tau}^+ \nabla_{\sigma}^- \pi(D)^* \xrightarrow{\delta'} \nabla_{\tau}^- \vartheta_{\sigma}^- \pi(D)^* \ra \nabla_{\tau}^- \pi(D)^* \ra \nabla_{\tau}^- \nabla_{\sigma}^- \pi(D)^* \ra 0.
\end{equation}
By the same argument as in the proof of Proposition \ref{sigmatau+}, $\Ima(\delta')$ is locally algebraic. As $\nabla_{\tau}^- \vartheta_{\sigma}^- \pi(D)^*$ is assumed pure, $\delta'=0$.  We deduce $\nabla_{\tau}^- \vartheta^-_{\sigma} \pi(D)^* \cong \vartheta^-_{\sigma} \nabla_{\tau}^- \pi(D)^*$ and a commutative diagram of exact sequences (which is just $\boxdot^- \pi(D)^*$)
\begin{equation}\label{Ctausigma3}
	\begin{CD}@. 0 @. 0 @. 0 @. \\
		@. @VVV @VVV @VVV @. \\		
		0 @>>>  \vartheta_{\tau}^- \vartheta_{\sigma}^-\pi(D)^* @>>> \vartheta_{\tau}^- \pi(D)^* @>>> \vartheta_{\tau}^-  \nabla_{\sigma}^- \pi(D)^* @>>> 0 \\ 
		@. @VVV @VVV @VVV @. \\
		0 @>>>  \vartheta_{\sigma}^- \pi(D)^* @>>> \pi(D)^* @>>>  \nabla_{\sigma}^- \pi(D)^* @>>> 0 \\
		@. @VVV @VVV @VVV @. \\
		0 @>>> \nabla_{\tau}^-\vartheta_{\sigma}^- \pi(D)^* @>>>  \nabla_{\tau}^- \pi(D)^* @>>> \nabla_{\tau}^- \nabla_{\sigma}^- \pi(D)^* @>>> 0
		\\ @. @VVV @VVV @VVV @. \\
		@. 0 @. 0 @. 0. @.
	\end{CD}
\end{equation}
This concludes the proof.
\end{proof}

\begin{remark}  (1) Suppose $D$ is de Rham, then $\pi(D)^*$ have the following structure:
$$	\begindc{\commdiag}[200]
\obj(0,0)[a]{$\vartheta_{\Sigma_K}^- \pi(D)^*$}
\obj(4,-2)[b]{$\vartheta_{\sigma} ^-\nabla_{\tau}^- \pi(D)^*$}
\obj(4,2)[d]{$\vartheta_{\tau} ^-\nabla_{\sigma}^- \pi(D)^*$}
\obj(8,0)[f]{$\pi(D)^{\lalg, *}$}
\mor{a}{b}{}[+1,\solidline]
\mor{a}{d}{}[+1,\solidline]
\mor{b}{f}{}[+1,\solidline]
\mor{d}{f}{}[+1,\solidline]
\enddc.$$

(2) If one replaces $\pi(D)^*$ by a certain self-extension $\tilde{\pi}^*=[\pi(D)^*\rule[2.5pt]{10pt}{0.5 pt} \pi(D)^*]$,  it seems possible that Hypothesis \ref{hypoCM} still holds, while Hypotheses \ref{newhypo} and \ref{Hnew2} do not hold anymore. 
\end{remark}
\subsubsection{Crystabelline $\pi(D)$ for $\GL_2(K)$ with $[K:\Q_p]=2$}\label{secpirho}

In the section we use $\boxdot^{\pm}(\pi(D)^*)$ to study  $\pi(D)^*$. 

We first recall some formulas of the Schneider-Teitelbaum dual of locally analytic principal series. 
Let $\psi=\psi_1 \otimes \psi_2$ be a smooth character of $T(K)$ such that $\psi_1 \psi_2^{-1}\neq 1, |\cdot|^2$. Then $(\Ind_{B^-(K)}^{\GL_2(K)} \psi_1 \otimes \psi_2)^{\infty}$ is irreducible. By \cite[Prop.~6.5]{ST-dual}, we have 
\begin{equation*}
\EE^i\big(\cF_{B^-}^{\GL_2}(M^-(-w\cdot \lambda), \psi)^*\big)\cong \begin{cases}
	\cF_{B^-}^{\GL_2}(M^-(-w_0 w\cdot \lambda^*), \psi^{-1} |\cdot|_K^{-1} \otimes |\cdot|_K)^* & i=3d_K \\
	0 & \text{otherwise.}
\end{cases}
\end{equation*}
Using an easy d\'evissage, we deduce the following proposition.
\begin{proposition}\label{dualformula}
For $I\subset \Sigma_K$, $w_I\in \sW_{2,I}$, $\cF_{B^-}^{\GL_2}(M_I^-(-w_I \cdot \lambda_I) \otimes_E L_{\Sigma_K \setminus I}^-(-\lambda_{\Sigma_K \setminus I}), \psi)^*$ is Cohen-Macaulay of grade $3d_K+\# I$ and
\begin{multline*}
	E^{3d_K+\# I}\big( \cF_{B^-}^{\GL_2}(M_I^-(-w_I \cdot \lambda_I) \otimes_E L_{\Sigma_K \setminus I}^-(-\lambda_{\Sigma_K \setminus I}), \psi)^* \big)\\
	\cong \cF_{B^-}^{\GL_2}(M_I^-(-w_0w_I \cdot \lambda_I^*) \otimes_E L_{\Sigma_K \setminus I}^-(-\lambda_{\Sigma_K \setminus I}^*), \psi^{-1} |\cdot|_K^{-1} \otimes |\cdot|_K)^*.
\end{multline*}
\end{proposition} 
Suppose $D$ is crystabelline of regular Sen weights (i.e. $\textbf{h}$ is strictly dominant), and let $\ul{\phi}_1=\phi_1 \otimes \phi_2$, $\ul{\phi}_2=\phi_2 \otimes \phi_1$ be the two refinements of $D$. Assume $D$ is generic, i.e. $\phi_1\phi_2^{-1} \neq 1, |\cdot|_K^{\pm}$. We recall some locally $\Q_p$-analytic representations associated to $D$ in \cite{Br}. For $\sigma\in \Sigma_K$, the refinement $\ul{\phi}_1=\phi_1 \otimes \phi_2$ is called \textit{$\sigma$-critical} if  $D_{\sigma}\cong \cR_{K,E}(\phi_1 \sigma(z)^{h_{\sigma,2}}) \oplus \cR_{K,E}(\phi_2\sigma(z)^{h_{\sigma,1}})$ (similarly for $\ul{\phi}_2$). Note the two refinements $\ul{\phi}_1$  and  $\ul{\phi}_2$ can not be $\sigma$-critical on the same time for a fixed embedding $\sigma$, as $h_{\sigma,1}\neq h_{\sigma,2}$. If $\ul{\phi}_1$ is $\sigma$-critical, we put
\begin{equation*}
\pi(D_{\sigma}, \sigma, \lambda):=\Big((\Ind_{B^-(K)}^{\GL_2(K)} \jmath(\ul{\phi}_1) \sigma(z)^{s_{\sigma}\cdot \lambda_{\sigma}})^{\sigma-\an} \oplus (\Ind_{B^-(K)}^{\GL_2(K)} \jmath(\ul{\phi}_2) \sigma(z)^{\lambda_{\sigma}})^{\sigma-\an}\Big)\otimes_E L_{\Sigma_K \setminus \{\sigma\}}(\lambda_{\Sigma_K \setminus \{\sigma\}}).
\end{equation*}
If $D$ is $\sigma$-critical for $\ul{\phi}_2$, we define $\pi(D_{\sigma},\sigma, \lambda)$ in a similar way exchanging $\ul{\phi}_1$ and $\ul{\phi}_2$. 
If $D$ is not $\sigma$-critical, i.e. not $\sigma$-critical for either $\ul{\phi}_1$ or $\ul{\phi}_2$, then $D_{\sigma}$ is isomorphic to the unique non-split de Rham extension
\begin{equation}\label{Dsigma}
D_{\sigma}\cong [\cR_{K,E}(\phi_1 \sigma(z)^{h_{\sigma,1}}) \rule[2.5pt]{10pt}{0.5 pt} \cR_{K,E}(\phi_2 \sigma(z)^{h_{\sigma,2}})] \cong [\cR_{K,E}(\phi_2 \sigma(z)^{h_{\sigma, 1}}) \rule[2.5pt]{10pt}{0.5 pt} \cR_{K,E}(\phi_1 \sigma(z)^{h_{\sigma, 2}})].
\end{equation}
Indeed, by \cite[\S~1.2]{Ding4}, $\dim_E \Ext^1_{g}\big(\cR_{K,E}(\phi_i \sigma(z)^{h_{\sigma,2}}), \cR_{K,E}(\phi_j \sigma(z)^{h_{\sigma,1}})\big)=1$ for $\{i,j\}=\{1,2\}$, and it is clear that if $D$ is not $\sigma$-critical, $D_{\sigma}$ is a non-split extension as in (\ref{Dsigma}).  In this case, we put 
\begin{multline*}
\pi(D_{\sigma}, \sigma, \lambda):=L_{\Sigma_K \setminus \{\sigma\}}(\lambda_{\Sigma_K \setminus \{\sigma\}})\\ \otimes_E \Big((\Ind_{B^-(K)}^{\GL_2(K)} \jmath(\ul{\phi}_1) \sigma(z)^{\lambda_{\sigma}})^{\sigma-\an} \oplus_{(\Ind_{B^-(K)}^{\GL_2(K)} \jmath(\ul{\phi}_1))^{\infty}\otimes_E L(\lambda_{\sigma})}  (\Ind_{B^-(K)}^{\GL_2(K)} \jmath(\ul{\phi}_2) \sigma(z)^{\lambda_{\sigma}})^{\sigma-\an} \Big).
\end{multline*}
For simplicity, write $\PS_{i,w}:=\cF_{B^-}^{\GL_2}(L^-(-w\cdot \lambda), \jmath(\ul{\phi}_i))$, and $\alg:=\PS_{1,1}\cong \PS_{2,1}$ (which should be the locally algebraic subrepresentation of $\pi(D)$). Then $\pi(D_{\sigma}, \sigma, \lambda)$ has the form
\begin{equation}\label{EDsigma}
\begindc{\commdiag}[200]
\obj(0,1)[a]{$\alg$}
\obj(4,0)[b]{$\PS_{2,s_{\sigma}}$}
\obj(4, 2)[c]{$\PS_{1,s_{\sigma}}$}
\mor{a}{b}{$a_{2,\sigma}$}[+1,\solidline]
\mor{a}{c}{$a_{1,\sigma}$}[+1,\solidline]
\enddc,
\end{equation}
where $a_{i,\sigma}\in \Ext^1(\PS_{i,s_{\sigma}},\alg)\cong E$ (cf. \cite{Sch10}). So $a_{i,\sigma}=0$ means the corresponding extension splits, which is equivalent to that the corresponding refinement is $\sigma$-critical.

Suppose $\alg\hookrightarrow \pi(D)$, by \cite{BH2} (see also \cite{Ding1} in the case of unitary Shimura curves), it then extends to an injection $\pi(D_{\sigma},\sigma, \lambda) \hookrightarrow \pi(D)$, which, by Corollary \ref{Ugfini} (2),  induces an injection
\begin{equation}\label{sigma0}
\pi(D_{\sigma},\sigma, \lambda) \hooklongrightarrow (\nabla^-_{\Sigma_K\setminus \{\sigma\}} \pi(D)^*)^*=\pi^-(D_{\sigma}, \sigma, \lambda)
\end{equation}
\begin{hypothesis}
\label{sigma}
The map (\ref{sigma0}) is an isomorphism, for all $\sigma\in \Sigma_K$.
\end{hypothesis}
\begin{remark}\label{RemQS}
By the recent work \cite{QS} (generalzing \cite{Pan4}), when $\rho$ appears in the cohomology of unitary Shimura curves, Hypothesis \ref{sigma} holds.
\end{remark}
Let $M_{\infty}$ be the patched module of \cite{CEGGPS1}, so $\Pi_{\infty}\cong \Hom_{\co_E}^{\cont}(M_{\infty},E)$. 
\begin{proposition}\label{Ppurity}
Assume Hypothesis \ref{hypoCM} $M_{\infty}$ is flat over $R_{\infty}$, and $d_K= 2$.\footnote{Which is known when $K=\Q_{p^2}$ under mild assumption.} Then Hypothesis \ref{sigma} implies Hypotheses \ref{newhypo} and \ref{Hnew2} both hold.
\end{proposition}
\begin{proof}
Hypothesis \ref{newhypo} is clear, since $\pi(D_{\sigma}, \sigma, \lambda)$ does not have non-zero locally algebraic quotient by definition.	By  Hypothesis \ref{sigma},  Proposition \ref{dualformula} and Proposition \ref{Pdual2}, one directly calculates
$\nabla_{\sigma}^? \nabla_{\tau}^{??} \pi(D)^*\cong  \alg^*=\big((\Ind_{B^-(K)}^{\GL_2(K)} \jmath(\ul{\phi}_1))^{\infty} \otimes_E L(\lambda)\big)^*$ for $?, ??\in \{+,-\}$, 
\begin{equation*}
	\pi^+(D_{\sigma}, \sigma, \lambda)^*=	\nabla^+_{\tau} \pi(D)^*\cong  E^{3d_K+1}(\nabla^-_{\tau} \pi(D)^*) \otimes_E \delta_D^{-1} \cong \pi^-(D_{\sigma}, \sigma, \lambda)^* (\cong \pi(D_{\sigma}, \sigma, \lambda)^*),
\end{equation*}
\begin{eqnarray}\label{sgima+tau+}
	&&\pi^+(\Delta, \sigma, \lambda)^*=	\vartheta^+_{\sigma} \nabla^+_{\tau} \pi(D)^*\\ &\cong& \Big(\big((\Ind_{B^-(K)}^{\GL_2(K)} \jmath(\ul{\phi}_1) \sigma(z)^{\lambda_{\sigma}})^{\sigma-\an} \oplus (\Ind_{B^-(K)}^{\GL_2(K)} \jmath(\ul{\phi}_2) \sigma(z)^{\lambda_{\sigma}})^{\sigma-\an}\big) \otimes_E L_{\tau}(\lambda_{\tau})\Big)^* \nonumber \\
	&\cong& [\PS_{1,s_{\sigma}}^* \rule[2.5pt]{10pt}{0.5 pt}\alg^*] \oplus 	[\PS_{2,s_{\sigma}}^* \rule[2.5pt]{10pt}{0.5 pt} \alg^*]. \nonumber
\end{eqnarray}
If $\nabla_{\tau}^- \vartheta_{\sigma}^- \pi(D)^*$ is not pure, by the second exact sequence in Proposition \ref{Pthetasigmatau}, $E^{3d_K+2}(\nabla_{\tau}^- \vartheta_{\sigma}^+ \pi(D)^*)\neq 0$. By (\ref{Ecokalg}), we deduce $E^{3d_K+2}(\nabla_{\tau}^- \vartheta_{\sigma}^+ \pi(D)^*) \otimes_E \delta_D^{-1}$ has to be isomorphic to $\nabla_{\tau}^+ \nabla_{\sigma}^- \pi(D)^* \cong \alg^*$. On the other hand, we have by Proposition \ref{sigmatau+} (1), Proposition \ref{Pnablasigmatau+}  (1):
\begin{equation*}
	E^{3d_K+1}(\nabla_{\tau}^+ \vartheta_{\sigma}^+ \pi(D)^*) \otimes_E \delta_D^{-1}  \cong E^{3d_K+1}(\vartheta_{\sigma}^+ \nabla_{\tau}^+ \pi(D)^*) \otimes_E \delta_D^{-1} \cong  \vartheta_{\sigma}^- \nabla_{\tau}^- \pi(D)^* \cong \PS_{1,s_{\sigma}}^* \oplus \PS_{2,\sigma}^*.
\end{equation*}
Using again Proposition \ref{Pthetasigmatau} (3), we get an exact sequence
\begin{equation*}
	0 \lra \alg^* \lra \nabla_{\tau}^- \vartheta_{\sigma}^- \pi(D)^* \lra (\PS_{1,s_{\sigma}} \oplus \PS_{2,s_\sigma})^* \lra 0.
\end{equation*}
Recall the dual of  $\nabla_{\tau}^-\vartheta_{\sigma}^- \pi(D)^*$ is the maximal $\text{U}(\ug_{\tau})$-finite subrepresentation of $(\vartheta_{\sigma}^- \pi(D)^*)^*=\pi(D)/\nabla_{\sigma}^- \pi(D)\cong \pi(D)/\pi(D_{\tau}, \tau, \lambda)$. By the above exact sequence, we see $\pi(D)$ contains a subrepresentation of the form
\begin{equation*}
	\alg   \rule[2.5pt]{10pt}{0.5 pt}  (\PS_{1,s_{\sigma}} \oplus \PS_{2,s_\sigma} \oplus \PS_{1,s_{\tau}} \oplus \PS_{2,s_\tau})  \rule[2.5pt]{10pt}{0.5 pt}  \alg.
\end{equation*}
But this contradicts Theorem \ref{Tnosurplus}.
\end{proof}

In the following, we assume $[K:\Q_p]=2$, $M_{\infty}$ is flat over $R_{\infty}$, Hypothesis \ref{hypoCM} and Hypothesis \ref{sigma}. 
Let $\Sigma_K=\{\sigma,\tau\}$.  As in the proof of Proposition \ref{Ppurity}, we have an explicit description of $\pi^{\pm}(\Delta, \sigma, \lambda)^*, \pi^+(\Delta, \emptyset,\lambda)^*$, $\pi^{\pm}(D_{\sigma},\sigma, \lambda)^*$. Also, it is straightforward to see
\begin{equation*}
\dim_E \Hom(\pi^+(\Delta, \sigma, \lambda)^*, \pi^+(\Delta, \emptyset, \lambda)^*)=\dim_E\Ext^1\big(\pi^-(\Delta, \emptyset, \lambda)^*, \pi^-(\Delta, \sigma, \lambda)^*\big)=2.
\end{equation*}
All these (together with the discussion in ``Cas cristallin" of \cite[\S~3.2]{Br16}) confirm  the corresponding part of Conjecture \ref{conjGL2} (2) (3). 
We mover to $\pi^+(\Delta, \Sigma_K, \lambda)^*)=\vartheta_{\Sigma_K}^+ \pi(D)^*$.  
We  describe the $\text{U}(\ug_{\sigma})$-finite quotient of $\vartheta_{\Sigma_K}^+ \pi(D)^*$ (noting  the representation   should have supersingular constituents).  Using $\boxdot^+(\pi(D)^*)$ and Theorem \ref{TtubeGL2}, we have an exact sequence
\begin{equation*}
0  \lra \vartheta_{\Sigma_K}^+ \pi(D)^*/\pi(D)^* \lra \vartheta_{\tau}^+\nabla_{\sigma}^+ \pi(D)^* \oplus \vartheta_{\sigma}^+\nabla_{\tau}^+ \pi(D)^* \lra \nabla_{\Sigma_K}^+ \pi(D)^* \lra 0.
\end{equation*}
Thus $\vartheta_{\Sigma_K}^+ \pi(D)^*/\pi(D)^*$ is 
isomorphic to an extension of $(\alg^{\oplus 3})^*$ by $\oplus_{i=1,2} (\PS_{i,s_{\sigma}}^* \oplus \PS_{i,s_{\tau}}^*)$.  Similarly, using $\boxdot^- (\pi(D)^*)$, we have an exact sequence
\begin{equation*}
0 \lra \pi(D)^*/\vartheta_{\Sigma_K}^- \pi(D)^* \lra \vartheta_{\sigma}^- \pi(D)^* \oplus \vartheta_{\tau}^- \pi(D)^* \lra \nabla_{\Sigma_K}^- \pi(D)^* \lra 0.
\end{equation*}
Hence $\pi(D)^*/\vartheta_{\Sigma_K}^- \pi(D)^*$ is isomorphic to an extension of $\alg^*$ by $\oplus_{i=1,2} (\PS_{i,s_{\sigma}}^* \oplus \PS_{i,s_{\tau}}^*)$. In fact, we have $\pi(D)^*/\vartheta_{\Sigma_K}^- \pi(D)^*\cong \big(\pi^-(D_{\sigma}, \sigma, \lambda) \oplus_{\alg} \pi^-(D_{\tau}, \tau, \lambda)\big)^*$ (which depends on $D_{\sigma}$ and $D_{\tau}$), and $\vartheta_{\Sigma_K}^+ \pi(D)^*/\pi(D)^*\cong \EE^7\big(\pi(D)^*/\vartheta_{\Sigma_K}^- \pi(D)^*\big) \otimes_E \delta_D^{-1}$.

By Lemma \ref{sigmataucomm} (1) (2), $\boxdot^+(\vartheta_{\Sigma_K}^+ \pi(D)^*)$  and $\boxdot^-(\vartheta_{\Sigma_K}^- \pi(D)^*)$  have a trivial structure. The following proposition shows  $\boxdot^-(\vartheta_{\Sigma_K}^+ \pi(D)^*)$ and $\boxdot^+(\vartheta_{\Sigma_K}^- \pi(D)^*)$ coincide:
\begin{equation*}\footnotesize
\begindc{\commdiag}[400]
\obj(6,0)[a2]{$\nabla_{\tau}^+\vartheta_{\Sigma_K}^- \pi(D)^*$}
\obj(8,0)[b2]{$\nabla_{\tau}^+\vartheta_{\sigma}^+ \vartheta_{\Sigma_K}^- \pi(D)^*$}
\obj(10,0)[c2]{$\nabla_{\Sigma_K}^+ \vartheta_{\Sigma_K}^- \pi(D)^*$}
\obj(6,1)[d2]{$\vartheta_{\tau}^+ \vartheta_{\Sigma_K}^- \pi(D)^*$}
\obj(8,1)[e2]{$\vartheta_{\Sigma_K}^+\vartheta_{\Sigma_K}^- \pi(D)^*$}
\obj(10,1)[f2]{$\nabla_{\sigma}^+ \vartheta_{\tau}^+ \vartheta_{\Sigma_K}^- \pi(D)^*$}
\obj(6,2)[g2]{$\vartheta_{\Sigma_K}^- \pi(D)^*$}
\obj(8,2)[h2]{$\vartheta_{\sigma}^+ \vartheta_{\Sigma_K}^- \pi(D)^*$}
\obj(10,2)[i2]{$\nabla_{\sigma}^+ \vartheta_{\Sigma_K}^- \pi(D)^*$}
\obj(0,0)[a1]{$\nabla_{\tau}^- \vartheta_{\sigma}^- \vartheta_{\Sigma_K}^+ \pi(D)^*$}
\obj(2,0)[b1]{$\nabla_{\tau}^- \vartheta_{\Sigma_K}^+ \pi(D)^*$}
\obj(4,0)[c1]{$\nabla_{\Sigma_K}^- \vartheta_{\Sigma_K}^+ \pi(D)^*$}
\obj(0,1)[d1]{$\vartheta_{\sigma}^- \vartheta_{\Sigma_K}^+ \pi(D)^*$}
\obj(2,1)[e1]{$\vartheta_{\Sigma_K}^+ \pi(D)^*$}
\obj(4,1)[f1]{$\nabla_{\sigma}^- \vartheta_{\Sigma_K}^+ \pi(D)^*$}
\obj(0,2)[g1]{$\vartheta_{\Sigma_K}^- \vartheta_{\Sigma_K}^+ \pi(D)^*$}
\obj(2,2)[h1]{$\vartheta_{\tau}^- \vartheta_{\Sigma_K}^+ \pi(D)^*$}
\obj(4,2)[i1]{$\nabla_{\sigma}^-\vartheta_{\tau}^- \vartheta_{\Sigma_K}^+ \pi(D)^*$}
\mor{a1}{b1}{}
\mor{b1}{c1}{}
\mor{d1}{a1}{}
\mor{e1}{b1}{}
\mor{f1}{c1}{}
\mor{d1}{e1}{}
\mor{e1}{f1}{}
\mor{g1}{d1}{}
\mor{h1}{e1}{}
\mor{i1}{f1}{}
\mor{g1}{h1}{}
\mor{h1}{i1}{}
\mor{a2}{b2}{}
\mor{b2}{c2}{}
\mor{d2}{a2}{}
\mor{e2}{b2}{}
\mor{f2}{c2}{}
\mor{d2}{e2}{}
\mor{e2}{f2}{}
\mor{g2}{d2}{}
\mor{h2}{e2}{}
\mor{i2}{f2}{}
\mor{g2}{h2}{}
\mor{h2}{i2}{}
\enddc.
\end{equation*} 
\begin{proposition}\label{PpiDelta}(1) We have isomorphisms $\vartheta_{\Sigma_K}^- \pi(D)^* 
\xrightarrow{\sim}\vartheta_{\Sigma_K}^- \vartheta_{\Sigma_K}^+ \pi(D)^*$, $\vartheta_{\Sigma_K}^+\vartheta_{\Sigma_K}^- \pi(D)^* \xrightarrow{\sim} \vartheta_{\Sigma_K}^+ \pi(D)^*$, $\vartheta_{\sigma}^- \vartheta_{\Sigma_K}^+ \pi(D)^*\cong \vartheta_{\sigma}^- \vartheta_{\tau}^+ \pi(D)^* \cong \vartheta_{\tau}^+ \vartheta_{\sigma}^- \pi(D)^* \cong \vartheta_{\tau}^+ \vartheta_{\Sigma_K}^- \pi(D)^*$. Consequently, the two hypercubes $\boxdot^{\pm} \vartheta_{\Sigma_K}^{\mp} \pi(D)^*$ are isomorphic (in an obvious sense).

(2) All the vertical and horizontal sequences in $\boxdot^{\pm}(\vartheta_{\Sigma_K}^{\mp} \pi(D)^*)$ are exact after adding $0$ at two ends, and have the form $0 \ra \vartheta_{\sigma'}^- M \ra M \ra \nabla_{\sigma'}^- M \ra 0$ and meantime have the form $0 \ra M' \ra \vartheta_{\sigma''}^+ M' \ra \nabla_{\sigma''}^+ M' \ra 0$.

(3) $\nabla_{\sigma}^- \vartheta_{\Sigma_K}^+ \pi(D)^* \cong \pi^+(\Delta, \tau, \lambda)^{*,\oplus 2}$ (the same holds with $\sigma$, $\tau$ exchanged), and $\nabla_{\Sigma_K}^- \vartheta_{\Sigma_K}^+ \pi(D)^*\cong \alg^{*, \oplus 4}$.  Consequently, $\vartheta_{\Sigma_K}^+ \pi(D)^*/\vartheta_{\Sigma_K}^- \pi(D)^*$ has cosocle $\alg^{*,\oplus 4}$ and is isomorphic to an extension of $\alg^{*, \oplus 4}$ by $(\oplus_{i=1,2} \PS_{i,s_{\sigma}}^* \oplus \PS_{i, s_{\tau}}^*)^{\oplus 2}$.
\end{proposition}
\begin{proof}
(1):	As $\vartheta_{\Sigma_K}^+ \pi(D)^*/\pi(D)^*$ (resp. $\pi(D)^*/\vartheta_{\Sigma_K}^- \pi(D)^*$) is generated by $\text{U}(\ug_{\sigma})$-finite vectors and $\text{U}(\ug_{\tau})$-finite vectors, $\Theta_{\Sigma_K}(\vartheta_{\Sigma_K}^+ \pi(D)^*/\pi(D)^*)=0$ (resp. $\Theta_{\Sigma_K}(\pi(D)^*/\vartheta_{\Sigma_K}^- \pi(D)^*)=0$). We then deduce 
$\vartheta_{\Sigma_K}^- \pi(D)^*\xrightarrow{\sim} \vartheta_{\Sigma_K}^-\vartheta_{\Sigma_K}^+ \pi(D)^*$ (resp. $\vartheta_{\Sigma_K}^+ \vartheta_{\Sigma_K}^- \pi(D)^* \xrightarrow{\sim} \vartheta_{\Sigma_K}^+ \pi(D)^*$). 
By Lemma \ref{lemsharp} (4) and Lemma \ref{sigmataucomm} (3), $\vartheta_{\tau}^+ \vartheta_{\Sigma_K}^- \pi(D)^*\cong \vartheta_{\tau}^+ \vartheta_{\sigma}^- \pi(D)^*$. By Lemma \ref{lemsharp} (3) and the fact $\vartheta_{\tau}^+ \pi(D)^*$ does not have non-zero $\text{U}(\ug_{\sigma})$-finite sub (see the proof of Theorem \ref{Tinj}),  $\vartheta_{\sigma}^- \vartheta_{\Sigma_K}^+ \pi(D)^*\cong \vartheta_{\sigma}^- \vartheta_{\tau}^+ \pi(D)^*$. By (\ref{Et+s-})  (noting $\delta'=0$) and the fact $\vartheta_{\tau}^+ \nabla_{\sigma}^- \pi(D)^*$ is $\text{U}(\ug_{\sigma})$-finite (as $\nabla_{\sigma}^- \pi(D)^*$ is so), we deduce $\vartheta_{\sigma}^- \vartheta_{\tau}^+ \pi(D)^* \hookrightarrow \vartheta_{\tau}^+ \vartheta_{\sigma}^- \pi(D)^*$ and $\nabla_{\sigma}^- \vartheta_{\tau}^+ \pi(D)^* \twoheadrightarrow \vartheta_{\tau}^+ \nabla_{\sigma}^- \pi(D)^*$ (where if one is an isomorphism, the other is as well).  By direct calculation, we have  \begin{equation}\label{nathst}\vartheta_{\tau}^+ \nabla_{\sigma}^- \pi(D)^* \cong   [\PS_{1,s_{\tau}}^* \rule[2.5pt]{10pt}{0.5 pt} \alg^*] \oplus 	[\PS_{2,s_{\tau}}^* \rule[2.5pt]{10pt}{0.5 pt} \alg^*]\cong \pi^+(\Delta,\tau, \lambda)^*.\end{equation}
By (\ref{Et+s+}), we then see that $\nabla_{\sigma}^- \vartheta_{\tau}^+ \pi(D)^*$ has the same constituents as $\vartheta_{\tau}^+ \nabla_{\sigma}^- \pi(D)^*$ hence  $\nabla_{\sigma}^- \vartheta_{\tau}^+ \pi(D)^* \xrightarrow{\sim}\vartheta_{\tau}^+ \nabla_{\sigma}^- \pi(D)^*$ and $\vartheta_{\sigma}^- \vartheta_{\tau}^+ \pi(D)^* \xrightarrow{\sim} \vartheta_{\tau}^+ \vartheta_{\sigma}^- \pi(D)^*$.
(1) follows.

(3): We first describe $\nabla_{\sigma}^- \vartheta_{\Sigma_K}^+ \pi(D)^*$. Consider the commutative diagram
\begin{equation*}
	\begin{CD} 0 @>>> \Theta_{\sigma} \vartheta_{\tau}^+ \pi(D)^* @>>> \Theta_{\sigma} \vartheta_{\Sigma_K}^+ \pi(D)^* @>>> \Theta_{\sigma} (\nabla_{\sigma}^+ \vartheta_{\tau}^+ \pi(D)^*) @>>> 0 \\
		@. @VVV @VVV @VVV @. \\
		0 @>>> \vartheta_{\tau}^+ \pi(D)^* @>>> \vartheta_{\Sigma_K}^+ \pi(D)^* @>>> \nabla_{\sigma}^+ \vartheta_{\tau}^+ \pi(D)^* @>>> 0. 
	\end{CD}
\end{equation*}
As $\nabla_{\sigma}^+\vartheta_{\tau}^+ \pi(D)^*$ is $\text{U}(\ug_{\sigma})$-finite, $ \Theta_{\sigma} (\nabla_{\sigma}^+ \vartheta_{\tau}^+ \pi(D)^*)=0$. We deduce an exact sequence
\begin{equation}\label{sigmaSigma+}
	0 \lra \nabla_{\sigma}^- \vartheta_{\tau}^+ \pi(D)^* \lra  \nabla_{\sigma}^- \vartheta_{\Sigma_K}^+ \pi(D)^* \lra \nabla_{\sigma}^+ \vartheta_{\tau}^+ \pi(D)^* \lra 0.
\end{equation}  This, together with (1), (\ref{nathst}) and (\ref{sgima+tau+}), imply that $\nabla_{\sigma}^- \vartheta_{\Sigma_K}^+ \pi(D)^*$ is a self-extension of $\pi^+(\Delta, \tau, \lambda)^*$.  By the  discussion above the proposition,  the   constituents of $\vartheta_{\Sigma_K}^+ \pi(D)^*/\vartheta_{\Sigma_K}^- \pi(D)^*$ consist of $4$ copies of $\alg^*$ and $2$-copies of $\oplus_{i=1,2} (\PS_{i,s_{\sigma}}^* \oplus \PS_{i,s_{\tau}}^*)$, and we have
\begin{equation}\label{dimgq}
	\dim_E \Hom_{\GL_2(K)}(\vartheta_{\Sigma_K}^+ \pi(D)^*, \alg^*) \geq 3. 
\end{equation}
We analyse the possible structure of $\nabla_{\sigma}^- \vartheta_{\Sigma_K}^+ \pi(D)^*$. Recall $\PS_{i,\tau}\cong [\alg  \rule[2.5pt]{10pt}{0.5 pt}  \PS_{i,s_{\tau}}]$, and  $\pi^+(\Delta, \tau, \lambda)=\PS_{1,\tau} \oplus \PS_{2,\tau}$ (see the proof of Proposition \ref{Ppurity}). Denote by $\Ext^1_{\inf, Z}\subset \Ext^1$ the subgroup consisting of extensions with both infinitesimal character and central character (for $\GL_2(K)$). Then $\dim_E  \Ext^1_{\inf,Z}(\PS_{i,\tau}^*, \PS_{j,\tau}^*)=\begin{cases}
	0 & i \neq j \\
	1 & i=j
\end{cases}$.  Moreover, when $i=j=1$, the unique (up to isomorphism) non-split self-extension (with central and infinitesimal character) of $\PS_{i,\tau}^*$ is the dual of the representation 
\begin{equation*}\tilde{\PS}_{i,\tau}:= \big(\Ind_{B^-(K)}^{\GL_2(K)} \jmath(\ul{\phi}_i) \tau(z)^{\lambda_{\tau}} \otimes_E (1+\psi \epsilon)\big)^{\tau-\an} \otimes_E L_{\sigma}(\lambda_{\sigma})
\end{equation*}
where $\psi: T(K) \ra E$ is the additive character sending $(a,d)$ to $\val_K(a/d)$. As $\nabla_{\sigma}^- \vartheta_{\Sigma_K}^+ \pi(D)^* \in \Ext^1_{\inf,Z}(\pi^+(\Delta, \tau, \lambda)^*, \pi^+(\Delta, \tau, \lambda)^*)$ and using (\ref{dimgq}), we see $\nabla_{\sigma}^- \vartheta_{\Sigma_K}^+ \pi(D)^*$ is either isomorphic to $\pi^+(\Delta, \tau, \lambda)^{*,\oplus 2}$ or $\PS_{i,\tau}^{*, \oplus 2} \oplus \tilde{\PS}_{j, \tau}^*$ for $\{i,j\}=\{1,2\}$.

Recall that $\vartheta_{\Sigma_K}^+ \pi(D)^*/\vartheta_{\Sigma_K}^- \pi(D)^*$ is essentially self-dual. If $$\vartheta_{\Sigma_K}^+ \pi(D)^*/\vartheta_{\Sigma_K}^- \pi(D)^* \twoheadlongrightarrow \nabla_{\sigma}^- \vartheta_{\Sigma_K}^+ \pi(D)^*\cong \PS_{i,\tau}^{*, \oplus 2} \oplus \tilde{\PS}_{j, \tau}^* ,$$ applying $E^{3d_K+1}(-)$ and (an easy variant of) \cite[Prop.~6.5]{ST-dual}, we get an injection
\begin{equation*}
	\tilde{\PS}_{i,{\tau}}^* \hooklongrightarrow \vartheta_{\Sigma_K}^+ \pi(D)^*/\vartheta_{\Sigma_K}^- \pi(D)^*,
\end{equation*}
a contradiction. So $\nabla_{\sigma}^- \vartheta_{\Sigma_K}^+ \pi(D)^*\cong \pi^+(\Delta, \tau, \lambda)^{*, \oplus 2}$, and $\nabla_{\Sigma_K}^- \vartheta_{\Sigma_K}^+ \pi(D)^* \cong \alg^{*, \oplus 4}$. The rest of (3) easily follows.

(2): First, it is clear that all the four maps at the left-upper corner of $\boxdot^-(\vartheta_{\Sigma_K}^+ \pi(D)^*)$ are all injective. By (3), it is not difficult to see the $\text{U}(\ug_{\tau})$-finite quotient $\nabla^-_{\tau}\vartheta_{\sigma}^- \vartheta_{\Sigma_K}^+ \pi(D)^*$ is isomorphic to $(\PS_{1,s_{\sigma}}^* \oplus \PS_{2,s_{\sigma}}^*)^{\oplus 2}$, and the map  $\nabla^-_{\tau}\vartheta_{\sigma}^- \vartheta_{\Sigma_K}^+ \pi(D)^* \ra \nabla_{\tau}^- \vartheta_{\Sigma_K}^+ \pi(D)^*$ factors through an isomorphism $\nabla^-_{\tau} \vartheta_{\sigma}^- \vartheta_{\Sigma_K}^+ \pi(D)^* \xrightarrow{\sim} \vartheta_{\sigma}^- \nabla_{\tau}^- \vartheta_{\Sigma_K}^+ \pi(D)^*$.  The same holds with $\sigma$ and $\tau$ exchanged. We see all the horizontal and vertical sequences in $\boxdot^-(\vartheta_{\Sigma_K}^+ \pi(D)^*)$ are exact and have the form $0 \ra \vartheta_{\sigma'}^- M \ra M \ra \nabla_{\sigma'}^- M \ra 0$. Looking at $\boxdot^+(\vartheta_{\Sigma_K}^- \pi(D)^*)$, it is clear that the top two horizontal sequences and the left two  vertical sequences have the form $0 \ra M' \ra \vartheta_{\sigma''}^+ M' \ra \nabla_{\sigma''}^+ M'' \ra 0$. Finally, similar statements (easily) hold for the bottom and right sequences by using the explicit structure of the representations.
%
\end{proof}
\begin{remark}\label{RemGL22}
By the proposition, $\pi(\Delta, \lambda)^*=\vartheta_{\Sigma_K}^+ \pi(D)^*$ has the following structure:
\begin{equation}	\label{diag2}
	\begindc{\commdiag}[200]
	\obj(0,3)[a]{$\vartheta_{\Sigma_K}^- \pi(D)^*$}
	\obj(4,0)[b]{$\PS_{2,s_{\sigma}}^{*,\oplus 2}$}
	\obj(4,2)[c]{$\PS_{1,s_{\tau}}^{*, \oplus 2}$}
	\obj(4,4)[d]{$\PS_{2,s_{\sigma}}^{*, \oplus 2}$}
	\obj(4,6)[e]{$\PS_{1,s_{\sigma}}^{*, \oplus 2}$}
	\obj(8,3)[f]{$\alg^{*, \oplus 4}$}
	\mor{a}{b}{}[+1,\solidline]
	\mor{a}{c}{}[+1,\solidline]
	\mor{a}{d}{}[+1,\solidline]
	\mor{a}{e}{}[+1,\solidline]
	\mor{f}{b}{}[+1,\solidline]
	\mor{f}{c}{}[+1,\solidline]
	\mor{f}{d}{}[+1,\solidline]
	\mor{f}{e}{}[+1,\solidline]
	\enddc
\end{equation}
We may furthermore expect \begin{multline*}
	\pi(\Delta, \lambda)^*/\pi_0(\Delta, \lambda)^*=\vartheta_{\Sigma_K}^+ \pi(D)^*/\vartheta_{\Sigma_K}^- \pi(D)^* \xlongrightarrow[\sim]{?} \\  
	\bigg(\begindc{\commdiag}[150]
	\obj(0,-1)[c]{$\PS_{1,s_{\tau}}^{*}$}
	\obj(0,1)[d]{$\PS_{1,s_{\sigma}}^{*}$}
	\obj(3,0)[f]{$\alg^{*}$}
	\mor{f}{c}{}[+1,\solidline]
	\mor{f}{d}{}[+1,\solidline]
	\enddc\bigg)
	\oplus 	\bigg(\begindc{\commdiag}[150]
	\obj(0,-1)[c]{$\PS_{2,s_{\tau}}^{*}$}
	\obj(0,1)[d]{$\PS_{2,s_{\sigma}}^{*}$}
	\obj(3,0)[f]{$\alg^{*}$}
	\mor{f}{c}{}[+1,\solidline]
	\mor{f}{d}{}[+1,\solidline]
	\enddc\bigg)
	\oplus 	\bigg(\begindc{\commdiag}[150]
	\obj(0,-1)[c]{$\PS_{1,s_{\tau}}^{*}$}
	\obj(0,1)[d]{$\PS_{2,s_{\sigma}}^{*}$}
	\obj(3,0)[f]{$\alg^{*}$}
	\mor{f}{c}{}[+1,\solidline]
	\mor{f}{d}{}[+1,\solidline]
	\enddc\bigg)
	\oplus 	\bigg(\begindc{\commdiag}[150]
	\obj(0,-1)[c]{$\PS_{2,s_{\tau}}^{*}$}
	\obj(0,1)[d]{$\PS_{1,s_{\sigma}}^{*}$}
	\obj(3,0)[f]{$\alg^{*}$}
	\mor{f}{c}{}[+1,\solidline]
	\mor{f}{d}{}[+1,\solidline]
	\enddc\bigg).\end{multline*}
In fact, by the knowledge on principal series (see below), it is not difficult to get the first two direct summands. The second two appear more subtle, as it should come from the translation of the supersingular constituent in $\pi(\Delta)$. 

\end{remark}
\begin{corollary}\label{Cext1} Let $i\in \{1,2\}$.

(1) $\dim_E \Ext^1_{\GL_2(K)}(\PS_{i,s_{\tau}}^*, \vartheta_{\Sigma_K}^- \pi(D)^*)=\dim_E \Hom_{\GL_2(K)}(\vartheta_{\sigma}^+ \vartheta_{\Sigma_K}^- \pi(D)^*, \PS_{i,s_\tau}^*)=2.$

(2)  $\dim_E \Ext^1_{\GL_2(K)}(\PS_{i,\sigma}^*, \vartheta_{\sigma}^+ \vartheta_{\Sigma_K}^- \pi(D)^*)=\dim_E \Hom_{\GL_2(K)}(\vartheta_{\Sigma_K}^+ \pi(D)^*, \PS_{i,\sigma}^*)=2$.
\end{corollary}
\begin{proof}
(1) By the proposition and (\ref{sgima+tau+}), the maximal $\text{U}(\ug_{\sigma})$-finite quotient $\nabla_{\sigma}^-(\vartheta_{\tau}^- \vartheta_{\Sigma_K}^+ \pi(D)^*)$ of $\vartheta_{\tau}^- \vartheta_{\Sigma_K}^+ \pi(D)^*\cong \vartheta_{\sigma}^+ \vartheta_{\Sigma_K}^- \pi(D)^*$ is isomorphic to $(\PS_{1,s_{\tau}}^* \oplus \PS_{2,s_{\tau}}^*)^{\oplus 2}$. The $\Hom$-part in (1) follows.  By Lemma \ref{lmnoalg}, $\Hom_{\GL_2(K)}(\PS_{i,s_{\tau}}^*, \vartheta_{\sigma}^+ \vartheta_{\Sigma_K}^- \pi(D)^*)=0$. In particular, all the subs of $\vartheta_{\sigma}^+ \vartheta_{\Sigma_K}^- \pi(D)^*$ given  by an extension of $\PS_{i,s_{\tau}}^*$ by $\vartheta_{\Sigma_K}^- \pi(D)^*$ are non-split. Let $V$ be the sub of $\vartheta_{\sigma}^+ \vartheta_{\Sigma_K}^- \pi(D)^*$ given by the extension of two copies of $\PS_{i, s_{\tau}}^*$ by $\vartheta_{\Sigma_K}^- \pi(D)^*$. We claim it is the \textit{universal} extension of $\PS_{i,s_{\tau}}^* $ by $\vartheta_{\Sigma_K}^- \pi(D)^*$. In fact, for any non-split extension $W\cong [\vartheta_{\Sigma_K}^-\pi(D)^* \rule[2.5pt]{10pt}{0.5 pt} \PS_{i,s_{\tau}}^*]$, we have by Lemma \ref{injlemm} (1): $W \hookrightarrow \vartheta_{\tau}^+ W \cong \vartheta_{\tau}^+ \vartheta_{\Sigma_K}^-\pi(D)^*$. As for $j\neq i$, $\Hom_{\GL_2(K)}(\vartheta_{\Sigma_K}^- \pi(D)^*,\PS_{j,s_{\tau}}^*)=\Hom_{\GL_2(K)}(\PS_{i,s_{\tau}^*}, \PS_{j,s_{\tau}}^*)=0$, we see $\Hom_{\GL_2(K)}(W,\PS_{j,s_{\tau}}^*)=0$ and the injection factors through $V$. The claim hence the $\Ext$-part follow. 

(2) As $\nabla_{\tau}^- \vartheta_{\Sigma_K} \cong (\PS_{1,\sigma}^* \oplus \PS_{2,\sigma}^*)^{\oplus 2}$, the $\Hom$-part follows. As $\vartheta_{\sigma}^+ \vartheta_{\Sigma_K}^- \pi(D)^*\cong \vartheta_{\tau}^- \vartheta_{\Sigma_K}^+ \pi(D)^*$, the $\Ext$-part follows by the same argument as in (1). 
\end{proof}
\begin{corollary}\label{CpiDsigma}
We have $\vartheta_{\tau}^+ \pi(D)^*/\vartheta_{\tau}^- \pi(D)^*\cong \pi(D_{\sigma}, \sigma, \lambda)^{*, \oplus 2}$. 
\end{corollary}
\begin{proof}We have an exact sequence $0 \ra \nabla_{\tau}^- \pi(D)^* \ra \vartheta^+_{\tau} \pi(D)^*/ \vartheta_{\tau}^- \pi(D)^* \ra \nabla_{\tau}^+ \pi(D)^* \ra 0$. Hence $\vartheta_{\tau}^+ \pi(D)^*/\vartheta_{\tau}^- \pi(D)^*$ is a self-extension of $\pi(D_{\sigma}, \sigma, \lambda)^*\cong \pi^{\pm}(D_{\sigma}, \sigma, \lambda)^*$. We also know 
\begin{enumerate}
	\item[(1)] $\cZ_K$ and $Z(K)$ act on $\vartheta_{\tau}^+ \pi(D)^*/\vartheta_{\tau}^- \pi(D)^*$ by a character,
	\item[(2)] $\dim_E \Hom(\vartheta_{\tau}^+ \pi(D)^*, \alg^*)=2$ (by Proposition \ref{PpiDelta} (3), using $\vartheta_{\tau}^+ \pi(D)^*/\vartheta_{\tau}^- \pi(D)^*$ is a subquotient of $\vartheta_{\Sigma_K}^+ \pi(D)^*/\vartheta_{\Sigma_K}^- \pi(D)^*$).	\end{enumerate}
We have a natural map $f: \Ext^1_Z(\pi(D_{\sigma}, \sigma,\lambda)^*, \pi(D_{\sigma}, \sigma,  \lambda)^*)\ra \Ext^1_Z(\alg^*, \pi(D_{\sigma}, \sigma, \lambda)^*)$.  By (2),  $f([\vartheta_{\tau}^+ \pi(D)^*/\vartheta_{\tau}^- \pi(D)^*])=0$. When $D$ is not $\sigma$-critical, by d\'evissage, $f$ is an isomorphism hence $\vartheta_{\tau}^+ \pi(D)^*/\vartheta_{\tau}^- \pi(D)^*]$ splits. If $D$ is $\sigma$-critical for $\ul{\phi}_i$, the kernel is one dimensional and generated by the representation
\begin{equation*}
	(\PS_{j,\sigma}^*)^{\oplus 2} \oplus \big(\big(\Ind_{B^-(K)}^{\GL_2(K)} \jmath(\ul{\phi}_i) \sigma(z)^{s_{\sigma} \cdot \lambda_{\sigma}}(1+\epsilon \val_K) \otimes (1-\epsilon \val_K)\big)^{\sigma-\la} \otimes_E L_{\tau}(\lambda_{\tau})\big)^*.
\end{equation*}
However, again by Proposition \ref{PpiDelta} (3), it can not be a subquotient of $\vartheta_{\Sigma_K}^+ \pi(D)^*/\vartheta_{\Sigma_K}^- \pi(D)^*$. The corollary follows. 
\end{proof} 
The following theorem provides a strong evidence towards Conjecture \ref{conjGL2} (3) (where we omit $\GL_2(K)$ in $\Hom$ and $\Ext$):
\begin{theorem}\label{TLin}
(1) $\dim_E \Hom(\pi^+(\Delta, \sigma, \lambda)^*, \pi^+(\Delta, \emptyset, \lambda)^*)=\dim_E \Ext^1(\pi^-(\Delta, \emptyset, \lambda)^*, \pi^-(\Delta, \sigma, \lambda)^*)=2$

(2) $\dim_E \Hom(\pi^+(\Delta, \Sigma_K, \lambda)^*, \pi^+(\Delta, \tau, \lambda)^*) =\dim_E \Ext^1(\pi^-(\Delta, \tau, \lambda)^*, \pi^-(\Delta, \Sigma_K, \lambda)^*)=4$.

(3) If $D$ is not $\sigma$-critical, then: $$\dim_E \Hom\big( \pi^+(D_{\sigma}, \Sigma_K, \lambda)^*, \pi^+(D_{\sigma}, \sigma, \lambda)^*\big)=\dim_E \Ext^1\big(\pi^-(D_{\sigma}, \sigma, \lambda)^*, \pi^-(D_{\sigma}, \Sigma_K, \lambda)^*\big)=2.$$
If $D$ is $\sigma$-critical, then: $$\dim_E \Hom\big( \pi^+(D_{\sigma}, \Sigma_K, \lambda)^*, \pi^+(D_{\sigma}, \sigma, \lambda)^*\big)=\dim_E \Ext^1\big(\pi^-(D_{\sigma}, \sigma, \lambda)^*, \pi^-(D_{\sigma}, \Sigma_K, \lambda)^*\big)=4.$$ 
\end{theorem}
\begin{proof}
(1) follows by \cite[Thm.~4.1 (3)]{Br}. (2) is a direct consequence of Corollary \ref{Cext1}.

(3) The $\Hom$-part follows from Corollary \ref{CpiDsigma}. The $\Ext$ part follows by the same argument as in Corollary \ref{Cext1} (1), noting $\dim_E \End_{\GL_2(K)}(\pi(D_{\sigma},\sigma, \lambda)=\begin{cases}
	1 & \text{$D$ is not $\sigma$-critical} \\
	2 & \text{$D$ is $\sigma$-critical}
\end{cases}$. Indeed, the argument shows that  $\vartheta_{\tau}^+ \vartheta_{\tau}^- \pi(D)^*\cong \vartheta_{\tau}^+ \vartheta_{\tau}^- \pi(D)^* $ is actually the universal extension of  $\pi(D_{\sigma},\sigma, \lambda)^{*}$ by $\vartheta_{\tau}^- \pi(D)^*$.
\end{proof}
\begin{remark}
Note $\pi(D)^*$ corresponds to a line in the $E$-vector spaces in (3). This gives a realisation of the Hodge parameter of $D$, which, roughly speaking, measures the relative position of the Hodge filtrations for different embeddings. 
\end{remark}
We look at the representation $\vartheta_{\Sigma_K}^- \pi(D)^*$. Let $\PS_i:=(\Ind_{B^-(K)}^{\GL_2(K)}z^{\lambda} \jmath(\ul{\phi}_i))^{\Q_p-\an}$. 
Recall that $\PS_i^*$ has the form (cf. \cite[Thm.~4.1]{Br}):
$\begindc{\commdiag}[200]
\obj(0,0)[a]{$\PS_{i,s_{\sigma}s_{\tau}}^*$}
\obj(3,-1)[b]{$\PS_{i,s_{\sigma}}^*$}
\obj(3,1)[c]{$\PS_{i,s_{\tau}}^*$}
\obj(6,0)[d]{$\alg^*$}
\mor{a}{b}{}[+1,\solidline]
\mor{a}{c}{}[+1,\solidline]
\mor{b}{d}{}[+1,\solidline]
\mor{c}{d}{}[+1,\solidline]
\enddc$. 
By \cite{BH2} or (\ref{Efs1}) (see also \cite{Ding1} for the unitary Shimura curves case), it is easy to see $\vartheta_{\Sigma_K}^- \pi(D)^* \twoheadrightarrow \PS_{1,s_{\sigma}s_{\tau}}^* \oplus \PS_{2,s_{\sigma}s_{\tau}}^*$. Let $\SSS_{s_{\sigma}s_{\tau}}^*$ be the kernel of the map. Using the fact that $\vartheta_{\Sigma_K}^- \pi(D)^*$ doesn't have non-zero $\text{U}(\ug_{\sigma/\tau})$-finite quotients and  results on extensions of locally analytic principal series, one can actually show that  $\SSS_{s_{\sigma}s_{\tau}}^*$ does not have non-zero quotient which is a subquotient of certain locally analytic principal series. 
\begin{corollary}\label{Pextss}Suppose  $D$ is not $\sigma$-critical.
Then for $i=1,2$,  $\dim_E\Ext^1_{\GL_2(K)}(\PS_{i,s_{\sigma}}^*, \SSS_{s_{\sigma}s_{\tau}}^*)=1$.
\end{corollary}
\begin{proof}
By d\'evissage and $\Ext^1_{\GL_2(K)}(\PS^*_{i,s_{\sigma}}, \PS_{j,s_{\sigma}s_{\tau}}^*)=0$ for $i \neq j$, we have 
\begin{equation*}
	0 \lra \Ext^1_{\GL_2(K)}(\PS_{i,s_{\sigma}}^*, \SSS_{s_{\sigma}s_{\tau}}^*) \lra \Ext^1_{\GL_2(K)}(\PS_{i,s_{\sigma}}^*, \vartheta_{\Sigma_K}^- \pi(D)^*) \xlongrightarrow{f} \Ext^1_{\GL_2(K)}(\PS_{i,s_{\sigma}}^*, \PS_{i,s_{\sigma} s_{\tau}}^*).
\end{equation*}
As $D$ is not $\sigma$-critical, by \cite{BH2}, $\pi(D)^*$ has a sub of the form $\vartheta_{\Sigma_K}^- \pi(D)^* \rule[2.5pt]{10pt}{0.5 pt} \PS_{i,s_{\sigma}}^*$, which admits a non-split quotient $\PS_{i,s_{\sigma}s_{\tau}}^* \rule[2.5pt]{10pt}{0.5 pt} \PS_{i,s_{\sigma}}^*$. Hence $f$ is surjective. By Corollary \ref{Cext1} (1) and the fact $\dim_E  \Ext^1_{\GL_2(K)}(\PS_{i,s_{\sigma}}^*, \PS_{i,s_{\sigma} s_{\tau}}^*)=1$, the corollary follows.
\end{proof}
We finally take Conjecture \ref{conjLG} into consideration and make some speculations. Let $\PS_{i,-\theta_K}:=(\Ind_{B^-(K)}^{\GL_2(K)} z^{-\theta_K} \jmath(\ul{\phi}_i))^{\Q_p-\an}$. Conjecture \ref{conjLG} implies $\pi(\Delta)$ has the form  $\PS_{1,-\theta_K}\oplus \PS_{2,-\theta_K} \oplus \SSS_{-\theta_K}$. By \cite[Thm.~4.2.12]{JLS}, $T_{-\theta_K^*}^{\lambda^*} \PS_{i,-\theta_K}^*[\cZ_K=\chi_{\lambda^*}]\cong \PS_i^*$. As $T_{-\theta_K^*}^{\lambda^*} \pi(\Delta)^*[\cZ_K=\chi_{\lambda^*}]=\vartheta_{\Sigma_K}^+ \pi(D)^*=\vartheta_{\Sigma_K}^+ \vartheta_{\Sigma_K}^- \pi(D)^*$,  we see $T_{-\theta_K^*}^{\lambda^*} \SSS_{-\theta_K}^* [\cZ_K=\chi_{\lambda^*}]\cong \vartheta_{\Sigma_K}^+ \SSS_{s_{\sigma}s_{\tau}}^*$ with the following form
\begin{equation*}	\label{diag5}
\begindc{\commdiag}[200]
\obj(0,3)[a]{$\SSS_{s_{\sigma}s_{\tau}}^*$}
\obj(4,0)[b]{$\PS_{2,s_{\tau}}^*$}
\obj(4,2)[c]{$\PS_{1,s_{\sigma}}^*$}
\obj(4,4)[d]{$\PS_{2,s_{\sigma}}^*$}
\obj(4,6)[e]{$\PS_{1,s_{\tau}}^*$}
\obj(8,3)[f]{$\alg^{*, \oplus 2}$}
\mor{a}{b}{}[+1,\solidline]
\mor{a}{c}{}[+1,\solidline]
\mor{a}{d}{}[+1,\solidline]
\mor{a}{e}{}[+1,\solidline]
\mor{f}{b}{}[+1,\solidline]
\mor{f}{c}{}[+1,\solidline]
\mor{f}{d}{}[+1,\solidline]
\mor{f}{e}{}[+1,\solidline]
\enddc.
\end{equation*}
We have $\Ext^1_{\GL_2(K)}(\pi_{\sigma}(\Delta, \lambda)^*, \vartheta_{\Sigma_K}^- \pi(D)^*)\cong \oplus_{i=1,2} \Ext^1_{\GL_2(K)}(\PS_{i,s_{\sigma}}^*, \vartheta_{\Sigma_K}^- \pi(D)^*)$. Conjecture \ref{conjGL2} (3) suggests there should be a natural isomorphism of the $2$-dimensional vector spaces
\begin{equation*}
\Ext^1_{\GL_2(K)}(\PS_{1,s_{\sigma}}^*, \vartheta_{\Sigma_K}^- \pi(D)^*) \xlongrightarrow[\sim]{?} \Ext^1_{\GL_2(K)}(\PS_{2,s_{\sigma}}^*, \vartheta_{\Sigma_K}^-  \pi(D)^*).
\end{equation*}
One may furthermore expect the above isomorphism comes from natural isomorphisms of one dimensional $E$-vector spaces for $i\neq j$:
\begin{equation*}
\Ext^1_{\GL_2(K)}(\PS_{i,s_{\sigma}}^*, \PS_{i,s_{\sigma}s_{\tau}}^*) \xlongrightarrow[\sim]{?} \Ext^1_{\GL_2(K)}(\PS_{j,s_{\sigma}}^*, \SSS_{s_{\sigma}s_{\tau}}^*).
\end{equation*}
Together with (\ref{EDsigma}) and the discussions in \cite[\S~4]{Br}, $\vartheta^-_{\sigma} \pi(D)^*$ should have the following form
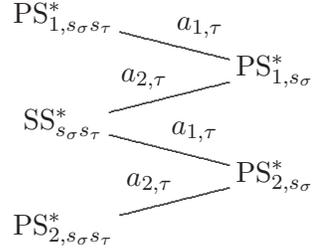
\begin{equation*}	
\begindc{\commdiag}[200]
\obj(0,3)[a]{$\SSS_{s_{\sigma}s_{\tau}}^*$}
\obj(4,2)[c]{$\PS_{2,s_{\sigma}}^*$}
\obj(4,4)[d]{$\PS_{1,s_{\sigma}}^*$}
\obj(0,5)[a2]{$\PS_{1,s_{\sigma}s_{\tau}}^*$}
\obj(0,1)[a1]{$\PS_{2,s_{\sigma}s_{\tau}}^*$}
\mor{a}{c}{$a_{1,\tau}$}[+1,\solidline]
\mor{a}{d}{$a_{2,\tau}$}[+1,\solidline]
\mor{a1}{c}{$a_{2,\tau}$}[+1,\solidline]
\mor{a2}{d}{$a_{1,\tau}$}[+1,\solidline]
\enddc.
\end{equation*}
Finally, assume  the isomorphism in Remark \ref{RemGL22} (1) holds. All these speculations suggest that $\pi(D)^*$ has the following form
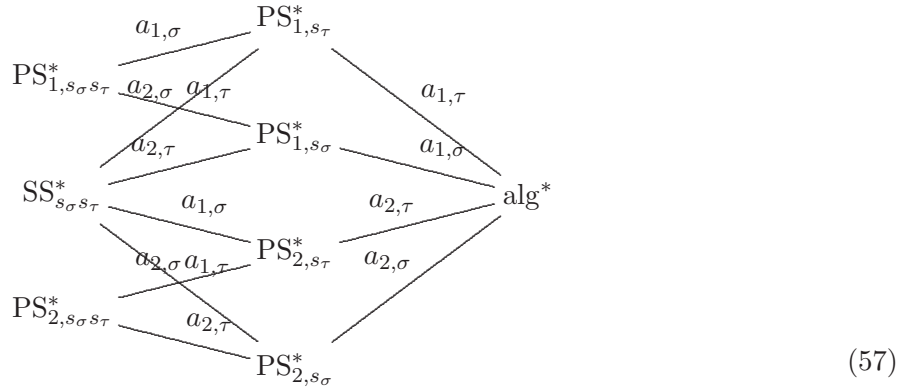
\begin{equation}\label{Dpirho}	
\begindc{\commdiag}[220]
\obj(0,3)[a]{$\SSS_{s_{\sigma}s_{\tau}}^*$}
\obj(4,0)[b]{$\PS_{2,s_{\sigma}}^*$}
\obj(4,2)[c]{$\PS_{2,s_{\tau}}^*$}
\obj(4,4)[d]{$\PS_{1,s_{\sigma}}^*$}
\obj(4,6)[e]{$\PS_{1,s_{\tau}}^*$}
\obj(8,3)[f]{$\alg^{*}$}
\obj(0,5)[a2]{$\PS_{1,s_{\sigma}s_{\tau}}^*$}
\obj(0,1)[a1]{$\PS_{2,s_{\sigma}s_{\tau}}^*$}
\mor{a}{b}{$a_{1,\tau}$}[+1,\solidline]
\mor{a}{c}{$a_{1,\sigma}$}[+1,\solidline]
\mor{a}{d}{$a_{2,\tau}$}[+1,\solidline]
\mor{a}{e}{$a_{2,\sigma}$}[+1,\solidline]
\mor{b}{f}{$a_{2,\sigma}$}[+1,\solidline]
\mor{c}{f}{$a_{2,\tau}$}[+1,\solidline]
\mor{d}{f}{$a_{1,\sigma}$}[+1,\solidline]
\mor{e}{f}{$a_{1,\tau}$}[+1,\solidline]
\mor{a1}{b}{$a_{2,\tau}$}[+1,\solidline]
\mor{a1}{c}{$a_{2,\sigma}$}[+1,\solidline]
\mor{a2}{d}{$a_{1,\tau}$}[+1,\solidline]
\mor{a2}{e}{$a_{1,\sigma}$}[+1,\solidline]
\enddc
\end{equation} 
where each labelled  line is a possibly-split extension, and the same label means that one splits if and only if the other one splits. 
We remark that such a form was  antecedently speculated in a personal note of Breuil (back to 2008!). 

We finally give a quick discussion for some other cases. 

\paragraph{Crystabelline case with general $K$} One may make similar speculations for general finite extension $K$ of $\Q_p$. In fact, the hypercubes $\boxdot^{\pm}(\pi(D)^*)$ appear to have the following  inductive feature with respect to $d_K$.  For $I \subset \Sigma_K$, and a $\text{U}(\ug_{\Sigma_K \setminus I})$-finite $\text{U}(\ug_{\Sigma_K})$-module $M$, we can define a $\#I$-dimensional hypercube, denoted by $\boxdot_I^{\pm} M$, by just using the wall-crossing functors for $\sigma\in I$. Then $\boxdot^{\pm}_I \nabla_{\Sigma_K \setminus I}^- \pi(D)^*$ appears to have a similar symmetric structure as $\boxdot^{\pm} \pi(D')^*$ for a rank two $(\varphi, \Gamma)$-module over $\cR_{K',E}$ with $d_{K'}=\# I$. In the generic crystabelline case, one may expect $\vartheta_I^-\nabla_{\Sigma_K \setminus I}^- \pi(D)^*$ has the form $\PS_{1, s_I}^* \oplus \SSS_{1,s_I}^* \oplus \cdots \oplus \SSS_{\#I-1, s_I}^* \oplus \PS_{2,s_I}^*$.  For example, when $[K:\Q_p]=3$, $\boxdot^- (\pi(D)^*)$ should have the following structure (where we only keep $\nabla_I^- \vartheta_{\Sigma_K \setminus I}^- \pi(D)^*$ and $\pi(D)^*$, so each line means an extension,  where the first number in the subscript counts $\PS$ or $\SSS$, and where  we also enumerate the embeddings, and $s_{abc}$ just  means $s_{\sigma_a} s_{\sigma_b} s_{\sigma_c}$)\footnote{The form was also speculated in a personal note of Breuil around 2008.}
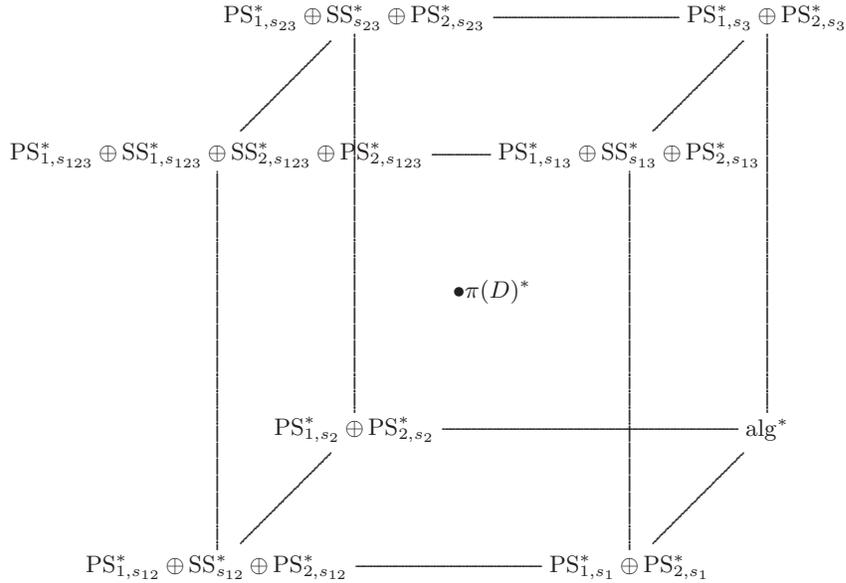
\begin{equation*}\footnotesize
\begindc{\commdiag}[260]
\obj(0,0)[a]{$\PS_{1, s_{12}}^* \oplus \SSS_{s_{12}}^* \oplus \PS_{2,s_{12}}^*$}
\obj(6,0)[c]{$\PS_{1,s_1}^* \oplus \PS_{2,s_1}^*$}
\obj(0,6)[e]{$\PS_{1, s_{123}}^* \oplus \SSS_{1, s_{123}}^* \oplus   \SSS_{2, s_{123}}^*\oplus \PS_{2, s_{123}}^*$}
\obj(6,6)[i]{$\PS_{1, s_{13}}^* \oplus \SSS_{s_{13}}^* \oplus \PS_{2,s_{13}}^*$}
\obj(4,4)[f1]{$\bullet \pi(D)^*$}
\obj(2,2)[a2]{$\PS_{1,s_2}^* \oplus \PS_{2,s_2}^*$}
\obj(8,2)[c2]{$\alg^*$}
\obj(2,8)[e2]{$\PS_{1, s_{23}}^* \oplus \SSS_{s_{23}}^* \oplus \PS_{2, s_{23}}^*$}
\obj(8,8)[i2]{$\PS_{1, s_3}^* \oplus \PS_{2,s_3}^*$}
\mor{a}{e}{}[+1,\solidline]
\mor{a}{c}{}[+1,\solidline]
\mor{c}{i}{}[+1, \solidline]
\mor{e}{i}{}[+1, \solidline]
\mor{a2}{e2}{}[+1, \solidline]
\mor{a2}{c2}{}[+1,\solidline]
\mor{c2}{i2}{}[+1, \solidline]
\mor{e2}{i2}{}[+1, \solidline]
\mor{a}{a2}{}[+1,\solidline]
\mor{c}{c2}{}[+1,\solidline]
\mor{e}{e2}{}[+1,\solidline]
\mor{i}{i2}{}[+1,\solidline]
\enddc.
\end{equation*}
We propose the following conjecture on the multiplicities of Orlik-Strauch representations in $\pi(\Delta, \lambda) =\vartheta_{\Sigma_K}^+ \pi(D)^*$ for general $\GL_2(K)$ in generic crystabelline case.
\begin{conjecture}[Multiplicities]
Suppose $\Delta$ is crystabelline and  generic, and let $\ul{\phi}_1$, $\ul{\phi}_2$ be the two refinements of $\Delta$. Then $\pi(\Delta, \lambda)$ contains a locally $\Q_p$-analytic representation which is a successive extension of $(\pi_{\infty}(\Delta) \otimes_E L(\lambda)\big)^{\oplus 2^{d_K}}$ and 
$$\bigoplus_{\substack{J\subset \Sigma_K\\ |J|=i}} \Big(\cF_{B^-}^{\GL_2}\big(L_J^-(-s_J \cdot \lambda_J) \otimes_E L_{J^c}^-(-\lambda_{J^c}), \jmath(\ul{\phi}_1)\big)^{\oplus 2^{d-i}} \oplus \cF_{B^-}^{\GL_2}\big(L_J^-(-s_J\cdot  \lambda_J) \otimes_E L_{J^c}^-(-\lambda_{J^c}), \jmath(\ul{\phi}_2)\big)^{\oplus 2^{d-i}}\Big)$$
for $i=1, \cdots, d_K$.
\end{conjecture}
\begin{remark}\label{Rk-mult}By the discussion above the conjecture, one may expect that for a generic crystabelline $D$ of regular Sen weights, $\pi(D)^*$ has $\sum_{i=0}^{d_K} (d_K+1-i) \binom{d_K}{i}=(2^{d_K}+ d_K 2^{d_K-1})$-number of (multiplicity free) irreducible constituents, and $\pi(\Delta, \lambda)^*$ has $\sum_{i=0}^{d_K} (d_K+1-i)2^i \binom{d_K}{i}=(3^{d_K}+d_K 3^{d_K-1})$-number of irreducible constituents.\end{remark}
\paragraph{de Rham non-trianguline case}
Assume now $D$ is de Rham non-trianguline (and $\textbf{h}$ is strictly dominant). For simplicity, we still write $\alg=\pi_{\infty}(\Delta) \otimes_E L(\lambda)$, that is the locally algebraic subrepresentation of $\pi(D)$.   The discussion in the crystabelline case in the precedent section suggests the following speculative structure of the representations in $\boxdot^{\pm}(\pi(D)^*)$ for de Rham non-trianguline $D$ for $[K:\Q_p]=2$:
\begin{equation*}	
\pi(\Delta, \lambda)^*: 
\begindc{\commdiag}[150]
\obj(0,0)[a]{$\SSS_{s_{\sigma}s_{\tau}}^*$}
\obj(4,-2)[b]{$\SSS_{s_{\sigma}}^{\oplus 2,*}$}
\obj(4,2)[d]{$\SSS_{s_{\tau}}^{\oplus 2,*}$}
\obj(8,0)[f]{$\alg^{\oplus 4, *}$}
\mor{a}{b}{}[+1,\solidline]
\mor{a}{d}{}[+1,\solidline]
\mor{b}{f}{}[+1,\solidline]
\mor{d}{f}{}[+1,\solidline]
\enddc,
\ \ \pi(D)^*:
\begindc{\commdiag}[150]
\obj(0,0)[a]{$\SSS_{s_{\sigma}s_{\tau}}^*$}
\obj(4,-2)[b]{$\SSS_{s_{\sigma}}^*$}
\obj(4,2)[d]{$\SSS_{s_{\tau}}^{*}$}
\obj(8,0)[f]{$\alg^{*}$}
\mor{a}{b}{}[+1,\solidline]
\mor{a}{d}{}[+1,\solidline]
\mor{b}{f}{}[+1,\solidline]
\mor{d}{f}{}[+1,\solidline]
\enddc.
\end{equation*}
For general $K$, and a de Rham non-trianguline $(\varphi, \Gamma)$-module  $D$ of regular Sen weights, one may expect $\pi(D)^*$ has $2^{d_K}$-number of (multiplicity free) irreducible constituents and $\pi(\Delta, \lambda)^*$ has $3^{d_K}$-number for irreducible constituents. We invite the reader to  compare it with Remark \ref{Rk-mult}.

For $K=\Q_p$, using $\boxdot^{\pm} \pi(D)^*$, we can reprove the following result on extension group (\cite[Thm.~2.5]{Ding12}:
\begin{proposition}\label{PextQp}
$\dim_E \Hom_{\GL_2(\Q_p)}(\pi(\Delta,\lambda)^*, \alg^*)=\dim_E \Ext^1_{\GL_2(\Q_p)}(\alg^*, \pi_0(\Delta, \lambda)^*)=2$.
\end{proposition}
\begin{proof}
As $\alg^*=\nabla^- \pi(D)^*$, the result on smooth dual and \cite[Cor.~3.6]{ST-dual} imply $\nabla^+ \pi(D)^* \cong \alg^*$. As $\Ext^1_Z(\alg^*, \alg^*)=0$, $\pi(\Delta, \lambda)^*=\vartheta^+ \pi(D)^*$ is an extension of $(\alg)^{*,\oplus 2}$ by $\pi_0(\Delta,\lambda)^*=\vartheta^- \pi(D)^*$. The proposition then follows by similar argument in Corollary \ref{Cext1} (showing that $\pi(\Delta, \lambda)^*$ is the universal extension of $\alg^*$ by $\pi_0(\Delta, \lambda)^*$).
\end{proof}

\appendix
\section{Lie calculations for $\gl_2$}\label{sec:lie-calculations-for-gl2}

We collect some facts on $\gl_2$-modules.  Let $u^+:=\begin{pmatrix}
0& 1 \\ 0 & 0
\end{pmatrix}$, $u^-=\begin{pmatrix}
0 & 0 \\ 1 & 0
\end{pmatrix}$, $\fz=\begin{pmatrix}
1 & 0 \\ 0 & 1
\end{pmatrix}$ and $\fh=\begin{pmatrix}
1 & 0 \\ 0 & -1
\end{pmatrix}$. Let $\lambda$ be an integral dominnat weight. Let $M\in \Mod({\text{U}(\gl_2)_{\chi_{\lambda}}})$, and consider $M \xrightarrow{\iota} T_{-\theta}^{ \lambda} T_{\lambda}^{-\theta} M\xrightarrow{\kappa} M$. 
\begin{lemma}\label{injlemm}(1) $Ker(\iota)$  is the submodule of $M$ generated by $\text{U}(\gl_2)$-finite vectors.

(2) $\Coker(\kappa)$ is generated by $\text{U}(\gl_2)$-finite vectors. 
\end{lemma}
\begin{proof}
Let $M_0\subset M$ be the submodule generated by $\text{U}(\gl_2)$-finite vectors. It is easy to see $T_{ \lambda}^{-\theta} M_0=0$. As the map $M_0 \hookrightarrow M \ra T_{-\theta}^{\lambda} T_{w_0\cdot \lambda}^{-\theta} M$ factors through $T_{-\theta}^{} T_{\lambda}^{-\theta} M_0=0$, $M_0 \subset \Fer(\iota)$.

Let $\ul{0}:=(0,0)$. Then $T_{\lambda}^{\ul{0}}$ induces an equivalence of categories $\Mod(\text{U}(\gl_2)_{\chi_{\lambda}}) \xrightarrow{\sim} \Mod({\text{U}(\gl_2)_{\chi_{\ul{0}}}})$ (with inverse $T_{\ul{0}}^{\lambda}$) and sends $L(\lambda)$ to the trivial representation $L(\ul{0})=L$ We also have  $T_{\ul{0}}^{\lambda} T_{-\theta}^{\ul{0}} T_{\ul{0}}^{-\theta} T_{\lambda}^{\ul{0}}\cong T_{-\theta}^{\lambda} T_{\lambda}^{-\theta}$. We reduce to prove the $\gl_2$-action is trivial on $\Fer(\iota)$ and $\Coker(\kappa)$ when $\lambda=\ul{0}$. 

Assume henceforth $\lambda=\ul{0}$, hence $-\theta-w_0\cdot \ul{0}=(0,-1)$ and $L(-\theta-w_0\cdot  \ul{0}) \otimes_E \dett$ is isomorphic to standard $2$-dimensional representation $V_1$ of $\gl_2$. Let $e_0$ be a lowest weight vector of $V_1$ and $e_1=u^+ e_0$. Let $e_i^*\in V_1^{\vee}$ be the dual basis of $e_i$ (so $e_0^*=-u^+ e_1^*$ is a highest weight vector of $V_1^{\vee}$). We have by the calculation in \cite[Lem.~2.17]{Ding14}: $(T_{\ul{0}}^{-\theta} M) \otimes_E \dett \cong (M\otimes_EV_1)[\fc=-1]=\{v_0 \otimes e_0 + v_1 \otimes e_1\ |\ (\fh-2)v_0=2 u^+ v_1, \ (\fh+2)v_1=-2u^- v_0\}$, and 
$T_{-\theta}^{\ul{0}} T_{\ul{0}}^{-\theta} M=T_{\ul{0}}^{-\theta} M \otimes_E V_1^{\vee}$. 
By direct calculations, we have 
\begin{equation*}
	\iota(v)=-(2u^+ v \otimes e_0 +\fh v \otimes e_1) \otimes e_1^* + (\fh v \otimes e_0 -2u^- v \otimes e_1) \otimes e_0^*,
\end{equation*}
\begin{equation*}
	\kappa(-(v_{0} \otimes e_0 + v_1 \otimes e_1) \otimes e_1^*+(v_0'\otimes e_0 +v_1' \otimes e_1) \otimes e_0^*)=v_0'-v_1.
\end{equation*}
It is clear that $\Fer(\iota)=M[\gl_2=0]$.  We see $$\Ima(\kappa)=\bigg\{x+y\ \Big|\ \exists x', y'\in M \text{ such that } \begin{cases} (\fh-2) x =2u^+ x' \\ -2u^- x=(\fh+2) x' \end{cases}\& \begin{cases} (\fh+2) y=-2u^-y'\\ 2u^+ y=(\fh-2)y'\end{cases}\bigg\}.$$ 
In particular, for any $v\in M$, $u^+ v\in  \Ima(\kappa)$ (with $x=u^+ v$, $x'=\frac{1}{2} \fh v$) and $u^- v\in \Ima(\kappa)$ (with $y=u^- v$, $y'=-\frac{1}{2} \fh v$). We easily deduce the $\gl_2$-action on $M/\Ima(\kappa)$ is trivial. 
\end{proof}
Let $M^{\sharp}:=(T_{-\theta}^{\lambda} T_{\lambda}^{-\theta} M)[\cZ=\chi_{\lambda}]$, and $M^{\flat}:=\Ima(\kappa)$. It is clear that the map $\iota$ factors through $\iota: M \ra M^{\sharp}$. 
\begin{lemma}\label{Alem2}
The cokernal $M^{\sharp}/\Ima(\iota)$ is generated by $\text{U}(\gl_2)$-finite vectors. 
\end{lemma}
\begin{proof}Similarly as in the proof of the above lemma, we reduce to the case $\lambda=\ul{0}$. We have 
\begin{eqnarray}\label{EMsharp}
	M^{\sharp}&=&\bigg\{v= -(v_{0} \otimes e_0 + v_1 \otimes e_1) \otimes e_1^*+(v_0'\otimes e_0 +v_1' \otimes e_1) \otimes e_0^*  \\
	&&\in T_{\ul{0}}^{-\theta} M \otimes_E V_1^{\vee} \ \Big|\begin{cases}
		(\fh-2)v_0=2u^+v_0' \\ \fh v_1=2v_0'+2u^+ v_1'	\end{cases}\& \begin{cases}
		\fh v_0'=-2(v_1+u^- v_0) \\ (\fh+2) v_1'=-2u^-v_1
	\end{cases}
	\bigg\}. \nonumber 
\end{eqnarray}
Recall the condition $-(v_0 \otimes e_0 + v_1 \otimes e_1) \otimes e_1^*+(v_0'\otimes e_0 +v_1' \otimes e_1) \otimes e_0^*\in T_{\ul{0}}^{-\theta} \otimes_E V_1^{\vee}$ also implies $(\fh-2)v_0=2u^+ v_1$, $(\fh+2)v_1=-2u^-v_0$. For $v\in M^{\sharp}$ as in (\ref{EMsharp}), we calculate
\begin{multline*}
	u^+ v=-(u^+v_0 \otimes e_0 +(v_0+u^+ v_1) \otimes e_1) \otimes e_1^* + ((v_0+u^+v_0') \otimes e_0 +(v_1 +v_0'+u^+v_1')\otimes e_1) \otimes e_0^* \\
	=\frac{1}{2}\big(-(2u^+ v_0 \otimes e_0 +\fh v_0 \otimes e_1) \otimes e_1^*+(\fh v_0 \otimes e_0-2u^- v_0 \otimes e_1) \otimes e_0^*\big)\in \Ima(\iota),
\end{multline*}
\begin{multline*}
	\fh v=-((\fh-2)v_0 \otimes e_0 +\fh v_1 \otimes e_1) \otimes e_1^* + (\fh v_0' \otimes e_0 +(\fh+2) v_1'\otimes e_1) \otimes e_0^* \\
	=-(2u^+ v_1\otimes e_0 +\fh v_1 \otimes e_1) \otimes e_1^*+(\fh v_1 \otimes e_0-2u^- v_1 \otimes e_1) \otimes e_0^*\in \Ima(\iota).
\end{multline*}
We deduce easily that the $\gl_2$-action on $M^{\sharp}/\Ima(\iota)$ is trivial.  
\end{proof}
\begin{lemma}\label{traexseq}
Suppose $M$ does not have non-zero $\text{U}(\gl_2)$-finite vectors, then  the Casimir operator $\fc$ on $T_{-\theta}^{\lambda} T_{\lambda}^{-\theta} M$ induces an exact sequence
\begin{equation*}
	0 \lra M^{\sharp} \lra  T_{-\theta}^{\lambda} T_{\lambda}^{-\theta} M \lra M^{\flat} \lra 0.
\end{equation*}
\end{lemma}
\begin{proof}
We reduce to the case $\lambda=\ul{0}$. It suffices to show $\Fer(\kappa)=M^{\sharp}$. Let $v= -(v_{0} \otimes e_0 + v_1 \otimes e_1) \otimes e_1^*+(v_0'\otimes e_0 +v_1' \otimes e_1) \otimes e_0^*\in M^{\sharp}$, it suffices to show $v_0'=v_1$. However, by the equations in  (\ref{EMsharp}), we have $u^+(v_0'-v_1)=u^-(v_0'-v_1)=\fh(v_0'-v_1)=0$ hence $v_0'-v_1=0$ by the assumption on $M$. 
\end{proof}
\begin{lemma}\label{lmnoalg}
The module $T_{-\theta}^{\lambda} T_{\lambda}^{-\theta} M$ does not have non-zero $\text{U}(\gl_2)$-finite vectors.
\end{lemma}
\begin{proof}
Similarly as in the proof of Lemma \ref{injlemm}, we reduce to the case where $\lambda=\ul{0}$. Let $v=-(v_0 \otimes e_0+v_1 \otimes e_1) \otimes e_1^* +(v_0'\otimes e_0+v_1' \otimes e_1) \otimes e_0^*\in T_{-\theta}^{\ul{0}} T_{\ul{0}}^{-\theta} M= (M\otimes_E V_1)[\cZ=\chi_{-\theta}]\otimes_EV_1^{\vee}$. We have
\begin{eqnarray*}
	u^+ v&=&-u^+v_0 \otimes e_0 \otimes e_1^*-(v_0+u^+v_1) \otimes e_1 \otimes e_1^*+(v_0+u^+ v_0')\otimes e_0 \otimes e_0^* +(v_1+v_0'+u^+ v_1') \otimes e_1 \otimes e_0^* \\
	u^-v&=&-(u^-v_0+v_1+v_0') \otimes e_0 \otimes e_1^*-(u^-v_1+v_1') \otimes e_1 \otimes e_1^*+(u^- v_0'+v_1') \otimes e_0 \otimes e_0^* + u^- v_1' \otimes e_1 \otimes e_0^*\\
	\fh v&=& -(\fh-2)v_0 \otimes e_0 \otimes e_1^*-\fh v_1 \otimes e_1 \otimes e_1^* +\fh v_0' \otimes e_0 \otimes e_0^* +(\fh+2) v_1' \otimes e_1 \otimes e_0^*. 
\end{eqnarray*}
As $\fc$ annihilates $M$,  $u^+v=\fh v=0$ (resp. $u^-v=\fh v=0$) implies $u^+v_0=0$, $\fh v_0=2 v_0$ (resp. $u^-v_1'=0$, $\fh v_1'=-2v_1'$) hence $v_0=0$ (resp. $v_1'=0$).  Then we deduce $u^+v_1=\fh v_1=u^-v_1=0$ (resp. $u^+ v_0'=u^-v_0'=\fh v_0'=0$). The lemma follows.  
\end{proof}

\begin{lemma}\label{lemsharp}
(1) The map $\iota: M^{\sharp} \ra (M^{\sharp})^{\sharp}$ is an isomorphism. 

(2) $\kappa: (M^{\flat})^{\flat} \ra M^{\flat}$ is an isomorphism.

(3) If $M$ does not have non-zero $\text{U}(\gl_2)$-finite vectors, $M \ra M^{\sharp}$ induces an isomorphism $M^{\flat} \xrightarrow{\sim} (M^{\sharp})^{\flat}$.

(4) $M^{\flat} \hookrightarrow M$ induces an isomorphism $(M^{\flat})^{\sharp} \xrightarrow{\sim} M^{\sharp}$.
\end{lemma}
\begin{proof}
(1) By Lemma \ref{injlemm}(1) and Lemma \ref{Alem2}, and the fact that $T_{\lambda}^{-\theta} M'=0$ if $M'$ is $\text{U}(\ug)$-finite, the map  $M \ra M^{\sharp}$ induces an isomorphism $T_{-\theta}^{\lambda} T_{\lambda}^{-\theta}  M \xrightarrow{\sim} T_{-\theta}^{\lambda} T_{\lambda}^{-\theta} M^{\sharp}$. Taking $\chi_{\lambda}$-eigenspace, we deduce $M^{\sharp} \xrightarrow{\sim} (M^{\sharp})^{\sharp}$ and the following diagram commutes
\begin{equation*}
	\begin{CD}
		M @> \iota >> M^{\sharp} \\
		@V \iota VV @V \iota VV	\\
		M^{\sharp} @> \sim >> (M^{\sharp})^{\sharp}\end{CD}.
\end{equation*}
It suffices to show the bottom map, denoted by $\iota'$, coincides with $\iota$. However, as their restrictions to $M$ are equal, and $M^{\sharp}/M$ is generated by $\text{U}(\gl_2)$-finite vectors, we see $\Ima(\iota'-\iota)$ is generated by $\text{U}(\gl_2)$-finite vectors. By Lemma \ref{lmnoalg}, $\iota'=\iota$. 

(2) By Lemma \ref{injlemm} (2), the map $M^{\flat} \hookrightarrow M$ induces  a commutative diagram
\begin{equation*}
	\begin{CD}T_{-\theta}^{\lambda} T_{\lambda}^{-\theta} M^{\flat} @> \sim >> T_{-\theta}^{\lambda} T_{\lambda}^{-\theta} M \\
		@V \kappa VV @V \kappa VV \\
		M^{\flat} @>>> M
	\end{CD}.
\end{equation*} 
We deduce the left $\kappa$ is surjective. (2) follows.

(3) Similarly, by Lemma \ref{injlemm} (1), the map $M \ra M^{\sharp}$ induces a commutative diagram
\begin{equation*}
	\begin{CD}
		T_{-\theta}^{\lambda} T_{\lambda}^{-\theta} M @> \sim >> T_{-\theta}^{\lambda} T_{\lambda}^{-\theta} M^{\sharp} \\ 
		@V \kappa VV @V \kappa VV \\
		M @> \iota >> M^{\sharp}
	\end{CD}.
\end{equation*}
By assumption, $\iota$ is injective. (3) follows.

(4) follows  from Lemma \ref{injlemm} (2).
\end{proof}
\section{Schneider-Teitelbaum dual for locally $J$-analytic representations}
Let $G$ be a locally $K$-analytic group of dimension $r$. We collect some results on the dimensions and Schneider-Teitelbuam duals for locally $J$-analytic representations of $G$ for the lack of references.

Note we can view $G$ as a $\Q_p$-analytic group of dimension $d_K r$. Let $\ug$ be the Lie algebra of $G$ (over $K$). For $\sigma\in \Sigma_K$, denote by $\ug_{\sigma} :=\ug \otimes_{K,\sigma} E$. In general, for $J\subset \Sigma_K$, denote by $\ug_J:=\prod_{\sigma\in J} \ug_{\sigma}$ and $\ug^J:=\prod_{\sigma\notin J} \ug_{\sigma}$. Denote by $I_J$ the kernel of $\text{U}(\ug) \ra \text{U}(\ug_J)$, which is also generated by the kernel of $\text{U}(\ug^J) \ra E$ via the natural injection $\text{U}(\ug^J) \hookrightarrow \text{U}(\ug)$). Let $\cC^{J-\la}(G,E)$ be the subspace of $\cC^{\Q_p-\an}(G,E)$ of locally $J$-analytic functions, those that are annihilated by $I_J$.  Let  $\cD_J(G,E):=\cC^{J-\la}(G,E)^*$ be the locally $J$-analytic distribution algebra. We have $\text{U}(\ug_J) \hookrightarrow \cD_J(G,E)$ and  $\cD_J(G,E)\cong \cD(G,E)/I_J \cD(G,E)$. Indeed, as $\cD_*(G,E) \cong \oplus_{g\in H\backslash G} \cD_*(H,E) \delta_g$ (the isomorphism being $\ug_{\Sigma_K}$-equivariant) for a compact open subgroup $H$ of $G$, we only need to show $\cD_J(H,E)\cong \cD(H,E)/I_J$. But it follows from \cite[Prop. 2.18]{Sch10} and the discuss following it. Let $H$ be a compact open subgroup of $G$,  recall that $\cD(H,E)$ is equipped with a family of multiplicative norms $\{q_s\}_{\frac{1}{p}<s<1}$ (cf. \cite[\S~4]{ST03}). Denote by $\cD_s(H,E)$ (resp. $\text{U}_s(\ug_K)$) the completion of $\cD(H,E)$ (resp. $\text{U}(\ug_K)$) with respect to $q_s$. 


Following \cite[\S~2]{ST-dual}, put $\cD_{J,c}(G,E):=\cC_c^{J-\la}(G,E)^*$ hence $\cD_{J,c}(G,E)\cong \prod_{g\in G/H} \cD_J(gH,E)$ for a compact open subgroup $H$ of $G$. We assume $G/H$ is countable. By the discussion in \cite[\S~2]{ST-dual}, $\cC_{c}^{J-\la}(G,E)$ is equipped with a natural topology of space of compact type. Moreover, the right and left translation of $G$ on $\cC_c^{J-\la}(G,E)$ induces separately continuous $\cD_{J}(G,E)$-module structures on $\cD_{J,c}(G,E)$. 
\begin{proposition}\label{Pliecoho}
We have  
$$\text{H}_q(\ug^J, \cD(G,E))=\begin{cases}
	\cD_J(G,E) & q=0 \\
	0 & q>0
\end{cases}, \text{ and } \text{H}_q(\ug^J, \cD_c(G,E))=\begin{cases} \cD_{J,c}(G,E) & q=0 \\ 0 & q>0
\end{cases}.$$
\end{proposition}
\begin{proof}
The proof is due to Zhixiang Wu. Let $H$ be a compact open subgroup of $G$. It suffices to prove a similar statement for $\cD(H,E)$. We consider the Chevalley-Eilenberg resolution $\wedge^{\bullet} \ug^J  \otimes_E \text{U}(\ug^J) \ra E$ of the trivial $\text{U}(\ug^J)$-module $E$, which induces $\wedge^{\bullet} \ug^J \otimes_E \text{U}(\ug_K) \ra \text{U}(\ug_J)$. For $s\in p^{\Q}$, $\frac{1}{p}<s<1$, taking completion with respect to $q_s$, the complex $\wedge^{\bullet} \ug^J \otimes_E \text{U}_s(\ug_K)$ is also exact except at the degree $0$ with $\text{H}_0=\text{U}_s(\ug_K)/I_J \text{U}_s(\ug_K)$. Applying $-\otimes_{\text{U}_s(\ug_K)} \cD_s(H,E)$ and using \cite[Thm.~1.4.2]{Koh2}, the complex $\wedge^{\bullet} \ug^J \otimes_E \cD_s(H,E)$ is still exact except at the degree $0$ with $\text{H}_0=\cD_s(H,E)/I_J \cD_s(H,E)$. By taking inverse limit over $s$ and using the topological Mittag-Leffler property, we finally obtain a complex $\wedge^{\bullet} \ug^J \otimes_E \cD(H,E)$ which is exact except at the degree $0$ with $\text{H}_0=\cD(H,E)/I_J \cD(H,E)\cong \cD_J(H,E)$. It is clear the complex calculates the $\ug^J$-homology of $\cD(H,E)$, hence $\text{H}_q(\ug^J, \cD(H,E))=\begin{cases} \cD_J(H,E)& q=0 \\ 0 & q>0 \end{cases}$. The proposition follows.
\end{proof}
By the proposition and the same argument as in the proof of  \cite[Cor.~3.6]{ST-dual}, for $M\in \cM_{G,J}$, we have an isomorphism in $D^+(\cM_G)$:
\begin{equation}\label{Edual1}
\RHom_{\cD(G,E)}(M, \cD_c(G,E))\cong \RHom_{\cD_J(G,E)}\big(M, \RHom_{\cD(G,E)}\big(\cD_J(G,E), \cD_c(G,E)\big)\big).
\end{equation}
On the other hand, by the same argument as for \cite[Prop.~3.5]{ST-dual} (using the resolution $\wedge^{\bullet} \ug^J \otimes_E \cD(H,E)\ra \cD_J(H,E)$), we have:
\begin{proposition}
$\RHom_{\cD(G,E)}\big(\cD_J(G,E), \cD_{c}(G)\otimes_E \fd_{G,J}\big)$ is naturally quasi-isomorphic to $\cD_{J,c}(G) \otimes_E |\Delta_{G,J}|_K^{-1}$ concentrated in degree $r(d_K-\#J)$, where $\Delta_{G,J}:=\wedge^{r(d_K-\#J)} \ug^J$ is equipped with a natural $\cD(G,E)$-action extending the adjoint action of $G$, and $\fd_{G,J}:=\Delta_G \otimes_E |\Delta_{G,J}|_K^{-1}$.
\end{proposition}
By the same argument in the proof of \cite[Cor.~3.6]{ST-dual}, the proposition together with (\ref{Edual1}) imply: 
\begin{corollary}
The following diagram commutes
\begin{equation*}\small
	\begin{CD}
		D^b(\cM_{G,J})  @> \text{can} >> D^b(\cM_G) \\
		@V \RHom_{\cD_J(G,E)}(\cdot, \cD_{J,c}(G,E) \otimes_E \delta_G^{-1}[-r(d_K-\#J)]) VV @VV \RHom_{\cD(G,E)}(\cdot, \cD_c(G,E) \otimes_E \fd_{G,J}) V\\
		D^b(\cM_{G,J}) @> \text{can} >> D^+(\cM_G).
	\end{CD}
\end{equation*}
\end{corollary}

\begin{corollary}\label{CJandual}
For $M\in \cM_{G,J}$, we have $\Ext^i_{\cD_J(H,E)}(M, \cD_J(H,E)) \cong \Ext^{r(d_K-\#J)+i}_{\cD(H,E)}(M, \cD(H,E))$, in particular, $\Ext^j_{\cD(H,E)}(M, \cD(H,E))\in \cM_{G,J}$. 
\end{corollary}

\section{Surplus locally algebraic constituents}\label{AppC}
We discuss the phenomenon of ``surplus" locally algebraic constituents, those that appear in $\pi(D)/\pi(D)^{\lalg}$. We show that for $\GL_2$, there don't exist such constituents right after a certain locally $\Q_p$-analytic subrepresentation $\pi_1(D)$ of $\pi(D)$  (assuming $M_{\infty}$ is flat over $R_{\infty}$), while for general $\GL_n$, there are quite many with a lower bound $(2^n-\frac{n(n+1)}{2}-1)d_K$. 

First let $D$ be a rank $2$ generic crystabelline \'etale $(\varphi, \Gamma)$-module over $\cR_{K,E}$. We use the notation ``$\alg$", ``$\pi(D_{\sigma}, \sigma ,\lambda)$" of \S~\ref{secpirho}. Let $\pi_1(D):=\bigoplus^{\sigma\in \Sigma_K}_{\alg} \pi(D_{\sigma}, \sigma, \lambda)$.
\begin{theorem} \label{Tnosurplus}Suppose $M_{\infty}$ is flat over $R_{\infty}$. Let $V$ be an extension of $\alg$ by $\pi_1(D)$. Then $V$ is not a subrepresentation of $\pi(D)$.
\end{theorem}
Note that the representation $V$, however, can be a subrepresentation of a certain self-extension of $\pi(D)$. By \cite{BHHMS1} (see also Corollary \ref{cor::3/6}), $M_{\infty}$ is flat over $R_{\infty}$ when $K$ is unramified over $\Q_p$ under mild hypothesis.
\begin{proof} 
Recall $\pi(D)=\Pi_{\infty}^{R_{\infty}-\an}[\fm]$ where $\fm$ is a maximal ideal of $R_{\infty}[1/p]$ associated to $\rho$. We assume $\pi_1(D) \hookrightarrow \pi(D)$ (otherwise there is noting to prove). We write $x=(x_{\wp} \times x^{\wp}) \in (\Spf R_{\infty})^{\rig}=(\Spf R_{\overline{\rho}}^{\square})^{\rig} \times (\Spf R_{\infty}^{\wp})^{\rig}$ (where $R_{\infty}^{\wp}$ is the prime to $\wp$ part of the patched deformation ring,  $\wp$ is the chosen $p$-adic place $\fp$ of \cite{CEGGPS1}, and $\overline{\rho}$ is a mod $p$ reduction of $\rho$), and $\fm_{x_{\wp}}$,  $\fm_{x^{\wp}}$ to be the respectively associated maximal ideal of $R_{\overline{\rho}}^{\square}[1/p]$ and $R_{\infty}^{\wp}[1/p]$. Let $\fa_x:=\fm_{x_{\wp}}^2+\fm_{x^{\wp}}$. We consider $\Pi_{\infty}[\fa_x]$. The proof is a bit long, and we give a summarization. We first construct in Step 1 the universal extension $\mathscr{E}(\alg, \pi_1(D))$ of $\alg$ by $\pi_1(D)$. Furthermore, using local-global compatibility for triangulation deformations, we show it can be injected into $\Pi_{\infty}[\fa_x]$. If $V\hookrightarrow \pi(D)$, the aforementioned injection will ``split" $V$ and lead to an extra multiplicity of $\alg$ in $\Pi_{\infty}[\fa_x]$, contradiction.

\noindent \textbf{Step 1}: We work out $\Ext^1_{\GL_2(K)}(\alg, \pi_1(D))$, and gives an explicit construction  of the universal extension of $\alg$ by $\pi_1(D)$. First it is easy to see (e.g. see \cite[Lem.~3.16]{BD1})
$$\dim_E \Ext^1_{\GL_2(K)}(\alg, \pi_1(D))=\dim_E \Ext^1_{\GL_2(K),Z}(\alg, \pi_1(D))+d_K+1.$$ Moreover we have an exact sequence by d\'evissage (see \S~\ref{secpirho} for $\PS_{i,s_{\sigma}}$)
\begin{equation*}
	0 \ra \Ext^1_{Z}(\alg,\alg) \ra \Ext^1_{Z}(\alg, \pi_1(D)) \ra \oplus_{i=1,2, \sigma\in \Sigma_K}\Ext^1_{Z}(\alg, \PS_{i,s_{\sigma}}) \ra \Ext^2_{Z}(\alg,\alg).
\end{equation*}
Using \cite[Prop.~4.7]{Sch11}, it is not difficult  to see $\Ext^2_{\GL_2(K),Z}(\alg,\alg)=0$ and $\dim_E \Ext^1_{\GL_2(K),Z}(\alg,\alg)=1$. By \cite[Cor.~4.9]{Sch11} and similar arguments as in \cite[Lem.~2.28]{Ding7}, $\dim_E \Ext^1_{\GL_2(K)}(\alg, \PS_{i,s_{\sigma}})=1$. Hence $\dim_E\Ext^1_G(\alg, \pi_1(D))=2+3d_K$. We construct a basis of it as follows. For $\sigma\in \Sigma_K$,  let $\Hom_{\sigma}(K^{\times},E)\subset \Hom(K^{\times}, E)$ be the $2$-dimensional subspace of locally $\sigma$-analytic characters (cf. \cite[\S~1.3.1]{Ding4}).  Let $\PS_{i,\sigma}:=(\Ind_{B^-(K)}^{\GL_2(K)} \jmath(\ul{\phi}_i) \sigma(z)^{\lambda_{\sigma}})^{\sigma-\an} \otimes_E L_{\Sigma_K\setminus \{\sigma\}}(\lambda_{\Sigma_K \setminus \{\sigma\}})$ (which is the unique non-split extension of $\PS_{1,s_{\sigma}}$ by $\alg$). We similarly define $\PS_{2,\sigma}$. Similarly as in \cite[Prop.~2.30]{Ding7}, we have
\begin{equation}\label{Eindsigma}\Hom_{\sigma}(K^{\times}, E) \xlongrightarrow{\sim}\Ext^1_{\GL_2(K),Z}(\alg, \PS_{1,\sigma})
\end{equation}
where the map is given as follows: for $a_{\sigma}\in \Hom_{\sigma}(K^{\times}, E)$, consider the representation over $E[\epsilon]/\epsilon^2$:  $\Big(\Ind_{B^-(K)}^{\GL_2(K)} \jmath(\ul{\phi}_i) \sigma(z)^{\lambda_{\sigma} }\otimes_E \big((1+ \epsilon a_{\sigma}) \otimes (1-\epsilon a_{\sigma})\big)\Big)^{\sigma-\an} \otimes_E L_{\Sigma_K\setminus \{\sigma\}}(\lambda_{\Sigma_K \setminus \{\sigma\}})$ which is a self-extension of $\PS_{1,\sigma}$.  Then (\ref{Eindsigma}) sends  $a_{\sigma}$ to the corresponding subquotient in the self-extension, denoted by $\pi_{1,a_{\sigma}}$. We define $\pi_{2,a_{\sigma}}$ in a similar way.  We fix $\psi_{\sigma}\in \Hom_{\sigma}(K^{\times},E)$  such that $\psi_{\sigma}$ and the smooth character $\val_K$ form a basis of $\Hom_{\sigma}(K^{\times},E)$. If $\PS_{i,\sigma}\hookrightarrow \pi(D_{\sigma}, \sigma, \lambda) \hookrightarrow \pi_1(D)$, let $\pi_1(D,i,\sigma)$ be the push-forward of $\pi_{i,\psi_{\sigma}}$ along the injection (depending on the choice of $\psi_{\sigma}$ of course); if $\PS_{i,s_{\sigma}}\hookrightarrow\pi(D_{\sigma}, \sigma, \lambda) \hookrightarrow \pi_1(D)$ (i.e. $a_{i,\sigma}=0$),  we use  $\pi_1(D,i,\sigma)$ to denote the push-forward of $\cF_{B^-}^{\GL_2}\big(M_{\sigma}^-(-\lambda_{\sigma})^{\vee} \otimes_E L_{\Sigma_K\setminus \{\sigma\}}(-\lambda_{\Sigma_K \setminus \{\sigma\}}), \jmath(\ul{\phi}_i)\big)$ (being the unique non-split extension  of $\alg$ by $\PS_{i,\sigma}$) via the injection.  We let $\pi_1(D,0, \psi_{\sigma})$ denote the extension of $\alg$ by $\pi_1(D)$ appearing in $\pi_1(D) \otimes_E (1+\epsilon \psi_{\sigma}) \circ \dett$. Let $\pi_1(D,\infty_1)$ (resp. $\pi_1(D,\infty_2)$) be the push-forward of $\widetilde{\alg}_1:=\Big(\Ind_{B^-(K)}^{\GL_2(K)} \jmath(\ul{\phi}_i) \otimes_E \big(1+\epsilon \val_E)\otimes (1-\epsilon \val_E)\big)\Big)^{\infty} \otimes_E L(\lambda)$ (resp. $\widetilde{\alg}_2:=\alg \otimes_E (1+\epsilon \val_E) \circ \dett$) via $\alg \hookrightarrow \pi_1(D)$. Then 
\begin{equation*}
	\mathscr{E}(\alg, \pi_1(D)):=	\bigoplus^{i=1,2}_{\pi_1(D)} \pi_1(D,\infty_i) \oplus_{\pi_1(D)} \bigoplus^{\substack{i=0,1,2 \\ \sigma\in \Sigma_K}}_{\pi_1(D)} \pi_1(D,i,\sigma)
\end{equation*}
is the universal extension of $\alg$ by $\pi_1(D)$. 

\noindent \textbf{Step  2:} We recall the description of the locally algebraic vectors $\Pi_{\infty}[\fa_x]^{\lalg}$ by \cite[Thm.~4.19, \S~4.28]{CEGGPS1}. Let $\xi$ be the inertial type $D$, and $R_{\overline{\rho}}^{\square, \textbf{\ul{h}}-\pcr}(\xi)$ be the corresponding potentially crystalline deformation ring. Recall by \cite[\S~4.28]{CEGGPS1}, the action of $R_{\overline{\rho}}^{\square}$ on $\Pi_{\infty}[\fa_x]^{\lalg}$ ($=(\Pi_{\infty}[\fa_x] \otimes_E L(\lambda)^{\vee})^{\sm} \otimes_E L(\lambda)$, which is non-zero as $\pi_1(D)\hookrightarrow \pi(D)$) factors through $R_{\overline{\rho}}^{\square,\textbf{\ul{h}}-\pcr}(\xi)$. Let $\psi_0$ be the smooth character of $\GL_2(\co_K)$ associated to $\xi$. As $(\Spf R_{\overline{\rho}}^{\square,\textbf{\ul{h}}-\pcr}(\xi))^{\rig}$ is smooth at the point $x_{\wp}$ (and $(\Spf R_{\infty}^{\wp})^{\rig}$ is smooth at $x^{\wp}$), by \cite[Lem.~4.18]{CEGGPS1}, the dual of $\Hom_{\GL_2(\co_K)}(L(\lambda) \otimes_E \psi_0, \Pi_{\infty}[\fa_x])=\Hom_{\GL_2(\co_K)}(\psi_0, \Pi_{\infty}\otimes_E L(\lambda)^{\vee}[\fa_x])$ is a free module of rank $1$ over $R_{\overline{\rho}}^{\square,\textbf{\ul{h}}-\pcr}(\xi)[1/p]/\fm_{x_{\wp}}^2$ (hence has dimension $4+d_K$ over $E$). Let $\cH:=\End(\cind_{\GL_2(\co_K)}^{\GL_2(K)} \psi_0)$, acting naturally on $\Hom_{\GL_2(\co_K)}(\psi_0,\Pi_{\infty}\otimes_E L(\lambda)^{\vee})$. There is a natural map $\cH \ra R_{\overline{\rho}}^{\square,\textbf{\ul{h}}-\pcr}(\xi)$, which induces a surjective map on tangent spaces at $x_{\wp}$ (cf. \cite[\S~4.1]{CEGGPS1} \cite[Prop.~4.2.9]{BD3}). By \cite[Thm.~4.19]{CEGGPS1}, the $\cH$-action on $\Hom_{\GL_2(\co_K)}(L(\lambda) \otimes_E \psi_0, \Pi_{\infty}[\fa_x])$ factors through $R_{\overline{\rho}}^{\square,\textbf{\ul{h}}-\pcr}(\xi)$. We deduce hence an isomorphism (the first two factors coming from the tangent space of $\cH$) 
\begin{equation}\label{Einjalg}
	(\widetilde{\alg}_1 \oplus_{\alg} \widetilde{\alg}_2) \oplus \alg^{\oplus 2+d_K} \xlongrightarrow{\sim} \Pi_{\infty}[\fa_x]^{\lalg}.
\end{equation}

\noindent \textbf{Step 3:} Using local-global compatibility for trianguline deformations of $D$, we construct an injection
\begin{equation}\label{EinjUniv}
	\mathscr{E}(\alg, \pi_1(D)) \oplus \alg^{\oplus 2+d_K} \hooklongrightarrow \Pi_{\infty}[\fa_x],
\end{equation} 
such that the intersection of its image and $\Pi_{\infty}[\fm]$ is equal to $\pi_1(D)$.

\noindent \textbf{Step 3(a):} For a rigid space $X$ and $y\in X$, we use $T_y X$ to denote the tangent space of $X$ at $y$. We construct a certain basis of $T_{x_{\wp}} (\Spf R_{\overline{\rho}}^{\square})^{\rig}$. Recall  $X_{\tri}(\overline{\rho})$ is the (framed) trianguline variety. For each refinement $\ul{\phi}_i$, one can associate a point $(\rho, \ul{\phi}_i z^{\textbf{h}})\in X_{\tri}(\overline{\rho}) \hookrightarrow (\Spf R_{\overline{\rho}}^{\square})^{\rig} \times \widehat{T}_K$ (where $\widehat{T}_K$ is the rigid space parametrizing locally $\Q_p$-analytic characters of $T(K)$). Recall the map $T_{z_i} X_{\tri}(\overline{\rho}) \ra T_{x_{\wp}} (\Spf R_{\overline{\rho}}^{\square})^{\rig}$ is injective (\cite[Lem.~4.16]{BHS2}). By \cite[Prop.~3.7 (i)]{HerSc}, $T_{z_1} X_{\tri}(\overline{\rho})  + T_{z_2} X_{\tri}(\overline{\rho})  =T_{x_{\wp}} (\Spf R_{\overline{\rho}}^{\square})^{\rig}$. Recall $\dim_E T_{z_i} X_{\tri}(\overline{\rho})=4+3d_K$, and $\dim_E T_{x_{\wp}} (\Spf R_{\overline{\rho}}^{\square})^{\rig}=4+4d_K$. 
There is an exact sequence
\begin{equation*}
	0 \lra W(\rho) \lra T_{x_{\wp}}(\Spf R_{\overline{\rho}}^{\square})^{\rig} \xlongrightarrow{f} \Ext^1_{(\varphi, \Gamma)}(D,D) \lra 0
\end{equation*}
where $W(\rho)$ is the subspace ``$K(r)$" in \cite[Lem.~4.13]{BHS2}, of dimension $4-\dim_E \End_{(\varphi,\Gamma)}(D)$. By the proof of \cite[Cor.~4.17]{BHS2}, $W(\rho)\subset T_{z_i} X_{\tri}(\overline{\rho})$ for both $z_i$. Let $\Ext^1_{g}(D,D)$ be the subspace of de Rham (hence crystalline) deformations. Then $ f^{-1}(\Ext^1_g(D,D))$ has dimension $4+d_K$ and is identified with the tangent space of $(\Spf R_{\overline{\rho}}^{\square,\textbf{\ul{h}}-\pcr}(\xi))^{\rig}$ at $\rho$. By \cite[Lem.~3.5]{HerSc},  $f^{-1}(\Ext^1_g(D,D))\subset  T_{z_i} X_{\tri}(\overline{\rho})$ for both $i$. For $\sigma\in \Sigma_K$,  let $a_{\sigma}:=\rho \otimes_E E[\epsilon]/\epsilon^2(1+\epsilon \psi_{\sigma})$ that is an element in $T_{x_{\wp}} (\Spf R_{\overline{\rho}}^{\square})^{\rig}$. It is easy to see $a_{\sigma} \in \cap_{i=1,2} T_{z_i}X_{\tri}(\overline{\rho})$ (using $(\rho \otimes_E E[\epsilon]/\epsilon^2(1+\epsilon \psi_{\sigma}), \delta_i(1+\epsilon \psi_{\sigma})\in T_{z_i} X_{\tri}(\overline{\rho})$).  Let $e_1, \cdots, e_{n^2+d_K}$  be a basis of  $T_{x_{\wp}}(\Spf R_{\overline{\rho}}^{\square,\textbf{\ul{h}}-\pcr}(\xi))^{\rig}$. It is easy to see the sets $\{e_i\}$ and $\{a_{\sigma}\}_{\sigma\in \Sigma_K}$ are linearly independent. As $\dim_E \cap_i T_{z_i} X_{\tri}(\overline{\rho})=4+2d_K$, $\{e_i\}$ and $\{a_{\sigma}\}$ form a basis of it.   For each $\sigma$ and $i$, if $\ul{\phi}_i$ is not $\sigma$-critical, then by \cite[Lem.~2.4]{Ding6}, there exists $b_{\sigma,i}\in T_{z_i} X_{\tri}(\overline{\rho})$, such that $f(b_{\sigma,i})$ is $\Sigma_K \setminus \{\sigma\}$-de Rham and the natural map
\begin{equation}\label{Etanget1}
	T_{z_i} X_{\tri}(\overline{\rho}) \lra \Hom(T(K), E)
\end{equation}
sends $b_{\sigma,i}$ to $(1+\epsilon \psi_{\sigma}) \otimes (1-\epsilon \psi_{\sigma})$. If $\ul{\phi}_i$ is $\sigma$-critical, by \cite[Thm.~2.4]{Ding6}, there exists $b_{\sigma,i}\in T_{z_i} X_{\tri}(\overline{\rho})$ such that $f(b_{\sigma,i})$ is $\Sigma_K \setminus \{\sigma\}$-de Rham, $\sigma$-Hodge-Tate non-de Rham and that $b_{\sigma,i}$ is sent to zero via (\ref{Etanget1}).  One can show that $\{e_j\}$, $\{a_{\sigma}\}$, $\{b_{\sigma,i}\}$ form a basis of $T_{z_i} X_{\tri}(\overline{\rho})$. Indeed, one just needs to check there are linearly independent, but this follows easily from the theory of local model (\cite{BHS3}). Consequently $\{e_j\}$, $\{a_{\sigma}\}$, $\{b_{\sigma,i}\}$, $i=1,2$ form a basis of $T_{x_{\wp}} (\Spf R_{\overline{\rho}}^{\square})^{\rig}$.

\noindent \textbf{Step 3(b):} 
Let $\cE$ be the  patched eigenvariety in our setting (i.e. the $X_{\wp}(\overline{\rho})$ of \cite[\S~4.1]{Ding7}). The dual of the Jacquet-Emerton module $J_B(\Pi_{\infty}^{R_{\infty}-\an})$ gives rise to a Cohen-Macaulay coherent sheaf $\cM$ over $\cE$. The injection $\alg\hookrightarrow \Pi_{\infty}[\fm]$ gives to two classical points $x_i=(\rho, \jmath(\ul{\phi}_i)\delta_B z^{\lambda}, x^{\wp})\in \cE\hookrightarrow (\Spf R_{\overline{\rho}}^{\square})^{\rig} \times \widehat{T}_K \times (\Spf R_{\infty}^{\wp})^{\rig}$, $i=1,2$.
Recall that $\cE$ is smooth at the point $x_i$ (\cite[Thm.~1.3]{BHS2}).  By the multiplicity one property for $\Pi_{\infty}$ in \cite{CEGGPS2}, it is not difficult to show $\cM$ is locally free of rank $1$ in a neighbourhood of $x_i$ (for example, using the fact that the fibre of  $\cM$ at each non-critical classical point is one-dimensional).  
There is a natural morphism $\cE \ra X_{\tri}(\overline{\rho}) \times (\Spf R_{\infty}^{\wp})^{\rig}$ sending $x_i$ to $(z_i, x^{\wp})$ (cf. \cite[Thm.~1.1]{BHS1}), which is a local isomorphism at $x_i$ by \cite[Thm~1.5]{BHS3}. We view  $\{e_j\}$, $\{a_{\sigma}\}$, $\{b_{\sigma,i}\}$ in the precedent step as elements in the tangent space of $\cE$ at $x_i$ (by adding $0$ to the $R_{\infty}^{\wp}$-factor).

\noindent \textbf{Step 3(c):} 
For an element $t=(D_t, \delta_t, \fa_t^{\wp}): \Spec E[\epsilon]/\epsilon^2 \hookrightarrow \cE$ in the tangent space at $x_i$, we have an injection by definition
\begin{equation}\label{Ejacinj}
	\delta_t \hooklongrightarrow J_B(\Pi_{\infty}[\fa_t])
\end{equation}
where $\fa_t$ (resp. $\fa_t^{\wp}$) is the associated ideal of $R_{\infty}[1/p]$ (resp. $R_{\infty}^{\wp}[1/p]$). 
We apply the adjunction for the Jacquet-Emerton functor. We only  consider the case where $t$ is zero for the $R_{\infty}^{\wp}$-factor, i.e. $\fa_t^{\wp}=\fm_{x^{\wp}}$ (so that we identify $t$ with its corresponding factor in $T_{z_i} X_{\tri}(\overline{\rho})$).  When $t$ lies the tangent space of $(\Spf R_{\overline{\rho}}^{\square,\textbf{\ul{h}}-\pcr}(\xi))^{\rig}$,  the map (\ref{Ejacinj}) factors through $J_B(\Pi_{\infty}[\fa_t]^{\lalg})$, hence is balanced in the sense of \cite[Def.~0.8]{Em2}. By \cite[Thm.~0.13]{Em2}, it induces a non-zero map (see \textit{loc. cit.} for  $I_{B^-(K)}^{\GL_2(K)}(-)$): 	\begin{equation*}
	I_{B^-(K)}^{\GL_2(K)}(\delta_t) \lra \Pi_{\infty}[\fa_t].
\end{equation*}
As $\delta_t$ is locally algebraic, $I_{B^-(K)}^{\GL_2(K)}(\delta_t)$ is a self-extension of $\alg$. In fact the map is already included in (\ref{Einjalg}). If $t=a_{\sigma}$, it is easy to see (\ref{Ejacinj}) corresponds to 
\begin{equation*}
	\pi_1(D) \otimes_E (1+\epsilon \psi_{\sigma}) \circ \dett \hooklongrightarrow \Pi_{\infty}[\fa_t].
\end{equation*}
If $t=b_{\sigma,i}$, suppose first $\ul{\phi}_i$ is not $\sigma$-critical. By the choice of $\psi_{\sigma}$ and \cite[Cor.~3.23]{Ding6}, the map (\ref{Ejacinj}) factors through (recalling $f(b_{\sigma,i})$ is $\Sigma_K \setminus \{\sigma\}$-de Rham) 
\begin{multline}\label{JpardR}
	J_B\big(\Pi_{\infty}^{R_{\infty}-\an}(\lambda^{\sigma}) \otimes_E L_{\Sigma_K\setminus \{\sigma\}}(\lambda_{\Sigma_K \setminus \{\sigma\}})[\fa_t]\big)\\ =J_B\big(\Pi_{\infty}^{R_{\infty}-\an} \otimes_E L_{\Sigma_K \setminus \{\sigma\}}(\lambda_{\Sigma_K \setminus \{\sigma\}})^{\vee})^{\sigma-\la} \otimes_E  L_{\Sigma_K \setminus \{\sigma\}}(\lambda_{\Sigma_K \setminus \{\sigma\}})\big)[\fa_t].
\end{multline}This together with the fact  $\PS_{i,s_{\sigma}}$ is not a subrepresentation of $\Pi_{\infty}[\fa_t]$ imply that  (\ref{Ejacinj}) is balanced. By definition, $I_{B^-(K)}^{\GL_2(K)}(\delta_t)$ in this case is just the representation $\pi_{i,\psi_{\sigma}}$ considered in Step (1). Using push-forward to $\pi_1(D)$, we get $\pi_1(D,i,\sigma)\hookrightarrow \Pi_{\infty}[\fa_t]$.
Now suppose $\ul{\phi}_i$ is $\sigma$-critical. By the choice of $\psi_{\sigma}$, $\delta_t\cong \delta \oplus \delta$ and the injection (\ref{Ejacinj}) again factors through (\ref{JpardR}). One difference is that (\ref{Ejacinj}) is, however, no longer balanced. Pick an injection $j: \delta \hookrightarrow \delta_t \hookrightarrow J_B\big(\Pi_{\infty}^{R_{\infty}-\an}(\lambda^{\sigma}) \otimes_E L_{\Sigma_K\setminus \{\sigma\}}(\lambda_{\Sigma_K \setminus \{\sigma\}})[\fa_t]\big)$ whose image is not contained in $J_B(\Pi_{\infty}^{R_{\infty}-\an}[\fm_x])\big)$. Using Breuil's adjunction formula \cite[Thm.~4.3]{Br13II}, the map $j$ induces an injection (see also the proof of \cite[Thm.~4.4]{Ding6}): $\cF_{B^-}^{\GL_2}\big(M_{\sigma}^-(-\lambda_{\sigma})^{\vee} \otimes_E L_{\Sigma_K\setminus \{\sigma\}}(-\lambda_{\Sigma_K \setminus \{\sigma\}}), \jmath(\ul{\phi}_i)\big)\hookrightarrow \Pi_{\infty}^{R_{\infty}-\an}[\fa_t]$. Using again the push-forward to $\pi_1(D)$, we get $\pi_1(D,i,\sigma)\hookrightarrow \Pi_{\infty}[\fa_t]$ in this case. Putting all these together (and using Step 3(a)), we get an injection as in (\ref{EinjUniv}). As $M_{\infty}$ is assumed to be flat over $R_{\infty}$, we have an exact sequence with respect to the basis $\{e_j, a_{\sigma}, b_{\sigma,i}\}$:
\begin{equation*}
	0 \lra \Pi_{\infty}^{R_{\infty}-\an}[\fm] \lra \Pi_{\infty}^{R_{\infty}-\an}[\fa] \xlongrightarrow{f} \pi(D)^{\oplus (4+4d_K)} \lra 0.
\end{equation*}
It is not difficult to see that the induced injection given as above for each element in $\{e_j, a_{\sigma}, b_{\sigma,i}\}$, composed with $f$, will induce an injection from $\alg$ to the corresponding direct summand in $\pi(D)^{\oplus (4+4d_K)}$. We then deduce the intersection of the image of (\ref{EinjUniv}) and $\Pi_{\infty}[\fm]$ is exactly $\pi_1(D)$.

We finally prove the theorem. Suppose $V \hookrightarrow \Pi(D)$. By (\ref{EinjUniv}),  we get
\begin{equation*}
	V \oplus_{\pi_1(D)} \mathscr{E}(\alg, \pi_1(D)) \oplus \alg^{\oplus (2+d_K)} \hooklongrightarrow \Pi_{\infty}[\fa_x].
\end{equation*}
As $\mathscr{E}(\alg, \pi_1(D))$ is universal, we have $V \oplus_{\pi_1(D)} \mathscr{E}(\alg, \pi_1(D)) \cong \alg \oplus \mathscr{E}(\alg, \pi_1(D))$. The above injection then contradicts (\ref{Einjalg}). 
\end{proof}
We consider the case of $\GL_n(K)$. Assume $D$ is a generic crystabelline $(\varphi, \Gamma)$-module of rank $n$ over $\cR_{K,E}$ (of regular Sen weights $\textbf{h}$). Let $\phi_1, \cdots, \phi_n$ be the smooth characters of $K^{\times}$ such that $\Delta\cong \oplus_{i=1}^n \cR_{K,E}(\phi_i)$ (cf. \S~\ref{Sphigamma}, recalling the generic assumption means $\phi_i\phi_j\neq 1, |\cdot|_K^{\pm 1}$). The refinements of $D$ are then given by $\{w(\phi)=\otimes_{i=1}^n \phi_{w^{-1}(i)}, \ w\in S_n\}$. For $i=1,\cdots, n-1$, $\sigma\in \Sigma_K$, and a refinement $w(\phi)$, put ($\lambda=\textbf{h}-\theta_K$ and see Example~\ref{ExFS} for $\jmath$)
\begin{multline*}
C(w(\phi), s_{\sigma,i}):=\cF_{B^-}^{\GL_n}(L_{\Sigma_K \setminus \{\sigma\}}^-(-\lambda_{\Sigma_K \setminus \sigma}) \otimes_E L_{\sigma}(-s_{\sigma, i} \cdot \lambda), \jmath(w(\phi))\big)\\
\cong \cF_{Q_i^-}^{\GL_n}(L_{\Sigma_K \setminus \{\sigma\}}^-(-\lambda_{\Sigma_K \setminus \sigma}) \otimes_E L_{\sigma}(-s_{\sigma, i} \cdot \lambda), (\Ind_{B^-(K)\cap L_{Q_i}(K)}^{L_{Q_i}(K)} \jmath(w(\phi)))^{\infty}),
\end{multline*}
where $Q_i$ is the standard maximal parabolic subgroup with simple roots $\{1, \cdots, n-1\}\setminus \{i\}$. 
Let $\sF_i:=\{w^{-1}(1), \cdots, w^{-1}(i)\}$, it is clear by the last isomorphism that for a different refinement $w'(\phi)$, if the associated set $\sF_i'$ is equal to $\sF_i$, then $C(w(\phi), s_{\sigma,i})\cong C(w'(\phi), s_{\sigma,i})$. Hence we will also denote $	C(w(\phi), s_{\sigma,i})$  by $C(\sF_i, s_{\sigma,i})$. By \cite[Thm.]{OS}, the $(2^n-2)d_K$ representations $\{C(\sF_i, s_{\sigma,i})\}_{\substack{\sigma\in \Sigma_K \\ i=1, \cdots, n-1 \\ \sF_i\subset \{1,\cdots, n\}, \# \sF_i=i}}$ are now pairwise distinct. Let  $\alg:=\cF_{B^-}^{\GL_n}(L^-(-\lambda), \jmath(\phi))$ ($\cong \cF_{B^-}^{\GL_n}(L^-(-\lambda), \jmath(w(\phi)))$ for all $w$). For $w\in S_n$, let $\PS_{w(\phi)}^{\mathrm{simple}}$ be the (unique) subrepresentation of $(\Ind_{B^-(K)}^{\GL_n(K)} \jmath(w(\phi))z^{\lambda})^{\Q_p-\an}$, with the form
\begin{equation*}
\PS_{w(\phi)}^{\mathrm{simple}}\cong [\alg \rule[2.5pt]{10pt}{0.5 pt} \oplus_{\substack{\sigma \in \Sigma_K\\ i=1, \cdots, n-1}} C(w(\phi), s_{\sigma,i})].
\end{equation*}
Indeed the existence and uniqueness follow easily from \cite[Thm.]{OS} (and the easy fact on multiplicities of the corresponding simple modules in $M^-(-\lambda)$). We assume $D$ is non-critical for all the refinements. Put $\pi_1(D)$ be the unique quotient of the amalgam $\oplus^{w\in S_n}_{\alg} \PS_{w(\phi)}^{\mathrm{simple}}$ with socle $\alg$. Recall that each $C(\sF_i,s_{\sigma,i})$ has multiplicity one in $\pi_1(D)$ (cf. \cite[Prop.~5.9]{BH2}). Indeed, $\pi_1(D)$ is a subrepresentation of $\pi(D)^{\fss}$ of \textit{loc. cit.} Note also $\PS_{w(\phi)}^{\mathrm{simple}}\hookrightarrow \pi_1(D)$ for all $w$. We collect some easy facts:
\begin{lemma}\label{Eextpi1}
(1) For any $w\in S_n$, there is a natural isomorphism
\begin{equation}\label{Eextind}
	\Hom(T(K), E) \xlongrightarrow{\sim} \Ext^1_{\GL_n(K)}\big(\alg, \PS_{w(\phi)}^{\mathrm{simple}}\big).
\end{equation}
In particular, $\dim_E \Ext^1_{\GL_n(K)}\big(\alg, \PS_{w(\phi)}^{\mathrm{simple}}\big)=nd_K$. 

(2) We have $\dim_E \Ext^1_{\GL_n(K)}(\alg, \pi_1(D))=(2^n-1) d_K+n$.
\end{lemma}
\begin{proof}
The  lemma follows from similar (and easier) arguments as in \cite[\S~2.2, 2.3]{Ding7}. We only give a sketch. 
Using Schraen's spectral sequence \cite[(4.39)]{Sch11} (and \cite[(4.43), Thm.~4.10]{Sch11} with the classical facts on Jacquet module of smooth principal series), we have a natural isomorphism
\begin{equation}\label{Eextind2}
	\Hom(T(K), E) \xlongrightarrow{\sim} \Ext^1_{\GL_n(K)}\big(\alg, (\Ind_{B^-(K)}^{\GL_n(K)}\jmath(w(\phi))z^{\lambda})^{\Q_p-\an}\big).
\end{equation}
By \cite[Lem.~2.26]{Ding7}, the right isomorphism is naturally isomorphic to $ \Ext^1_{\GL_n(K)}\big(\alg, \PS_{w(\phi)}^{\mathrm{simple}}\big)$, (1) follows.

By similar (and easier) arguments as in the proof of \cite[Lem.~2.28]{Ding7}, 
$\dim_E \Ext^1(\alg, C(w(\phi),s_{\sigma,i}))=1$ for all $w$, $s_{\sigma,i}$. 	Using \cite[Prop.~4.7]{Sch11}, $\dim_E \Ext^1_{\GL_n(K)}(\alg, \alg)=d_K+n$ (where ``$1+d_K$" comes from the self-extensions of the central character), $\dim_E\Ext^1_{\GL_n(K), Z}(\alg, \alg)=n-1$ and $\Ext^2_{\GL_n(K),Z}(\alg, \alg)=0$ (where the subscript $Z$ stands for ``with central character").  By d\'evissage, we have hence an exact sequence
\begin{equation}\label{Eextdiv}
	0 \ra \Ext^1_{\GL_n(K),Z}(\alg, \alg) \ra \Ext^1_{\GL_n(K),Z}(\alg,\pi_1(D)) \ra \oplus_{i,\sigma , \sF_i} \Ext^1_{\GL_n(K),Z}(\alg, C(\sF_i, s_{\sigma,i})) \ra 0.
\end{equation}
So $\dim_E \Ext^1_{\GL_n(K), Z}(\alg, \pi_1(D))=(2^n-2)d_K+n-1$.
Finally, by the same argument in \cite[Lem.~3.16]{BD1}, we get $\dim_E \Ext^1_{\GL_n(K)}(\alg, \pi_1(D))=\Ext^1_{\GL_n(K), Z}(\alg, \pi_1(D))+d_K+1=(2^n-1)d_K+n$. 
\end{proof}
\begin{remark}\label{Rextfern}
(1) We can explicit describe (\ref{Eextind}) as follows. For $\psi\in \Hom(T(K),E)$, consider the following representations over $E[\epsilon]/\epsilon^2$: 
\begin{equation*}
	I_{B^-(K)}^{\GL_n(K)} \big(\jmath(w(\phi))z^{\lambda} (1+\psi \epsilon)\big)\hooklongrightarrow \big(\Ind_{B^-(K)}^{\GL_n(K)} \jmath(w(\phi))z^{\lambda} (1+\psi \epsilon)\big)^{\Q_p-\an}
\end{equation*}
where $I_{B^-(K)}^{\GL_n(K)}(-)$ is Emerton's representation defined in \cite{Em2}, that is the closed subrepresentation of the induced representation on the right hand side generated by 
\begin{equation*}
	\jmath(w(\phi))z^{\lambda} (1+\psi \epsilon)\delta_B \hooklongrightarrow J_B\big(\big(\Ind_{B^-(K)}^{\GL_n(K)} \jmath(w(\phi))z^{\lambda} (1+\psi \epsilon)\big)^{\Q_p-\an}\big).
\end{equation*}
It is not difficult to see (e.g. by \cite[Lem.~4.12]{Ding7}, and the discussion after (\ref{Eextind2}))  $I_{B^-(K)}^{\GL_n(K)}\big(\jmath(w(\phi))z^{\lambda} (1+\psi \epsilon)\big)$ is contained in a unique extension $\PS_{w(\phi)}^{\mathrm{simple}}(\psi)$ of $\alg$ by $\PS_{w(\phi)}^{\mathrm{simple}}$. The map (\ref{Eextdiv}) is then given by sending $\psi$ to $\PS_{w(\phi)}^{\mathrm{simple}}(\psi)$.

(2) From the proof of Lemma \ref{Eextpi1}, it is easy to see that the following natural map (by push-forward) is injective:
\begin{equation*}
	\Ext^1_{\GL_n(K)}(\alg, \PS_{w(\phi)}^{\mathrm{simple}}) \hooklongrightarrow \Ext^1_{\GL_n(K)}(\alg, \pi_1(D)).
\end{equation*}
And $\sum_w \Ext^1_{\GL_n(K)}\big(\alg, \PS_{w(\phi)}^{\mathrm{simple}}\big)=\Ext^1_{\GL_n(K)}(\alg, \pi_1(D))$.
\end{remark}
Recall $\Ext^1_{\GL_n(K), \lalg}(\alg, \alg)\cong \Hom_{\infty}(T(K),E)$, where $\Hom_{\infty}$ denotes the space of smooth characters, and ``$\lalg$" in the subscript means ``locally algebraic extensions". Let $\widetilde{\alg}$ be the universal locally algebraic extension of $\alg^{\oplus n}$ by $\alg$. Let $\mathscr{E}\big(alg, \PS_{w(\phi)}^{\mathrm{simple}}\big)$ be the universal extension of $\alg^{\oplus n d_K}$ by $\PS_{w(\phi)}^{\mathrm{simple}}$, $\mathscr{E}(\alg, \pi_1(D))$ be the universal extension of $\alg^{\oplus (n+(2^n-1)d_K)}$ by $\pi_1(D)$. We have natural injections by push-forward
\begin{equation*}
\widetilde{\alg} \hooklongrightarrow \mathscr{E}\big(\alg, \PS_{w(\phi)}^{\mathrm{simple}}\big) \hooklongrightarrow \mathscr{E}(\alg, \pi_1(D)).
\end{equation*}
By Remark \ref{Rextfern} (2), $\mathscr{E}(\alg, \pi_1(D))$ can be generated by (the push-forward) of $\mathscr{E}\big(\alg, \PS_{w(\phi)}^{\mathrm{simple}}\big)$ with $w$ varying.

Suppose $\pi(D)^{\lalg}\neq 0$, hence $\pi(D)^{\lalg}\cong \alg$. 
For each refinement $w(\phi)$, let $x_w=(x_{w,\wp}, x_w^{\wp})=(\jmath(w(\phi))\delta_B z^{\lambda}, \rho, x_w^{\wp})\in \cE\hookrightarrow (\Spf R_{\infty}^{\wp})^{\rig} \times (\Spf R_{\overline{\rho}}^{\square})^{\rig} \times \widehat{T}_K$ be the associated classical point on the patched eigenvariety $\cE$. Let $\cM$ be the coherent sheaf over $\cE$ associated to $J_B(\Pi_{\infty}^{R_{\infty}-\an})$. As $D$ is non-critical for all embeddings,  $\cE$ is smooth at each point $x_w$ and $\cM$ is locally free of rank one at $x_w$ (using \cite[Lem.~3.8]{BHS2} and  the multiplicity one property in the construction in \cite{CEGGPS1}). Let $\fm_{\wp}$, $\fm^{\wp}$ be the respective maximal ideal of $R_{\overline{\rho}}^{\square}[1/p]$ and $ R_{\infty}^{\wp}[1/p]$ associated to $x_w$. 
Let $\fa:=\fm_{\wp}^2+ \fm^{\wp}\subset R_{\infty}[1/p]$. We have $\dim_E \fm/\fa=\dim_E \fm_{\wp}/\fm_{\wp}^2=n^2+n^2d_K$. Consider $\Pi_{\infty}^{R_{\infty}-\an}[\fa]$. By definition, there is an exact sequence (with respect to a basis of $\fm/\fa$)
\begin{equation}\label{Eflat?}
0 \lra \Pi_{\infty}^{R_{\infty}-\an}[\fm] \lra \Pi_{\infty}^{R_{\infty}-\an}[\fa] \xlongrightarrow{f} \Pi_{\infty}^{R_{\infty}-\an}[\fm]^{\oplus n^2+n^2d_K}.
\end{equation}
Let $\Pi_{\infty}^{R_{\infty}-\an}[\fa]_1:=f^{-1}\big(\alg^{\oplus (n^2+n^2d_K)}\big)$. 
The following lemma follows from the same argument in Step 2 of the proof of Theorem \ref{Tnosurplus} (based on \cite[\S~4.28]{CEGGPS1}).
\begin{lemma}\label{Llalg1}
We have $\Pi_{\infty}^{R_{\infty}-\an}[\fa]_1^{\lalg} =\Pi_{\infty}^{R_{\infty}-\an}[\fa]^{\lalg} \cong \alg^{\oplus (n^2+\frac{n(n-1)}{2} d_K-n)} \oplus \widetilde{\alg}$. 
\end{lemma}
Recall by \cite[Thm.~1.1]{BH2}, $\alg \hookrightarrow \pi(D)$ extends to a unique injection $\pi_1(D) \hookrightarrow \pi(D)$. 
\begin{proposition}
The injection $\pi_1(D) \hookrightarrow \Pi_{\infty}^{R_{\infty}-\an}[\fm]=\pi(D)$ extends to an injection
\begin{equation*}
	\mathscr{E}(\alg, \pi_1(D)) \hooklongrightarrow \Pi_{\infty}^{R_{\infty}-\an}[\fa]_1.    
\end{equation*}
\end{proposition}
\begin{proof}
The proof is similar (and easier, as we only deal with the non-critical case) as in the Step 3 (c) of the proof of Thm.~\ref{Tnosurplus}. We give a sketch. For each point $x_w$, the tangent map of $\cE \ra (\Spf R_{\overline{\rho}}^{\square})^{\rig} \times (\Spf R_{\infty}^{\wp})^{\rig}$ is injective (cf. \cite[Lem.~4.16]{BHS2}), and the tangent map of $\cE \ra \widehat{T}_K$ at $x_w$ is surjective (for example, using similar argument as in \cite[Prop.~4.3]{Ding7} and \cite[Thm.~2.5.10]{BCh}). For each $\psi\in \Hom(T(K),E)$, let  $t=(\tilde{D}, \jmath(w(\phi))\delta_Bz^{\lambda}(1+\psi_{\epsilon}), \fm^{\wp}): \Spec E[\epsilon]/\epsilon^2 \ra \cE$ be an element in the tangent space of $\cE$ at $x_w$ (whose $R_{\infty}^{\wp}$-factor is zero). By definition, it gives an injection
\begin{equation}\label{Ebalance}
	\jmath(w(\phi))\delta_B z^{\lambda} (1+\psi \epsilon) \hooklongrightarrow J_B(\Pi_{\infty}^{R_{\infty}-\an}[\fa_t])\hookrightarrow J_B(\Pi_{\infty}^{R_{\infty}-\an}[\fa])
\end{equation}
where $\fa_t$ is the ideal of $R_{\infty}[1/p]$ associated to $t$. As $\cM$ is locally free of rank one at $x_w$, we have in particular
\begin{equation*}
	E\cong \Hom_{T(\Q_p)}(\jmath(w(\phi))\delta_B z^{\lambda}, J_B(\pi(D)^{\lalg}))\cong \Hom_{T(\Q_p)}(\jmath(w(\phi))\delta_B z^{\lambda}, J_B(\pi(D))).
\end{equation*}
Together with (\ref{Eflat?}), we see (\ref{Ebalance}) actually has image in $J_B(\Pi_{\infty}^{R_{\infty}-\an}[\fa]_1)$. 
As $D$ is non-critical at $w$, (\ref{Ebalance}) is balanced (see for example \cite[Lem.~4.11]{Ding7}). By \cite[Thm.~0.13]{Em2},  it  induces an injection 
\begin{equation*}
	\iota_{w,\psi}: I_{B^-(K)}^{\GL_n(K)}\big(	\jmath(w(\phi))z^{\lambda} (1+\psi \epsilon)\big) \hooklongrightarrow \Pi_{\infty}^{R_{\infty}-\an}[\fa]_1.
\end{equation*}
Letting $\psi$ vary, and using (\ref{Eextind}) and Remark \ref{Rextfern} (1), all the above maps amalgamate to an injection
\begin{equation*}
	\iota_w: 	\mathscr{E}\big(\alg, \PS_{w(\phi)}^{\mathrm{simple}}\big)\hooklongrightarrow \Pi_{\infty}^{R_{\infty}-\an}[\fa]_1.
\end{equation*}
Finally, letting $x_w$ vary, and using Remark \ref{Rextfern} (2) (and the discussion after it), these $\iota_w$'s amalgamate to the wanted injection in the proposition.
\end{proof}
We finally get the following theorem, giving a lower bound of the multiplicities of $\alg$ in $\pi(D)$:
\begin{theorem}\label{Tsurplus2}
Suppose $D$ is non-critical for all refinements, and $\pi(D)^{\lalg}\neq 0$. Then the multiplicity of $\alg$ in $\pi(D)$ is no smaller than $1+(2^n-\frac{n(n+1)}{2}-1) d_K$. 
\end{theorem}
\begin{proof}
By the above proposiiton and Lemma \ref{Llalg1}, we have an injection
\begin{equation*}
	\mathscr{E}(\alg, \pi_1(D)) \oplus \alg^{\oplus (n^2+\frac{n(n-1)}{2} d_K-n)} \hooklongrightarrow \Pi_{\infty}^{R_{\infty}-\an}[\fa]_1.
\end{equation*} 
By (\ref{Eflat?}), we easily see the multiplicity of $\alg$ in $\pi(D)$ has to be no smaller than $1+((2^n-1)d_K+n)+(n^2+\frac{n(n-1)}{2}d_K -n) - (n^2+n^2d_K)=1+(2^n-\frac{n(n+1)}{2}-1)d_K$. 
\end{proof}
\begin{remark}(1) The existence of surplus locally algebraic constituents in the very critical case for $\GL_3$ was previously announced by Hellmann-Hernandez-Schraen (see for example the discussion below \cite[Conj.~9.6.37]{EGH}).

(2) The extension of surplus locally algebraic constituents with $\pi_1(D)$ within $\pi(D)$ carries certain information of the Hodge filtration of $D$ (even the full information when $K=\Q_p$). We will discuss this in a upcoming work (\cite{Ding16}).
\end{remark}


\newpage

\titleformat{\section}[block]{\Large\filcenter}{\thesection}{1em}{}
\section[Cohen-Macaulayness and duality by Yiwen Ding, Yongquan Hu, Haoran Wang]{Cohen-Macaulayness and duality of some $p$-adic representations}\label{AppCM}
\renewcommand{\thefootnote}{\fnsymbol{footnote}} 
\setcounter{footnote}{1}
\footnotetext{We thank Liang Xiao for helpful discussions. Part of this work was done during a visit of H. W. to I.H.\'E.S.. He thanks Ahmed Abbes for the invitation and I.H.É.S. for the hospitality.
	
	Y. D. is supported by the NSFC
	Grant No. 8200800065 and No. 8200907289. Y. H. is partially supported by  National Key R$\&$D Program of China 2020YFA0712600, CAS Project for Young Scientists in Basic Research, Grant No. YSBR-033; National Natural Science Foundation of China Grants 12288201 and 11971028; National Center for Mathematics and Interdisciplinary Sciences and Hua Loo-Keng Key Laboratory of Mathematics, Chinese Academy of Sciences. H. W. is partially supported by National Key R\&D Program of China 2023YFA1009702 and National Natural Science Foundation of China Grants 12371011.}
\begin{center}
	Yiwen Ding, Yongquan Hu, Haoran Wang
\end{center}

	\begin{abstract}Let $K$ be an unramified extension of $\Q_p$. We show that the $p$-adic Banach space and locally $\Q_p$-analytic representations of $\GL_2(K)$ (associated to two dimensional Galois representations) are Cohen-Macaulay and essentially self-dual.
	\end{abstract}
	\localtableofcontents

	\subsection{Notations and preparations}
	
	We fix a prime number $p.$ Let $E/\Q_p$ be a finite extension in $\bQp,$ with ring of integers $\co_E$ and residue field $\F \defn \co_E/(\varpi_E)$ where $\varpi_E$ is a fixed uniformizer of $\co_E$. We assume that $E$ and $\F$ are sufficiently large.  
	
	Let $G_0$ be a compact $p$-adic analytic group. The ring-theoretic properties of the Iwasawa algebra $ \co_E [[G_0]] $ are established by the fundamental works of Lazard \cite{Lazard} and Venjakob \cite{Ven}. In particular, if $G_0$ has no element of order $p$,  then $  \co_E [[G_0]]$ is an 
	Auslander regular ring of global dimension $ {\rm gld}(\co_E[[G_0]]) = 1+\dim_{\Q_p} G_0$, where $\dim_{\Q_p} G_0$ is the dimension of $G_0$ as a $p$-adic analytic group. Let $M$ be a \emph{nonzero} left $\co_E[[G_0]]$-module, the \emph{grade} $j_{G_0}(M) =j_{  \co_E [[G_0]]} (M) $ of $M$ over $\co_E[[G_0]]$ is defined by
	\[
	j_{G_0}(M)=\mathrm{inf}\{i \in \N ~|~ \Ext^i_{\co_E[[G_0]]}(M,\co_E[[G_0]])\neq0\}.
	\]
	We always have
	\[
	0 \leq j_{  G_0} (M) \leq 1+\dim_{\Q_p} G_0.
	\]
	The {\em dimension} of $M$ is defined by 
	\begin{equation*}
		\d_{G_0}(M):= \d_{ \co_E [[G_0]]}( M)  = {\rm gld}(\co_E[[G_0]]) - j_{G_0}(M) = 1+\dim_{\Q_p} G_0- j_{G_0}(M).
	\end{equation*}
	Recall that a finitely generated nonzero left $\cO_E[[G_0]]$-module $M$ is  \emph{Cohen-Macaulay} if the module  $\Ext^i_{\co_E[[G_0]]}(M,\co_E[[G_0]])$ is nonzero for just one degree $i.$ Denote by $E[[G_0]] : = \co_E [[G_0]]\otimes_{\co_E} E .$

	Let $G$ be a $p$-adic analytic group with a fixed open compact subgroup $G_0 \subseteq G.$ Set
	\[
	\L(G) :=  \plim_n \left( \co_E/\varpi_E^{n}[G]\otimes_{\co_E/\varpi_E^n[G_0]}  \co_E/\varpi_E^n [[G_0]]\right) =  \co_E[G]\otimes_{\co_E[G_0]}  \co_E [[G_0]].
	\]
	The discussion in \cite[\S1]{Ko} for $k[G]\otimes_{k[G_0]}  k [[G_0]]$ with $k$ a discrete field, extends to $\co_E/\varpi_E^{n}[G]\otimes_{\co_E/\varpi_E^n[G_0]}  \co_E/\varpi_E^n [[G_0]].$ In particular, $\co_E/\varpi_E^{n}[G]\otimes_{\co_E/\varpi_E^n[G_0]}  \co_E/\varpi_E^n [[G_0]]$ and hence $\L(G)$ do not depend on the choice of $G_0$. We consider finitely generated $\Lambda(G_0)$-modules which carries a jointly continuous $\co_E$-linear action of $G$ defined in \cite[Def. 2.1.6]{EmertonOrd1}. Recall that if $M$ is such a module, we set \[j_G(M) = j_{\co_E[[G_0]]} (M),\quad \EE^i(M) : = \Ext^i_{\Lambda(G_0)} (M, \Lambda(G_0)).\] Then $\EE^i (M)$ carries naturally a jointly continuous action of $G$ and is still finitely generated over $\L(G_0)$ by \cite[\S3]{Ko}. The aim of this section is to extend some results of \cite[\S11.3]{HW-CJM} to a broader setting.

	\begin{lemma}\label{lemma--duality-Gorenstein}
		Let $(A, \frak{m})$ be a local Noetherian $\co_E$-algebra such that $A$ is Gorenstein of Krull dimension $1$ and that the $\co_E$-algebra structure map  induces an isomorphism $ \co_E /\varpi_E \cong A/\frak{m }.$ Assume $x: A \to \co_E$ is a surjective homomorphism of $\co_E$-algebras. Then $\co_E$ is an $A$-module via $x$, and there is an isomorphism of $A$-modules $\co_E \cong \Hom_A (\co_E, A).$ 
	\end{lemma}
	\begin{proof}
		Since $A$ is Gorenstein of Krull dimension $1,$ 
		\[
		\Hom_A (\co_E/\varpi_E, A) =  \Ext^2_A (\co_E/\varpi_E , A)=  0,\quad \Ext^1_A (\co_E/\varpi_E , A) \cong \co_E/\varpi_E.
		\]
		The short exact sequence $0 \to \co_E \To{\varpi_E} \co_E \to \co_E/\varpi_E \to 0$ gives an exact sequence of $A$-modules
		\[
		0 \to \Hom_A (\co_E, A) \To{\varpi_E} \Hom_A (\co_E, A) \to \co_E/\varpi_E  \to \Ext^1_A (\co_E, A) \To{\varpi_E} \Ext^1_A (\co_E, A)  \to 0.
		\]
		By Nakayama's lemma, we get $\Ext^1_A (\co_E, A) = 0$, and hence a short exact sequence of $A$-modules
		\[
		0 \to \Hom_A (\co_E, A) \To{\varpi_E} \Hom_A (\co_E, A) \to \co_E/\varpi_E  \to 0.
		\]
		Applying $\Hom_A(\co_E,-)$ to the short exact sequence of $A$-modules $0 \to \Ker(x) \to A \To{x} \co_E\to 0,$ we obtain a homomorphism of $A$-modules
		\[
		\alpha: \Hom_A(\co_E, A) \to  \Hom_A(\co_E, \co_E) \cong \co_E
		\]
		which fits into a commutative diagram
		\[
		\xymatrix{
			0  \ar[r] & \Hom_A (\co_E, A) \ar[d]^{\a}  \ar[r]^{\varpi_E} & \Hom_A (\co_E, A)  \ar[d]^{\a} \ar[r] & \co_E/\varpi_E \ar[d]^{\cong} \ar[r] & 0\\
			0  \ar[r] &  \co_E \ar[r]^{\varpi_E} & \co_E  \ar[r] & \co_E/\varpi_E  \ar[r] & 0.
		}
		\]
		By Nakayama's lemma, $\Ker(\a) = \Coker(\a) = 0,$ and hence $\a$ is an isomorphism. 
	\end{proof}
	
	\begin{lemma}\label{lemma--CM-spe}
		Let $(A,\frak{m})$ be a flat local Noetherian $\co_E$-algebra of Krull dimension $1$ such that the $\co_E$-algebra structure map induces an isomorphism $\co_E/\varpi_E \cong A/\frak{m}.$ Let $M$ be a finitely generated $\L(G_0)$-module equipped with a homomorphism of $\cO_E$-algebras $\alpha_M: A \to \End_{ \L(G_0)}(M)$ such that $M$ is flat over $A,$ Cohen-Macaulay over $\L(G_0)$ of grade $c.$ Let $x: A\to \co_E$ be a surjective homomorphism of $\co_E$-algebras. Then $M \otimes_{A, x} \co_E$ is Cohen-Macaulay over $\L(G_0)$ of grade $c.$
	\end{lemma}
	\begin{proof}
		We may assume $G_0$ has no element of order $p$ by \cite[Lem. A.7]{Gee-Newton}. Then $\L(G_0)$ is Auslander regular with finite global dimension. We first prove that $M\otimes_{A , x} \co_E/\varpi_E$ is Cohen-Macaulay over $\L(G_0) $ of grade $c+1.$ Since $A$ is $\varpi_E$-torsion free and $M$ is flat over $A,$ $M\otimes_A A/\varpi_E $ is finitely generated and Cohen-Macaulay over $\L(G_0)$ of grade $c+1$ by \cite[Lem. A.15]{Gee-Newton} and its proof. Since $A/\varpi_E$ is an artinian ring, $M\otimes_{A} A/\varpi_E$ has a finite filtration with graded pieces isomorphic to $M\otimes_A A/\frak{m}.$  As a consequence, $j_{G_0}(M\otimes_A A/\frak{m}) = j_{G_0}(M\otimes_A A/\varpi_E) = c+1$ by \cite[Prop. 1.8]{Bjork-LNM}.   By \cite[Ch.~X, \S~8.1, Cor.~2]{Bourbaki} $M\otimes_A A/\varpi_E$ and $M\otimes_A A/\frak{m}$ have the same projective dimension over $\L(G_0)$. Hence $M\otimes_{A,x} \co_E/\varpi_E \cong M\otimes_A A/\frak{m} $ is Cohen-Macaulay  of grade $c+1$. 
		
		Consider the short exact sequence (obtained by the flatness of $M$ over $A$)
		\[
		0\to M\otimes_{A,x} \co_E \To{\varpi_E}  M\otimes_{A,x} \co_E \to   M\otimes_{A,x} \co_E/\varpi_E  \to 0.
		\]
		Since $M\otimes_{A,x} \co_E/\varpi_E$ is Cohen-Macaulay of grade $c+1,$ we have surjections $\varpi_E: \EE^i (M\otimes_{A,x} \co_E) \onto \EE^i (M\otimes_{A,x} \co_E)$ for $i\neq c,$ and a short exact sequence $0 \to \EE^{c}(M\otimes_{A,x} \co_E) \To{\varpi_E} \EE^c (M\otimes_{A,x} \co_E) \to \EE^{c+1} (M\otimes_{A,x} \co_E/\varpi_E) \to 0.$ Since $M$ is finitely generated over $ \L(G_0),$ so is $\EE^{i} (M\otimes_{A,x} \co_E).$ Then $\EE^{i} (M\otimes_{A,x} \co_E) = 0  $ for $i\neq c$ by Nakayama's lemma, and $\EE^{c} (M\otimes_{A,x} \co_E) \neq 0.  $ 
	\end{proof}

	\begin{proposition}\label{thm-self-duality}
		Let $(A, \frak{m})$ be a local Noetherian $\co_E$-algebra which is Gorenstein of Krull dimension $d \geq 1$ such that the $\co_E$-algebra structure map $i:\co_E \to A$ induces an isomorphism $ \co_E /\varpi_E \cong A/\frak{m }.$ Let $M,N$ be $ \L(G)$-modules equipped with homomorphisms of $\cO_E$-algebras $\alpha_M: A\to \End_{\L(G)}(M)$ and $\alpha_N: A\to \End_{\L(G)}(N)$ such that $M$ and $N$ are  flat over $A$,  finitely generated and Cohen-Macaulay of grade $c$ over $\L(G_0).$ Assume $\epsilon:M \to \EE^c(N)$ is an isomorphism of $A\otimes_{\co_E} \L(G)$-modules. Let $x: A\to \co_E$ be a surjective homomorphism of $\co_E$-algebras. Then  $M\otimes_{A,x}\cO_E$ and $N\otimes_{A,x}\cO_E$ are Cohen-Macaulay of grade $c+d-1$ over $\Lambda(G_0)$. Moreover, the isomorphism $\epsilon$ induces an isomorphism of $\L(G)$-modules $M\otimes_{A, x} \co_E \cong \EE^{c + d-1}(N \otimes_{A, x} \co_E).$ 
	\end{proposition}
	
	\begin{proof}
		Since $A$ is assumed to be Gorenstein, it is Cohen-Macaulay. We also have $A$ is flat over $\co_E$ as $\Hom_A (\co_E/\varpi_E, A) = 0$. Let $\frak{p}: = \Ker(A \To{x} \co_E).$ One can choose $x_i\in \frak{p},$ $1\leq i\leq d-1$ such that $\{ x_1,\ldots, x_{d-1},\varpi_E \}$ is a regular sequence of $A.$ Indeed, let $\{\varpi_E, y_1, \ldots, y_{d-1}\}$ be any system of parameters of $A.$ Since $\frak{m} = (\varpi_E)+\frak{p},$ we may choose $x_i \in \frak{p}$ so that $x_i \equiv y_i \pmod \varpi_E.$ Then $\{ \varpi_E, x_1,\ldots, x_{d-1} \}$ is also a system of parameters of $A$, hence a regular sequence by \cite[Thm.~17.4(iii)]{MatsumuraCRT}. Equivalently, $\{ x_1,\ldots, x_{d-1},\varpi_E \}$ is a regular sequence of $A.$ Since $M$ (resp. $N$) is flat over $A,$ $\{x_1,\ldots, x_{d-1},\varpi_E\}$ is an $M$-sequence (resp. $N$-sequence).
		
		By a standard argument (\cite[Exer.~18.1]{MatsumuraCRT}), $A / (x_1,\ldots, x_{d-1})$ is a flat local Noetherian $\co_E$-algebra which is Gorenstein of Krull dimension $1.$ Since $M$ and $N$ are flat over $A,$ $M / (x_1,\ldots, x_{d-1})$ and $N/(x_1,\ldots, x_{d-1})$ are Cohen-Macaulay over $\L(G_0)$ of grade $c+(d-1),$ and are flat over $A / (x_1,\ldots, x_{d-1}).$ By induction,  $\epsilon$ induces an isomorphism $$M / (x_1,\ldots, x_{d-1}) \cong  \EE^{c+(d-1)} (N/ (x_1,\ldots, x_{d-1})).$$ So we are reduced to the case where $A$ has Krull dimension $1$ and $M,N$ are Cohen-Macaulay of some grade $j$ over $\L(G_0).$ By Lemma \ref{lemma--CM-spe} $M \otimes_{A, x} \co_E$ and $N \otimes_{A,x} \co_E$ are Cohen-Macaulay of grade $j.$ We choose a finite presentation of $\co_E$ as 
		$A$-module:
		\begin{equation}\label{eq:appen-F}
			A^n\To{f} A\To{x} \co_E \to 0,
		\end{equation}
		which induces an exact sequence
		\[0\ra \EE^j(N \otimes_{A, x} \co_E)\to \EE^j(N \otimes_{A} A)\overset{f^{*}}{\ra} \EE^j(N \otimes_A A^n).\]
		It is easy to see that the map $f^*$ is equal to \[(A\overset{f^{T}}{\ra} A^n)\otimes\EE^j(N) \]
		where $f^{T}$ denotes the transpose of $f$.
		
		On the other hand, applying $\Hom_A(-,A)$ to \eqref{eq:appen-F} gives an exact sequence
		\[
		0\ra \Hom_A(\co_E,A)\ra A\overset{f^T}{\ra} A^n.
		\]
		Noticing that $\Hom_A(\co_E,A)\cong \co_E$ by Lemma \ref{lemma--duality-Gorenstein}, and that $M$ is $A$-flat, we obtain an exact sequence
		\[0\ra M\otimes_{A,x} \co_E \ra M\overset{f^T}{\ra} M^n.\]
		The explicit description of maps shows that the diagram
		\[
		\xymatrix{M \ar^{f^T}[r]\ar_{\cong}^{\epsilon}[d]& M^n \ar^{\epsilon}_{\cong}[d] \\
			\EE^j(N) \ar^{f^*}[r]&  \EE^j(N^n)}
		\]
		is commutative. Then $\epsilon$ induces an isomorphism $  M\otimes_{A,x} \co_E \cong \EE^j(N\otimes_{A,x} \co_E).$
	\end{proof}

	\begin{corollary}\label{cor--self-duality}
		Under the assumption of Proposition \ref{thm-self-duality}, $(M\otimes_{A, x} \co_E)[1/\varpi_E] $ and $(N\otimes_{A,x} \co_E)[1/\varpi_E] )$ are Cohen-Macaulay modules over $E[[G_0]].$ And we have an isomorphism of $\L(G)[1/p]$-modules $(M\otimes_{A, x} \co_E)[1/\varpi_E] \cong \EE^{c+d-1}( (N\otimes_{A,x} \co_E)[1/\varpi_E] ).$
	\end{corollary}
	\begin{proof}
		This follows from \cite[Prop. 3.28]{Ven}.
	\end{proof}
	The following lemma will be useful in our application.
	\begin{lemma}\label{Lextravaria}
		Let  $M$ be a $\Lambda(G)$-module, finitely generated over $\Lambda(G_0)$. Let $s\in \Z_{\geq 1}$ and suppose $M$ is equipped with a $\Z_p^s$-action via a certain character $\chi$. Then we have a natural $G$-equivariant isomorphism of $\Lambda(\Z_p^s \times G_0)$-modules:
		\begin{equation}\label{extravaria}
			\Ext^d_{\Lambda(G_0)}(M, \Lambda(G_0)) \cong \Ext^{d+s}_{\Lambda(\Z_p^s \times G_0)}(M, \Lambda(\Z_p^s \times G_0))
		\end{equation}
		where $\Z_p^s$ acts on the left hand side via $\chi^{-1}$.
	\end{lemma}
	\begin{proof}
		By induction, we are reduced to the case $s=1$. We show first (\ref{extravaria})  for $d=0$. Let $x\in \co_E[[\Z_p]]$ be the generator of the ideal corresponding to $\chi$. Thus $M$ is annihilated by $x$. Consider the exact sequence
		\begin{equation*}
			0 \ra \Lambda(\Z_p \times G_0) \xrightarrow{x} \Lambda(\Z_p \times G_0) \ra \Lambda(G_0) \ra 0.
		\end{equation*} 
		Applying $\Hom_{\Lambda(\Z_p \times G_0)}(M,-)$, we get 
		\begin{multline*}
			0 \ra \Hom_{\Lambda(\Z_p \times G_0)}(M, \Lambda(\Z_p \times G_0)) \xrightarrow{x} \Hom_{\Lambda(\Z_p \times G_0)}(M, \Lambda(\Z_p \times G_0))\ra  \Hom_{\Lambda(\Z_p \times G_0)}(M, \Lambda(G_0)) \\
			\ra \Ext_{\Lambda(\Z_p \times G_0)}^1(M, \Lambda(\Z_p \times G_0)) \xrightarrow{x} \Ext_{\Lambda(\Z_p \times G_0)}^1(M, \Lambda(\Z_p \times G_0)).
		\end{multline*}
		As $M$ is annihilated by $x$, the first two terms are zero, and the last map is also zero. We obtain  hence 
		\begin{equation}\label{isoHo}
			\Hom_{\Lambda(G_0)}(M, \Lambda(G_0)) \cong  \Hom_{\Lambda(\Z_p \times G_0)}(M, \Lambda(G_0)) 
			\xrightarrow{\sim} \Ext_{\Lambda(\Z_p \times G_0)}^1(M, \Lambda(\Z_p \times G_0)).
		\end{equation}
		If $M$ is a finitely generated projective $\Lambda(G_0)$-module, then $ \Lambda(\Z_p)\widehat{\otimes}_{\co_E} M$ is a finitely generated projective $\Lambda(\Z_p \times G_0)$-module. Applying $\Hom_{\Lambda(\Z_p \times G_0)}(-, \Lambda(\Z_p \times G_0))$ to the exact sequence
		\begin{equation*}
			0 \ra \Lambda(\Z_p) \widehat{\otimes}_{\co_E} M \xrightarrow{x}  \Lambda(\Z_p) \widehat{\otimes}_{\co_E}  M\ra M \ra 0,
		\end{equation*}
		we see $\Ext^i_{\Lambda(\Z_p \times G_0)}(M, \Lambda(\Z_p \times G_0))=0$ for $i\neq 1$. Thus $\Ext^{\bullet}_{\Lambda(G_0)}(-,\Lambda(G_0))$ and $\Ext^{1+\bullet}_{\Lambda(\Z_p \times G_0)}(-, \Lambda(\Z_p \times G_0))$ are both universal $\delta$-functors from finitely generated $\Lambda(G_0)$-modules to $\Lambda(\Z_p \times G_0)$-modules, and hence are isomorphic by (\ref{isoHo}).
	\end{proof}

	\subsection{Patched representations and duality}
	
	If $F$ is a field, we let $\Gal_F := \Gal(\Fov/F)$ denote its absolute Galois group.  If $M$ is a linear-topological $\co_E$-module, we let $M^{\vee}:=  \Hom^{\rm cont}_{\co_E} (M, E/\co_E)$ denote its Pontryagin dual.  If $M$ is a torsion free linear-topological $\co_E$-module, we let $M^{d}:=  \Hom^{\rm cont}_{\co_E} (M, \co_E)$ denote its Schikhof dual.

	We use the setting of \cite[\S 2.3]{CEGGPS1} (with slight different notation). In particular, we have the following data
	\begin{equation*}
		\{\overline{r}, \overline{\rho}, K, F/F^+, \widetilde{G}, v_1, S_p, \fp, \xi, \tau\},
	\end{equation*}
	where
	\begin{itemize} \item $K$ is a finite extension of $\Q_p$, $F$ is a CM field with maximal totally real subfield $F^+$ such that all the finite places of $F^+$ are unramified in $F$, that all the $p$-adic places of $F^+$ split in $F$, and the completion of $F^+$ at the $p$-adic places are all isomorphic to $K$;
		\item $S_p$ is the set of $p$-adic places of $F^+$, and $\fp\in S_p$;	
		\item $v_1$ is a certain finite place of $F^+$ prime to $p$ that splits in $F$;
		\item $\overline{r}:\Gal_F \ra \GL_n(\F)$ is a continuous representation and $\overline{\rho}:\Gal_{F^+} \to \mathcal{G}_n(\F)$ is a suitable globalization of $\overline{r}$ as in \cite[\S~2.1]{CEGGPS1};
		\item $\widetilde{G}$ is a certain definite unitary group over $F^+$ with a model over $\co_{F^+}$;
		\item $\xi$ is an integral dominant weight (with respect to the upper triangular Borel subgroup) of $\Res^F_{\Q_p} \GL_n$, $\tau$ is an inertial type of $K$,  $W_{\xi,\tau}$ be the representation of $\prod_{v\in S_p \setminus \{\fp\}} \widetilde{G}(\co_{F^+_v})$ over $\co$ associated to $\xi$ and $\tau$. Note that the definition of $W_{\xi,\tau}$ depends on the choice of an $\GL_n(\co_K)$-stable $\co_E$-lattice of $\s(\tau)^{\vee},$ where $\s(\tau)$ is the smooth type given by the inertial local Langlands correspondence. 	
	\end{itemize}
	Let $G_0$ be a compact open subgroup of $\GL_n(K)$ with no element of order $p$. Let $U^{\fp}=U^{p} U^{\fp}_p=\prod_{v\nmid p} U_v \times \prod_{v\in S_p \setminus \{\fp\}} U_v$ be a compact open subgroup of $\widetilde{G}(\A_{F^+}^{\infty, \fp})$ with $U_v\cong \GL_n(\co_F)$ for $v\in S_p\setminus \{\fp\}$. For  $k\in \Z_{\geq 1}$ and a compact open subgroup $U_{\fp}$ of $\widetilde{G}(\co_{F^+_{\fp}})$, consider the $\co/\varpi_E^k$-module \begin{equation}\label{ESxitau}
		S_{\xi,\tau}(U^{\fp} U_{\fp}, \co/\varpi_E^k)=\{f: \widetilde{G}(F^+)\backslash \widetilde{G}(\A_{F^+}^{\infty}) \ra W_{\xi,\tau}/\varpi_E^k\ |\ f(gu)=u^{-1} f(g), \ \forall g\in \widetilde{G}(\A_{F^+}^{\infty}), u\in U^{\fp}U_{\fp}\}
	\end{equation}
	where $U^{\fp}U_{\fp}$ acts on $W_{\xi,\tau}/\varpi_E^k$ via the projection $U^{\fp}U_{\fp} \ra \prod_{v\in S_p \setminus \{\fp\}} U_v$.  Put 
	\[\widehat{S}_{\xi,\tau}(U^{\fp}, \co):=\varprojlim_k  S_{\xi,\tau}(U^{\fp}, \co/\varpi_E^k):=\varprojlim_k \varinjlim_{U_{\fp}} S_{\xi,\tau}(U^{\fp}U_{\fp}, \co/\varpi_E^k).\] Denote by $\mathbb{T}(U^{\fp})$ the polynomial $\co$-algebra generated by the spherical Hecke operators at places $v$ such that $\widetilde{G}(\co_{F^+_v})\cong \GL_n(\co_{F^+_v})$. Then $S_{\xi,\tau}(U^{\fp}U_{\fp}, \co/\varpi_E^k)$ and $\widehat{S}_{\xi,\tau}(U^{\fp}, \co)$ are equipped with a natural action of $\mathbb{T}(U^{\fp})$. Denote by $S_{\xi,\tau}(U^{\fp}U_{\fp}, \co/\varpi_E^k)^{\vee}$ the Pontryagin dual of $S_{\xi,\tau}(U^{\fp}U_{\fp}, \co/\varpi_E^k)$. Hence $\widehat{S}_{\xi,\tau}(U^{\fp},\co)^d\cong \varprojlim_k \varprojlim_{U_{\fp}} S_{\xi,\tau}(U^{\fp} U_{\fp}, \co/\varpi_E^k)^{\vee}$.  Let $\xi'$ (resp. $\tau'$) be the dual of $\xi$ (resp. of $\tau$). We define $S_{\xi', \tau'}(U^{\fp}U_{\fp}, \co/\varpi_E^k)$ etc. in a similar way replacing $W_{\xi,\tau}$ by $W_{\xi',\tau'}:= W_{\xi,\tau}^d$. The following lemma is well-known.
	\begin{lemma}\label{Poincaredual}
		Assume $U^{\fp}$ is sufficiently small in the sense of \cite{CHT}. Let $V^p=\prod_{v\nmid p} V_v$ be a compact open normal subgroup of $U^p$, $V^{\fp}:=V^p U_p^{\fp}$ and $V_{\fp}$ be a compact open normal subgroup of $U_{\fp}$. There is a natural $\mathbb{T}(V^{\mathfrak{p}}) \times U_{\fp} /V_{\fp} \times U^{p}/V^{p}$-equivariant isomorphism
		\begin{equation}\label{dual00}
			\Hom_{\co_E[U_{\fp}/V_{\fp} \times U^{p}/V^{p}]}\big(S_{\xi,\tau}(V_{\fp}V^{\mathfrak{p}}, \co_E/\varpi_E^k)^{\vee}, \co_E/\varpi_E^k[U_{\fp}/V_{\fp} \times U^p/V^p]\big) \xrightarrow{\sim} S_{\xi',\tau'}(V_{\fp}V^{\mathfrak{p}}, \co_E/\varpi_E^k)^{\vee}
		\end{equation}
		Moreover, if $V_{\mathfrak{p}}'\subset V_{\mathfrak{p}}$ is another compact open normal subgroup of $U_{\mathfrak{p}}$, then for $*=\xi,\tau$ or $\xi', \tau'$, there is a natural isomorphism
		\[S_*(V_{\fp}V^{\fp}, \co_E/\varpi_E^k)^{\vee} \cong S_*(V_{\fp}'V^{\fp}, \co_E/\varpi_E^k)^{\vee} \otimes_{\co_E[U_{\fp}/V_{\fp}']} \co_E[U_{\fp}/V_{\fp}]\] and the following diagram commutes
		\begin{equation*}
			\begin{array}{ccc}
				\Hom_{\co_E[U_{\fp}/V_{\fp}' \times U^{p}/V^{p}]}\big(S_{\xi,\tau}(V_{\fp}'V^{\mathfrak{p}}, \co_E/\varpi_E^k)^{\vee}, \co_E/\varpi_E^k[U_{\fp}/V_{\fp}' \times U^{p}/V^{p}]\big) & \xrightarrow{\sim} & S_{\xi',\tau'}(V_{\fp}'V^{\mathfrak{p}}, \co_E/\varpi_E^k)^{\vee} \\
				\downarrow & & \downarrow\\
				\Hom_{\co_E[U_{\fp}/V_{\fp} \times U^{p}/V^{p}]}\big(S_{\xi,\tau}(V_{\fp}V^{\mathfrak{p}}, \co_E/\varpi_E^k)^{\vee}, \co_E/\varpi_E^k[U_{\fp}/V_{\fp} \times U^{p}/V^{p}]\big) &\xrightarrow{\sim} & S_{\xi',\tau'}(V_{\fp}V^{\mathfrak{p}}, \co_E/\varpi_E^k)^{\vee}.
			\end{array}
		\end{equation*}
	\end{lemma}
	\begin{proof}
		Let $\mathcal{U}_p^{\fp}\subset U_p^{\fp}$ be a compact open subgroup such that the action of $\mathcal{U}_p^{\fp}$ on $W_{\xi,\tau}/\varpi_E^k$ and $W_{\xi',\tau'}/\varpi_E^k$ are both trivial. The statements in lemma follow easily from similar statements with $U_p^{\fp}$ repaced by $\mathcal{U}_p^{\fp}$. Let $X:=\widetilde{G}(F^+) \backslash \widetilde{G}(\mathbb{A}_{F^+}^{\infty})/(V_{\fp}V^p \mathcal{U}_p^{\fp})$ that is a finite set equipped with a right $(U_{\fp}/V_{\fp} \times U^p/V^p)$-action. We have $(U_{\fp}/V_{\fp} \times U^p/V^p)$-equivariant isomorphisms $S_{\xi,\tau}(V_{\fp} V^p \mathcal{U}_p^{\fp}, \co_E/\varpi_E^k)\cong\mathscr{C}(X, W_{\xi,\tau}/\varpi_E^k)$, and $S_{\xi',\tau'}(V_{\fp} V^p \mathcal{U}_p^{\fp}, \co_E/\varpi_E^k)\cong \mathscr{C}(X, W_{\xi',\tau'}/\varpi_E^k)$ where $\mathscr{C}(-,-)$ denotes the set of maps. Let $[-,-]$ denote the pairing
		\begin{equation*}
			\mathscr{C}(X, W_{\xi,\tau}/\varpi_E^k) \times \mathscr{C}(X, W_{\xi',\tau'}/\varpi_E^k) \lra \co_E/\varpi_E^k
		\end{equation*}
		given by $[f_1,f_2]=\sum_{x\in X} (f_1(x), f_2(x))$, where $(-,-)$ denotes the natural pairing $W_{\xi,\tau}/\varpi_E^k \times W_{\xi',\tau'}/\varpi_E^k \ra \co_E/\varpi_E^k$. Finally, we define a pairing
		\begin{equation*}
			\langle-, -\rangle: \mathscr{C}(X, W_{\xi,\tau}/\varpi_E^k) \times \mathscr{C}(X, W_{\xi',\tau'}/\varpi_E^k) \lra \co_E/\varpi_E^k[U_{\fp}/V_{\fp} \times U^p/V^p]
		\end{equation*}
		given by $\langle f_1, f_2 \rangle=\sum_{h\in U_{\fp}/V_{\fp} \times U^p/V^p} [f_1, h f_2] [h^{-1}]\in \co_E/\varpi_E^k[U_{\fp}/V_{\fp} \times U^p/V^p]$. It is straightforward to check that the pairing induces an isomorphism as in (\ref{dual00}) and all the properties in the lemma hold.
	\end{proof}
	By taking limit with respect to $k\in \Z_{\geq 1}$ and $U_{\fp}$, we deduce the following corollary (see \cite[\S~3]{Ko} for the $\GL_n(K)$-action).
	
	\begin{corollary}Assume $U^{\fp}$ is sufficiently small. 
		Let $H$ be a compact open subgroup of $\GL_n(K)$, and $V^p$ be a compact open normal subgroup of $U^p$, there is a natural $\mathbb{T}(V^{\fp}) \times \GL_n(K) \times U^p/V^p$-equivariant isomorphism
		\begin{equation*}
			\Hom_{\co_E[[H]] \times \co_E[U^{p}/V^{p}]}\big(\widehat{S}_{\xi,\tau}(V^{\fp}, \co_E)^d, \co_E[[H]] \otimes_{\co_E} \co_E[U^{p}/V^{p}]\big)\cong \widehat{S}_{\xi',\tau'}(V^{\fp}, \co_E)^d.
		\end{equation*}
	\end{corollary}
	Let $\overline{\rho}': \Gal_{F^+}\ra \mathcal{G}_n(\F)$ be the representation isomorphic to  $\overline{\rho}^c$ (where $\overline{\rho}^c(g)=\overline{\rho}(cgc^{-1})$,  $c$ being the complex conjugation) such that (the $\GL_n(\F)$-factor of) $\overline{\rho}'|_{\Gal_{F}}$ satisfies $\overline{\rho}'(g)=\overline{\varepsilon}(g)^{1-n}w(\overline{\rho}(g)^{-1})^Tw,$ where $\overline{\varepsilon}$ denotes the mod $p$ reduction of the $p$-adic cyclotomic character  $\varepsilon$ and $w=\begin{pmatrix}
		0 & \cdots & 0 & 1 \\
		0 & \cdots &1 & 0 \\
		\vdots &\iddots & \vdots &\vdots \\
		1& \cdots & 0 & 0
	\end{pmatrix}$. Let $\fm_{\overline{\rho}}$ (resp. $\fm_{\overline{\rho}'}$) be the maximal ideal of $\mathbb{T}(U^{\fp})$ associated to $\overline{\rho}$ (resp. to $\overline{\rho}'$), which depends only on (the $\GL_ n$-factor of) the restriction $\overline{\rho}|_{\Gal_{F}}: \Gal_{F} \ra \GL_n(\F)$ 
	(resp. $\overline{\rho}'|_{\Gal_{F}}\cong \overline{\rho}|_{\Gal_{F}}^{\vee} \otimes_{\F} \overline{\varepsilon}^{1-n}$). It is straightforward to see that (\ref{dual00}) induces an isomorphism
	\begin{equation*}
		\Hom_{\co_E[U_{\fp}/V_{\fp} \times U^{p}/V^{p}]}\big(S_{\xi,\tau}(V_{\fp}V^{\mathfrak{p}}, \co_E/\varpi_E^k)_{\fm_{\overline{\rho}}}^{\vee}, \co_E/\varpi_E^k[U_{\fp}/V_{\fp} \times U^p/V^p]\big) \xrightarrow{\sim} S_{\xi',\tau'}(V_{\fp}V^{\mathfrak{p}}, \co_E/\varpi_E^k)^{\vee}_{\fm_{\overline{\rho}'}},
	\end{equation*}
	and the statements in lemma hold after taking the respective localizations. As $\overline{\rho}$ is automorphic (in the sense of \cite[Def.~5.3.1]{EG14}), we deduce $\overline{\rho}'$ is also automorphic. In particular, $\overline{\rho}'$ is a suitable globalisation of $\overline{r}^{\vee} \otimes_{\F} \overline{\varepsilon}^{1-n}$.

	We recall the patching argument in \cite{CEGGPS1}. Following \cite{Scholze}, we use ultrafilters. For a finite place $w$ of $F$, we denote by $\overline{\rho}_{w}$ (resp. $\overline{\rho}'_w$) the $\GL_n(\F)$-factor of the restriction $\overline{\rho}|_{\Gal_{F_{w}}}$ (resp. $\overline{\rho}'|_{\Gal_{F_w}}$).
	For $v\in S_p\cup \{v_1\}$, we denote by $R_{\overline{\rho}_{\widetilde{v}}}^{\square}$ the maximal reduced and $p$-torsion free quotient of the universal $\co$-lifting ring of $\overline{\rho}_{\widetilde{v}}$ ($\cong \overline{r}$). 
	For $v\in S_p\setminus \{\fp\}$, we denote by $R_{\overline{\rho}_{\widetilde{v}}}^{\square, \xi,\tau}$   the reduced and $p$-torsion free quotient of $R_{\overline{\rho}_{\widetilde{v}}}^{\square}$ corresponding to potentially crystalline lifts of weight $\xi$ and inertial type $\tau$. We define in a similar way $R_{\overline{\rho}'_{\widetilde{v}}}^{\square}$ and $R_{\overline{\rho}'_{\widetilde{v}}}^{\square, \xi',\tau'}$. We have  isomorphisms $\eta_{\widetilde{v}}: R_{\overline{\rho}_{\widetilde{v}}}^{\square} \xrightarrow{\sim} R_{\overline{\rho}'_{\widetilde{v}}}^{\square}$ and $\eta_{\widetilde{v}}: R_{\overline{\rho}_{\widetilde{v}}}^{\square, \xi,\tau} \xrightarrow{\sim} R_{\overline{\rho}'_{\widetilde{v}}}^{\square, \xi',\tau'}$ (for $v\in S_p \setminus \{\fp\}$) sending $\rho'$ to $[g\mapsto \varepsilon^{1-n}(g) w\rho'(g^{-1})^Tw]$. 
	Let $R^{\loc}:=R_{\overline{\rho}_{\widetilde{\fp}}}^{\square} \widehat{\otimes} \Big(\widehat{\otimes}_{S_p\setminus \{\fp\}}R_{\overline{\rho}_{\widetilde{v}}}^{\square, \xi,\tau}\Big)\widehat{\otimes} R_{\overline{\rho}_{\widetilde{v_1}}}^{\square}$, and $R'^{\loc}:=R_{\overline{\rho}'_{\widetilde{\fp}}}^{\square} \widehat{\otimes} \Big(\widehat{\otimes}_{S_p\setminus \{\fp\}}R_{\overline{\rho}'_{\widetilde{v}}}^{\square, \xi',\tau'}\Big)\widehat{\otimes} R_{\overline{\rho}'_{\widetilde{v_1}}}^{\square}$. Then $(\eta_{\widetilde{v}})_{v\in S_p\cup \{v_1\}}$ induces an isomorphism $\eta: R^{\loc} \xrightarrow{\sim} R'^{\loc}$. Let $\cS$ denote the global deformation problem as in \cite[\S ~2.4]{CEGGPS1}, $R_{\overline{\rho},\cS}$ be the universal deformation ring and $R_{\overline{\rho},\cS}^{\square_T}$ be the universal $T$-framed deformation ring with $T=S_p\cup \{v_1\}$. We define $\cS'$ in a similar way as $\cS$ with $\overline{\rho},$ $\xi$ and $\tau$ replaced by $\overline{\rho}'$, $\xi'$ and $\tau'$ respectively. Let $R_{\overline{\rho}^{\vee},\cS'}$ and $R_{\overline{\rho}^{\vee},\cS'}^{\square_T}$ be the corresponding deformation rings. We have  isomorphisms $\eta: R_{\overline{\rho},\cS}\xrightarrow{\sim} R_{\overline{\rho}',\cS'}$ and $\eta: R_{\overline{\rho},\cS}^{\square_T} \xrightarrow{\sim} R_{\overline{\rho}',\cS'}^{\square_T}$ defined in a similar way as above. 
	
	Let $q\geq [F^+:\Q]\frac{n(n-1)}{2} $ be as in \cite[\S~2.6]{CEGGPS1}. 
	For each $N\in \Z_{\geq 1}$, let $Q_N$, $\widetilde{Q}_N$ be certain sets (of cardinality $q$) of primes of $F^+$ and $F$ satisfying the properties in \textit{loc. cit.} For each $v\in Q_N$ and $i=0,1$, we have compact open subgroups $U_i(Q_N)_v\subset \widetilde{G}(\co_{F_v^+})$ (cf. \textit{loc. cit.}, noting $U_0(Q_N)_v/U_1(Q_N)_v\cong \Z/p^N \Z$). Let $\cS_{Q_N}$ be the 
	deformation problem considered in \textit{loc. cit.}, and $\cS_{Q_N}'$ be  the similar one with $\overline{\rho},$ $\xi$ and $\tau$ replaced by $\overline{\rho}',$ $\xi'$ and $\tau'$ respectively. Let $R_{\overline{\rho},\cS_{Q_N}}$, $R_{\overline{\rho}',\cS'_{Q_N}}$ (resp. $R_{\overline{\rho},\cS_{Q_N}}^{\square_T}$, $R_{\overline{\rho}', \cS'_{Q_N}}^{\square_T}$) be the corresponding universal deformation rings (resp. $T$-framed deformation rings).  Similarly as above, we have isomorphisms $\eta: R_{\overline{\rho},\cS_{Q_N}}\xrightarrow{\sim} R_{\overline{\rho}',\cS'_{Q_N}}$ and $\eta: R_{\overline{\rho},\cS_{Q_N}}^{\square_T}\xrightarrow{\sim} R_{\overline{\rho}', \cS'_{Q_N}}^{\square_T}$.  Let $g:=q-[F^+:\Q]\frac{n(n-1)}{2}$, and $R_{\infty}:=R^{\loc}[[x_1, \cdots, x_g]]$, $R'_{\infty}:=R'^{\loc}[[x_1, \cdots, x_g]]$. The isomorphism $\eta$ extends to an isomorphism $\eta: R_{\infty} \xrightarrow{\sim} R'_{\infty}$ sending $x_i$ to $x_i$. Recall that $R_{\overline{\rho},\cS_{Q_N}}^{\square_T}$ can be topologically generated over $R^{\loc}$ by $g$ elements (see \cite[\S~2.6]{CEGGPS1}). We fix for each $N$  a surjection $R_{\infty} \twoheadrightarrow R_{\overline{\rho}, \cS_{Q_N}}^{\square_T}$, which then induces a surjection $R_{\infty}' \twoheadrightarrow R_{\overline{\rho}', \cS'_{Q_N}}^{\square_T}$ such that the following diagram commutes:
	\begin{equation*}
		\begin{CD}R_{\infty} @>>> R_{\overline{\rho}, \cS_{Q_N}}^{\square_T} \\
			@V \cong V \eta V @V \cong V \eta V \\
			R_{\infty}' @>>> R_{\overline{\rho}', \cS'_{Q_N}}^{\square_T}.
		\end{CD}
	\end{equation*}Let $\Delta_{Q_N}:=(\Z/p^N \Z)^{\oplus q}$. By the definition of $\cS_{Q_N}$ (resp. $\cS'_{Q_N}$), there is a natural morphism $\co_E[\Delta_{Q_N}]\ra R_{\overline{\rho},\cS_{Q_N}}$ (resp. $\co_E[\Delta_{Q_N}] \ra R_{\overline{\rho}', \cS'_{Q_N}}$). Moreover, letting $\eta: \co_E[\Delta_{Q_N}]\xrightarrow{\sim} \co_E[\Delta_{Q_N}]$ be the isomorphism sending $[a]$ to $[-a]$ for $a\in \Delta_{Q_N}$, the following diagram commutes
	\begin{equation*}
		\begin{CD}
			\co_E[\Delta_{Q_N}] @>>> R_{\overline{\rho},\cS_{Q_N}} \\
			@V \eta VV @V \eta VV \\
			\co_E[\Delta_{Q_N}]@>>> R_{\overline{\rho}', \cS'_{Q_N}}.
		\end{CD}
	\end{equation*}
	For each  $N$, we choose a surjection $\co_E[[y_1, \cdots, y_q]]\twoheadrightarrow \co_E[\Delta_{Q_N}]$ which  induces a morphism (which is the identity on the variables $z_i$) 
	\begin{equation}\label{Sinfty1}S_{\infty}:=\co_E[[z_1, \cdots, z_{n^2 \# T}, y_1, \cdots, y_q]]\lra R_{\overline{\rho},\cS_{Q_N}}^{\square_T}\cong R_{\overline{\rho},\cS_{Q_N}}\widehat{\otimes}_{\co_E} \co_E[[z_1, \cdots, z_{n^2 \# T}]].
	\end{equation}
	Let $\eta: S_{\infty} \xrightarrow{\sim} S_{\infty}$ be the involution sending $z_i$, $y_i$ to $-z_i$, $-y_i$ respectively. Remark that its restriction on $\co_E[[y_1, \cdots, y_q]]$ extends $\eta$ on $\co_E[\Delta_{Q_N}]$.  There is a unique morphism $S_{\infty} \ra R_{\overline{\rho}', \cS_{Q_N}'}^{\square_T}$ such that the diagram commutes:
	\begin{equation*}
		\begin{CD}S_{\infty} @>>> R_{\overline{\rho},\cS_{Q_N}}^{\square_T} \\
			@V \eta VV @V \eta VV \\
			S_{\infty} @>>> R_{\overline{\rho}', \cS_{Q_N}}^{\square_T}.
		\end{CD}
	\end{equation*}

	For each $N$, and $i=0,1$, let $U_i(Q_N)^{\fp}$ be the prime to $\fp$-part of $U_i(Q_N)_0$ of \cite{CEGGPS1} (in particular,  $U_0(Q_N)^{\fp}/U_1(Q_N)^{\fp}\cong \Delta_{Q_N}$). Let $\mathbb{T}^{S_p\cup Q_N}$ be the $\co_E$-polynomial Hecke algebra  as in \cite[\S~2.3]{CEGGPS1} (denoted by $\mathbb{T}^{S_p \cup Q_N, \mathrm{univ}}$ there). Denote by $\fm_{Q_N}$ (resp, $\fm'_{Q_N}$) the maximal ideal of $\mathbb{T}^{S_p \cup Q_N}$ associated to $\overline{\rho}$ (resp. to $\overline{\rho}^{\vee}$). For a compact open normal subgroup $U_{\fp}$ of $G_0$, let $M_N(U_{\fp},k):=\mathrm{pr}(S_{\xi,\tau}(U_1(Q_N)^{\fp} U_{\fp}, \co_E/\varpi_E^k)_{\fm_{Q_N}})^{\vee}$ and $M_N'(U_{\fp},k):=\mathrm{pr}(S_{\xi',\tau'}(U_1(Q_N)^{\fp}U_{\fp}, \co_E/\varpi_E^k)_{\fm_{Q_N}'})^{\vee}$, where  $\mathrm{pr}$ is the operator defined  in \cite[\S~2.6]{CEGGPS1}. Put $M_N(U_{\fp}):=\varprojlim_k M_N(U_{\fp},k)$, and $M_N'(U_{\fp}):=\varprojlim_k M_N'(U_{\fp},k)$, which are finite $\co_E$-modules equipped with a natural action of $R_{\overline{\rho},\cS_{Q_N}}$ and $R_{\overline{\rho}^{\vee}, \cS_{Q_N}'}$ respectively. As $\mathrm{pr}$ is defined using Hecke operators at places in $Q_N$, we deduce from (\ref{dual00}) a natural  isomorphism
	\begin{equation}\label{dual01}
		\Hom_{\co_E[G_0/U_{\fp} \times \Delta_{Q_N}]}\big(M_N(U_{\fp},k), \co_E/\varpi_E^k[G_0/U_{\fp} \times \Delta_{Q_N}]\big) \xrightarrow{\sim} M_N'(U_{\fp},k),
	\end{equation}
	which is $\co_E[G_0/U_{\fp} \times \Delta_{Q_N}] \times R_{\overline{\rho},\cS_{Q_N}}$-equivariant, where $R_{\overline{\rho},\cS_{Q_N}}$ acts on the right hand side via the action induced by $\eta: R_{\overline{\rho},\cS_{Q_N}}\xrightarrow{\sim} R_{\overline{\rho}', \cS'_{Q_N}}$.
	
	Now for each $N$, let $M_N^{\square}(U_{\fp}):=R_{\overline{\rho}, \cS_{Q_N}}^{\square_T} \otimes_{R_{\overline{\rho}, \cS_{Q_N}}} M_N(U_{\fp})$ 
	(resp. $M_N'^{\square}(U_{\fp}):= R_{\overline{\rho}^{\vee}, \cS_{Q_N}'}^{\square_T} \otimes_{R_{\overline{\rho}^{\vee}, \cS_{Q_N}'}} M_N'(U_{\fp})$) which is hence equipped with an $S_{\infty}$-action via $S_{\infty} \ra R_{\overline{\rho},\cS_{Q_N}}^{\square_T}$ (resp. $S_{\infty} \ra R_{\overline{\rho}^{\vee},\cS_{Q_N}'}^{\square_T}$). For an open ideal $I\subset S_{\infty}$, let $M_{N,I}(U_{\fp}):=M_N^{\square}(U_{\fp}) \otimes_{S_{\infty}} S_{\infty}/I$. We fix a non-principal ultrafilter $\mathfrak{F}$ of $\Z_{\geq 1}$, and let $\prod_{N\geq 1} S_{\infty}/I \ra S_{\infty}/I$ be the localization at the maximal ideal corresponding to $\mathfrak{F}$. Put $M_{\infty,I}(U_{\fp}):=(\prod_{N\geq 1} M_{N,I}(U_{\fp})) \otimes_{\prod_{N\geq 1} S_{\infty}/I} S_{\infty}/I$, $M_{\infty,I}:=\varprojlim_{U_{\fp}} M_{\infty,I}(U_{\fp})$, and $M_{\infty}:=\varprojlim_I M_{\infty,I}$ \big(resp. $M_{\infty,I}'(U_{\fp}):=(\prod_{N\geq 1} M'_{N,I}(U_{\fp})) \otimes_{\prod_{N\geq 1} S_{\infty}/I} S_{\infty}/I$, $M_{\infty,I}':=\varprojlim_{U_{\fp}} M_{\infty,I}'(U_{\fp})$, and $M_{\infty}':=\varprojlim_I M_{\infty,I}'$\big). All these are equipped with a natural $S_{\infty}$-action. Recall $M_{\infty}$ and $M'_{\infty}$ are equipped with a natural $S_{\infty}$-linear action of $\GL_n(K)$, and the both are finite type projective $S_{\infty}[[G_0]]$-modules. Moreover, $M_{\infty}$ (resp. $M'_{\infty}$) is also equipped with a natural action of $R_{\infty}$ (resp. $R'_{\infty}$). The patching procedure of $M_{\infty}$ produces a morphism $S_{\infty} \ra R_{\infty}$ satisfying that the $R_{\infty}$-action on $M_{\infty}$ is $S_{\infty}$-linear (e.g. see the arguments in \cite[\S~4.4]{Gee-Newton}). Let $S_{\infty} \ra R_{\infty}'$ be the morphism such that the diagram commutes:
	\begin{equation*}
		\begin{CD}
			S_{\infty} @>>> R_{\infty} \\
			@V \eta VV @V \eta VV \\
			S_{\infty} @>>> R_{\infty}'.\end{CD}
	\end{equation*}
	By looking at the action on each $M_{\infty,I}'(U_{\fp})$ (or using the following proposition), one can show that the $R_{\infty}'$-action on $M_{\infty}'$ is $S_{\infty}$-linear. 
	\begin{proposition}\label{prop::duality-3/6}
		Let $U_{\fp}$ be a compact open normal subgroup of $G_0$, we have an $R_{\infty} \times G_0/U_{\fp}$-equivariant $S_{\infty}$-linear isomorphism
		\begin{equation*}
			\Hom_{S_{\infty}/I[G_0/U_{\fp}]}\big(M_{\infty,I}(U_{\fp}), S_{\infty}/I[G_0/U_{\fp}]\big) \cong M'_{\infty,I}(U_{\fp}),
		\end{equation*}
		where $R_{\infty}$ acts on the right hand side via the induced action by $\eta: R_{\infty} \ra R_{\infty}'$. 
	\end{proposition}
	\begin{proof}
		Let $r$ be sufficiently large, such that for any $N \geq r$, $I$ contains the kernel of $S_{\infty}\twoheadrightarrow \co_E[\Delta_{Q_N}]$.  We have $M_{\infty,I}^*(U_{\fp})\cong (\prod_{N\geq r} M^*_{N,I}(U_{\fp})) \otimes_{\prod_{N\geq r} S_{\infty}/I} S_{\infty}/I$, $*\in \{\emptyset, '\}$. By (\ref{dual01}), for all $N\geq r$, we have a natural $R_{\infty} \times G_0/U_{\fp}$ equivariant isomorphism 
		\begin{equation}\label{isofini}
			\Hom_{S_{\infty}/I[G_0/U_{\fp}]}\big(M_{N,I}(U_{\fp}), S_{\infty}/I[G_0/U_{\fp}]\big) \cong M'_{N,I}(U_{\fp}), 
		\end{equation}
		where $R_{\infty}$ acts on the right hand side via $\eta: R_{\infty} \ra R'_{\infty} \ra R_{\overline{\rho}', \cS_{Q_N}}^{\square_T}$.
		Let $R:=\prod_{N\geq r} S_{\infty}/I$, and $R_{\mathfrak{F}}\cong S_{\infty}/I$ be the localization of $R$ at the maximal ideal $\fm_{\mathfrak{F}}$ associated to $\mathfrak{F}$. The isomorphism (\ref{isofini}) induces
		\begin{equation*}
			\Hom_{R}\Big(\prod_{N\geq r} M_{N,I}(U_{\fp}), \prod_{N\geq r} S_{\infty}/I[G_0/U_{\fp}]\Big)\cong \prod_{N\geq r}\Hom_{S_{\infty}/I}\Big(M_{N,I}(U_{\fp}), S_{\infty}/I[G_0/U_{\fp}]\Big)\cong \prod_{N\geq r} M_{N,I}'(U_{\fp}).
		\end{equation*}
		It suffices to show the following map sending $(f_N)$ to  $(x_N) \mapsto (f_N(x_N))$ is an  isomorphism:
		\begin{multline*}
			\Hom_{R}\Big(\prod_{N\geq r} M_{N,I}(U_{\fp}), \prod_{N\geq r} S_{\infty}/I[G_0/U_{\fp}]\Big) \otimes_R R_{\mathfrak{F}}  \\ \xrightarrow{\sim} \Hom_{R_{\mathfrak{F}}}\Big((\prod_{N\geq r} M_{N,I}(U_{\fp})) \otimes_R R_{\mathfrak{F}}, (\prod_{N\geq r} S_{\infty}/I[G_0/U_{\fp}]) \otimes_R R_{\mathfrak{F}}\Big).
		\end{multline*}
			It is straightforward to see it is injective.  For any $f$ in the target, and $(x_N)\in \prod_{N\geq r} M_{N,I}(U_{\fp})$, as $M_{N,I}(U_{\fp})$ has uniformly bounded cardinality, there exists $J \in \mathfrak{F}$ such that $x_N$, $N\in J$ are all equal, denoted by $x_{\infty}$. One check the map $(x_N)\mapsto \prod_{N\in J} f(x_{\infty}) \times 0$ is a preimage of $f$. The claim follows.
		\end{proof}
		\begin{corollary}\label{cor::self-dual-M_infty}
			There is an $R_{\infty} \times \GL_n(K)$-equivariant $S_{\infty}$-linear isomorphism
			\begin{equation*}
				\Hom_{S_{\infty}[[G_0]]}\big(M_{\infty}, S_{\infty}[[G_0]]\big)\cong M_{\infty}',
			\end{equation*}
			where $R_{\infty}$ acts on $M'_{\infty}$ via $\eta: R_{\infty} \ra R_{\infty}'$.
		\end{corollary}
		
		\begin{proof}
			This follows from Proposition \ref{prop::duality-3/6}.
		\end{proof}
		
		\begin{theorem}\label{thm::duality-patching}
			Assume that
			\begin{itemize}
				\item $R_{\infty}$ is Gorenstein,
				\item $M_{\infty}$ and $M'_{\infty}$ are flat over $R_{\infty}.$
			\end{itemize}
			Then for any surjective homomorphism of $\co_E$-algebras $x: R_{\infty}\to \co_E,$ $(M_{\infty}\otimes_{R_{\infty},x} \co_E)[1/\varpi_E] $ and $(M'_{\infty}\otimes_{R_{\infty}, x} \co_E)[1/\varpi_E]$ are Cohen-Macaulay modules  of grade $\frac{n(n+1)}{2} d_F$ over $E[[G_0]]$, and there is a $\GL_n(K)$-equivariant isomorphism of $E[[ G_0]]$-modules
			\[
			\EE^{\frac{n(n+1)}{2} d_F}( (M_{\infty}\otimes_{R_{\infty},x} \co_E)[1/\varpi_E] ) \simeq (M'_{\infty}\otimes_{R_{\infty}, x} \co_E)[1/\varpi_E] ,
			\]
			where $d_K: = [K:\Q_p].$
		\end{theorem}
		
		\begin{proof}
			By taking $A = R_{\infty},$ $M = M'_{\infty}, $ $N= M_{\infty}$ and $G = \Z_p^q \times \GL_n(\co_K)$ in Proposition \ref{thm-self-duality} and using $\dim R_{\infty}=1+q+n^2 \# T+\frac{n(n+1)}{2} d_F$, this follows from  Corollary \ref{cor--self-duality}, Corollary \ref{cor::self-dual-M_infty} and Lemma \ref{Lextravaria}.
		\end{proof}
		
		\subsection{Applications to $\GL_2$}
		
		Assume $n=2$ from now on. Let $U=U^{\fp} U_{\fp}=\prod U_v$ be a sufficiently small compact open subgroup of  $\widetilde{G}(\A_{F^+}^{\infty})$ such that $\det(u)\in U$ for all $u\in U$. For $f\in S_{\xi,\tau}(U, \co_E/\varpi_E^k)$, consider the composition
		\begin{equation}\label{Edetertwis}
			\widetilde{G}(F^+) \backslash \widetilde{G}(\A_{F^+}^{\infty})\xrightarrow{[g\mapsto \det(g)^{-1} g]}\widetilde{G}(F^+) \backslash \widetilde{G}(\A_{F^+}^{\infty}) \xrightarrow{f} W_{\xi,\tau}/\varpi_E^k.
		\end{equation}
		Denote by $W_{\xi,\tau}^{\tw}$ the representation over $\co_E$ whose underlying $\co_E$-module is $W_{\xi,\tau}$, with the action given by $g\in \prod_{v\in S_p \setminus \{\fp\}} \widetilde{G}(\co_{F^+_v})$ acting via $\det(g)^{-1} g$. It is straightforward to see (\ref{Edetertwis})  lies in $S_{\xi,\tau}^{\tw}(U,\co_E/\varpi_E^k)$, the module defined as in (\ref{ESxitau}) with $W_{\xi,\tau}$ replaced by $W_{\xi,\tau}^{\tw}$. 
		We obtain hence a bijection (letting $U_{\fp}$-vary)
		\begin{equation}\label{isotwist}
			S_{\xi,\tau}(U^{\fp}, \co_E/\varpi_E^k) \xrightarrow{\sim} S_{\xi,\tau}^{\tw}(U^{\fp}, \co_E/\varpi_E^k).
		\end{equation}
		For  $N\in \Z_{\geq 1}$, let $\rho^{\rm{univ}}$ be the universal representation of $\Gal_{F^+}$ over $R_{\overline{\rho}, \cS_{Q_N}}$, and  $\chi_N$  denote the  character 
		\[\GL_2(K)\cong \GL_2(F_{\widetilde{\fp}}) \xrightarrow{\det} F_{\widetilde{\fp}}^{\times} \ra \Gal_{F_{\widetilde{\fp}}}^{\mathrm{ab}} \ra \Gal_{F}^{\mathrm{ab}} \xrightarrow{\varepsilon\det(\rho^{\rm{univ}})} R_{\overline{\rho},\cS_{Q_N}}^{\times}.\] 
		Note the morphism $\eta^{-1}$ coincides with $R_{\overline{\rho}',\cS_{Q_N}'} \ra R_{\overline{\rho},\cS_{Q_N}} $, $\rho \mapsto (\varepsilon\det(\rho))^{-1} \rho $ (the same holds also for the framed $\eta$).
		By directly checking the $\mathbb{T}^{S_p\cup Q_N}\times \GL_2(F_{\widetilde{\fp}})$-action, we obtain the following lemma.
		\begin{lemma}\label{lisotwist}
			For $N\in \Z_{\geq 1}$, the isomorphism (\ref{isotwist}) (applied to $U^{\fp}=U_1(Q_N)^{\fp}$) induces an isomorphism
			\begin{equation*}
				S_{\xi,\tau}(U^{\fp}, \co_E/\varpi_E^k)_{\fm_{Q_N}} \xrightarrow{\sim} S_{\xi,\tau}^{\tw}(U^{\fp}, \co_E/\varpi_E^k)_{\fm_{Q_N}'}
			\end{equation*}
			which is moreover $\GL_2(K) \times R_{\overline{\rho},\cS_{Q_N}}$-equivariant if  $R_{\overline{\rho},\cS_{Q_N}}$ acts on the right hand side via $\eta: R_{\overline{\rho},\cS_{Q_N}} \xrightarrow{\sim} R_{\overline{\rho}',\cS'_{Q_N}}$, and the $\GL_2(K)$-action on the right hand side is twisted by $\chi_N$.
		\end{lemma}
		We construct $M_{\infty}^{\tw}$ in the same way as $M_{\infty}$ with $S_{\xi,\tau}$ replaced by $S_{\xi,\tau}^{\tw}$ (and $R_{\infty}$ replaced by $R_{\infty}'$). 
		Let $\chi_{\fp}: \GL_2(F_{\fp}) \ra R_{\infty}^{\times}$ be the character given by the composition
		\begin{equation*}
			\GL_2(K) \xrightarrow{\det} K^{\times} \ra \Gal_{K}^{\mathrm{ab}} \xrightarrow{\varepsilon \wedge^2 \rho_{\widetilde{\fp}}^{\rm{univ}}} (R_{\overline{\rho}_{\widetilde{\fp}}}^{\square})^{\times} \ra R_{\infty}^{\times}.
		\end{equation*}
		Note that the composition of $\chi_{\fp}$ with $R_{\infty} \ra R_{\overline{\rho}, \cS_{Q_N}}^{\square_T}$ coincides with $\chi_N$. By Lemma \ref{lisotwist}, we deduce the following proposition.
		
		\begin{proposition}
			Keep the above situation. There is a  $\GL_2(K) \times R_{\infty}$-equivariant isomorphism 
			\begin{equation*}
				M_{\infty}\xrightarrow{\sim} M_{\infty}^{\tw}
			\end{equation*}
			where $R_{\infty}$ acts on $M_{\infty}^{\tw}$ via $\eta$, and the $\GL_2(K)$-action on $M_{\infty}^{\tw}$ is twisted by the character $\chi_{\fp}^{-1}$. In particular, for a surjective homomorphism of $\co_E$-algebras $x: R_{\infty} \ra \co_E$, there is a natural $\GL_2(K)$-equivariant isomorphism 
			\begin{equation*}
				M_{\infty} \otimes_{R_{\infty}, x} \co_E \cong (M_{\infty}^{\tw} \otimes_{R_{\infty},x} \co_E) \otimes_{\co_E} (x\circ \chi_{\fp})^{-1}.
			\end{equation*}
		\end{proposition}
		\begin{remark}Let $x$ be as in the above proposition, and $\rho_x$ be the associated $\Gal_F$-representation. Then $x\circ \chi_{\fp}=\varepsilon(\wedge^2 \rho_x)$. By an easier variant of \cite{DPS} (using the density of locally algebraic vectors in $M_{\infty}^d$), $x\circ \chi_{\fp}$ is in fact the central character of $(M_{\infty} \otimes_{R_{\infty}, x} \co_E)^d$. 
		\end{remark}
		We compare $M_{\infty}^{\tw}$ with $M_{\infty}'$. Recall $W_{\xi',\tau'} : = W_{\xi,\tau}^d.$
		As representation of $ \prod_{v\in S_p \setminus \{\fp\}} \widetilde{G}(\co_{F^+_v})$ over $E$, we have $W_{\xi,\tau}^{\tw}\otimes_{\co_E} E \cong W_{\xi',\tau'} \otimes_{\co_E} E$. Multiplying  $W_{\xi',\tau'}$ be a certain $p$-th power, we can and do assume there exists $r\in \Z_{\geq 1}$ such that $p^r W_{\xi',\tau'} \subset W_{\xi, \tau}^{\tw} \subset W_{\xi',\tau'}$. For $n$ sufficiently large, there exist $ \prod_{v\in S_p \setminus \{\fp\}} \widetilde{G}(\co_{F^+_v})$-representations $M_1$, $M_2$ (independent of $n$), finite over $\co_E$, such that we have an exact sequence
		\begin{equation*}
			0 \ra M_1 \ra W_{\xi,\tau}^{\tw}/\varpi_E^n \ra W_{\xi',\tau'}/\varpi_E^n \ra M_2 \ra 0.
		\end{equation*}
		Applying the patching construction, we deduce an injection
		\begin{equation*}
			M_{\infty}^{\tw}  \hookrightarrow M_{\infty}'
		\end{equation*}
		which is $R_{\infty}' \times \GL_2(K)$-equivariant and is bijective after inverting $p$ (using $M_1$ and $M_2$ are annihilated by $p^r$ for a constant $r$).  We get hence:
		\begin{corollary}
			Keep the above situation. For a surjective homomorphism of $\co_E$-algebras $x: R_{\infty} \ra \co_E$, there is a natural $\GL_2(K)$-equivariant isomorphism 
			\begin{equation*}
				M_{\infty} \otimes_{R_{\infty}, x} \co_E [1/\varpi_E]\cong (M_{\infty}' \otimes_{R_{\infty},x} \co_E[1/\varpi_E]) \otimes_{\co_E} (x\circ \chi_{\fp})^{-1}.
			\end{equation*}
		\end{corollary}
			%
			
			%

			\begin{corollary} \label{cor::3/6}
				Assume $K$ is unramified over $\Q_p$ of degree $d_K.$ Assume the assumptions in \cite[Thm. 1.1]{BHHMS1} if $\brho_{\frak{p}}$ is semisimple; assume the assumptions in \cite[Thm 1.1]{HW-CJM} or \cite[Thm. 1.1]{Yitong} if $\brho_{\frak{p}}$ is non-semisimple. Then for any surjective homomorphism of $\co_E$-algebras $x: R_{\infty} \ra \co_E$, $(M_{\infty}\otimes_{R_{\infty},x} \co_E)[1/\varpi_E]$ is Cohen-Macaulay of grade $3d_K$ over $E[[G_0]]$, and  there is a $\GL_2(K)$-equivariant isomorphism of $E[[G_0]]$-modules:
				\[
				\EE^{3 d_K}( (M_{\infty}\otimes_{R_{\infty},x} \co_E)[1/\varpi_E] ) \simeq (M_{\infty}\otimes_{R_{\infty}, x} \co_E)[1/\varpi_E] \otimes_E (x\circ \chi_{\fp}).
				\]		
			\end{corollary}
			\begin{proof}
				Under the assumptions of the corollary, $(M_{\infty}/\frak{m}_{R_\infty})^{\vee}$ and $(M'_{\infty}/\frak{m}_{R_\infty})^{\vee}$ both have Gelfand-Kirillov dimension $d_K$ by \cite{BHHMS1} \cite{HW-CJM} \cite{Yitong}, where $\frak{m}_{R_\infty}$ is the maximal ideal of $R_{\infty}$. Although \cite{BHHMS1} \cite{HW-CJM} \cite{Yitong} work with patched modules associated to quaternion algebras over totally real field, the same proof applies to our $M_{\infty}$ and $M'_{\infty}.$ Then by \cite[Prop. A.30]{Gee-Newton} $R_{\infty}$ is a regular local ring, and $M_{\infty}$ and $M'_{\infty}$ are flat over $R_{\infty}$. We conclude by Theorem \ref{thm::duality-patching}.
			\end{proof}
			
			\begin{remark}
				Under suitable assumptions, similar discussion may apply to some $p$-adic Banach space representations of $D^{\times},$ where $D$ is the nonsplit quaternion algebra over $K.$ In particular, an analogue statement of Corollary \ref{cor::3/6} holds for patched modules described in \cite[\S 5]{HW-JL1}, under the assumption of \cite[Thm. 1.1]{HW-JL1}.
			\end{remark}
			Keep the assumption in Corollary \ref{cor::3/6}, and let $\widehat{\pi}(x):=\Hom_{\co_E}^{\cont}(M_{\infty}\otimes_{R_{\infty},x} \co_E, E)$. By Corollary \ref{cor::3/6} and \cite[Prop.~7.2]{ScSt} (see also the proof of Lem.~2.2 of \textit{loc. cit.}), we have
			\begin{corollary}\label{C-CMdual}
				The dual $(\widehat{\pi}(x)^{\rm{la}})^*$ of the locally analytic subrepresentation of $\widehat{\pi}(x)$ is a Cohen-Macaulay module of grade $3d_K$ over the distribution algebra $\mathcal{D}(G_0,E)$, and $\EE^{3 d_K}(\pi(x)^*)\cong \pi(x)^*\otimes_E (x\circ \chi_{\fp})$.
			\end{corollary}

\end{document}